\documentclass[11pt,letterpaper]{amsart}
\usepackage[T1]{fontenc}
\usepackage[latin9]{inputenc}
\usepackage{geometry}
\usepackage{wrapfig}
\usepackage{multicol}
\usepackage{enumerate}   
\usepackage{graphicx}
\usepackage{soul}
\usepackage{xcolor}
\usepackage{amssymb}

\newtheorem{theorem}{Theorem}[section]
\newtheorem{lemma}[theorem]{Lemma}
\newtheorem{proposition}[theorem]{Proposition}
\newtheorem{corollary}[theorem]{Corollary}
{ \theoremstyle{definition}
\newtheorem{definition}[theorem]{Definition}}
{ \theoremstyle{remark}
\newtheorem{remark}[theorem]{Remark}}
\usepackage{placeins}
\usepackage{cite}
\usepackage{caption}
\usepackage{enumerate}
\usepackage{afterpage}
\usepackage{enumitem}
\usepackage{bmpsize}
\usepackage{hyperref}
\usepackage{tabu}
\usepackage{enumitem}   
\numberwithin{equation}{section}
\usepackage{stmaryrd}
\usepackage{tikz}
\usetikzlibrary{math}
\usetikzlibrary{matrix,graphs,arrows,positioning,calc,decorations.markings,shapes.symbols}

\newcommand{\im}{\mathsf{i}}
\newcommand{\Real}{\mathsf{Re}}
\newcommand{\Imag}{\mathsf{Im}}
\newcommand{\Arg}{\mathsf{Arg}}
\usepackage{ifthen}

\title{Airy wanderer line ensembles}
\date{\today}
\author{Evgeni Dimitrov}

\begin{document}

\maketitle 

\begin{abstract}
In (J. Stat. Phys. 132, 275-290, 2008) Borodin and P{\'e}ch{\'e} introduced a generalization of the extended Airy kernel based on two infinite sets of parameters. For an arbitrary choice of parameters we construct determinantal point processes on $\mathbb{R}^2$ for these generalized kernels. In addition, for a subset of the parameter space we show that the point processes can be lifted to line ensembles on $\mathbb{R}$, which satisfy the Brownian Gibbs property. Our ensembles generalize the wanderer line ensembles introduced by Corwin and Hammond in (Invent. Math. 195, 441-508, 2014).
\end{abstract}

\tableofcontents

%%%%%%%%%%%%%%%%%%%%%%%%%%%%%%%%%%%%%%%%%%%%%%%%%%%%%%%%%%%%%%%%%%%%%
%
%    Section 1
%
%%%%%%%%%%%%%%%%%%%%%%%%%%%%%%%%%%%%%%%%%%%%%%%%%%%%%%%%%%%%%%%%%%%%%
\section{Introduction and main results}\label{Section1}

%%%%%%%%%%%%%%%%%%%%%%%%%%%%%%%%%%%%%%%%%%%%%%%%%%%%%%%%%%%%%%%%%%%%%
%
%    Section 1.1
%
%%%%%%%%%%%%%%%%%%%%%%%%%%%%%%%%%%%%%%%%%%%%%%%%%%%%%%%%%%%%%%%%%%%%%
\subsection{Introduction}\label{Section1.1} The {\em Airy line ensemble} $\mathcal{L}^{\mathrm{Airy}}$ is a sequence of real-valued random continuous functions $\{\mathcal{L}^{\mathrm{Airy}}_{i}\}_{i \geq 1}$, which are all defined on $\mathbb{R}$ and are strictly ordered in the sense that $\mathcal{L}^{\mathrm{Airy}}_{i}(t) > \mathcal{L}^{\mathrm{Airy}}_{i+1}(t)$ for all $i \geq 1$ and $t \in \mathbb{R}$. The Airy line ensemble arises as the edge scaling limit of Wigner matrices \cite{Sod15}, lozenge tilings \cite{AH21}, avoiding Brownian bridges (also known as {\em Brownian watermelons}) \cite{CorHamA}, as well as various integrable models of non-intersecting random walkers and last passage percolation \cite{DNV19}. The list of models becomes quite vast if one further includes those that converge to the various projections of $\mathcal{L}^{\mathrm{Airy}}$ such as $(\mathcal{L}^{\mathrm{Airy}}_{1}(t): t \in \mathbb{R})$ (known as the {\em Airy process}), $(\mathcal{L}^{\mathrm{Airy}}_{i}(t_0): i \geq 1)$ for a fixed $t_0 \in \mathbb{R}$ (known as the {\em Airy point process}) and $\mathcal{L}^{\mathrm{Airy}}_{1}(t_0)$ for a fixed $t_0$ (known as the {\em Tracy-Widom distribution}). Due to its appearance as a universal scaling limit, and its role in the construction of the {\em Airy sheet} in \cite{DOV22}, the Airy line ensemble has become a central object in the {\em Kardar-Parisi-Zhang (KPZ)} universality class \cite{CU2}.

One of the salient features of the Airy line ensemble is that it has the structure of a {\em determinantal point process}, see Section \ref{Section2} for more background on these objects. More specifically, if one fixes a finite set $\mathcal{A} = \{s_1, \dots, s_m\} \subset \mathbb{R}$ with $s_1 < \cdots < s_m$, then the random measure on $\mathbb{R}^2$, defined by
\begin{equation}\label{RMS1}
M(A) = \sum_{i \geq 1} \sum_{j = 1}^m {\bf 1} \left\{\left(s_j, \mathcal{L}^{\mathrm{Airy}}_{i}(s_j) \right) \in A \right\},
\end{equation}
is a determinantal point process on $\mathbb{R}^2$ with reference measure $\mu_{\mathcal{A}} \times \lambda$, where $\mu_{\mathcal{A}}$ is the counting measure on $\mathcal{A}$, and $\lambda$ is the usual Lebesgue measure on $\mathbb{R}$, and whose correlation kernel is given by the {\em extended Airy kernel}, defined for $x_1, x_2 \in \mathbb{R}$ and $t_1, t_2 \in \mathcal{A}$ by 
\begin{equation}\label{S1EAK}
\begin{split}
K^{\mathrm{Airy}}(t_1,x_1; t_2,x_2) = & -  \frac{{\bf 1}\{ t_2 > t_1\} }{\sqrt{4\pi (t_2 - t_1)}} \cdot e^{ - \frac{(x_2 - x_1)^2}{4(t_2 - t_1)} - \frac{(t_2 - t_1)(x_2 + x_1)}{2} + \frac{(t_2 - t_1)^3}{12} } \\
& + \frac{1}{(2\pi \im)^2} \int_{\Gamma_{\alpha}^+} d z \int_{\Gamma_{\beta}^-} dw \frac{e^{z^3/3 -x_1z - w^3/3 + x_2w}}{z + t_1 - w - t_2}.
\end{split}
\end{equation}
In (\ref{S1EAK}) the contours $\Gamma_{\alpha}^+$ and $\Gamma_{\beta}^-$ consist of two rays starting from $\alpha$ and $\beta$ and departing at angles $\pi/4$, where $\alpha, \beta \in \mathbb{R}$ are arbitrary subject to $\alpha + t_1 > \beta + t_2$. The contours are formally introduced in Definition \ref{DefContInf} and we refer to Figure \ref{S11} for their depiction. We mention that there are various formulas for the extended Airy kernel, and the one in (\ref{S1EAK}) comes from \cite[Proposition 4.7 and (11)]{BK08} under the change of variables $u \rightarrow z + t_1$ and $w \rightarrow w + t_2$. We also mention that usually one takes the contours to depart at angles $\pi/3$ from the real line; however, one can deform those to $\Gamma_{\alpha}^+$ and $\Gamma_{\beta}^-$ without affecting the value of the integral by Cauchy's theorem, with the decay required to perform the deformation near infinity coming from the cubic terms in the exponential. 

The introduction of the extended Airy kernel is frequently attributed to \cite{Spohn} where it arises in the context of the polynuclear growth model, although it appeared earlier in \cite{Mac94} and \cite{FNH99}. In \cite{Spohn} and in more detail in \cite{J03} it was shown that the Airy process $(\mathcal{L}^{\mathrm{Airy}}_{1}(t): t \in \mathbb{R})$ has a continuous version by estimating certain fourth moments and working directly with the fourth order correlation functions. The fact that the Airy line ensemble $\mathcal{L}^{\mathrm{Airy}}$ has a continuous version was proved in \cite{CorHamA} by utilizing the locally avoiding Brownian bridge structure of the Brownian watermelons whose edge limit gives rise to $\mathcal{L}^{\mathrm{Airy}}$. This local description of the Brownian watermelons is now referred to as the {\em Brownian Gibbs property}, and it is enjoyed by the {\em parabolic Airy line ensemble} $\mathcal{L}^{\mathrm{pAiry}}$ whose relationship to $\mathcal{L}^{\mathrm{Airy}}$ is explicitly given by
\begin{equation}\label{PALE}
\mathcal{L}^{\mathrm{pAiry}}_i(t) = 2^{-1/2} \cdot \mathcal{L}^{\mathrm{Airy}}(t) - 2^{-1/2} \cdot t^2 \mbox{ for $i \geq 1$ and $t \in \mathbb{R}$}.
\end{equation}

\begin{figure}[h]
    \centering
     \begin{tikzpicture}[scale=2.7]

        \def\tra{3} 
        % Picture on the left
        % Coordinate System
        \draw[->, thick, gray] (-1.4,0)--(1.4,0) node[right]{$\Real$};
        \draw[->, thick, gray] (0,-1.2)--(0,1.2) node[above]{$\Imag$};

        % The contour t_2 + Gamma_{beta}^-
        \draw[-,thick][black] (-0.2,0) -- (-0.6,-0.4);
        \draw[->,thick][black]  (-1.2,-1) -- (-0.6,-0.4);
        \draw[black, fill = black] (-0.2,0) circle (0.02);
        \draw (-0.2,-0.275) node{$\beta + t_2$};
        \draw[->,very thin][black] (-0.2,-0.2) -- (-0.2, -0.05);
        \draw[->,thick][black] (-0.2,0) -- (-0.6,0.4);
        \draw[-,thick][black]  (-1.2,1) -- (-0.6,0.4);

        % The contour Gamma_{alpha}^+
        \draw[-,thick][black] (0.25, 0) -- (0.75,-0.5);
        \draw[->,thick][black] (1.25, -1) -- (0.75, -0.5);
        \draw[black, fill = black] (0.25,0) circle (0.02);
        \draw (0.25,-0.275) node{$\alpha + t_1$};
        \draw[->,very thin][black] (0.25,-0.2) -- (0.25, -0.05);
        \draw[->,thick][black] (0.25,0) -- (0.75,0.5);
        \draw[-,thick][black]  (0.75,0.5) -- (1.25,1);

        % Poles to be avoided
        \draw[black, fill = black] (0.50,0) circle (0.02);
        \draw[black, fill = black] (0.70,0) circle (0.02);
        \draw[black, fill = black] (0.90,0) circle (0.02);
        \draw[black, fill = black] (1.1,0) circle (0.02);
        \draw (0.50,-0.1) node{$x_1$};
        \draw (0.70,-0.1) node{$x_2$};
        \draw (0.92,-0.1) node{$\cdots$};
        \draw (1.2,-0.12) node{$x_{J_1}$};
        \draw (1.35,-0.3) node{(or $1/a_i^+$'s and $a_i^-$'s)};
        
        \draw[black, fill = black] (-.50,0) circle (0.02);
        \draw[black, fill = black] (-0.70,0) circle (0.02);
        \draw[black, fill = black] (-0.90,0) circle (0.02);
        \draw[black, fill = black] (-1.1,0) circle (0.02);
        \draw (-0.50,-0.12) node{$y_{J_2}$};
        \draw (-0.70,-0.1) node{$\cdots$};
        \draw (-0.92,-0.1) node{$y_2$};
        \draw (-1.1,-0.1) node{$y_{1}$};
        \draw (-1.35,-0.3) node{(or $-1/b_i^+$'s and $-b_i^-$'s)};

        % Contour names
        \draw (0.65,0.825) node{$t_1 + \Gamma^+_{\alpha}$};
        \draw (-1.4,0.825) node{$t_2 + \Gamma^-_{\beta}$};

        % Angles
        \draw[-,thick][black] (0.45,0) arc (0:45:0.2);
        \draw[-,thick][black] (-0.4,0) arc (180:135:0.2);
        \draw (0.6,0.12) node{$\pi/4$};
        \draw (-0.55,0.12) node{$\pi/4$};

    \end{tikzpicture} 
    \caption{The figure depicts the contours $\Gamma_{\alpha}^+$ and $\Gamma_{\beta}^-$ shifted by $t_1$ and $t_2$, respectively. The figure also shows the relative location of the poles in (\ref{S1EAKS}) and (\ref{DefPhiS11}) when shifted by $t_1$ (for the poles in $z$) and by $t_2$ (for the poles in $w$).}
    \label{S11}
\end{figure}
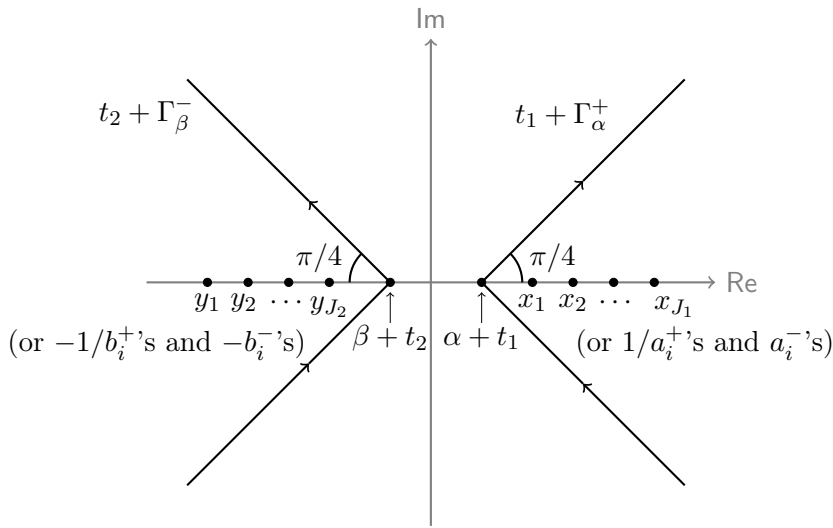

In \cite{BP08} the authors introduced a many-parameter generalization of the extended Airy kernel. Specifically, given finite $J_1, J_2 \in \mathbb{Z}_{\geq 0}$ and parameters $X = \{x_i\}_{i = 1}^{J_1}$, $Y = \{y_j\}_{j = 1}^{J_2}$ such that 
$$\underline{x} := \min(x_i:i = 1, \dots, J_1) > \max (y_j: j = 1, \dots, J_2) =: \overline{y}$$ 
they defined
\begin{equation}\label{S1EAKS}
\begin{split}
&K^{\mathrm{Airy}}_{X,Y}(t_1,x_1; t_2,x_2) =  -  \frac{{\bf 1}\{ t_2 > t_1\} }{\sqrt{4\pi (t_2 - t_1)}} \cdot e^{ - \frac{(x_2 - x_1)^2}{4(t_2 - t_1)} - \frac{(t_2 - t_1)(x_2 + x_1)}{2} + \frac{(t_2 - t_1)^3}{12} } \\
& + \frac{1}{(2\pi \im)^2} \int_{\Gamma_{\alpha}^+} d z \int_{\Gamma_{\beta}^-} dw \frac{e^{z^3/3 -x_1z - w^3/3 + x_2w}}{z + t_1 - w - t_2} \cdot \prod_{i = 1}^{J_1} \frac{w + t_2 - x_i}{z + t_1 - x_i} \cdot \prod_{j = 1}^{J_2} \frac{z + t_1 - y_j}{w + t_2- y_j},
\end{split}
\end{equation}
where as before $\alpha, \beta \in \mathbb{R}$ should satisfy $\alpha + t_1 > \beta + t_2$, but also $\alpha + t_1 < \underline{x}$ and $\beta + t_2 > \overline{y}$. In words, the last condition requires that $x_i - t_1$ to be to the right of $\Gamma_{\alpha}^+$ and $y_j - t_2$ to be to the left of $\Gamma_{\beta}^-$, see Figure \ref{S11}. The kernel in (\ref{S1EAKS}) arose in \cite{BP08} by considering the scaling limit of an exponential last passage percolation model with defective rows and columns. In the case of $J_1 = J_2 \in \mathbb{Z}_{\geq 0}$ the kernel was also obtained in \cite{AFM10} as the scaling limit of a Brownian watermelon model with finitely many outliers. For special cases (especially $t_1 = t_2$) the kernel in (\ref{S1EAKS}) arises in the asymptotics of sample covariance matrices \cite{BBP05}, finite rank perturbations of random matrices \cite{Peche06}, asymmetric exclusion processes \cite{BFS09,IS07} and directed random polymers \cite{BCD21, TV20}.

If one replaces $t_1, t_2$ with $t_1 + \Delta , t_2 + \Delta$ in (\ref{S1EAKS}), then by choosing a suitable $\Delta$ one can obtain an analogous kernel with $x_i > 0$ for $i = 1, \dots, J_1$ and $y_j < 0$ for $j = 1, \dots, J_2$. Using the identities
$$\frac{z - y}{w-y} = \frac{1 - z/y}{1-w/y}, \hspace{2mm}\frac{w - x}{z-x} = \frac{1 - w/x}{1 - z/x} \mbox{ and } \frac{(w-x)(z-y)}{(w-y)(z-x)} = \frac{(1-x/w)(1-y/z)}{(1-y/w)(1-x/z)},$$
one observes that an appropriate limit transition of the integrand in (\ref{S1EAKS}), when $J_1, J_2 \rightarrow \infty$ and possibly some of the $x,y$ parameters go to zero or infinity yields the family of functions
\begin{equation}\label{DefPhiS11}
\begin{split}
&\frac{e^{z^3/3 -x_1z - w^3/3 + x_2w}}{z + t_1 - w - t_2} \cdot \frac{\Phi_{a,b,c}(z + t_1) }{\Phi_{a,b,c}(w + t_2)}, \mbox{ where } \Phi_{a,b,c}(z) = e^{c^+z + c^-/ z} \cdot \prod_{i = 1}^{\infty} \frac{(1 + b_i^+ z) (1 + b_i^- /z)}{(1 - a_i^+ z) ( 1 - a_i^{-}/z)}.
\end{split}
\end{equation}
In (\ref{DefPhiS11}) we have $c^{\pm} \geq 0$ and the four non-negative sequences $\{a_i^+\}_{ i \geq 1}$, $\{a_i^-\}_{ i \geq 1}$, $\{b_i^+\}_{ i \geq 1}$, $\{b_i^-\}_{ i \geq 1}$ satisfy $\sum_{i \geq 1}(a_i^{\pm} + b_i^{\pm}) < \infty$. The infinite-parameter generalization in (\ref{DefPhiS11}) was discussed briefly in \cite[Remark 2]{BP08}, where it was pointed out that the functions $\Phi_{a,b,c}$ from (\ref{DefPhiS11}) arise in stationary extensions of the discrete sine kernel \cite{Bor07}, as generating functions of totally positive doubly infinite sequences \cite{Edrei53}, and as indecomposable characters of the infinite-dimensional unitary group, see \cite{OO98} and the references therein.

As mentioned in \cite[Remark 2]{BP08}, in order to define a kernel as in (\ref{S1EAKS}) with the integrand replaced with (\ref{DefPhiS11}) one needs to choose the contours $\Gamma_{\alpha}^+$ and $\Gamma_{\beta}^-$ so that the poles at $1/a_i^+ - t_1$ and $a_i^- - t_1$ are to the right of $\Gamma_{\alpha}^+$ and the poles at $-1/b_i^+ - t_2$ and $-b_i^- - t_2$ are to the left of $\Gamma_{\beta}^-$, see Figure \ref{S11}. The latter choice of contours is problematic when infinitely many $a_i^-$ and $b_i^-$ are positive. Indeed, in this case the summability of these parameters means they need to converge to zero, so that if for simplicity we take $t_1 = t_2 = 0$, the origin needs to be to the left of $\Gamma_{\beta}^-$ and to the right of $\Gamma_{\alpha}^+$. In particular, such a choice of contours is not possible without forcing $\Gamma_{\beta}^-$ and $\Gamma_{\alpha}^+$ to cross each other, which in turn produces extra residues due to the $z + t_1 - w - t_2$ term in the denominator in (\ref{DefPhiS11}). A similar issue arises if $c^- > 0$. This point appears to have been overlooked in \cite{BP08}. \\

In the present paper we resolve the issue of the previous paragraph and formally construct determinantal point processes corresponding to the infinite-parameter generalization of the extended Airy kernel proposed in \cite{BP08}. We give a precise formulation of this kernel, denoted $K_{a,b,c}$, in equation (\ref{3BPKer}) in the next section, and call the corresponding point processes the {\em Airy wanderer point processes}. We do this for arbitrary choices of $c^{\pm} \geq 0$ and summable non-negative sequences $\{a_i^+\}_{ i \geq 1}$, $\{a_i^-\}_{ i \geq 1}$, $\{b_i^+\}_{ i \geq 1}$, $\{b_i^-\}_{ i \geq 1}$. Under the assumption that $c^- = 0$ and all but finitely many of the parameters $a_i^-, b_i^-$ are equal to zero, we also show that there is a sequence of processes $\{Y_{i}\}_{i \geq 1}$ on $\mathbb{R}$ that are strictly ordered and whose corresponding random measures as in (\ref{RMS1}) but with $\mathcal{L}^{\mathrm{Airy}}_i(s_j)$ replaced with $Y_i(s_j)$ are determinantal with correlation kernel $K_{a,b,c}$. In this case, we also show that the processes $\{Y_{i}\}_{i \geq 1}$ have a continuous version, called {\em Airy wanderer line ensembles}, such that the rescaled and parabolically shifted line ensembles $\mathcal{L}^{a,b,c}$ given by
\begin{equation}
\mathcal{L}^{a,b,c}_i(t) = 2^{-1/2} \cdot Y_i(t) - 2^{-1/2} \cdot t^2 \mbox{ for $i \geq 1$ and $t \in \mathbb{R}$}
\end{equation}
enjoy the Brownian Gibbs property. Our line ensemble construction generalizes the one from \cite[Proposition 3.12]{CorHamA}, which was performed for the kernels in (\ref{S1EAKS}) for the special case of $J_1 = 0$, or $J_2= 0$, or $J_1 = J_2 \in \mathbb{Z}_{\geq 0}$.

Overall, the goals of the paper are to construct the Airy wanderer point processes for general parameters, and lift them to line ensembles in the case of finitely many ``minus'' parameters. A detailed discussion of our approach and the key ideas that go into it is given in Section \ref{Section1.3}, after we have introduced some notation and formulated our main results in Section \ref{Section1.2}. In the remainder of this section we give an informal description of the ensembles $\mathcal{L}^{a,b,c}$ and explain how one should view their parameters. We mention that the following discussion is quite informal and rigorously establishing the description of $\mathcal{L}^{a,b,c}$ below will be the subject of research efforts in the near future.  \\

\begin{figure}[h]
    \centering
    \begin{tikzpicture}[scale=0.8]

  %defining variables
  \def\r{0}
  \def\s{0.04}

    \foreach \x [remember=\r as \rlast (initially 0)] in {-200,...,83}
    {
        \tikzmath{\r = 5*rand*\s; \y = \x*\s; \z = (\x+1)*\s;}
        \ifthenelse{\x < -25}
        {
            \draw[-,thin][black] (\y, \y + \rlast) -- (\z, \z + \r);
        }
        {
            \ifthenelse{\x < 25}{
                \draw[-,thin][black] (\y, -\y*\y*0.25 + \rlast - 0.75) -- (\z, -\z*\z*0.25 -0.75 + \r);
            }
            {
                \draw[-,thin][black] (\y, -0.7*\y*\y + \rlast -0.3) -- (\z, -0.7*\z*\z + \r - 0.3);
            }
        }
    }
    
    \foreach \x [remember=\r as \rlast (initially 0)] in {-200,...,200}
    {
        \tikzmath{\r = 5*rand*\s; \y = \x*\s; \z = (\x+1)*\s;}
        \ifthenelse{\x < -25}
        {
            \draw[-,thin][black] (\y, 0.5*\y + \rlast) -- (\z, 0.5*\z + \r);
        }
        {
            \ifthenelse{\x < 25}{
                \draw[-,thin][black] (\y, -\y*\y*0.25 + \rlast - 0.25) -- (\z, -\z*\z*0.25 -0.25 + \r);
            }
            {
                \draw[-,thin][black] (\y, -\y*0.7 + \rlast + 0.2 ) -- (\z, -\z*0.7 + \r +0.2);
            }
        }
    }

    \foreach \x [remember=\r as \rlast (initially 0)] in {-200,...,200}
    {
        \tikzmath{\r = 5*rand*\s; \y = \x*\s; \z = (\x+1)*\s;}
        \ifthenelse{\x < -25}
        {
            \draw[-,thin][black] (\y, 0.2*\y + \rlast+0.25) -- (\z, 0.2*\z + \r +0.25);
        }
        {
            \ifthenelse{\x < 25}{
                \draw[-,thin][black] (\y, -\y*\y*0.25 + \rlast + 0.35) -- (\z, -\z*\z*0.25 + 0.35 + \r);
            }
            {
                \draw[-,thin][black] (\y, -\y*0.3 + \rlast + 0.4 ) -- (\z, -\z*0.3 + \r +0.4);
            }
        }
    }

    % Contour names
    \draw (-6, -0.25) node{$\mathcal{L}^{a,b,c}_1$};
    \draw (-6, -2.25) node{$\mathcal{L}^{a,b,c}_2$};
    \draw (-6, -5) node{$\mathcal{L}^{a,b,c}_3$};

    % Slopes
    \draw (-9, -0.75) node{slope $\beta_1$};
    \draw (-9, -3.5) node{slope $\beta_2$};
    \draw (-9, -7.5) node{slope $\beta_3$}; 

    \draw (8, -1.25) node{slope $\alpha_1$};
    \draw (8, -4.25) node{slope $\alpha_2$};
    \draw (5.25, -8) node{slope $\alpha_3 \equiv -\infty$}; 
  
\end{tikzpicture}
\caption{The figure depicts the top three curves in $\mathcal{L}^{a,b,c}$ from Theorem \ref{T3}. The slopes $\beta_i$ are the sorted in ascending order numbers $\{\sqrt{2} (b_i^+)^{-1}\}_{ i\geq 1}$ and the non-zero elements in $\{\sqrt{2} b_i^-\}_{i \geq 1}$ (counted with multiplicities and with $+\infty$ allowed). The slopes $\alpha_i$ are the sorted in descending order numbers $\{-\sqrt{2} (a_i^+)^{-1}\}_{ i\geq 1}$ and the non-zero elements in $\{- \sqrt{2} a_i^-\}_{i \geq 1}$ (counted with multiplicities and with $-\infty$ allowed).}
    \label{S12}
\end{figure}
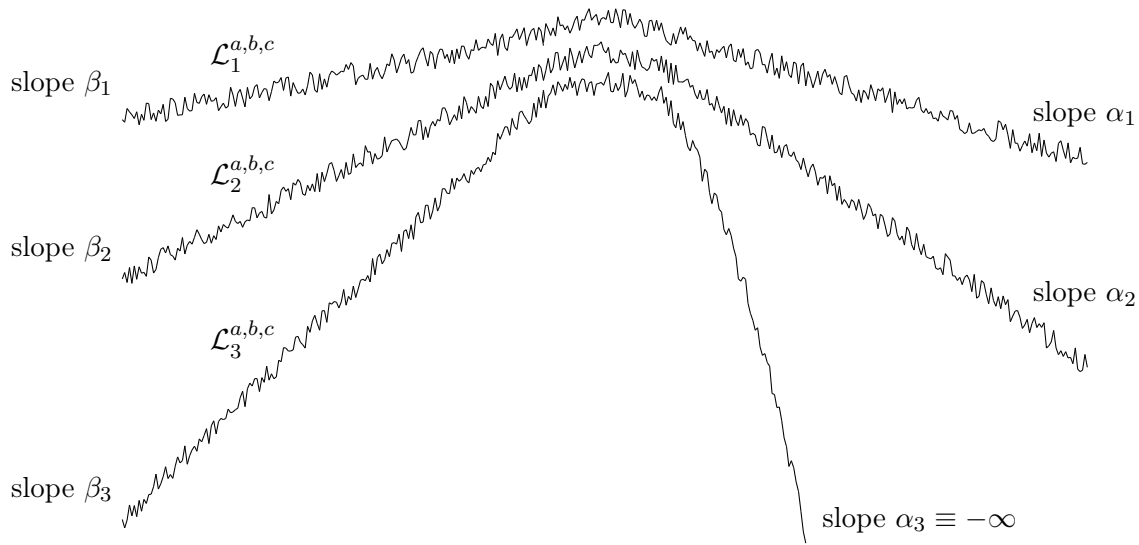

When all parameters are equal to zero, the ensemble $\mathcal{L}^{0,0,0}$ is precisely the parabolic Airy line ensemble from (\ref{PALE}). Since $\mathcal{L}^{\mathrm{Airy}}$ is stationary, we see that $\mathcal{L}^{0,0,0}$ globally looks like the inverted parabola $-2^{-1/2} t^2$. Once we start adding $a$ and $b$ parameters the behavior of $\mathcal{L}^{a,b,c}$ near $\pm \infty$ changes from parabolic to {\em linear}. Based on informal computations involving the kernel $K_{a,b,c}$ we expect that adding an extra $a_i^+> 0$ parameter creates an additional line whose slope for large $t$ is $-\sqrt{2} (a_i^+)^{-1}$, while adding an extra $a_i^- > 0$ creates an additional line with slope $-\sqrt{2}a_i^-$ for large $t$. Similarly, adding an extra $b_i^+$ or $b_i^-$ parameter creates an additional line of slope $\sqrt{2} (b_i^+)^{-1}$ or $\sqrt{2}b_i^-$, respectively, for small negative $t$. If we have only finitely many non-zero parameters, only the top few curves are linear away from the origin, and the others should look like the inverted parabola $-2^{-1/2} t^2$. We also think that it is possible to have curves that on one side of the origin are linear and on the other are parabolic, formally corresponding to a slope of $\pm \infty$, see Figure \ref{S12}. 

We mention that recently it was established in \cite{AH23} that $\mathcal{L}^{0,0,0}$ is (up to an independent affine shift) the unique Brownian Gibbsian line ensemble whose top curve globally looks like $-2^{-1/2} t^2$. Our line ensembles $\mathcal{L}^{a,b,c}$ for general parameters obviously violate this condition, but it would be interesting to see if they too can be characterized by some global feature, with asymptotic slopes of the curves being an obvious candidate.

We lastly discuss the case when $c^- > 0$ or infinitely many of the parameters $a_i^-, b_i^-$ are positive. While the point processes exist in this case, it is not clear whether one can lift them to line ensembles. As we add more ``minus'' parameters, we are creating more and more curves whose asymptotic slopes $\sqrt{2}b_i^-$ and $-\sqrt{2}a_i^-$ get closer and closer to zero. Since the curves avoid each other this causes a repulsion in the bulk, which we believe asymptotically will force the top curves to go to infinity. In this case, we think it possible that there is still a line ensemble, although it is not indexed by $\mathbb{N}$, but rather $\mathbb{Z}$. I.e. there is no longer a ``top'' curve but instead the ensemble has infinitely many curves in both vertical directions. On the level of formulas, having $c^- > 0$ or infinitely many infinitely positive $a_i^-, b_i^-$ causes the function $\Phi_{a,b,c}$ in (\ref{DefPhiS11}) to have an {\em essential} singularity at the origin, which would support a qualitative change in the nature of the line ensembles.

%%%%%%%%%%%%%%%%%%%%%%%%%%%%%%%%%%%%%%%%%%%%%%%%%%%%%%%%%%%%%%%%%%%%%
%
%    Section 1.2
%
%%%%%%%%%%%%%%%%%%%%%%%%%%%%%%%%%%%%%%%%%%%%%%%%%%%%%%%%%%%%%%%%%%%%%
\subsection{Main results}\label{Section1.2}
We begin by fixing the parameters of the model and some notation.
\begin{definition}\label{DLP} 
We assume that we are given four sequences of non-negative real numbers $\{a_i^+\}_{ i \geq 1}$, $\{a_i^-\}_{ i \geq 1}$, $\{b_i^+\}_{ i \geq 1}$, $\{b_i^-\}_{ i \geq 1}$ such that 
\begin{equation}\label{ParProp}
\sum_{i = 1}^{\infty} (a_i^+ + a_i^- + b_i^+ + b_i^-) < \infty \mbox{ and } a_{i}^{\pm} \geq a_{i+1}^{\pm},  b_{i}^{\pm} \geq b_{i+1}^{\pm} \mbox{ for all } i \geq 1,
\end{equation}
as well as two non-negative parameters $c^+, c^-$. We let $J_a^{\pm} = \inf \{ k \geq 1: a_{k}^{\pm} = 0\} - 1$ and $J_b^{\pm} = \inf \{ k \geq 1: b_{k}^{\pm} = 0\} - 1$. In words, $J_a^{\pm}$ is the largest index $k$ such that $a_{k}^{\pm} > 0$, with the convention that $J_a^{\pm} = 0$ if all $a_k^{\pm} = 0$ and $J_a^{\pm} = \infty$ if all $a_k^{\pm} > 0$, and analogously for $J_b^{\pm}$.

Define 
$$\underline{a} = \begin{cases} 0  &\hspace{-3.5mm}  \mbox{ if } a_1^- + b_1^- + c^- > 0, \\   \infty & \hspace{-3.5mm} \mbox{ if }   a_1^- + b_1^- + c^- = 0\mbox{ and }  a_1^+ = 0, \\ 1/a_1^+ & \hspace{-3.5mm}  \mbox{ if }a_1^- + b_1^- + c^- =  0 \mbox{ and } a_1^+ > 0, \end{cases} \hspace{1mm} \mbox{ and }\hspace{1mm} \underline{b} = \begin{cases} 0 &\hspace{-3.5mm}  \mbox{ if }  a_1^- + b_1^- + c^- > 0, \\  -\infty & \hspace{-3.5mm} \mbox{ if } a_1^- + b_1^- + c^- =  0 \mbox{ and } b_1^+ = 0, \\ -1/b_1^+ & \hspace{-3.5mm} \mbox{ if } a_1^- + b_1^- + c^- =  0 \mbox{ and } b_1^+ > 0. \end{cases}$$
 Observe that $\underline{a} \in [0, \infty]$ and $\underline{b} \in [- \infty, 0]$. 
\end{definition}

For $z \in \mathbb{C} \setminus \{0\} $ we define the function
\begin{equation}\label{DefPhi}
\Phi_{a,b,c}(z) = e^{c^+z + c^-/ z} \cdot \prod_{i = 1}^{\infty} \frac{(1 + b_i^+ z) (1 + b_i^- /z)}{(1 - a_i^+ z) ( 1 - a_i^{-}/z)}.
\end{equation}
From (\ref{ParProp}) and \cite[Chapter 5, Proposition 3.2]{Stein}, we have that the above defines a meromorphic function on $\mathbb{C} \setminus \{0\}$ whose zeros are at $\{-(b_i^+)^{-1}\}_{i =1}^{J_b^+}$ and $\{- b_i^-\}_{i =1}^{J_b^-}$, while its poles are at $\{(a_i^+)^{-1}\}_{i =1}^{J_a^+}$ and $\{ a_i^-\}_{i =1}^{J_a^-}$. We also observe that $\Phi_{a,b,c}(z) $ is analytic in $\mathbb{C} \setminus [\underline{a}, \infty)$, and its inverse is analytic in $\mathbb{C} \setminus (-\infty, \underline{b}]$, where $\underline{a}, \underline{b}$ are as in Definition \ref{DLP}.

The following definitions present the Airy wanderer kernel, introduced in \cite{BP08}, starting with the contours that appear in it. 
\begin{definition}\label{DefContInf}  Fix $a \in \mathbb{R}$. We let $\Gamma_a^+$ denote the union of the contours $\{a + y e^{\pi \im /4}\}_{y \in \mathbb{R}_+}$ and $\{a + y e^{-\pi \im/4}\}_{y \in \mathbb{R}_+}$, and $\Gamma_a^-$ the union of the contours $\{a + y e^{ 3\pi \im /4}\}_{y \in \mathbb{R}_+}$ and $\{a + y e^{-3 \pi \im/4}\}_{y \in \mathbb{R}_+}$. Both contours are oriented in the direction of increasing imaginary part. See Figure \ref{S11}.
\end{definition}

\begin{definition}\label{3BPKernelDef} Assume the same notation as in Definition \ref{DLP}. For $t_1, t_2,x_1,x_2 \in \mathbb{R}$ we define 
\begin{equation}\label{3BPKer}
\begin{split}
&K_{a,b,c} (t_1, x_1; t_2, x_2) = K^1_{a,b,c} (t_1, x_1; t_2, x_2) +  K^2_{a,b,c} (t_1, x_1; t_2, x_2) +  K^3_{a,b,c} (t_1, x_1; t_2, x_2), \mbox{ with } \\
& K^1_{a,b,c} (t_1, x_1; t_2, x_2) = \frac{1}{2\pi \im} \int_{\gamma}dw \cdot  e^{(t_2 - t_1)w^2 + (t_1^2 - t_2^2) w + w (x_2-x_1) + x_1 t_1 - x_2 t_2 - t_1^3/3 + t_2^3/3}    \\
&K^2_{a,b,c} (t_1, x_1; t_2, x_2)  = -  \frac{{\bf 1}\{ t_2 > t_1\} }{\sqrt{4\pi (t_2 - t_1)}} \cdot e^{ - \frac{(x_2 - x_1)^2}{4(t_2 - t_1)} - \frac{(t_2 - t_1)(x_2 + x_1)}{2} + \frac{(t_2 - t_1)^3}{12} }; \\
& K^3_{a,b,c} (t_1, x_1; t_2, x_2) = \frac{1}{(2\pi \im)^2} \int_{\Gamma_{\alpha }^+} d z \int_{\Gamma_{\beta}^-} dw \frac{e^{z^3/3 -x_1z - w^3/3 + x_2w}}{z + t_1 - w - t_2} \cdot \frac{\Phi_{a,b,c}(z + t_1) }{\Phi_{a,b,c}(w + t_2)}.
\end{split}
\end{equation}
In (\ref{3BPKer}) $\alpha, \beta  \in \mathbb{R}$ are such that $\alpha + t_1 < \underline{a}$ and $\beta + t_2 > \underline{b}$, the function $\Phi_{a,b,c}$ is as in (\ref{DefPhi}) and the contours of integration in $K^3_{a,b,c}$ are as in Definition \ref{DefContInf}. If $\Gamma^+_{\alpha + t_1} (=t_1 + \Gamma^+_{\alpha})$ and $\Gamma^-_{\beta + t_2} (= t_2 + \Gamma^+_{\beta} )$ have zero or one intersection points, i.e. they look as in Figure \ref{S11}, we take $\gamma = \emptyset$ and then $K^1_{a,b,c} \equiv 0$. Otherwise, $\Gamma^+_{\alpha + t_1} $ and $\Gamma^-_{\beta + t_2}$ have exactly two intersection points, which are complex conjugates, and $\gamma$ is the straight vertical segment that connects them with the orientation of increasing imaginary part. See Figure \ref{S12}.
\end{definition}
\begin{figure}[h]
    \centering
     \begin{tikzpicture}[scale=2.7]

        \def\tra{3} 
        % Picture on the left
        % Coordinate System
        \draw[->, thick, gray] (-1.2,0)--(1.2,0) node[right]{$\Real$};
        \draw[->, thick, gray] (0,-1.2)--(0,1.2) node[above]{$\Imag$};

         % The contour Gamma_{beta + t_2}^-
        \draw[-,thick][black] (0.6,0) -- (0.2,-0.4);
        \draw[->,thick][black] (-0.4,-1) -- (0.2,-0.4);
        \draw[black, fill = black] (0.6,0) circle (0.02);
        \draw (0.6,-0.3) node{$\beta + t_2$};
        \draw[->,very thin][black] (0.6,-0.2) -- (0.6, -0.05);
        \draw[->,thick][black] (0.6,0) -- (0.2,0.4);
        \draw[-,thick][black]  (-0.4,1) -- (0.2,0.4);       

        % The contour Gamma_{alpha}^+
        \draw[-,thick][black] (-0.75, 0) -- (-0.25,-0.5);
        \draw[->,thick][black] (0.25, -1) -- (-0.25, -0.5);
        \draw[black, fill = black] (-0.75,0) circle (0.02);
        \draw (-0.75,-0.275) node{$\alpha + t_1$};
        \draw[->,very thin][black] (-0.75,-0.2) -- (-0.75, -0.05);
        \draw[->,thick][black] (-0.75,0) -- (-0.25,0.5);
        \draw[-,thick][black]  (-0.25,0.5) -- (0.25,1);

        % the contour gamma
        \draw[->,thick][black] (-0.075, -0.675) -- (-0.075,0.2);
        \draw[-,thick][black] (-0.075, 0.2) -- (-0.075,0.675);

        % Intersection point
        \draw[black, fill = black] (-0.075,0.675) circle (0.02);
        \draw[black, fill = black] (-0.075,-0.675) circle (0.02);
        \draw (-0.075,0.8) node{$u_+$};
        \draw (-0.075,-0.825) node{$u_-$};

        % Contour names
        \draw (0.35,0.825) node{$\Gamma^+_{\alpha + t_1}$};
        \draw (-0.4,0.825) node{$\Gamma^-_{\beta + t_2}$};
        \draw (-0.15,0.225) node{$\gamma$};

    \end{tikzpicture} 
    \caption{The figure depicts the contours $\Gamma_{\alpha + t_1}^+, \Gamma_{\beta + t_2}^-$ when they have two intersection points, denoted by $u_-$ and $u_+$. The contour $\gamma$ is the segment from $u_-$ to $u_+$.}
    \label{S12}
\end{figure}
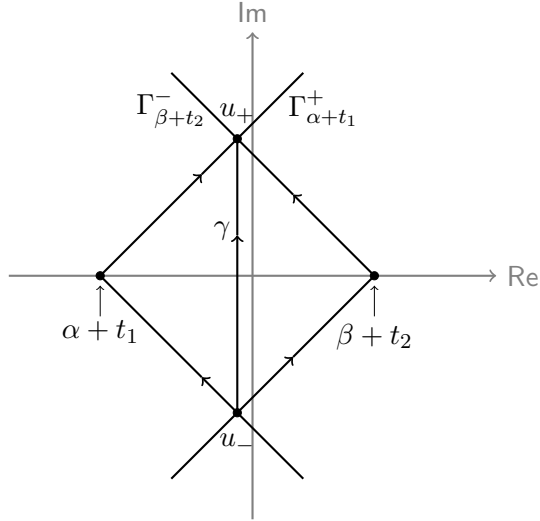

The following lemma summarizes some basic properties of the kernel $K_{a,b,c}$ in Definition \ref{3BPKernelDef}, and in particular shows that it is well-defined. It is proved in Section \ref{Section4.1}.
\begin{lemma}\label{WellDefKer} Assume the same notation as in Definition \ref{DLP}. For each $t_1, t_2,x_1,x_2 \in \mathbb{R}$ we have that the double integral in the definition of $K^3_{a,b,c}$ in (\ref{3BPKer}) is convergent. The value of $K_{a,b,c}(t_1, x_1; t_2, x_2) $ does not depend on the choice of $\alpha$ and $\beta$ as long as $\alpha + t_1 < \underline{a}$ and $\beta + t_2 > \underline{b}$. Moreover, for each fixed $t_1, t_2 \in \mathbb{R}$ we have that $K_{a,b,c}(t_1, \cdot; t_2, \cdot)$ is continuous in $(x_1, x_2) \in \mathbb{R}^2$.
\end{lemma}

If $c^+ = c^- = 0$, $a_i^{-} = b_i^- = 0$ for all $i \geq 1$, and $J_a^+ = J_1 \in \mathbb{Z}_{\geq 0}$, $J_b^+ = J_2 \in \mathbb{Z}_{\geq 0}$, a straightforward calculation gives 
$$\frac{\Phi_{a,b,c}(z ) }{\Phi_{a,b,c}(w )} = \prod_{i = 1}^{J_2} \frac{1 + b_i^+ z}{ 1 + b_i^+ w} \cdot \prod_{i = 1}^{J_1} \frac{1 - a_i^+ w}{ 1 - a_i^+ z} = \prod_{i = 1}^{J_1} \frac{w - x_i}{z - x_i} \cdot \prod_{i = 1}^{J_2} \frac{z - y_i}{w - y_i},$$
where $x_i = (a_i^+)^{-1}$ for $i = 1, \dots, J_1$ and $y_j = - (b_j^+)^{-1}$ for $j = 1, \dots, J_2$. In this case, we have $\underline{a} = \min(x_i:i = 1, \dots, J_1) \in (0,\infty)$ and $\underline{b} = \max (y_j: j = 1, \dots, J_2) \in (-\infty, 0)$. In particular, we can find $\alpha, \beta$ such that $\underline{b} < \beta + t_2 < \alpha + t_1 < \underline{a}$. We see that $\Gamma_{\alpha + t_1}^+$ and $\Gamma_{\beta + t_2}^-$ do not intersect, and so $K^1_{a,b,c} = 0$, which shows $K_{a,b,c}$ agrees with the kernel $K^{\mathrm{Airy}}_{X,Y}$ from (\ref{S1EAKS}). If one further replaces $t_1, t_2$ with $t_1 + \Delta , t_2 + \Delta$ in the above formula, one would obtain the same kernel $K^{\mathrm{Airy}}_{X,Y}$ but with $x$'s and $y$'s being arbitrary reals such that $\min(x_i:i = 1, \dots, J_1) > \max (y_j: j = 1, \dots, J_2)$.

The discussion in the previous paragraph shows that our kernels $K_{a,b,c}$ indeed generalize $K^{\mathrm{Airy}}_{X,Y}$. When the contours $\Gamma^+_{\alpha + t_1}$ and $\Gamma^-_{\beta + t_2}$ are disjoint, which occurs when all ``minus'' parameters are zero, they take the form proposed in \cite[Remark 2]{BP08}. When some of the ``minus'' parameters are positive, the function $\Phi_{a,b,c}(z)$ from (\ref{DefPhi}) has a (possibly essential) singularity at the origin. In this case the contours $\Gamma^+_{\alpha + t_1}$ and $\Gamma^-_{\beta + t_2}$ are deformed past each other to the correct side of the origin they need to pass through. As the contours cross, they give rise to a residue coming from $z + t_1 - w - t_2$, which is what produces the term $K^1_{a,b,c}$ in (\ref{3BPKer}). Our particular choice of contours and insertion of the extra term $K^1_{a,b,c}$ resolves the issue mentioned in Section \ref{Section1.1} and Lemma \ref{WellDefKer} shows that the kernel $K_{a,b,c}$ is well-defined and behaved.\\

The first main result of the paper is that there exists a determinantal point process whose correlation kernel is given by $K_{a,b,c}$ as in Definition \ref{3BPKernelDef}. The precise statement is given as Theorem \ref{T1} below. We refer the reader to Section \ref{Section2} for background on point processes, and specifically to Definition \ref{DPP} for the definition of a determinantal point process.
\begin{theorem}\label{T1} Assume the same notation as in Definition \ref{DLP}. Fix $m \in \mathbb{N}$, $s_1, \dots, s_m \in \mathbb{R}$ with $s_1 < s_2 < \cdots < s_m$, and set $\mathcal{A} = \{s_1, \dots, s_m\}$. There exists a determinantal point process $M$ on $\mathbb{R}^2$, whose correlation kernel is $K_{a,b,c}$ as in (\ref{3BPKer}), and with reference measure given by $\mu_{\mathcal{A}} \times \lambda$, where $\mu_{\mathcal{A}}$ is the counting measure on $\mathcal{A}$ and $\lambda$ is the Lebesgue measure on $\mathbb{R}$.  
\end{theorem}
\begin{remark}\label{T1R1} We mention that in Definition \ref{DPP} we require that the correlation kernel of a determinantal point process and its reference measures are both locally finite. The latter is clear for the measure $\mu_{\mathcal{A}} \times \lambda$, and for $K_{a,b,c}$ it follows from the continuity statement in Lemma \ref{WellDefKer}.
\end{remark}
\begin{remark}\label{T1R2} While Theorem \ref{T1} establishes the {\em existence} of a determinantal point process $M$ with the prescribed correlation kernel and reference measure, we have from part (3) of Proposition \ref{PropLem}, that this point process is {\em unique}. 
\end{remark}

The next two theorems show that if $c^- = 0$ and $J_a^- + J_b^- < \infty$, which is the case when $\Phi_{a,b,c}(z)$ has a removable singularity or a pole at the origin (i.e. not an essential singularity), one can lift the point processes of Theorem \ref{T1} to {\em line ensembles}, i.e. random elements in $C(\mathbb{N} \times \mathbb{R})$. We refer the reader to \cite[Definition 2.1]{CorHamA} and also \cite[Definition 2.1]{DEA21} (where we take $\Sigma = \mathbb{N}$ and $\Lambda = \mathbb{R}$) for a formal definition of line ensembles.

\begin{theorem}\label{T2} Assume the same notation as in Definition \ref{DLP}. Assume further that $c^- = 0$ and $J_a^- + J_b^- < \infty$. There exists a sequence $\{Y_i\}_{i \geq 1}$ of processes on $\mathbb{R}$ that are all defined on the same probability space $(\Omega, \mathcal{F}, \mathbb{P})$, such that the following hold. For each $m \in \mathbb{N}$, and $s_1, \dots, s_m \in \mathbb{R}$ with $s_1 < s_2 < \cdots < s_m$ we have that the random measure
\begin{equation}\label{T2E1}
M(\omega, A) = \sum_{i \geq 1} \sum_{j = 1}^m {\bf 1}\{ (s_j, Y_i(s_j, \omega)) \in A\}
\end{equation}
is a determinantal point process on $\mathbb{R}^2$, with correlation kernel $K_{a,b,c}$ as in (\ref{3BPKer}), and reference measure $\mu_{\mathcal{A}} \times \lambda$ as in Theorem \ref{T1}. In addition, 
\begin{equation}\label{T2E2}
Y_i(t, \omega) > Y_{i+1}(t, \omega) \mbox{ for each $i \in \mathbb{N}$, $t \in \mathbb{R}$ and $\omega \in \Omega$. }
\end{equation}
\end{theorem}
\begin{remark}\label{T2R1} From part (3) of Proposition \ref{PropLem} and Corollary \ref{CorWC2} we have that if $\{\tilde{Y}_i\}_{i \geq 1}$ is another sequence of processes on $\mathbb{R}$ that satisfies the properties in Theorem \ref{T2}, then $\{\tilde{Y}_i\}_{i \geq 1}$ has the same finite-dimensional distribution as $\{Y_i\}_{i \geq 1}$. I.e. the sequence of processes in Theorem \ref{T2} is unique.
\end{remark}

\begin{theorem}\label{T3} Assume the same notation as in Definition \ref{DLP}. Assume further that $c^- = 0$ and $J_a^- + J_b^- < \infty$. There exists a line ensemble $\mathcal{L}^{a,b,c} = \{\mathcal{L}_i^{a,b,c}\}_{i \geq 1}$ such that we have the following equality in the sense of finite-dimensional distributions
\begin{equation}\label{S1FDE}
\left(\sqrt{2} \cdot \mathcal{L}_i^{a, b, c}(t) + t^2: i \geq 1, t \in \mathbb{R} \right)  = \left(Y_i(t): i \geq 1 , t \in \mathbb{R} \right),
\end{equation}
where $\{Y_i\}_{i \geq 1}$ are as in Theorem \ref{T2}. Moreover, $\mathcal{L}^{a,b,c}$ satisfies the Brownian Gibbs property as in \cite[Definition 2.2]{CorHamA}.
\end{theorem}
\begin{remark} From \cite[Lemma 3.1]{DM21} we have that he finite-dimensional distributions of a line ensemble uniquely determine its law. Thus, the line ensemble $\mathcal{L}^{a,b,c}$ in Theorem \ref{T3} is unique.
\end{remark}

\begin{remark} As mentioned in Section \ref{Section1.1}, we are not certain that there exist line ensembles corresponding to the point processes from Theorem \ref{T1} when $c^- > 0$ or $J_a^- + J_b^- = \infty$. It seems plausible to us that there is a $\mathbb{Z}$-indexed line ensemble, i.e. a random element in $C(\mathbb{Z} \times \mathbb{R})$, corresponding to these point processes, but we do now know how to establish such a statement presently.
\end{remark}

%%%%%%%%%%%%%%%%%%%%%%%%%%%%%%%%%%%%%%%%%%%%%%%%%%%%%%%%%%%%%%%%%%%%%
%
%    Section 1.3
%
%%%%%%%%%%%%%%%%%%%%%%%%%%%%%%%%%%%%%%%%%%%%%%%%%%%%%%%%%%%%%%%%%%%%%
\subsection{Key ideas and paper outline}\label{Section1.3} The way we construct the point processes in Theorem \ref{T1} is by taking an appropriate limit of a sequence of point processes that arise in a directed last passage percolation model with geometric weights that has several defective rows and columns. This approach is similar to the original one in \cite{BP08} who worked with exponential weights, although our asymptotic analysis is notably harder as we directly work with infinitely many parameters. The essential property of geometric last passage percolation is that it has a distributional equality with a Schur process, which in turn has the structure of a determinantal point process. The correlation kernel for this process has a double-contour integral formula, derived by different means in \cite{Agg15} and \cite{BR05}, and which is suitable for performing our desired asymptotics. In Section \ref{Section2} we summarize some basic results about determinantal point processes, and in Section \ref{Section3} we recall the Schur process and the formula for its correlation kernel. The parameter scaling we perform for the Schur processes is presented in Section \ref{Section4}, and the pointwise limit of their correlation kernels is established in Proposition \ref{LimitKernelProp}, whose proof is the content of Section \ref{Section5}. Proposition \ref{LimitKernelProp} shows that the Schur correlation kernels converge uniformly over compact sets to $K_{a,b,c}$ from Definition \ref{3BPKernelDef}, which together with some results in Section \ref{Section2} implies the existence of the determinantal point processes in Theorem \ref{T1} and that they are weak limits of the Schur determinantal point processes.

The way we construct the processes $\{Y_i\}_{i \geq 1}$ in Theorem \ref{T2} is by {\em improving} the convergence of Schur point processes to finite-dimensional convergence of their point {\em locations}. Part of Section \ref{Section2} is devoted to establishing several results that can be used to establish such finite-dimensional convergence for general point processes, and we hope the latter will find numerous applications in related problems. In our context, in addition to the point process convergence we need to establish tightness of the locations of the various points. To obtain tightness from above, one can directly investigate the one-point distribution functions, which can be written as Fredholm determinants, derive upper tail estimates for the kernels, and use those to obtain upper-tail probability estimates -- we do this in Section \ref{Section6}. Tightness from below is harder to establish directly from the Fredholm determinant, which in the lower-tail regime involves summing infinitely many oscillating and asymptotically exploding terms. To circumvent this issue we derive an alternative tightness criterion in Proposition \ref{TightnessCrit}, which states that tightness from below is ensured under knowledge of weak convergence of point processes, tightness from above, and that the limiting point process is almost surely infinite. A priori, we do not know that the determinantal point processes from Theorem \ref{T1} have infinitely many points for general parameters as in Theorem \ref{T2}. One way to approach this is to show that the traces of $K_{a,b,c}$ are infinite and apply \cite[Theorem 4]{Sosh00}. As we could not estimate the traces from below we take a different approach.

The basic idea is to note that when all $a,b,c$ parameters are zero, $K_{a,b,c}$ is precisely the extended Airy kernel, and we know that the Airy point process has infinitely many points almost surely -- see (\ref{S72E2}). In this case we can apply Proposition \ref{TightnessCrit} and conclude the tightness of our discrete models in this case. When $J_a^- = J_b^- = 0$ we can construct a {\em monotone coupling} of the weights in the last passage percolation model, for which the zero-parameter case is a lower bound. In particular, the lower tails of the zero-parameter model are heavier than those of the model with non-zero parameters, which allows us to conclude tightness from below for the latter from that of the former. This monotone coupling is established in Proposition \ref{MonCoup} and relies on the connection between the geometric last passage percolation, the Schur process, the Robinson-Schensted-Knuth (RSK) correspondence and Greene's theorem, see Section \ref{Section3.3}. Once tightness is established we obtain the processes in Theorem \ref{T2} for any parameters such that $J_a^- = J_b^- = 0$. The processes for $J_a^{-} + J_b^- < \infty$ are obtained from those with $J_a^- = J_b^- = 0$ by appropriate translations, see Step 2 in the proof of Theorem \ref{T2} in Section \ref{Section7.1}.

At this point the construction of the line ensembles in Theorem \ref{T3} is relatively straightforward and performed in Section \ref{Section8}. In the special case when $J_a^+ = J_b^+ = m \in \mathbb{Z}_{\geq 0}$ and $J_a^- = J_b^- = c^- = 0$, these ensembles were constructed in \cite[Proposition 3.12]{CorHamA} as a weak limit of the Brownian watermelons with outliers from \cite{AFM10}. The correlation kernels of the latter can be used to approximate any kernel $K_{a,b,c}$ with $J_a^- = J_b^- = c^- = 0$, which together with the results from Section \ref{Section2} implies convergence on the level of point processes. By obtaining upper tail estimates for $K_{a,b,c}$ and using the {\em now available} by Theorem \ref{T2} statement that the corresponding point processes are infinite almost surely, we can apply Proposition \ref{TightnessCrit} to conclude tightness and ultimately finite-dimensional convergence. At this point we have verified the hypotheses in \cite[Definition 3.3]{CorHamA}, which in turn enables us to construct the ensembles in Theorem \ref{T3} as weak limits of those with $J_a^+ = J_b^+ = m \in \mathbb{Z}_{\geq 0}$ when $J_a^- = J_b^- = c^- = 0$ using various results from \cite{CorHamA}. The ensembles for general $J_a^- + J_b^- < \infty$ are obtained from those with $J_a^- = J_b^- = 0$ by appropriate translations, similarly to the proof of Theorem \ref{T2}.\\

In the remainder of this section we discuss briefly some of the advantages of our approach to the results in Section \ref{Section1.2}. As mentioned earlier, in the case when $J_a^+ = J_b^+ = m \in \mathbb{Z}_{\geq 0}$ and $J_a^- = J_b^- = c^- = 0$ the ensembles in Theorem \ref{T3} were constructed in \cite[Proposition 3.12]{CorHamA}. The key asymptotic input in that construction is the convergence of the correlation kernels for Brownian watermelons with outliers from \cite{AFM10}. A crucial feature of the kernels in \cite{AFM10} is that they can be expressed as the sum of the kernel for Brownian watermelons {\em without} outliers and a {\em finite} $m$-dependent sum, whose asymptotics can be handled by hand. If $J_a^+ \neq J_b^+$ and especially if $J_a^+ + J_b^+ = \infty$, it is not clear how to perform the asymptotics in \cite{AFM10}, as the previously mentioned finite $m$-dependent sum asymptotically becomes infinite. The reason this analysis is tractable in our case is due to the double-contour integral formula for the Schur correlation kernel, in which even asymptotically infinitely many parameters enter in a benign way. 

One could have constructed the processes $\{Y_i\}_{ i \geq 1}$ in Theorem \ref{T2} directly from the ensembles in \cite[Proposition 3.12]{CorHamA}, although there is a subtle issue with establishing tightness from below and hence finite-dimensional convergence. Specifically, if we sought to apply Proposition \ref{TightnessCrit} we would require knowledge that the point processes with correlation kernels $K_{a,b,c}$ have almost surely infinitely many points, which we do not know how to obtain directly. The way this is established in the proof of Theorem \ref{T2} is by utilizing the monotone coupling for Schur processes with different parameters, relying on the combinatorial structure of the Schur process and most notably the RSK correspondence. A similar monotone coupling is available for the Brownian watermelons in view of \cite[Lemmas 2.6 and 2.7]{CorHamA}, and one could in principle translate it to the ensembles in \cite[Proposition 3.12]{CorHamA} ultimately resolving the tightness from below issue mentioned above.

Overall, our approach of constructing the processes $\{Y_i\}_{ i \geq 1}$ in Theorem \ref{T2} and $\mathcal{L}^{a,b,c}$ in Theorem \ref{T3} from Schur processes seems to be the most direct (although using the exponential last passage percolation model from \cite{BP08} would have been comparable).  In addition, in the course of the proof of Theorem \ref{T2} we see that the Schur process converges in the finite-dimensional sense to $\mathcal{L}^{a,b,c}$, which allows one to transfer statements known about Schur processes to the line ensembles. The latter is potentially quite useful, as the Schur process has rich algebraic and combinatorial structures that can be exploited to obtain various couplings, identities and inequalities, Proposition \ref{MonCoup} being just one example.

%%%%%%%%%%%%%%%%%%%%%%%%%%%%%%%%%%%%%%%%%%%%%%%%%%%%%%%%%%%%%%%%%%%%%
%
%    Section 1.4
%
%%%%%%%%%%%%%%%%%%%%%%%%%%%%%%%%%%%%%%%%%%%%%%%%%%%%%%%%%%%%%%%%%%%%%
\subsection*{Acknowledgments}\label{Section1.4} The author would like to thank Amol Aggarwal, Alexei Borodin and Ivan Corwin for many fruitful discussions. This work was partially supported by the NSF grant DMS:2230262.

%%%%%%%%%%%%%%%%%%%%%%%%%%%%%%%%%%%%%%%%%%%%%%%%%%%%%%%%%%%%%%%%%%%%%
%
%    Section 2
%
%%%%%%%%%%%%%%%%%%%%%%%%%%%%%%%%%%%%%%%%%%%%%%%%%%%%%%%%%%%%%%%%%%%%%
\section{Point processes and convergence}\label{Section2} In this section we summarize some of the basic definitions and notation for determinantal point processes and establish a few convergence results for the latter. We mention that some of the results in this section are well-known; however, many (especially the ones concerning convergence) are new. In order to clearly state and rigorously establish our new results we require a robust framework for determinantal point processes and we formulate one using \cite[Section 2]{J06} as a basis. Our framework slightly deviates from the ones in \cite{HKP} and \cite{Sosh00}, although we take the time to explain how it is consistent with both. Overall, we have attempted to have this section, and Appendix \ref{AppendixA} (where the more technical results are established), be a self-contained exposition of determinantal point processes and their convergence. To follow the results in this section and the appendix, a reader is essentially only expected to be familiar with \cite[Chapter 1]{Billing}, and a few statements from \cite{Cinlar} and \cite{Kall}, which will be recalled as needed. All of these are standard graduate probability textbooks.

%%%%%%%%%%%%%%%%%%%%%%%%%%%%%%%%%%%%%%%%%%%%%%%%%%%%%%%%%%%%%%%%%%%%%
%
%    Section 2.1
%
%%%%%%%%%%%%%%%%%%%%%%%%%%%%%%%%%%%%%%%%%%%%%%%%%%%%%%%%%%%%%%%%%%%%%
\subsection{Point processes}\label{Section2.1} Throughout this section we assume that $(E, \mathcal{E})$ is the space $\mathbb{R}^k$ with the usual topology and corresponding Borel $\sigma$-algebra for some fixed $k \in \mathbb{N}$. We also fix a probability space $(\Omega, \mathcal{F}, \mathbb{P})$.

We say that $M$ is a {\em random measure} on $(E, \mathcal{E})$ if $M: \Omega \times \mathcal{E} \rightarrow [0, \infty]$ is such that $M_{\omega} = M(\omega, \cdot)$ is a measure on $\mathcal{E}$ for each $\omega \in \Omega$ and $M(\cdot, A) = M(A)$ is an (extended) random variable for each $A \in \mathcal{E}$. We also define the {\em mean} of $M$ to be the measure $\mu$ on $(E,\mathcal{E})$ given by
$$\mu(A) = \mathbb{E}[M(A)] \mbox{ for } A\in \mathcal{E}.$$

We say that $M$ is {\em locally finite} or {\em locally bounded} if for each bounded Borel set $B \in \mathcal{E}$ we have that $M(B)$ is a real random variable (i.e. $M(\omega, B) < \infty$ for each $\omega \in \Omega$). In particular, if $M$ is locally finite, then $M_\omega \in \mathcal{M}_{E}$ -- the space of locally bounded measures on $E$ with the vague topology. We refer the reader \cite[Chapter 4]{Kall} for background on the space $\mathcal{M}_{E}$ and the vague topology. For brevity we denote $\mathcal{M}_E$ by $S$ and its Borel $\sigma$-algebra by $\mathcal{S}$. From  \cite[Chapter 4]{Kall} we have that $S$ is a Polish space and so a locally finite random measure $M$ can be viewed as a random element in $(S,\mathcal{S})$ in the sense of \cite[Section 3]{Billing}. 

We say that a random element $M$ in $(S,\mathcal{S})$ is a {\em point process} if $M(B) \in \mathbb{Z}_{\geq 0}$ for each $\omega \in \Omega$ and each bounded Borel set $B \in \mathcal{E}$. We further say that $M$ is a {\em simple point process} if it is a point process and $M(\omega, \{ x\}) \leq 1$ for all $x \in E$ and $\omega \in \Omega$. \\

Suppose that $(\bar{E}, \bar{\mathcal{E}})$ is a measurable space such that $E \subseteq \bar{E}$ and $\mathcal{E} \subseteq \bar{\mathcal{E}}$. We say that a sequence of random elements $\{X_n \}_{n \geq 1}$ in $(\bar{E}, \bar{\mathcal{E}})$, all defined on $(\Omega, \mathcal{F}, \mathbb{P})$, {\em forms} the random measure $M$ if 
\begin{equation}\label{FormX}
M(\omega, A) = \sum_{n \geq 1} {\bf 1} \{X_n(\omega) \in A\}
\end{equation}
for each $\omega \in \Omega$ and $A \in \mathcal{E}$. Unless otherwise specified, we will assume that $\bar{E} = E \cup \{\partial\}$ and $\bar{\mathcal{E}} = \sigma(\mathcal{E}, \{\partial\})$ where $\partial$ is an extra point we add to the space $E$. The addition of an extra point $\partial$ to $E$ allows us to form finite random measures. For example, if $\{X_n \}_{n = 1}^N$ is a finite sequence of random elements in $(E, \mathcal{E})$, then we say that it forms the random measure $M$ if (\ref{FormX}) holds with $X_n(\omega) = \partial$ for $n \geq N+1$. Since $A \in \mathcal{E}$, any $X_n(\omega) = \partial$ does not contribute to the sum in (\ref{FormX}).

The following result states that any point process takes the form (\ref{FormX}) for some sequence of random elements. Its proof is given in Appendix \ref{AppendixA2}.
\begin{lemma}\label{FormL} Suppose that $M$ is a point process on $(E, \mathcal{E})$. There exists a sequence of random elements $\{X_n \}_{n \geq 1}$ in $(\bar{E}, \bar{\mathcal{E}})$ that satisfy (\ref{FormX}).
\end{lemma}

Suppose that $M$ is a point process on $(E, \mathcal{E})$ and let $\{X_n \}_{n \geq 1}$ be as in Lemma \ref{FormL}. With this data we can construct for any $n \in \mathbb{N}$ a new point process $M_n$ on $(E^n, \mathcal{E}^{\otimes n})$ through
\begin{equation}\label{RM1}
M_n(\omega, A) = \sum_{\substack{i_1, \dots, i_n = 1 \\ |\{i_1, \dots, i_n\}| = n} }^\infty {\bf 1} \{ (X_{i_1}(\omega), \dots, X_{i_n} (\omega)) \in A) \} \mbox{ for } \omega \in \Omega \mbox{ and } A \in \mathcal{E}^{\otimes n}.
\end{equation}
We also denote the mean of $M_n$ by $\mu_n$ -- this is a measure on $(E^n, \mathcal{E}^{\otimes n})$. From (\ref{RM1}) we have that $M_n$ depends only on the measure $M$ (and not the particular definition and order of the variables $X_n$) and for each $\omega \in \Omega$ the measure $M_n(\omega, \cdot)$ is symmetric in the sense that 
$$M_n(\omega, B_1 \times \cdots \times B_n) = M_n(\omega, B_{\sigma(1)} \times \cdots \times B_{\sigma(n)})$$
for any permutation $\sigma \in S_n$ and Borel sets $B_1, \dots, B_n \in \mathcal{E}$. The mean $\mu_n$ is also symmetric. \\

We next show that for each $n_1, \dots, n_m \in \mathbb{N}$ and bounded pairwise disjoint $A_1, \dots, A_m \in \mathcal{E}$
\begin{equation}\label{RM3}
\mu_n( A_1^{n_1} \times \cdots \times A_m^{n_m}) = \mathbb{E} \left[ \prod_{j = 1}^m \frac{M(A_j)!}{(M(A_j) - n_j)!} \right].
\end{equation}
In (\ref{RM3}) both sides are allowed to be infinite, and we use the convention that $\frac{1}{r!} = 0$ for $r < 0$. We also mention that if $Z$ is a random variable taking value in $\mathbb{Z}_{\geq 0}$, then 
\begin{equation}\label{RM4}
\mathbb{E} \left[ \frac{Z!}{(Z-n)!}\right] = \mathbb{E}\left[ Z(Z-1) \cdots (Z- n +1) \right]
\end{equation}
is called the $n$-th {\em factorial moment}. In particular, (\ref{RM3}) shows that the joint (factorial) moments of the variables $M(A_1), \dots, M(A_m)$ are completely described by the mean measures $\mu_n$.

To see why (\ref{RM3}) holds set $U_i(\omega) = \{r \in \mathbb{N}: X_r(\omega) \in A_i\}$ for $i = 1, \dots, m$ and let $N_i(\omega) = |U_i(\omega)|$ be the cardinality of these random sets. From (\ref{RM1}) we get a contribution of one for every ordered tuple $(i_1, \dots, i_n)$ such that the first $n_1$ indices are in $U_1(\omega)$, the next $n_2$ indices are in $U_2(\omega)$ etc. The latter implies
$$M_n(\omega,A_1^{n_1} \times \cdots \times A_m^{n_m}) = \prod_{i = 1}^m \binom{N_i(\omega)}{n_i} \cdot n_i! = \prod_{j = 1}^m \frac{M(A_j)!}{(M(A_j) - n_j)!}, $$
where in the last equality we used that $N_i(\omega) = M(\omega, A_i)$ in view of (\ref{RM1}). Taking expectations on both sides of the above identity yields (\ref{RM3}).\\

In the remainder of this section we give sufficient conditions that ensure that a sequence of point processes $\{M^N\}_{N \geq 1}$ converge weakly to a point process $M^{\infty}$. The essential condition we require is that the mean measures $\mu^N_n$ of $M^N$, or equivalently by (\ref{RM3}) the joint factorial moments, converge on a rich enough family of sets, and that the limits do not grow too quickly in $n$. 

We start by defining the ``rich enough'' family of sets. If $T \subset \mathbb{R}$ is countable, we let
\begin{equation}\label{rect}
\mathcal{I}_T = \{ (a_1, b_1] \times \cdots \times (a_k, b_k]: a_i, b_i \not \in T \mbox{ for } i = 1, \dots, k \},
\end{equation}
which is the set of rectangles in $E$ whose corner coordinates avoid the set $T$. One readily observes that $\mathcal{I}_T$ is a {\em dissecting semi-ring}. We recall from \cite{Kall} that a family of sets $\mathcal{C}$ is {\em dissecting} if
\begin{itemize}
    \item every open $G \subseteq E$ is a countable union of sets in $\mathcal{C}$,
    \item every bounded Borel $B \in \mathcal{E}$ is covered by finitely many sets in $\mathcal{C}$,
\end{itemize}
and it is a {\em semi-ring} if
\begin{itemize}
    \item it is closed under finite intersections,
    \item any proper difference between sets in $\mathcal{C}$ is a finite disjoint union of sets in $\mathcal{C}$.
\end{itemize}
We next state the key convergence result for point processes that we use. It's proof is given in Appendix \ref{AppendixA3}, and we discuss the significance of the assumptions after the statement.
\begin{proposition}\label{PPConv} Let $T \subset \mathbb{R}$ be countable and $\mathcal{I}_T$ as in (\ref{rect}). Suppose that $\{M^N\}_{N \geq 1}$ is a sequence of point processes such that for each $n_1, \dots, n_m \in \mathbb{N}$ and pairwise disjoint $A_1, \dots, A_m \in \mathcal{I}_T$ the following limits all exist
\begin{equation}\label{RM11}
\lim_{N \rightarrow \infty} \mathbb{E} \left[ \prod_{i = 1}^m \frac{M^N(A_i)!}{(M^N(A_i) - n_i)!}\right] =: C(n_1, \dots, n_m; A_1, \dots, A_m) \in [0, \infty),
\end{equation}
where $n = n_1 + \cdots + n_m$. Then, we can find a point process $M^{\infty}$ and a countable set $T^{\infty}$ such that $T \subseteq T^{\infty} \subset \mathbb{R}$ and for all pairwise disjoint $A_1, \dots, A_m \in \mathcal{I}_{T^{\infty}}$
\begin{equation}\label{RM12}
\mathbb{E} \left[ \prod_{i = 1}^m \frac{M^{\infty}(A_i)!}{(M^{\infty}(A_i) - n_i)!}\right]= C(n_1, \dots, n_m; A_1, \dots, A_m).
\end{equation}
If we further suppose that for each $B \in \mathcal{I}_{T}$ there exists $\epsilon > 0$ such that
\begin{equation}\label{RM13}
\sum_{n = 1}^{\infty} \frac{\epsilon^n}{n!} \cdot C(n; B) < \infty,
\end{equation}
then $M^N$ converge weakly to a point process $M^{\infty}$, which satisfies (\ref{RM12}) for some countable $T^{\infty}$ containing $T$. In the last sentence the weak convergence is that of random elements in $(S, \mathcal{S})$.
\end{proposition}
\begin{remark} As the proof in Appendix \ref{AppendixA3} will reveal, (\ref{RM11}) ensures that the sequence $\{M^N\}_{N \geq 1}$ is tight and that each subsequential limit has a modification to a point process that satisfies (\ref{RM12}) for a suitable $T^{\infty}$. The condition (\ref{RM13}) should be viewed as a moment growth condition, which if satisfied ensures there is at most one point process that satisfies (\ref{RM12}). The latter together with tightness ensures that $M^{N}$ have a unique weak subsequential limit $M^{\infty}$ to which they need to converge.
\end{remark}

The conditions in Proposition \ref{PPConv} are chosen to handle quite general situations and we explain in a few simple examples how some pathological cases are covered. In the examples below $E = \mathbb{R}$.

Firstly, if one lets $M^N$ be deterministic measures with $M^N = \delta_{(-1)^N/N}$, then these measures converge vaguely to the measure $M^{\infty} = \delta_0$. Observe that $M^N((0,1]) = {\bf 1}\{ N \mbox{ is even} \}$. In particular, we see that the convergence in (\ref{RM11}) does not need to happen for all rectangles. This motivates the introduction of $\mathcal{I}_T$ that allows one to exclude certain pathological rectangles from having a limit.

If we instead let $M^N = \delta_{1/N}$, then as before these measures converge vaguely to the measure $M^{\infty} = \delta_0$. One can directly verify that the limit in (\ref{RM11}) exists for each choice of $A_1, \dots, A_m$ and $n_1, \dots, n_m$; however, $\mu_n^N((0,1]) = M^{N}((0,1]) = 1$ for each $N$, while $\mu_n^{\infty}((0,1]) = M^{\infty}((0,1]) = 0$. This shows that the equality in (\ref{RM12}) does not need to hold for all the rectangles for which the convergence in (\ref{RM11}) occurs. This motivates the introduction of $T^{\infty}$ and $\mathcal{I}_{T^{\infty}} \subseteq \mathcal{I}_{T}$.

We end this section with an easy corollary to Proposition \ref{PPConv}.
\begin{corollary}\label{CorUnique} Let $T \subset \mathbb{R}$ be countable and $\mathcal{I}_T$ as in (\ref{rect}). Suppose that $M', M''$ are two point processes such that for each $n_1, \dots, n_m \in \mathbb{N}$ and pairwise disjoint $A_1, \dots, A_m \in \mathcal{I}_T$ we have
\begin{equation}\label{RM15}
 \mathbb{E} \left[ \prod_{i = 1}^m \frac{M'(A_i)!}{(M'(A_i) - n_i)!}\right] =\mathbb{E} \left[ \prod_{i = 1}^m \frac{M''(A_i)!}{(M''(A_i) - n_i)!}\right],
\end{equation}
and also for each $B \in \mathcal{I}_T$ there exists $\epsilon > 0$ such that 
\begin{equation}\label{RM16}
\sum_{n = 1}^{\infty} \frac{\epsilon^n}{n!} \cdot \mathbb{E} \left[ \frac{M'(B)!}{(M'(B) - n)!}\right] < \infty.
\end{equation}
Then, $M'$ and $M''$ have the same distribution.
\end{corollary}
\begin{proof} Fix $A_1, \dots, A_m \in \mathcal{I}_T$ that are pairwise disjoint and let $B \in \mathcal{I}_{T}$ be sufficiently large so that $\cup_{i = 1}^m A_i \subseteq B$. Note that for $n = n_1 + \cdots + n_m$ and any point process $M$ we have from (\ref{RM3})
\begin{equation}\label{RM17}
\mathbb{E} \left[ \frac{M(B)!}{(M(B)-n)!} \right] = \mu_n(B^n) \geq \mu_n (A_1^{n_1} \times \cdots \times A_m^{n_m}) = \mathbb{E} \left[\prod_{i = 1}^m \frac{M(A_i)!}{(M(A_i) - n_i)!} \right],
\end{equation}
where $\mu_n$ denotes the mean measure of $M_n$ as earlier in the section. From (\ref{RM16}) and (\ref{RM17}) we conclude that 
$$C(n_1, \dots, n_m; A_1, \dots, A_m): = \mathbb{E} \left[ \prod_{i = 1}^m \frac{M'(A_i)!}{(M'(A_i) - n_i)!}\right] =\mathbb{E} \left[ \prod_{i = 1}^m \frac{M''(A_i)!}{(M''(A_i) - n_i)!}\right]$$
are all finite. Setting $M^N = M'$ if $N$ is odd and $M^N = M''$ if $N$ is even, we see that $\{M^N\}_{N \geq 1}$ satisfies the conditions of Proposition \ref{PPConv}. The latter implies that $M^N$ converge weakly to some $M^{\infty}$. Since along odd indices the limit is $M'$, and along even ones it is $M''$, and weak limits are unique we conclude that $M'$ has the same distribution as $M''$.
\end{proof}

%%%%%%%%%%%%%%%%%%%%%%%%%%%%%%%%%%%%%%%%%%%%%%%%%%%%%%%%%%%%%%%%%%%%%
%
%    Section 2.2
%
%%%%%%%%%%%%%%%%%%%%%%%%%%%%%%%%%%%%%%%%%%%%%%%%%%%%%%%%%%%%%%%%%%%%%
\subsection{Determinantal point processes}\label{Section2.2} In this section we introduce determinantal point processes. We start with an important definition.

\begin{definition}\label{CorrelationFunctions} Suppose that $\lambda$ is a locally finite measure on $(E, \mathcal{E})$, called a {\em reference measure}. Suppose that $M$ is a point process on $(E, \mathcal{E})$ and let $\mu_n$ be as in Section \ref{Section2.1}. If $\mu_n$ is absolutely continuous with respect to $\lambda^n$ (the product measure of $n$ copies of $\lambda$) on $E^n$, then we call its density $\rho_n$ the {\em $n$-th correlation function of $M$} (with respect to the reference measure $\lambda$). Specifically, we have that $\rho_n : E^n \rightarrow [0, \infty]$ is $\mathcal{E}^{\otimes n}$-measurable and satisfies for all $A \in \mathcal{E}^{\otimes n}$
\begin{equation}\label{RN1}
\mu_n(A) = \int_A \rho_n(x_1, \dots, x_n) \lambda^n(dx).
\end{equation} 
\end{definition}
\begin{remark}
Given a point process $M$ it is not necessary that functions $\rho_n$ satisfying (\ref{RN1}) exist. However, if they exist, then they are unique $\lambda^n$-a.e. -- this is a consequence of the uniqueness of Radon-Nikodym derivatives. In addition, we have that $\rho_n(x_1, \dots, x_n)$ are symmetric $\lambda^n$-a.e. as the measures $\mu_n$ are symmetric.
\end{remark}

\begin{remark}
If a point process $M$ has correlation functions $\rho_n$ for each $n \geq 1$, then from (\ref{RM3}) 
\begin{equation}\label{RN2}
 \mathbb{E} \left[ \prod_{j = 1}^m \frac{M(A_j)!}{(M(A_j) - n_j)!} \right] = \int_{A_1^{n_1} \times \cdots \times A_m^{n_m}}  \rho_n(x_1 ,\dots, x_n) \lambda^n(dx)
\end{equation}
for each $n_1, \dots, n_m \in \mathbb{N}$, $n = n_1 + \cdots + n_m$ and bounded pairwise disjoint $A_1, \dots, A_m \in \mathcal{E}$.
\end{remark}
Some authors take (\ref{RN2}) as the definition of correlation functions -- see e.g. \cite[Definition 2]{Sosh00}. The next statement, whose proof is given in Appendix \ref{AppendixA4}, shows that this definition is consistent with the one we had before.
\begin{lemma}\label{Cons} Let $M$ be a point process on $E$, $\mathcal{I} \subset \mathcal{E}$ a dissecting semi-ring, and $\lambda$ a locally finite measure on $(E,\mathcal{E})$. Let $u_n: E^n \rightarrow [0, \infty]$ be a $\lambda^n$-a.e. symmetric measurable function such that 
\begin{equation}\label{RN3}
 \mathbb{E} \left[ \prod_{j = 1}^m \frac{M(A_j)!}{(M(A_j) - n_j)!} \right] = \int_{A_1^{n_1} \times \cdots \times A_m^{n_m}}  u_n(x_1 ,\dots, x_n) \lambda^n(dx) < \infty
\end{equation}
for each $n_1, \dots, n_m \in \mathbb{N}$, $n = n_1 + \cdots + n_m$ and bounded pairwise disjoint $A_1, \dots, A_m \in \mathcal{I}$. Then, $u_n$ is the $n$-th correlation function of $M$ in the sense of Definition \ref{CorrelationFunctions}.
\end{lemma}

For the next statement, we will assume that $E = \mathbb{R}$ and $M$ is a point process on $\mathbb{R}$ that has correlation functions $\rho_n$ for each $n \geq 1$. We define 
\begin{equation}\label{DefLast}
X_{\mathsf{max}}(\omega) = \inf\{ t \in \mathbb{R} : M(\omega, (t,\infty)) = 0 \}.
\end{equation}
Observe that $X_{\mathsf{max}}$ is an extended random variable on $[-\infty, \infty]$. The following result provides a formula for the cumulative distribution function of $X_{\mathsf{max}}$.

\begin{proposition}\label{PropLast}\cite[Proposition 2.4]{J06} Assume the same notation as in the preceding paragraph. Suppose that for each $t \in \mathbb{R}$ we have that 
\begin{equation}\label{IntegrableLast}
1 + \sum_{n = 1}^{\infty} \frac{1}{n!}  \int_{(t, \infty)^n} \rho_n(x_1, \dots, x_n) \lambda^n(dx) < \infty.
\end{equation}
Then, for each $t \in \mathbb{R}$ we have 
\begin{equation} \label{cdf}
\begin{split}
&\mathbb{P} \left( X_{\mathsf{max}} \leq t \right) =  1 + \sum_{n = 1}^{\infty} \frac{(-1)^n}{n!}  \int_{(t, \infty)^n}\rho_n(x_1, \dots, x_n)\lambda^n(dx).
\end{split}
\end{equation}
\end{proposition}
\begin{remark} We mention that in general the second line in (\ref{cdf}) does not define a {\em probability} distribution function. The series in (\ref{cdf}) converges, is continuous in $t$, lies in $[0,1]$, is increasing and converges to $1$ when $t \rightarrow \infty$. The only condition that is not automatically satisfied is that the series converges to $0$ as $t \rightarrow - \infty$. By monotonicity this series {\em does} have a limit in $[0,1]$ as $t \rightarrow -\infty$, however, this limit may be strictly positive. In particular, if this limit is equal to $\alpha$, then $\mathbb{P}(X_{\mathsf{max}} = -\infty) = \alpha$ and this is the same as the $\mathbb{P}(  M(\omega, \mathbb{R} ) = 0) = \alpha$. Essentially, condition (\ref{IntegrableLast}) ensures that the extended random variable in (\ref{DefLast}) is bounded from above almost surely, but it does not prevent it to be $-\infty$ with some positive probability.
\end{remark}

We now turn to the main definition of the section.
\begin{definition}\label{DPP} Suppose that $M$ is a point process on $(E ,\mathcal{E})$ and $\lambda$ is a locally finite measure on $(E, \mathcal{E})$. Suppose further that $K: E \times E \rightarrow \mathbb{C}$ is a locally bounded measurable function, called the {\em correlation kernel}. We say that $M$ is a {\em determinantal point process} with correlation kernel $K$ and reference measure $\lambda$ if the $n$-th correlation function $\rho_n$ as in Definition \ref{CorrelationFunctions} exists for each $n \geq 1$ and 
\begin{equation}\label{DP1}
\rho_n(x_1, \dots, x_n) = \det \left[ K(x_i, x_j) \right]_{i, j = 1}^n.
\end{equation}
\end{definition}
\begin{remark}\label{DPR1}
We mention that one does not in general need to assume that $K$ is locally bounded, but it will be convenient for us and this condition holds for all of the examples we are interested in. In addition, we note that the equations (\ref{RN1}) define the correlation functions $\rho_n$ on the support of $\lambda^n$. In particular, the value of $K$ outside of the support of $\lambda$ is immaterial. In some of our applications, we have that $\lambda$ is supported on some lattice in $E$ and we only define $K$ on that set.
\end{remark}

The following proposition summarizes some of the basic properties of determinantal point processes. Its proof occupies the remainder of this section.
\begin{proposition}\label{PropLem} Suppose that $M$ is a determinantal point process on $(E,\mathcal{E})$ as in Definition \ref{DPP}. We have the following statements:
\begin{enumerate}
\item $M$ is a simple point process $\mathbb{P}$-almost surely.
\item Suppose $D \in \mathcal{E}$ is any Borel set and define 
\begin{equation}\label{RestrictEqn}
N(\omega, A)= M(\omega, A \cap D) \mbox{ for each $A \in \mathcal{E}$ and $\omega \in \Omega$}.
\end{equation}
Then, $N$ is a determinantal point process with kernel $\tilde{K}(x,y) = {\bf 1}\{x \in D\} \cdot K(x,y) \cdot {\bf 1}\{y \in D\}$ and reference measure $\lambda$. 
\item The law of $M$ (as a random element in the space of locally bounded measures $(S, \mathcal{S})$ as in Section \ref{Section2.1}) is uniquely determined by the correlation kernel $K$ and reference measure $\lambda$.
\item If $f: E \rightarrow \mathbb{C} \setminus \{0\}$ is such that $f(x)$ and $1/f(x)$ are locally bounded, then $M$ is also a determinantal point process with correlation kernel 
\begin{equation}\label{GaugeTrans}
\tilde{K}(x,y) = \frac{f(x)}{f(y)} \cdot K(x,y) \mbox{ or alternatively } \tilde{K}(x,y) = \frac{f(x)}{f(y)} \cdot K(y,x).
\end{equation}
\item Let $\phi: E \rightarrow E$ be a measurable bijection with a measurable inverse such that $\phi$ and $\phi^{-1}$ are locally bounded. Then, the pushforward measure $M \phi^{-1}$ is a determinantal point process with correlation kernel $\tilde{K}(x,y) = K(\phi^{-1}(x), \phi^{-1}(y))$ and reference measure $\lambda \phi^{-1}$. 
\item If $\tilde{\lambda} = c \cdot \lambda$ for some $c > 0$, then $M$ is a determinantal point process with correlation kernel $\tilde{K}(x,y) = c^{-1} \cdot K(x,y)$ and reference measure $\tilde{\lambda}$.
\end{enumerate}
\end{proposition}
\begin{remark}\label{PropRem} Most of the statements in Proposition \ref{PropLem} are straightforward and known. Some authors take it as part of the definition of determinantal point processes that they are simple, see e.g. \cite[Definition 4.2.1]{HKP}. For us this property comes out of Definition \ref{DPP}. The kernel transformation in (\ref{GaugeTrans}) is sometimes called a {\em gauge transformation}. The last two statements in Proposition \ref{PropLem} essentially state that the determinantal point process structure is maintained under changes of variables and rescaling the reference measure.
\end{remark}
\begin{proof} Throughout we assume that $\{X_n\}_{n \geq 1}$ is a sequence of random elements in $(\bar{E}, \bar{\mathcal{E}})$ that forms $M$ as in (\ref{FormX}). Its existence is ensured by Lemma \ref{FormL}. We also let $\mathcal{I}$ be the set of all rectangles as in (\ref{rect}) for $T = \emptyset$.\\

{\raggedleft Part (1).} Let $A_n = [-n, n]^k$ and note that by (\ref{RN1}) and (\ref{DP1}) we have 
$$ \mathbb{E}\left[ M_2(A_n^2 \cap \{x_1 = x_2\}) \right] =  \mu_2(A_n^2 \cap \{x_1 = x_2\}) = \int_{A_n^2 \cap \{x_1 = x_2\}} \det \left[ K(x_i, x_j) \right]_{i,j = 1}^2 \lambda^2(dx) = 0,$$
where we used that the determinant is zero when $x_1 = x_2$. Adding the above over $n \geq 1$ we conclude that $M_2(\{x_1 = x_2\}) = 0$ $\mathbb{P}$-almost surely. From (\ref{RM1}) we conclude that 
$$\mathbb{P}\left(X_{i_1} = X_{i_2} \in E \mbox{ for some } i_1 \neq i_2 \right) = 0,$$
which implies that $M$ is a simple point process $\mathbb{P}$-a.s.\\

{\raggedleft Part (2).} If $A_1, \dots, A_m \in \mathcal{I}$ are pairwise disjoint and $n_1, \dots, n_m \in \mathbb{N}$, we have from (\ref{RN2})
\begin{equation*}
\begin{split}
&\mathbb{E} \left[ \prod_{i = 1}^m \frac{N(A_i)!}{(N(A_i)-n_i)!} \right] = \mathbb{E} \left[ \prod_{i = 1}^m \frac{M(A_i \cap D)!}{(M(A_i \cap D)-n_i)!} \right] = \int_{(A_1 \cap D)^{n_1} \times \cdots \times (A_m \cap D)^{n_m}} \hspace{-15mm} \det \left[ K(x_i, x_j) \right]_{i,j = 1}^n \lambda^n(dx) \\
& = \int_{A_1^{n_1} \times \cdots \times A_m^{n_m}} \prod_{i = 1}^n {\bf 1}\{x_i \in D\} \cdot \det \left[ K(x_i, x_j) \right]_{i,j = 1}^n \lambda^n(dx) = \int_{A_1^{n_1} \times \cdots \times A_m^{n_m}} \hspace{-3mm} \det [ \tilde{K}(x_i, x_j) ]_{i,j = 1}^n \lambda^n(dx) ,
\end{split}
\end{equation*}
which by Lemma \ref{Cons} implies the statement in the proposition.\\

{\raggedleft Part (3).} If $A$ is an $n \times n$ matrix with complex entries and columns $v_1, \dots, v_n$, we have by Hadamard's inequality, see \cite[Corollary 33.2.1.1.]{Prasolov}, that 
\begin{equation}\label{Hadamard}
 |\det A| \leq \prod_{i = 1}^n \|v_i\|
\end{equation}
where $\|x\| = (|x_1|^2 + \cdots + |x_n|^2)^{1/2}$ for $x = (x_1, \dots, x_n) \in \mathbb{C}^n$. Using (\ref{Hadamard}) and that $K$ and $\lambda$ are locally bounded, we conclude that for each $B \in \mathcal{I}$ there is a constant $C > 0$ such that for $n \geq 1$
$$ \int_{B^n} \left| \det [ K(x_i, x_j) ]_{i,j = 1}^n \right| \lambda^n(dx) \leq n^{n/2} \cdot C^n.$$
Combining the latter with (\ref{RN2}), we see that (\ref{RM16}) is satisfied for $M' = M$, and hence the statement in the proposition follows from Corollary \ref{CorUnique}.\\

{\raggedleft Part (4).} The statement follows directly from Definition \ref{DPP} and the equality $\det [ K(x_i, x_j) ]_{i,j = 1}^n = \det [ \tilde{K}(x_i, x_j) ]_{i,j = 1}^n$ for each $n \geq 1$ and $x_1, \dots, x_n \in E$.\\

{\raggedleft Part (5).} Put $N = M \phi^{-1}$ and note that by our assumptions on $\phi$, we have that $N$ is a point process on $E$. Fix any pairwise disjoint bounded Borel sets $A_1, \dots, A_m \in \mathcal{E}$ and let $B_i = \phi^{-1}(A_i)$. Note that $B_1, \dots, B_m$ are also pairwise disjoint and bounded Borel sets by our assumptions on $\phi$. By definition we have $N(A_i) = M(B_i)$ for $i =1, \dots, m$ and so by changing variables we conclude for any $n_1, \dots, n_m \in \mathbb{N}$ that
\begin{equation*}
\begin{split}
&\mathbb{E} \left[ \prod_{i = 1}^m \frac{N(A_i)!}{(N(A_i)-n_i)!} \right] = \mathbb{E} \left[ \prod_{i = 1}^m \frac{M(B_i)!}{(M(B_i)-n_i)!} \right] = \int_{B_1^{n_1} \times \cdots \times B_m^{n_m}} \hspace{-1mm} \det \left[ K(x_i, x_j) \right]_{i,j = 1}^n \lambda^n(dx) \\
& = \int_{A_1^{n_1} \times \cdots \times A_m^{n_m}} \hspace{-2mm} 
 \det \left[ K(\phi^{-1}(y_i), \phi^{-1}(y_j)) \right]_{i,j = 1}^n (\lambda \phi^{-1})^n(dy) = \int_{A_1^{n_1} \times \cdots \times A_m^{n_m}}  \hspace{-2mm}  \det [ \tilde{K}(y_i, y_j) ]_{i,j = 1}^n \tilde{\lambda}^n(dy) ,
\end{split}
\end{equation*}
where $\tilde{\lambda} = \lambda \phi^{-1}$ and $\tilde{K}(x,y) = K(\phi^{-1}(x), \phi^{-1}(y))$. The statement in the proposition now follows from Lemma \ref{Cons} applied to the dissecting semi-ring of all bounded Borel sets.\\

{\raggedleft Part (6).} For each pairwise disjoint $A_1, \dots, A_m \in \mathcal{I}$ and $n_1, \dots, n_m \in \mathbb{N}$ we have from (\ref{RN2}), the definition of $\tilde{K}$, $\tilde{\lambda}$, and the multi-linearity of the determinant that
$$\mathbb{E} \left[ \prod_{i = 1}^m \frac{M(A_i)!}{(M(A_i)-n_i)!} \right] = \int_{A_1^{n_1} \times \cdots \times A_m^{n_m}}  \det \left[ \tilde{K}(x_i, x_j) \right]_{i,j = 1}^n \tilde{\lambda}^n(dx).$$
The last equality implies the statement in the proposition in view of Lemma \ref{Cons}.
\end{proof}

%%%%%%%%%%%%%%%%%%%%%%%%%%%%%%%%%%%%%%%%%%%%%%%%%%%%%%%%%%%%%%%%%%%%%
%
%    Section 2.3
%
%%%%%%%%%%%%%%%%%%%%%%%%%%%%%%%%%%%%%%%%%%%%%%%%%%%%%%%%%%%%%%%%%%%%%
\subsection{Convergence of determinantal point processes}\label{Section2.3} In this section we investigate various questions of convergence for a sequence of determinantal point processes. We start with a result that ensures the weak convergence of a sequence of determinantal point processes under the assumption that their kernels converge uniformly over compact sets to a continuous kernel, and their reference measures converge vaguely. Its proof is given in Appendix \ref{AppendixA45}.

\begin{proposition}\label{PropWC0} Let $\lambda_N$ be a sequence of locally finite measures on $E$ that converge vaguely to $\lambda$, and let $M^N$ be a sequence of determinantal point processes with correlation kernels $K_N$ and reference measures $\lambda_N$. Suppose further that there is a function $K: E \times E \rightarrow \mathbb{C}$ such that for each compact set $\mathcal{V} \subset E$ we have
\begin{equation}\label{EWC0}
\lim_{N \rightarrow \infty}  \sup_{x,y \in \mathcal{V}} \left| K_N(x,y) -  K(x,y) \right| = 0,
\end{equation}
and such that $K(x,y)$ is continuous on $E^2$. Then, there exists a determinantal point process $M$ on $E$ with correlation kernel $K$ and reference measure $\lambda$. Furthermore, $M^N$ converge weakly to $M$ as random elements in $(S,\mathcal{S})$.
\end{proposition}

In the remainder of this section we consider a specific setup, which is suitable for the applications we have in mind, and turn to explaining that first.
\begin{definition}\label{DefSlices}
We suppose that $E =\mathbb{R}^2$, $r \in \mathbb{N}$, $t_1, \dots, t_r \in \mathbb{R}$ are fixed with $t_1 < \cdots < t_r$ and put $\mathcal{T} = \{t_1, \dots, t_r\}$. if $\nu = (\nu_{t_1}, \dots, \nu_{t_r})$ is an $r$-tuple of locally finite measures on $\mathbb{R}$, we define the (necessarily locally bounded) measure $\mu_{\mathcal{T}, \nu}$ on $\mathbb{R}^2$ by
\begin{equation}\label{ERG1}
\mu_{\mathcal{T}, \nu} = \sum_{t \in \mathcal{T}} \nu_t(A_t) \mbox{, where } A_t = \{y \in \mathbb{R}: (t,y) \in A\}.
\end{equation}
Below we consider determinantal point processes on $E$ with reference measures of the form $\mu_{\mathcal{T}, \nu}$. Note that that only the values of the kernels on $ (\mathcal{T} \times \mathbb{R}) \times (\mathcal{T} \times \mathbb{R}) $ matter as $\mathcal{T} \times \mathbb{R}$ contains the support of the reference measure $\mu_{\mathcal{T}, \nu}$, cf. Remark \ref{DPR1}. In many cases we only focus on the values on this set,  but for concreteness, one can extend the kernels to be equal to $0$ outside of it.
\end{definition}

The following result shows that the measures in Definition \ref{DefSlices} behave well under projections to a fixed $t \in \mathcal{T}$.
\begin{lemma}\label{LemmaSlice} Assume the same notation as in Definition \ref{DefSlices}, and let $M$ be a determinantal point process with correlation kernel $K$ and reference measure $\mu_{\mathcal{T}, \nu}$. Fix $t \in \mathcal{T}$ and consider the random measure on $\mathbb{R}$, given by $M^t(A) := M(\{t\} \times A)$. Then, $M^t$ is a determinantal point process on $\mathbb{R}$ with correlation kernel $K_t(x,y) = K(t,x; t,y)$ and reference measure $\nu_t$.
\end{lemma}
\begin{proof} Using the definition of $M^t$ and (\ref{RN2}) we have for all pairwise disjoint bounded Borel sets $A_1, \dots, A_h $ in $\mathbb{R}$ and $n_1, \dots, n_h \geq 0$ with $n_1 + \cdots + n_h = n$ that
\begin{equation*}
\mathbb{E} \left[ \prod_{j = 1}^h \frac{M^t(A_j)!}{(M^t(A_j) - n_j)!} \right] = \int_{A_1^{n_1} \times \cdots \times A_h^{n_h}} \det \left[ K(t, x_i; t, x_j) \right]_{i,j = 1}^n  \nu_t^n(dx).
\end{equation*}
The result now follows from the last equation and Lemma \ref{Cons}, where $\mathcal{I}$ is the collection of bounded Borel sets in $\mathbb{R}$.
\end{proof}

The next result provides sufficient conditions for a sequence of determinantal point processes as in Definition \ref{DefSlices} to converge weakly. Its proof is given in Appendix \ref{AppendixA5}.
\begin{proposition}\label{PropWC1}Assume the same notation as in Definition \ref{DefSlices}. Let $\nu^N = (\nu^N_{t_1}, \dots, \nu^N_{t_r})$ be a sequence of $r$-tuples of locally finite measures on $\mathbb{R}$ that converge vaguely to $\nu = (\nu_{t_1}, \dots, \nu_{t_r})$, and let $M^N$ be a sequence of determinantal point processes with correlation kernels $K_N$ and reference measures $\mu_{\mathcal{T}, \nu^N}$. Suppose further that there is a function $K: (\mathcal{T} \times \mathbb{R}) \times (\mathcal{T} \times \mathbb{R}) \rightarrow \mathbb{C}$ such that for each $A > 0$ we have
\begin{equation}\label{DP2}
\lim_{N \rightarrow \infty} \max_{s,t \in \mathcal{T}} \sup_{- A \leq x,y \leq A} \left| K_N(s, x; t ,y) -  K(s, x; t ,y) \right| = 0,
\end{equation}
and such that $K(s,\cdot; t, \cdot)$ is continuous on $\mathbb{R}^2$ for each $s,t \in \mathcal{T}$. Then, there exists a determinantal point process $M$ on $E$ with correlation kernel $K$ and reference measure $\mu_{\mathcal{T}, \nu}$. Furthermore, $M^N$ converge weakly to $M$ as random elements in $(S,\mathcal{S})$.
\end{proposition}

We suppose next that $X^N = \left(X_i^{j, N}: i \geq 1 \mbox{ and } j = 1,\dots, r\right)$ is a sequence of random elements in $(\mathbb{R}^{\infty}, \mathcal{R}^{\infty})$, see \cite[Example 1.2]{Billing}, such that 
\begin{equation}\label{OrderedX}
X_i^{j,N}(\omega) \geq X_{i+1}^{j,N}(\omega) \mbox{ for each } \omega \in \Omega,  i \geq 1 \mbox{ and } j = 1,\dots, r.
\end{equation}
We further suppose that $t_1 < \dots < t_r$ and the random measures $M^N$ on $\mathbb{R}^2$ defined by
\begin{equation}\label{FormedM}
M^N(\omega, A) = \sum_{i \geq 1} \sum_{j = 1}^r {\bf 1} \{ (t_j, X_i^{j, N}(\omega)) \in A \}
\end{equation}
are locally finite, and hence point processes. The following statement shows that if $M^N$ converge weakly and $X^N$ is tight, then the elements $X^N$ in fact converge weakly. 
\begin{proposition}\label{PropWC2} Assume the same notation as in the preceding paragraph. Suppose that for each $i \geq 1$ and $j \in \{1, \dots, r\}$ the sequence $\{X_i^{j,N}\}_{N \geq 1}$ is tight and that $M^N$ as in (\ref{FormedM}) converge weakly to a point process $M$. Then, the sequence $X^N$ converges weakly to some $X$, which also satisfies (\ref{OrderedX}), as random elements in $(\mathbb{R}^{\infty}, \mathcal{R}^{\infty})$. Moreover, the random measure
\begin{equation}\label{FormedM2}
 \tilde{M}(\omega, A) = \sum_{i \geq 1} \sum_{j = 1}^r {\bf 1} \{ (t_j, X_i^{j}(\omega)) \in A \}
\end{equation}
has the same distribution as $M$.
\end{proposition}
\begin{proof} We note that the tightness of $\{X_i^{j,N}\}_{N \geq 1}$ for each $i,j$ implies tightness of $X^N$ as random elements in $(\mathbb{R}^{\infty}, \mathcal{R}^{\infty})$. The latter is an easy consequence of Tychonoff's theorem, see \cite[Theorem 37.3]{Munkres}. From the converse part of Prohorov's theorem, see \cite[Theorem 5.2]{Billing}, we know that $M^N$ being convergent implies that it is tight. In the last statement we used that $S$ is a Polish space. Consider the sequence of pairs $(X^N, M^N)$, which we can regard as random elements in $(\mathbb{R}^{\infty} \times S, \mathcal{R}^{\infty} \otimes \mathcal{S})$. Note that since $\mathbb{R}^{\infty}$ and $S$ are Polish spaces, the same is true for $\mathbb{R}^{\infty} \times S$ and also the Borel $\sigma$-algebra on $\mathbb{R}^{\infty} \times S$ agrees with the product one, see \cite[M10]{Billing}. Finally, the individual tightness of $X^N$ and $M^N$ implies that of $(X^N,M^N)$. From the direct part of Prohorov's theorem, see \cite[Theorem 5.1]{Billing}, we conclude that $(X^N,M^N)$ is relatively compact.

Let $(X^{N_v},M^{N_v})$ be a subsequence weakly converging to some $(X^{\infty}, M^{\infty})$. By Skorohod's representation theorem, see \cite[Theorem 6.7]{Billing}, we may assume that $(X^{N_v},M^{N_v})$ and $(X^{\infty}, M^{\infty})$ are defined on the same probability space $(\Omega, \mathcal{F}, \mathbb{P})$ and for each $\omega \in \Omega$ we have
\begin{equation}\label{PO1}
M^{N_v}(\omega) \xrightarrow{v} M^{\infty}(\omega) \mbox{ and } X^{j,N_v}_i(\omega) \rightarrow X^{j,\infty}_i(\omega).
\end{equation}
We mention that the application of \cite[Theorem 6.7]{Billing} is justified as $\mathbb{R}^{\infty} \times S$ is Polish and we also note that $M^{\infty}$ has the same distribution as $M$.

Our first task is to show that $X^{\infty}$ satisfies the conditions of the proposition. The coordinate-wise convergence of $X^N(\omega)$ to $X^{\infty}(\omega)$ implies that (\ref{OrderedX}) is satisfied with $N = \infty$. Define for $j \in \{ 1, \dots, r\}$, $a < b$ and $v \in \mathbb{N}$
\begin{equation}\label{PO2}
F_v(j,a, b) = M^{N_v}(\omega)(\{ t_j \} \times (a,b)) \mbox{, and also } F_{\infty}(j,a, b) =  M^{\infty}(\omega)(\{ t_j \} \times (a,b)).
\end{equation}
Note that since $F_{\infty}(j,a,b) \in \mathbb{Z}_{\geq 0}$ for each real $a < b$, we have from the vague convergence in (\ref{PO1}) and \cite[Lemma 4.1]{Kall} that there exists $H(j, a,b) \in \mathbb{Z}_{\geq 0}$ (depending on $\omega$) such that 
\begin{equation}\label{PO3}
H(j,a,b) \geq F_v(j,a, b) \mbox{ for all } v \in \mathbb{N} \cup \{\infty\}.
\end{equation}
In addition, from the pointwise convergence in (\ref{PO1}) and the inequalities in (\ref{OrderedX}) we can find $B \in \mathbb{R}$ (depending on $\omega$) such that 
\begin{equation}\label{PO4}
B > X^{j, N_v}_i(\omega) \mbox{ for all } v \in \mathbb{N} \cup \{\infty\}, i \geq 1, j \in \{ 1, \dots, r\}.
\end{equation}
From (\ref{PO3}) and (\ref{PO4}), and the definition of $M^N$ in (\ref{FormedM}), we see that for any $b > B > a$, $j \in \{1, \dots, r\}$ and $v \in \mathbb{N}$
\begin{equation}\label{PO5}
X_{H(j,a,b) + 1}^{j, N_v}(\omega) \leq a, \mbox{ and hence } X_{H(j,a,b) + 1}^{j, \infty}(\omega) \leq a.
\end{equation}
Equations (\ref{PO4}) and (\ref{PO5}) show that $\tilde{M}^{\infty}(\omega)$ as in (\ref{FormedM2}) with $X$ replaced with $X^{\infty}$ is locally finite and hence a point process. 

We next show that the measure $\tilde{M}^{\infty}(\omega)$ is equal to $M^{\infty}(\omega)$. Fix a continuous function $g$ on $\mathbb{R}^2$ of compact support. Pick $b > B > a$ with $b$ sufficiently large and $a$ sufficiently small so that the support of $g$ is contained in $\mathbb{R} \times (a,b)$. If $H = \max_{j = 1, \dots, r} H(j,a,b)$, we note by the pointwise convergence in (\ref{PO1}) that
$$ \lim_{v \rightarrow \infty} \sum_{i = 1}^H \sum_{j = 1}^r g(t_j, X_i^{j,N_v}(\omega)) = \sum_{i = 1}^H \sum_{j = 1}^rg(t_j, X_i^{j,\infty}(\omega)) = \sum_{i = 1}^{\infty} \sum_{j = 1}^rg(t_j, X_i^{j,\infty}(\omega)) = \tilde{M}^{\infty}(\omega) g, $$
where the middle equality used (\ref{PO5}) and that $g$ vanishes outside of $\mathbb{R} \times (a,b)$. On the other hand, by (\ref{PO5}) and the vague convergence in (\ref{PO1}) we get
$$\lim_{v \rightarrow \infty} \sum_{i = 1}^H \sum_{j = 1}^r g(t_j, X_i^{j,N_v}(\omega)) = \lim_{v \rightarrow \infty} \sum_{i = 1}^{\infty} \sum_{j = 1}^r g(t_j, X_i^{j,N_v}(\omega))  =  \lim_{v \rightarrow \infty} M^{N_v}(\omega) g = M^{\infty}(\omega)g.$$
The last two equations show that $\tilde{M}^{\infty}(\omega) = M^{\infty}(\omega)$. The last equality and (\ref{PO4}) show for each $K \in \mathbb{N}$ and $a_{i,j} \in \mathbb{R}$
\begin{equation*}\label{PO6}
\begin{split}
&\{\omega \in \Omega: X^{\infty, j}_{i}(\omega) \leq a_{i,j} \mbox{ for } i = 1, \dots, K \mbox{ and } j = 1, \dots, r\} \\
& = \{\omega \in \Omega: M^{\infty}(\omega)(\{t_j\} \times (a_{i,j}, \infty)) < i \mbox{ for } i = 1, \dots, K \mbox{ and } j = 1, \dots, r\}.
\end{split}
\end{equation*}
Since $M^{\infty}$ has the same law as $M$ we conclude
\begin{equation}\label{PO6}
\begin{split}
&\mathbb{P}\left(X^{\infty, j}_{i}\leq a_{i,j} \mbox{ for } i = 1, \dots, K \mbox{ and } j = 1, \dots, r \right) \\
& =\mathbb{P}\left( M(\{t_j\} \times (a_{i,j}, \infty)) < i \mbox{ for } i = 1, \dots, K \mbox{ and } j = 1, \dots, r \right).
\end{split}
\end{equation}

At this point most of the work is done, and we just need to wrap up the weak convergence part of the proposition. Let $\hat{X}^{\infty}$ be a subsequential limit of $\{X^N\}_{N \geq 1}$ and let $\{ X^{R_v}\}_{v \geq 1}$ be a subsequence converging weakly to $\hat{X}^{\infty}$. From our work above we know that $(X^{R_v}, M^{R_v})$ is tight in $\mathbb{R}^{\infty} \times S$, and by possibly passing to a subsequence, which we continue to call $R_v$, we may assume that $(X^{R_v}, M^{R_v})$ converge weakly to $(\hat{X}^{\infty}, \hat{M}^{\infty})$. From (\ref{PO6}) we conclude that $\hat{X}^{\infty}$ has the same finite-dimensional distribution as $X^{\infty}$. Since finite-dimensional sets form a separating class in $\mathbb{R}^{\infty}$, see \cite[Example 1.2]{Billing}, we conclude that $\hat{X}^{\infty}$ has the same law as $X^{\infty}$. The last argument shows that $\{X^N\}_{N \geq 1}$ has at most one subsequential limit. By tightness (and hence relative compactness), it has at least one subsequential limit, which implies that $\{X^N\}_{N \geq 1}$ has exactly one subsequential limit to which the sequence weakly converges. Calling this limit $X$, we may pass to a subsequence so that $(X^{N_v},M^{N_v})$ jointly weakly converge to $(X^{\infty}, M^{\infty})$ as above. As weak limits are unique, we have that $X^{\infty}$ has the same law as $X$ and $\tilde{M}^{\infty}$ has the same law as $\tilde{M}$. Our earlier work now shows that $X$ satisfies (\ref{OrderedX}) and $\tilde{M}$ has the same law as $M$.
\end{proof}

We have the following immediate consequence to Proposition \ref{PropWC2}.
\begin{corollary}\label{CorWC2} Fix $r \in \mathbb{N}$, and $t_1, \dots, t_r \in \mathbb{R}$ with $t_1 < \dots < t_r$. Suppose that
$$X = \left(X_i^{j}: i \geq 1 \mbox{ and } j = 1,\dots, r\right), \hspace{3mm} Y = \left(Y_i^{j}: i \geq 1 \mbox{ and } j = 1,\dots, r\right)$$
are two random elements in $(\mathbb{R}^{\infty}, \mathcal{R}^{\infty})$, such that 
\begin{equation}\label{OrderedXY}
 X_i^{j}(\omega) \geq X_{i+1}^{j}(\omega), \mbox{ and } Y_i^{j}(\omega) \geq Y_{i+1}^{j}(\omega)  \mbox{ for each } \omega \in \Omega,  i \geq 1 \mbox{ and } j = 1,\dots, r.
\end{equation}
Suppose further that the random measures $M^X, M^Y$ on $\mathbb{R}^2$ defined by
\begin{equation}\label{FormedMXY}
M^X(\omega, A) = \sum_{i \geq 1} \sum_{j = 1}^r {\bf 1} \{ (t_j, X_i^{j}(\omega)) \in A \} \mbox{ and } M^Y(\omega, A) = \sum_{i \geq 1} \sum_{j = 1}^r {\bf 1} \{ (t_j, Y_i^{j}(\omega)) \in A \}
\end{equation}
are locally finite, and hence point processes. Then, $M^X$ has the same law as $M^Y$ as random elements in $(S,\mathcal{S})$ if and only if $X$ and $Y$ have the same law as random elements in $(\mathbb{R}^{\infty}, \mathcal{R}^{\infty})$, i.e. they have the same finite-dimensional distributions. 
\end{corollary}
\begin{proof} Suppose first that $M^X$ and $M^Y$ have the same law. Let $X^N$ be the sequence of random elements in $(\mathbb{R}^{\infty}, \mathcal{R}^{\infty})$ such that $X^N = X$ for odd $N$ and $X^N = Y$ for even $N$. We see that $X^N$ satisfies the conditions of Proposition \ref{PropWC2}. The latter implies that $X^N$ converge weakly to some $X^{\infty}$ as random elements in $(\mathbb{R}^{\infty}, \mathcal{R}^{\infty})$. Since along odd indices the limit is $X$, and along even indices it is $Y$, and weak limits are unique, we conclude that $X$ has the same distribution as $Y$.

Suppose now that $X$ and $Y$ have the same law. If $f_1, \dots, f_m$ are continuous functions of compact support on $\mathbb{R}^2$, we have 
for each $n \in \mathbb{N}$ the equality in law of random vectors in $\mathbb{R}^m$ 
$$\left( \sum_{i = 1}^n \sum_{j = 1}^r f_1(t_j, X_i^j), \dots, \sum_{i = 1}^n \sum_{j = 1}^r f_m(t_j, X_i^j) \right) \overset{\mathrm{Law}}{=} \left( \sum_{i = 1}^n \sum_{j = 1}^r f_1(t_j, Y_i^j), \dots, \sum_{i = 1}^n \sum_{j = 1}^r f_m(t_j, Y_i^j) \right).$$
Since we assumed that $M^X$, $M^Y$ are locally finite, we see that the above vectors converge almost surely as $n \rightarrow \infty$ and we get
$$\left( \int_{\mathbb{R}^2} f_1(x) M^X(dx), \dots, \int_{\mathbb{R}^2} f_m(x) M^X(dx) \right) \overset{\mathrm{Law}}{=} \left( \int_{\mathbb{R}^2} f_1(x) M^Y(dx), \dots, \int_{\mathbb{R}^2} f_m(x) M^Y(dx) \right).$$
The latter and \cite[Lemma 4.7]{Kall} imply that $M^X$ and $M^Y$ have the same law.
\end{proof}

Before we state our next result we explain briefly how it relates to the ones we already presented. Later in the paper we will consider certain sequences of determinantal point processes $M^N$ of the form (\ref{FormedM}) for some sequence of random elements $X^N$. These random measures are purely atomic and the {\em locations} of their atoms are precisely the points of the form $(t_j, X_i^{j,N}(\omega))$ for $i \geq 1$ and $j = 1,\dots, r$. We will be interested in showing that the locations themselves converge weakly (in $\mathbb{R}^{\infty}$) as $N \rightarrow \infty$. These determinantal point processes will have correlation kernels $K_N$ for which we will be able to show (\ref{DP2}). From Proposition \ref{PropWC1} this will imply that $M^N$ converge weakly to some determinantal point process $M$. The fact that the point processes converge does not {\em a priori} imply that the atom locations (i.e. the $X^N$) weakly converge. The essence of Proposition \ref{PropWC2} is that the $X^N$ {\em also} need to converge if in addition to $M^N \Rightarrow M^{\infty}$ we also assume that the variables $\{X_i^{j, N}\}_{N \geq 1}$ are tight (in $N$) for each fixed $i \geq 1$ and $j \in \{1, \dots, r\}$. We are thus led to the question of how to ensure tightness of  $\{X_i^{j, N}\}_{N \geq 1}$ under the assumption that $M^N \Rightarrow M^{\infty}$. 

We can at this point project the entire problem from $E = \mathbb{R}^2$ to $\{t_j\} \times \mathbb{R} \cong \mathbb{R}$ for some fixed $j \in \{1, \dots, r\}$ using Lemma \ref{LemmaSlice}. In this case, the question becomes to find a tightness criterion for sequences of random elements in $\mathbb{R}^{\infty}$, which form point processes $M^N$ in $\mathbb{R}$ that are known to converge weakly. This is the essence of the following result.
\begin{proposition}\label{TightnessCrit} Suppose that $X^N = (X_i^N: i \geq 1)$ is a sequence of random elements in $\mathbb{R}^{\infty}$ such that $X_1^N(\omega) \geq X_2^N(\omega) \geq \cdots$. Denote by
\begin{equation}\label{SDForm}
M^N (\omega, A) = \sum_{i \geq 1} {\bf 1}\{ X_i^N (\omega) \in A\}
\end{equation}
the corresponding random measures and suppose that $M^N$ are point processes on $\mathbb{R}$. Assume that
\begin{enumerate}
\item  $M^N$ converge weakly to a point process $M$ on $\mathbb{R}$ as $N \rightarrow \infty$;
\item $\mathbb{P}(M( \mathbb{R}) = \infty) = 1$;
\item $\lim_{a \rightarrow \infty} \limsup_{N \rightarrow \infty} \mathbb{P}(X_1^N \geq a) = 0$.
\end{enumerate}
Then $\{X^N\}_{N \geq 1}$ forms a tight sequence of random elements in $\mathbb{R}^{\infty}$. 
\end{proposition} 
\begin{remark}\label{TightCritR} In words, Proposition \ref{TightnessCrit} says that given a sequence of random elements $X^N$ in $(\mathbb{R}^{\infty}, \mathcal{R}^{\infty})$ that form point processes $M^N$ as in (\ref{SDForm}) one can conclude tightness of $\{X^N\}_{N \geq 1}$ if one assumes convergence of the point processes, that the limiting point process has infinitely many points and that the rightmost particle locations $X^N_1$ are bounded from above with high probability.
\end{remark}
\begin{proof} We can think of $X^N$ as random elements in $[-\infty, \infty)^{\infty}$, which is homeomorphic to $[0,\infty)^{\infty} \subset \mathbb{R}^{\infty}$ under the coordinate-wise map $x \rightarrow e^x$. Since $\mathbb{R}^{\infty}$ is Polish, we conclude the same for $[-\infty, \infty)^{\infty}$. Since $X_1^N(\omega) \geq X_2^N(\omega) \geq \cdots$, we see that condition (3) implies that $X^N$ is a tight sequence in $[-\infty,\infty)^{\infty}$. Suppose that $X^{N_v}$ is a subsequence that weakly converges to some $\hat{X}^{\infty}$. The statement of the proposition would follow if we can show for each $n \in \mathbb{N}$
\begin{equation}\label{QO1}
\mathbb{P}(\hat{X}_n^{\infty} = -\infty) = 0.
\end{equation}

Arguing as in the proof of Proposition \ref{PropWC2} we may pass to a subsequence, which we still call $N_v$, such that $(X^{N_v}, M^{N_v})$ converge weakly to some $(X^{\infty}, M^{\infty})$ -- a random element in $[-\infty, \infty)^{\infty} \times S$. Furthermore, by Skorohod's representation theorem, we may assume that $(X^{N_v}, M^{N_v})$ and $(X^{\infty}, M^{\infty})$ are all defined on the same probability space $(\Omega, \mathcal{F}, \mathbb{P})$ and for each $\omega \in \Omega$ 
\begin{equation}\label{QO2}
M^{N_v}(\omega) \xrightarrow{v} M^{\infty}(\omega) \mbox{ and } X^{N_v}_i(\omega) \rightarrow X^{\infty}_i(\omega) \in [-\infty, \infty).
\end{equation}
Since weak limits are unique, we observe that $\hat{X}^{\infty}$ has the same distribution as $X^{\infty}$ and $M^{\infty}$ has the same distribution as $M$ (as in the statement of the proposition).

From the pointwise convergence in (\ref{QO2}) we conclude that 
\begin{equation}\label{QO3}
X_1^{\infty}(\omega) \geq X_2^{\infty}(\omega) \geq \cdots, 
\end{equation}
and that we can find $B > 0$, depending on $\omega$, such that 
\begin{equation}\label{QO4}
B > X_i^{N_v}(\omega) \mbox{ for all } v \in \mathbb{N} \cup \{\infty\} \mbox{ and } i \in \mathbb{N}.
\end{equation}
Let $U$ be the set of $\omega$ such that $M^{\infty}(\omega) (\mathbb{R}) = \infty$. From condition (2) we know that $\mathbb{P}(U) = 1$. Fixing $n \in \mathbb{N}$ and $\omega \in U$, we can find $a < B < b$ with $a$ small and $b$ large enough so that 
\begin{equation*}
M^{\infty}(\omega)((a,b)) \geq n.
\end{equation*}
From \cite[Lemma 4.1]{Kall}, and the vague convergence in (\ref{QO2}) we can find $v_0 \in \mathbb{N}$ (depending on $\omega$) such that for $v \geq v_0$ we have 
$$M^{N_v}(\omega)((a,b)) \geq n.$$
Combining the last statement with (\ref{QO4}) and the identity (\ref{SDForm}), we conclude for all $\omega \in U$ and $v \geq v_0$ that  
$X_n^{N_v}(\omega) \geq a,$
which by the pointwise convergence in (\ref{QO2}) implies 
$$U \cap \{X^{\infty}_n(\omega) = -\infty\} = \emptyset.$$
The last equation, the fact that $\mathbb{P}(U) = 1$ and that $X^{\infty}$ has the same law as $\hat{X}^{\infty}$ imply (\ref{QO1}).
\end{proof}

%%%%%%%%%%%%%%%%%%%%%%%%%%%%%%%%%%%%%%%%%%%%%%%%%%%%%%%%%%%%%%%%%%%%%
%
%    Section 3
%
%%%%%%%%%%%%%%%%%%%%%%%%%%%%%%%%%%%%%%%%%%%%%%%%%%%%%%%%%%%%%%%%%%%%%
\section{Schur processes}\label{Section3} In this section we introduce certain measures on sequences of partitions, which are special cases of the Schur processes in \cite{OR03}. In Section \ref{Section3.1} we introduce our measures following the notation in \cite{BR05}, and in Section \ref{Section3.2} we show that certain point processes associated with these measures are determinantal with correlation kernels that have a double contour integral form, see Proposition \ref{PropCK}. In Section \ref{Section3.3} we explain how one can relate our measures to geometric last passage percolation, and we utilize this connection in Section \ref{Section3.4} to prove a certain monotone coupling between our measures with different sets of parameters.

%%%%%%%%%%%%%%%%%%%%%%%%%%%%%%%%%%%%%%%%%%%%%%%%%%%%%%%%%%%%%%%%%%%%%
%
%    Section 3.1
%
%%%%%%%%%%%%%%%%%%%%%%%%%%%%%%%%%%%%%%%%%%%%%%%%%%%%%%%%%%%%%%%%%%%%%
\subsection{Definition and core notation}\label{Section3.1} We start by recalling Schur symmetric polynomials and some of their basic properties. Our exposition follows parts of \cite[Chapter I]{Mac} and we refer the interested reader to the latter for more details.

A {\em partition} is a sequence $\lambda = (\lambda_1, \lambda_2,\dots)$ of non-negative integers such that $\lambda_1 \geq \lambda_2 \geq \cdots$ and all but finitely many elements are zero. We denote the set of all partitions by $\mathbb{Y}$. The {\em length} $\ell (\lambda)$ is the number of non-zero $\lambda_i$ and the {\em weight} is given by $|\lambda| = \lambda_1 + \lambda_2 + \cdots$ . There is a single partition of weight $0$, which we denote by $\emptyset$. A {\em Young diagram} is a graphical representation of a partition $\lambda$, with $\lambda_1$ left justified boxes in the top row, $\lambda_2$ in the second row and so on. In general, we do not distinguish between a partition $\lambda$ and the Young diagram representing it. Given two diagrams $\lambda$ and $\mu$ such that $\mu \subseteq \lambda$ (as a collection of boxes), we call the difference $\theta = \lambda - \mu$ a {\em skew Young diagram}. A skew Young diagram $\theta$ is a {\em horizontal $m$-strip} if $\theta$ contains $m$ boxes and no two lie in the same column. If $\lambda - \mu$ is a horizontal strip we write $\lambda \succeq \mu$, which is equivalent to $\lambda_1 \geq \mu_1 \geq \lambda_2 \geq \mu_2 \geq \cdots$.

A (column-strict) {\em skew tableau} $T$ is a sequence of partitions 
$$\mu = \lambda^{0} \subseteq  \lambda^{1} \subseteq \cdots \subseteq \lambda^{r} = \lambda,$$
such that each skew diagram $\theta^{i} = \lambda^{i} - \lambda^{i-1}$ for $1 \leq i \leq r$ is a horizontal strip. The skew diagram $\lambda - \mu$ is called the {\em shape} of $T$ and the sequence $(|\theta^{1}|, \dots, |\theta^{r}|)$ is the {\em weight} of $T$. We can visualize a skew tableau as a filling of the skew Young diagram $\lambda-\mu$ with numbers in the set $\{1, \dots, r\}$ so that the boxes in $\theta^{i}$ contain the number $i$. The condition that $\theta^i$ are horizontal strips is seen to be equivalent to the statement that the numbers in the filling weakly increase along rows and strictly increase along columns. If $\mu = \emptyset$, we call $T$ a {\em semi-standard Young tableau (SSYT)} or just a {\em tableau} and we let $SSYT(r)$ denote the set of all tableaux whose entries are in $\{1, \dots, r\}$.

Given finitely many variables $x_1, \dots, x_n$, we define the {\em skew Schur polynomials} via
\begin{equation}\label{UP1}
s_{\lambda/ \mu}(x_1, \dots, x_n) = \sum_{\mu = \lambda^{0} \preceq  \lambda^{1} \preceq \cdots \preceq \lambda^{n} = \lambda} \prod_{i = 1}^n x_i^{|\lambda^{i} - \lambda^{i-1}|}.
\end{equation}
In particular, skew Schur polynomials are indexed by pairs of partitions $\lambda$ and $\mu$ and are equal to zero unless $\mu \subseteq \lambda$. When $\mu = \emptyset$ we drop it from the notation, and write $s_{\lambda}$, which is then the Schur polynomial indexed by $\lambda$. From \cite[Section I.5]{Mac} we know that $s_{\lambda/\mu}$ are symmetric homogeneous polynomials of degree $|\lambda - \mu|$ and they satisfy the identities
\begin{equation}\label{UP2}
s_{\lambda/ \mu}(x_1, \dots, x_m, y_1, \dots, y_n) = \sum_{\kappa \in \mathbb{Y}} s_{\lambda/ \kappa}(x_1, \dots, x_m) \cdot s_{\kappa/ \mu}(y_1, \dots, y_n).
\end{equation}
From \cite[Chapter I, (4.3)]{Mac} the Schur polynomials satisfy the {\em Cauchy identity}
\begin{equation}\label{UP3}
\sum_{ \lambda \in \mathbb{Y}} s_{\lambda}(x_1, \dots, x_m) \cdot s_{\lambda}(y_1, \dots, y_n) = \prod_{i = 1}^m \prod_{j = 1}^n \frac{1}{(1 - x_i y_j)},
\end{equation}
The identity in (\ref{UP3}) is a special case of the {\em skew Cauchy identity}, see \cite[Section I.5, Example 26]{Mac}, which states
\begin{equation}\label{UP4}
\begin{split}
&\sum_{ \mu \in \mathbb{Y}} s_{\mu /\lambda}(x_1, \dots, x_m) \cdot s_{\mu/ \nu}(y_1, \dots, y_n)  \\
& = \prod_{i = 1}^m \prod_{j = 1}^n \frac{1}{(1 - x_i y_j)} \sum_{\kappa \in \mathbb{Y} } s_{\lambda /\kappa}(y_1, \dots, y_n) \cdot s_{\nu/ \kappa}(x_1, \dots, x_m) .
\end{split}
\end{equation}
We mention that {\em a priori} (\ref{UP3}) and (\ref{UP4}) should be understood as equalities of formal power series. However, if $x_i, y_j \in \mathbb{C}$ and $|x_i y_j| < 1$ for $i = 1, \dots, m$ and $j = 1, \dots, n$, then the series in (\ref{UP3}) and (\ref{UP4}) converge absolutely and the two sides are numerically equal.

With the above notation in place we can define our measures.
\begin{definition}\label{SPD} Fix $M, N \in \mathbb{N}$ and suppose that $\vec{X} = \{x_i\}_{i = 1}^M$, $\vec{Y} = \{y_i \}_{i  = 1}^N$ are such that $x_i, y_j \geq 0$ and $x_i y_j < 1$ for all $i = 1, \dots, M$, and $j = 1, \dots, N$. With this data we define the measure
\begin{equation}\label{SP}
\mathbb{P}_{\vec{X},\vec{Y}} (\lambda^1, \dots, \lambda^{M+N-1}) = \prod_{i = 1}^M \prod_{j = 1}^N (1 - x_i y_j) \prod_{i = 1}^M s_{\lambda^i/ \lambda^{i-1}}(x_i) \cdot \prod_{j = M}^{M+N-1} s_{\lambda^{j}/ \lambda^{j+1}}(y_{j-M+1}),
\end{equation}
where $\lambda^0 = \lambda^{M+N} = \emptyset$ and $\lambda^i \in \mathbb{Y}$. Note that the $\mathbb{P}_{\vec{X},\vec{Y}}$ is non-negative in view of (\ref{UP1}) and the sum over $\lambda$'s is equal to one in view of (\ref{UP2}) and (\ref{UP3}). In addition, in view of (\ref{UP1}) we know that $\mathbb{P}_{\vec{X},\vec{Y}}$ is supported on sequences such that
$$\emptyset \preceq \lambda^1 \preceq \lambda^2 \preceq \cdots \preceq \lambda^{M-1} \preceq \lambda^M \succeq \lambda^{M+1} \succeq \cdots \succeq \lambda^{M+N-1} \succeq \emptyset.$$
\end{definition}
\begin{remark} The measures $\mathbb{P}_{\vec{X},\vec{Y}}$ in (\ref{SP}) are special cases of the Schur processes in \cite{OR03}. In the notation in \cite[(2.2)]{BR05} they correspond to setting $T = M+N -1$, and using the single-variable specializations 
$$\rho_i^+ = x_{i+1} \mbox{ for } i = 0, \dots, M-1, \hspace{2mm} \rho_i^+ = 0 \mbox{ for } i = M, \dots, M+N-2,$$
$$  \rho_i^- = 0 \mbox{ for }i = 1, \dots, M-1, \hspace{2mm} \rho_i^- = y_{i-M+1} \mbox{ for }i = M, \dots, M+N-1.$$
Notice that in this case $\lambda^i = \mu^i$ for $i = 1, \dots, M-1$ and $\lambda^{i+1} = \mu^i$ for $i = M, \dots, M+N-2$.

\end{remark}

%%%%%%%%%%%%%%%%%%%%%%%%%%%%%%%%%%%%%%%%%%%%%%%%%%%%%%%%%%%%%%%%%%%%%
%
%    Section 3.2
%
%%%%%%%%%%%%%%%%%%%%%%%%%%%%%%%%%%%%%%%%%%%%%%%%%%%%%%%%%%%%%%%%%%%%%
\subsection{Determinantal structure}\label{Section3.2} Suppose $\lambda = (\lambda^1, \dots, \lambda^{M+N-1})$ is distributed according to (\ref{SP}). For any $m \in \mathbb{N}$ with $m \leq M$ and $1 \leq M_1 < M_2 < \cdots < M_m \leq M$ we have from (\ref{UP2}) and (\ref{UP4}) 
\begin{equation}\label{ASPProj}
\mathbb{P}_{\vec{X},\vec{Y}} \left(  \cap_{r = 1}^m \{ \lambda^{M_r} = \mu^r \} \right)= \prod_{i = 1}^{M_m} \prod_{j = 1}^N (1 - x_i y_j) \prod_{r = 1}^m s_{\mu^r/ \mu^{r-1}}\left(x_{M_{r-1} + 1}, \dots, x_{M_r} \right) \cdot s_{\mu^{m}}(y_1, \dots, y_N),
\end{equation}
where $\mu^0 = 0$ and $M_0 = 0$. Consider the point process $\mathfrak{S}(\lambda)$ on $\{1, \dots, m\} \times \mathbb{Z} \subset \mathbb{R}^2$ 
\begin{equation}\label{PPDef}
\mathfrak{S}(\lambda) = \sum_{i = 1}^m \sum_{j = 1}^{\infty} \delta\left(i, \lambda^{M_i}_j - j\right).
\end{equation}
As is shown in Proposition \ref{PropCK}, the point process in (\ref{PPDef}) is a determinantal point process with an explicit correlation kernel. The fact that the Schur process is determinantal was originally proved in \cite{OR03} using Fock space techniques. In \cite{BR05} this statement was established through the Eynard-Mehta theorem and in \cite{Agg15} it was shown using Macdonald difference operators. The proposition below is essentially a restatement of \cite[Theorem 1.1.2]{Agg15} with a mild modification that we justify in the proof.
\begin{proposition}\label{PropCK} Suppose that $\mathfrak{S}(\lambda)$ is the point process in (\ref{PPDef}) with the same notation and assumptions as earlier in the section. Then, $\mathfrak{S}(\lambda)$ is a determinantal point process on $\mathbb{R}^2$ with reference measure given by the counting measure on $\{1, \dots, m\} \times \mathbb{Z}$ and correlation kernel 
\begin{equation}\label{CKDef}
K(u,x; v, y) = \frac{1}{(2\pi \im)^2} \oint_{C_{r_1}}dz \oint_{C_{r_2}}dw \frac{1}{z-w} \cdot \frac{\prod_{k = 1}^{N} (1 - y_k/z) \prod_{k = 1}^{M_v} (1 - x_kw)   }{\prod_{k = 1}^{N} (1 - y_k/w)  \prod_{k = 1}^{M_u} (1 - x_k z)} \cdot w^y z^{-x -1},
\end{equation}
where $C_{r_1}, C_{r_2}$ are positively oriented, zero-centered circles of radii $r_1, r_2$ that are positive reals satisfying the following conditions. If $u > v$ we have that $\max_{1 \leq i \leq N} y_i < r_1 < r_2 < \min_{1 \leq j \leq M_m} x_j^{-1}$; otherwise $\max_{1 \leq i \leq N} y_i < r_2 < r_1 < \min_{1 \leq j \leq M_m} x_j^{-1}$. 
\end{proposition}
\begin{remark} Notice that from our assumption that $x_i y_j < 1$ for all $i = 1, \dots, M$, and $j = 1, \dots, N$ we necessarily have that $\max_{1 \leq i \leq N} y_i < \min_{1 \leq j \leq M_m} x_j^{-1}$ and so one can always find $r_1, r_2$ as in Proposition \ref{PropCK}.
\end{remark}
\begin{proof} By assumption we have $x_i y_j < 1$ for all $i = 1, \dots, M$, and $j = 1, \dots, N$. Define $C = \epsilon + \max_{1 \leq i \leq N} y_i \in (0, \infty)$, where $ \epsilon > 0$ is chosen small enough so that $C \cdot \max_{1 \leq i \leq M} x_i < 1$. If we replace in (\ref{ASPProj}) $x_i$ with $\tilde{x}_i = x_i \cdot C$ and $y_j$ with $\tilde{y}_j = y_j \cdot C^{-1}$, then the measure remains unchanged by the homogeneity of the Schur polynomials. Furthermore, $\tilde{x}_i \in [0,1)$ and $\tilde{y}_j \in [0, 1)$ for all $i = 1, \dots, M$, and $j = 1, \dots, N$. From \cite[Theorem 1.1.2]{Agg15} applied to $X^{(k)} = (\tilde{x}_{M_{k-1} +1}, \dots, \tilde{x}_{M_k})$ for $k = 1, \dots,m$; $Y^{(k)} = 0$ for $k = 1, \dots, m-1$ and $Y^{(m)} = (\tilde{y}_1, \dots, \tilde{y}_N)$ we conclude that $\mathfrak{S}(\lambda)$ is a determinantal point process with kernel $\tilde{K}(u,x;v,y)$ as in (\ref{CKDef}) but with all $x,y$ parameters replaced with $\tilde{x}, \tilde{y}$. If we now change variables $z \rightarrow z/C$ and $w \rightarrow w/C$, we see that $\tilde{K}(u,x;v,y) = C^{x-y} K(u,x;v,y)$, which is a gauge transformation as in Proposition \ref{PropLem}, and so $\mathfrak{S}(\lambda)$ is a determinantal point process with kernel $K(u,x;v,y)$ as well.
\end{proof}
\begin{remark} We mention that \cite[Theorem 1.1.2]{Agg15} has two typos. Firstly, the expression on the right side of \cite[Equation (1.6)]{Agg15} should be multiplied by $-1$. This mistake appears to originate from \cite[Proposition 2.2.3]{Agg15}, where a formula from \cite[Remark 2.2.11]{BorCor} is recalled with a sign error, and propagates throughout the paper. The second typo is that the ``$s < t$'' in that theorem should be replaced with ``$s > t$''. Both of these typos were discovered after the paper was accepted, but they have been corrected in the arxiv version of the paper. As an alternative check, one can instead use \cite[Theorem 2.2]{BR05}, which agrees with (\ref{CKDef}) up to transposing the kernel, i.e. swapping $(u,x)$ and $(v,y)$ in that formula. From Proposition \ref{PropLem} a determinantal point process with one kernel is also determinantal with the transpose of that kernel.
\end{remark}

%%%%%%%%%%%%%%%%%%%%%%%%%%%%%%%%%%%%%%%%%%%%%%%%%%%%%%%%%%%%%%%%%%%%%
%
%    Section 3.3
%
%%%%%%%%%%%%%%%%%%%%%%%%%%%%%%%%%%%%%%%%%%%%%%%%%%%%%%%%%%%%%%%%%%%%%
\subsection{Geometric last passage percolation}\label{Section3.3} In this section we establish a distributional equivalence between the Schur process $\mathbb{P}_{\vec{X},\vec{Y}}$ in Definition \ref{SPD} and {\em geometric last passage percolation}. We mention that this connection is a somewhat well-known consequence of the {\em Robinson-Schensted-Knuth (RSK)} correspondence and {\em Greene's theorem}, and with some effort can be read off from \cite[Section 5]{J06}. As we could not find the precise statement in the literature, we decided to include a short proof here, which was communicated to us privately by Ivan Corwin. We begin by first describing geometric last passage percolation and then we state and prove our desired distributional equivalence in Proposition \ref{PEquiv}. 

For $a \in [0,1)$, we let $p_a$ denote the geometric distribution with parameter $a$, i.e. 
\begin{equation}\label{DefGeom}
p_a(k) = (1-a) \cdot a^k \mbox{ for } k \geq 0.
\end{equation}
We fix two integers $M, N \in \mathbb{N}$, and parameters $\vec{X} = (x_1, \dots, x_M)$ and $\vec{Y} = (y_1, \dots, y_N)$ such that $x_i \geq 0$, $y_j \geq 0$ and $x_iy_j \in [0,1)$. We let $W = (W_{i,j}: i = 1, \dots, M \mbox{ and } j = 1, \dots, N)$ denote an array of independent random variables with law $p_{a_{i,j}}$, where $a_{i,j} = x_i y_j$. A {\em north-east (NE) chain} is a sequence of vertices $(c_1, r_1), (c_2, r_2), \dots, (c_k, r_k) \in \mathbb{Z}^2$ such that $c_1 \leq c_2 \leq \cdots \leq c_k$, and $r_1 \leq r_2 \leq \cdots \leq r_k$ for $i = 2, \dots, k$, see Figure \ref{geomLPP}. We let $\Pi_{M,N}$ denote the set of NE chains $\pi$ in $\mathbb{Z}^2$, whose vertices are contained in $\{1, \dots, M\} \times \{1, \dots, N\}$ and note that a chain may contain no vertices. For an integer $k \in \mathbb{N}$, we also let $\Pi_{M,N}^k$ denote the set of $k$ disjoint NE chains in $\Pi_{M,N}$.
\begin{figure}
	\begin{center}
		\begin{tikzpicture}[scale=0.7]
		\begin{scope}
		\draw (0,0) node[anchor = north east]{{\footnotesize $(1,1)$}};
		\draw (9,6) node[anchor=south west]{{\footnotesize $(M,N)$}};
		\draw[->, >=stealth'] (9,2.5) node[anchor=west]{{ $W_{i,j}$ with distribution $p_{a_{i,j}}$}} to[bend right] (7,2);
		\draw (12,1.8) node{$a_{i,j}=x_iy_j$};
		\draw[dotted, gray] (0,0) grid (9,6);
		
            \draw[ultra thick] (2,1) circle [radius=0.2]; 
            \draw[ultra thick] (2,2) circle [radius=0.2];
            \draw[ultra thick] (4,2) circle [radius=0.2];
            \draw[ultra thick] (4,3) circle [radius=0.2];
            \draw[ultra thick] (5,3) circle [radius=0.2];
            \draw[ultra thick] (7,5) circle [radius=0.2];
		\draw (7, -0.4) node{$i$};
		\draw (-0.4,2) node{$j$};
		\foreach \x in {1,2,...,10}
		\draw (\x-1,-1) node{$x_{\x}$};
		\foreach \x in {1,2,...,7}
		\draw (-1,\x-1) node{$y_{\x}$};
		\end{scope}
%		\begin{scope}[xshift=11cm]
%		\draw (9,6) node[anchor=west]{{\footnotesize $(n,k)$}};
%		\draw (9,5) node[anchor=west]{{\footnotesize $(n,k-1)$}};
%		\draw (9,4) node[anchor=west]{{\footnotesize $(n,k-2)$}};
%		\draw[dotted, gray] (0,0) grid (9,6);
%		\draw[ultra thick]  (3,1) -- (4,1) -- (4,2) -- (5,2) -- (6,2) -- (6,3) -- (7,3) -- (8,3) -;
%		\draw[ultra thick] (0,1) -- (1,1) -- (1,3) -- (5,3) -- (5,4) -- (7,4) -- (7,5) -- (9,5);
%		\draw[ultra thick] (0,2) -- (0,4) -- (3,4) --(3,5) -- (6,5) -- (6,6) -- (9,6) ;
%		\end{scope}
		\end{tikzpicture}
	\end{center}
	\caption{The NE chain $(3,2), (3,3), (5,3), (5,4), (6,4), (8,6)$ }
	\label{geomLPP}
\end{figure}
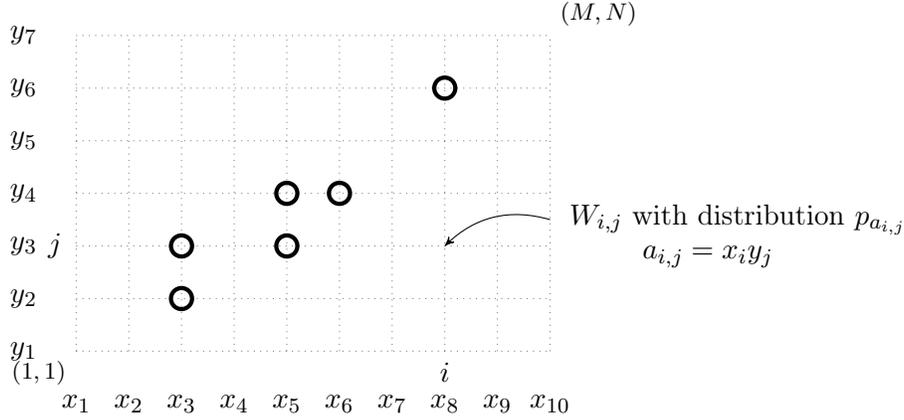

Given a NE chain $\pi \in \Pi_{M,N}$, we define its (random) {\em weight} by
\begin{equation}\label{PathWt}
w(\pi) = \sum_{(i,j) \in \pi} W_{i,j},
\end{equation}
where the sum is over all vertices $(i,j) \in \pi$. If $\pi$ has no vertices, the weight is zero. We finally define for any $k \in \mathbb{N}$, $m \in \{1, \dots, M\}$ and $n \in \{1, \dots, N\}$ 
\begin{equation}\label{GDef}
G^k(m,n) = \max_{(\pi_1, \dots, \pi_k) \in \Pi_{m,n}^k} w(\pi_1) + \cdots + w(\pi_k).
\end{equation}
We observe the following immediate consequences of the above definition:
\begin{itemize}
\item $G^{k+1}(m,n) \geq G^{k}(m,n) \in \mathbb{Z}_{\geq 0}$ for all $k \geq 1$,
\item $G^{k+1}(m,n) = G^{k}(m,n) = \sum_{i= 1}^m \sum_{j = 1}^n W_{i,j}$ for all $k \geq \min(m,n)$.
\end{itemize}
The latter observations show that if we define $\lambda_k(m,n)$ through
\begin{equation}\label{GtoL}
\lambda_1(m,n) = G^1(m,n) \mbox{ and } \lambda_k(m,n) = G^k(m,n) - G^{k-1}(m,n) \mbox{ for } k \geq 2,
\end{equation}
then $\{\lambda_k(m,n)\}_{k \geq 1}$ is a sequence of non-negative integers, which is eventually zero and whose terms sum up to $\sum_{i= 1}^m \sum_{j = 1}^n W_{i,j}$. In fact, this sequence is {\em decreasing}, and so $\lambda(m,n)$ is a partition of weight $\sum_{i= 1}^m \sum_{j = 1}^n W_{i,j}$. The last statement can be deduced by applying the RSK correspondence, \cite[Theorem 4.8.2]{Sagan}, and a generalization of Greene's theorem, \cite[Theorem 4.8.10]{Sagan}.

With the above notation in place we can state the main result of the section.
\begin{proposition}\label{PEquiv} The sequence $(\lambda(1,N), \lambda(2,N), \dots, \lambda(M,N), \lambda(M, N-1), \dots, \lambda(M,1))$, defined in (\ref{GtoL}) has distribution $\mathbb{P}_{\vec{X},\vec{Y}}$ as in Definition \ref{SPD}
\end{proposition}
\begin{proof} We start by recalling the RSK correspondence, which establishes a bijection between $M \times N$ matrices with entries in $\mathbb{Z}_{\geq 0}$ and pairs of SSYT $(P,Q)$ such that $P$ has the same shape as $Q$, $P \in SSYT(N)$ and $Q \in SSYT(M)$.

If $w = (w_{i,j}: i = 1, \dots, M, j = 1, \dots, N)$ is a matrix in $\mathbb{Z}_{\geq 0}^{M \times N}$, we construct the two-row array 
$\omega$, which consists of $w_{i,j}$ columns $\binom{i}{j}$ with $(i,j) \in \{1, \dots, M\} \times \{1, \dots, N\}$ traversed in lexicographic order. Assuming that $\omega$ takes the form 
$$\omega = \begin{bmatrix} u_1 & u_2 & \cdots & u_r \\ v_1 & v_2 & \cdots & v_r \end{bmatrix},$$
we define using the second row $P_k \in SSYT(N)$ for $k = 1, \dots, r$ via
$$P_k = (( \dots (( \hspace{1mm} \boxed{v_1} \leftarrow v_2) \leftarrow v_3) \dots ) \leftarrow v_{k-1} ) \leftarrow v_k,$$
where ``$\leftarrow$'' denotes {\em row insertion}, see \cite[Chapter 1]{Fulton} for a description. We also define tableaux $Q_k$ for $k = 1, \dots, r$, sometimes called a {\em recording tableaux}, from the first row in $\omega$. The $Q_k$'s are defined via $Q_1 = \boxed{u_1}$ and if $Q_{k-1}$ has been constructed, then we obtain $Q_k$ from $Q_{k-1}$ by adding the box that is in $P_k$ but not in $P_{k-1}$ and writing $u_k$ inside this box. As explained in \cite[Chapter 4]{Fulton} the above procedure associates to each integer matrix $w$ a pair $(P,Q) = (P_r, Q_r) \in SSYT(N) \times SSYT(M)$ of two SSYT of the same shape and the latter map is a bijection. We denote by $\mathrm{RSK}(w)$ the result of applying the RSK correspondence to $w$. 

We next list some properties of the above RSK bijection. Writing $\mu^k$ for the Young diagram in $Q$, formed by the boxes with entries that are at most $k$, we have directly from the definition of the RSK correspondence
\begin{equation}\label{WtP}
\emptyset = \mu^0 \preceq \mu^1 \preceq \mu^2 \preceq \cdots \preceq \mu^M \mbox{ and } |\mu^i - \mu^{i-1}| = \sum_{j = 1}^N w_{i,j} \mbox{ for } i = 1, \dots, M.
\end{equation}
Transposing the matrix $w$ to $w^t$, corresponds to swapping the two rows in $\omega$, and then by the Symmetry theorem, see \cite[Chapter 4]{Fulton}, we have that $\mathrm{RSK}(w^t) = (Q,P)$. Combining the latter with (\ref{WtP}) we conclude that if $\nu^k$ is the Young diagram in $P$, formed by the boxes with entries that are at most $k$, then
\begin{equation}\label{WtP2}
\emptyset = \nu^0 \preceq \nu^1 \preceq \nu^2 \preceq \cdots \preceq \nu^N \mbox{ and } |\nu^i - \nu^{i-1}| = \sum_{j = 1}^M w_{j,i} \mbox{ for } i = 1, \dots, N.
\end{equation}
Notice that the shape of $P$ is $\nu^N$ and that of $Q$ is $\mu^M$, so that $\mu^M = \nu^N$.

Let $g^k(m,n;w)$ be as in (\ref{GDef}) with $W_{i,j}$ replaced with $w_{i,j}$ and let $\lambda(m,n;w)$ be the partitions defined in (\ref{GtoL}) with $G^k(m,n)$ replaced with $g^{k}(m,n;w)$. By a generalization of Greene's theorem, see \cite[Theorem 4.8.10]{Sagan}, we have 
\begin{equation}\label{Greene}
\lambda(m, N;w) = \mu^m \mbox{ for } m = 1\dots, M \mbox{ and } \lambda(M, n;w) = \nu^n \mbox{ for } n = 1, \dots, N.
\end{equation}

We are now ready to complete the proof of the proposition. Let $(\Omega, \mathcal{F},\mathbb{P})$ be a probability space on which we have defined an $M \times N$ matrix $W$ with entries $W_{i,j}$ that are independent random variables with law $p_{a_{i,j}}$, where $a_{i,j} = x_i y_j$. We then have for $Q = \emptyset \preceq \mu^1 \preceq \mu^2 \preceq \cdots \preceq \mu^M$ and $P = \emptyset \preceq \nu^1 \preceq \nu^2 \preceq \cdots \preceq \nu^N$
\begin{equation}\label{S3Big}
\begin{split}
&\mathbb{P}\left( \lambda(1,N) = \mu^1, \dots, \lambda(M,N) = \mu^M = \nu^N, \lambda(M,N-1) = \nu^{N-1}, \dots, \lambda(M,1) = \nu^1 \right)  \\
& = \sum_{w \in \mathbb{Z}_{\geq 0}^{M \times N} } \prod_{i = 1}^M \prod_{j = 1}^N (1 - x_i y_j)  (x_i y_j)^{w_{i,j}} \times {\bf 1}\{ \mathrm{RSK}(w) = (P,Q) \} \\
& = \prod_{i = 1}^M \prod_{j = 1}^N (1 - x_i y_j) \sum_{w \in \mathbb{Z}_{\geq 0}^{M \times N} }   \prod_{i = 1}^M x_i^{|\mu^i - \mu^{i-1}|} \prod_{j = 1}^N y_j^{|\nu^{i}- \nu^{i-1}|}\times {\bf 1}\{ \mathrm{RSK}(w) = (P,Q) \} \\
& = \prod_{i = 1}^M \prod_{j = 1}^N (1 - x_i y_j)   \prod_{i = 1}^M x_i^{|\mu^i - \mu^{i-1}|} \prod_{j = 1}^N y_j^{|\nu^{j}- \nu^{j-1}|} \\
& = \mathbb{P}_{\vec{X},\vec{Y}} \left(  \lambda^1= \mu^1, \dots, \lambda^M = \mu^M = \nu^N, \lambda^{M+1} = \nu^{N-1}, \dots, \lambda^{M+N-1} = \nu^1  \right).
\end{split}
\end{equation}
For other sequences of partitions the probabilities in (\ref{S3Big}) are both equal to zero. We mention that in going from the first to the second line in (\ref{S3Big}) we used (\ref{Greene}); in going from the second to the third we used (\ref{WtP}) and (\ref{WtP2}); in going from the third to the fourth we used that $\mathrm{RSK}$ is a bijection, and so the sum over $w$ becomes equal to one; in going from the fourth to the fifth we used Definition \ref{SPD} and the formula for skew Schur polynomials in (\ref{UP1}).
\end{proof}

%%%%%%%%%%%%%%%%%%%%%%%%%%%%%%%%%%%%%%%%%%%%%%%%%%%%%%%%%%%%%%%%%%%%%
%
%    Section 3.4
%
%%%%%%%%%%%%%%%%%%%%%%%%%%%%%%%%%%%%%%%%%%%%%%%%%%%%%%%%%%%%%%%%%%%%%
\subsection{Monotone coupling for Schur processes}\label{Section3.4} The goal of this section is to establish a certain monotone coupling between Schur processes as in Definition \ref{SPD} for different sets of parameters. The precise statement is as follows.

\begin{proposition}\label{MonCoup} Fix $M, N \in \mathbb{N}$. There exists a probability space $(\Omega, \mathcal{F}, \mathbb{P})$ and a family of random sequences $(\lambda^1[\vec{X}, \vec{Y}], \dots, \lambda^{M+N-1}[\vec{X}, \vec{Y}]) \in \mathbb{Y}^{M+N-1}$, indexed by $\vec{X} = (x_1, \dots, x_M) \in \mathbb{R}^M_{\geq 0}$ and $\vec{Y}=  (y_1, \dots, y_N) \in \mathbb{R}^N_{\geq 0}$ with $x_i y_j \in [0,1)$ for all $i = 1, \dots, M$ and $j = 1, \dots, N$ so that the following hold. Under $\mathbb{P}$ the distribution of $(\lambda^1[\vec{X}, \vec{Y}], \dots, \lambda^{M+N-1}[\vec{X}, \vec{Y}])$ is $\mathbb{P}_{\vec{X}, \vec{Y}}$ as in Definition \ref{SPD}. In addition, we have for each $\omega \in \Omega$, $k \in \mathbb{N}$ and $j \in \{1, \dots, M+N-1\}$
\begin{equation}\label{MonIneq}
\sum_{i = 1}^k \lambda_i^j[\vec{X}, \vec{Y}](\omega) \geq \sum_{i = 1}^k \lambda_i^j[\vec{X}', \vec{Y}'](\omega),
\end{equation}
provided that $x_iy_j \geq x_i'y_j'$ for all $i = 1, \dots, M$ and $j = 1, \dots, N$.
\end{proposition}
\begin{remark}\label{RSKMonR1} In the proof of Proposition \ref{MonCoup} we construct a coupling of random arrays with independent geometric random variables $W[\vec{X}, \vec{Y}]$ that is monotone in the parameters $\vec{X}, \vec{Y}$. Subsequently, we run the same RSK correspondence on these coupled arrays to produce coupled Schur processes. The inequality in (\ref{MonIneq}) is then a direct consequence of Greene's theorem and the fact that $G^k(m,n)$ in (\ref{GtoL}) are monotone in the entries of the matrix $W$.
\end{remark}
\begin{proof} We let $(\Omega, \mathcal{F}, \mathbb{P})$ be a probability space on which we have defined $MN$ i.i.d. exponential random variables with parameter $1$, denoted by $Z_{i,j}$ for $i = 1, \dots, M$ and $j = 1, \dots, N$. Let us fix two vectors $\vec{X}$ and $\vec{Y}$ as in the statement of the proposition, and let $a_{i,j} = x_i y_j$. We define $b_{i,j} =  \frac{1}{-\log a_{i,j}}$ with the convention that $b_{i,j} = 0$ if $a_{i,j} = 0$ and put 
$$W_{i,j}[\vec{X}, \vec{Y}] = \lfloor b_{i,j} \cdot Z_{i,j} \rfloor.$$
By a direct computation we check that $W_{i,j}[\vec{X}, \vec{Y}]$ are independent and have the geometric distribution $p_{a_{i,j}}$ as in (\ref{DefGeom}). We now define $(\lambda^1[\vec{X}, \vec{Y}], \dots, \lambda^{M+N-1}[\vec{X}, \vec{Y}])$ to be the sequence of partitions $(\lambda(1,N), \lambda(2,N), \dots, \lambda(M,N), \lambda(M, N-1), \dots, \lambda(M,1))$ as in (\ref{GDef}) and (\ref{GtoL}) for the matrix $W[\vec{X}, \vec{Y}]$. From Proposition \ref{PEquiv} we know that under $\mathbb{P}$ the sequence $(\lambda^1[\vec{X}, \vec{Y}], \dots, \lambda^{M+N-1}[\vec{X}, \vec{Y}])$ is distributed according to $\mathbb{P}_{\vec{X}, \vec{Y}}$. This proves the first part of the proposition.

If we now have two sets of parameters $(\vec{X}, \vec{Y})$ and $(\vec{X}', \vec{Y}')$ as in the second part of the proposition, then we have $W_{i,j}[\vec{X}, \vec{Y}] \geq W_{i,j}[\vec{X}', \vec{Y}'],$ 
which from (\ref{GDef}) and (\ref{GtoL}) implies (\ref{MonIneq}).
\end{proof}
\begin{remark}\label{RSKMonR2} It is worth mentioning that while the inequality in (\ref{MonIneq}) holds almost surely, our coupling does not ensure that
$$\lambda_k^j[\vec{X}, \vec{Y}](\omega)  \geq \lambda_k^j[\vec{X}', \vec{Y}'](\omega).$$
For example, if $M = N = 3$, it is possible for our coupling to produce with positive probability
$$W[\vec{X}, \vec{Y}] = \begin{bmatrix} 2 & 0 & 0 \\ 2 & 0 & 0 \\ 0 & 2 & 2 \end{bmatrix} \mbox{ and } W[\vec{X}', \vec{Y}'] = \begin{bmatrix} 2 & 0 & 0 \\ 2 & 1 & 0 \\ 0 & 2 & 2 \end{bmatrix}.$$
In this case, $\lambda^3[\vec{X}, \vec{Y}] = 44$ and $\lambda^3[\vec{X}', \vec{Y}'] = 531$, so that 
$$\lambda^3_2[\vec{X}, \vec{Y}] = 4 > 3 = \lambda^3_2[\vec{X}', \vec{Y}'].$$
\end{remark}

%%%%%%%%%%%%%%%%%%%%%%%%%%%%%%%%%%%%%%%%%%%%%%%%%%%%%%%%%%%%%%%%%%%%%
%
%    Section 4
%
%%%%%%%%%%%%%%%%%%%%%%%%%%%%%%%%%%%%%%%%%%%%%%%%%%%%%%%%%%%%%%%%%%%%%
\section{Convergence of point processes}\label{Section4} The goal of this section is to establish Theorem \ref{T1}. In Section \ref{Section4.1} we prove Lemma \ref{WellDefKer}, which in particular establishes the well-posedness of the kernel $K_{a,b,c}$ that appears in Theorem \ref{T1}. In Section \ref{Section4.2} we introduce a particular scaling of the parameters of the Schur process from Section \ref{Section3.1}, and derive an alternative formula for the correlation kernel from Proposition \ref{PropCK} that is suitable for asymptotic analysis -- this is Lemma \ref{PrelimitKernel}. In Section \ref{Section4.3} we state what the kernels from Lemma \ref{PrelimitKernel} converge to under our specified scaling, this is Proposition \ref{LimitKernelProp}, and use that statement in conjunction with statements from Section \ref{Section2} to prove Theorem \ref{T1}. Proposition \ref{LimitKernelProp} is the main technical asymptotic result of the paper and its proof is the content of Section \ref{Section5}.

%%%%%%%%%%%%%%%%%%%%%%%%%%%%%%%%%%%%%%%%%%%%%%%%%%%%%%%%%%%%%%%%%%%%%
%
%    Section 4.1
%
%%%%%%%%%%%%%%%%%%%%%%%%%%%%%%%%%%%%%%%%%%%%%%%%%%%%%%%%%%%%%%%%%%%%%
\subsection{Limiting kernel properties}\label{Section4.1} In this section we prove Lemma \ref{WellDefKer}. We isolate the following useful inequalities that will be used within the proof
\begin{equation}\label{RatBound}
|1 + z| \leq e^{|z|} \mbox{ for $z \in \mathbb{C}$ and } |1 - z|^{-1} \leq e^{|z|(2/d + 2)} \mbox{ if } z \in \mathbb{C}, |1- z| \geq d > 0.
\end{equation}
The first inequality in (\ref{RatBound}) is trivial and for the second inequality we note that 
$$|1- z|^{-1} \leq 1/d \leq e^{1/d} \leq e^{2|z|/d} \leq e^{|z|(2/d + 2)} \mbox{ if $|z| \geq 1/2$,}$$
while for $|z| \leq 1/2$ we have 
$$|1 - z|^{-1} \leq (1 - |z|)^{-1} \leq 1 + 2|z| \leq e^{2|z|} \leq e^{|z|(2/d + 2)} .$$

In the remainder of the section we prove Lemma \ref{WellDefKer}. 

\begin{proof}[Proof of Lemma \ref{WellDefKer}] We assume that $t_1, t_2, x_1, x_2 \in \mathbb{R}$ are fixed and $\alpha, \beta \in \mathbb{R}$ are such that $\alpha + t_1 < \underline{a}$ while $\beta + t_2 > \underline{b}$ with $\underline{a}, \underline{b}$ as in Definition \ref{DLP}.

The integral in the definition of $K^1_{a,b,c}$ is well-defined and finite as the integrand is entire and the contour $\gamma$ a finite straight segment. 

We next investigate the integral in the definition of $K^3_{a,b,c}$. As mentioned after (\ref{DefPhi}), we have that the function $\Phi_{a,b,c}(z)$ is a meromorphic function on $\mathbb{C} \setminus \{0\}$ whose zeros are at $\{-(b_i^+)^{-1}\}_{i =1}^{J_b^+}$ and $\{- b_i^-\}_{i =1}^{J_b^-}$, while its poles are at $\{(a_i^+)^{-1}\}_{i =1}^{J_a^+}$ and $\{ a_i^-\}_{i =1}^{J_a^-}$. From the definition of $\Gamma_{\alpha}^+$ and $\Gamma_{\beta}^-$ in Definition \ref{DefPhi}, and the fact that $\alpha + t_1 < \underline{a}$ while $\beta + t_2 > \underline{b}$, we have that for $w \in \Gamma_{\beta}^-$ the variable $w + t_2$ is bounded away from the zeros of $\Phi_{a,b,c}$ and for $z \in \Gamma_{\alpha}^+$ the variable $z + t_2$ is bounded away from the poles of $\Phi_{a,b,c}$. The latter means that the function
$$e^{z^3/3 -x_1z - w^3/3 + x_2w}\cdot \frac{\Phi_{a,b,c}(z + t_1) }{\Phi_{a,b,c}(w + t_2)}$$
is continuous for $(z,w) \in \Gamma_{\alpha}^+ \times \Gamma_{\beta}^-$. The term $(z + t_1  -w - t_2)^{-1}$ that appears in the integrand defining $K^3_{a,b,c}$ is also continuous on $ \Gamma_{\alpha}^+ \times \Gamma_{\beta}^-$ if the contours $\Gamma_{\alpha + t_1}^+$ and $\Gamma_{\beta + t_2}^-$ are disjoint, which happens precisely when $\beta + t_2 < \alpha + t_1$. Instead, if we have that $\alpha + t_1 \leq \beta + t_2$, then the contours $\Gamma_{\alpha + t_1}^+$ and $\Gamma_{\beta + t_2}^-$ intersect at the points $z_{\pm} = c \pm \im d$, where $c = (\alpha + t_1 + \beta + t_2)/2$ and $d = (\beta + t_2 - \alpha - t_1)/2$. This causes the term $(z + t_1  -w - t_2)^{-1}$ to have a singularity at the points $(z,w) = (z_+ - t_1, z_+ - t_2)$ and $(z,w) = (z_- - t_1, z_- - t_2)$. Setting $z + t_1 = z_{\pm} + x + \im x$ and $w + t_2 = z_{\pm} + y - \im y$ provides a local parametrization of the contours near $(z_{\pm}, z_{\pm})$ and one directly observes that 
\begin{equation}\label{IntSing}
\frac{1}{|z + t_1  -w - t_2|} = \frac{1}{|(x + \im x)- (y - \im y)|} = \frac{1}{\sqrt{(x-y)^2 + (x + y)^2}} = \frac{1}{\sqrt{2} \sqrt{x^2 + y^2}}.
\end{equation}
We thus conclude that $(z + t_1  -w - t_2)^{-1}$ has an {\em integrable} singularity near the points $(z_{\pm} - t_1, z_{\pm} - t_2)$. In particular, we conclude that the integrand 
\begin{equation}\label{IntegrandKer}
e^{z^3/3 -x_1z - w^3/3 + x_2w}\cdot \frac{\Phi_{a,b,c}(z + t_1) }{\Phi_{a,b,c}(w + t_2)} \cdot \frac{1}{z + t_1  -w - t_2}
\end{equation}
in the definition of $K_{a,b,c}^3$ is {\em locally} integrable over $\Gamma_{\alpha}^+ \times \Gamma_{\beta}^-$. \\

To see why the function in (\ref{IntegrandKer}) is integrable over $\Gamma_{\alpha}^+ \times \Gamma_{\beta}^-$ it suffices to dominate it (in absolute value) by a function that is integrable over these contours. We seek to apply (\ref{RatBound}). Note that by the definition of the contours we can find $d_1 > 0$ (depending on $\alpha, t_1, a_1^+$)  such that 
$$|1 - a_i^+ (z + t_1)| \geq d_1 \mbox{ and } |1 - a_i^-/(z+t_1)| \geq d_1$$
for all $i \geq 1$ and $z \in \Gamma_{\alpha}^+$. Similarly, we can find $d_2 > 0$ (depending on $\beta, t_2, b_1^+$)  such that 
$$|1 + b_i^+ (w + t_2)| \geq d_2 \mbox{ and } |1 + b_i^-/(w+t_2)| \geq d_2$$
for all $i \geq 1$ and $w \in \Gamma_{\beta}^-$. We can apply (\ref{RatBound}) to conclude that for some $C > 0$ (depending on $d_1, d_2$) 
\begin{equation}\label{BoundPhiRat}
\left|\frac{\Phi_{a,b,c}(z + t_1) }{\Phi_{a,b,c}(w + t_2)}  \right| \leq e^{ (|z+t_1| + |w+t_2|) \cdot (c^+ + C B^+ + C A^+)} \cdot e^{(|z + t_1|^{-1} + |w+t_2|^{-1}) \cdot (c^- + C B^- + C A^-)  },
\end{equation}
where $A^{\pm} = \sum_{i \geq 1}a_i^{\pm}$ and $B^{\pm} = \sum_{i \geq 1}b_i^{\pm}$. In particular, the second exponential term on the right side of (\ref{BoundPhiRat}) is only present if $c^- + b_1^- + a_1^- > 0$ in which case $\underline{a} = \underline{b} = 0$, $\alpha + t_1 < 0$, $\beta + t_2 > 0$, the contours $\Gamma_{\alpha + t_1}^+$, $\Gamma_{\beta + t_2}^-$ are bounded away from $0$ and hence this term is uniformly bounded by a constant. In short, the left side of (\ref{BoundPhiRat}) grows at most exponentially fast in $|z| + |w|$ for $(z,w) \in \Gamma_{\alpha}^+ \times \Gamma_{\beta}^-$.

Writing $z = \alpha + y + \im \epsilon_z y$ and $w = \beta - x + \im \epsilon_w x$ where $\epsilon_z , \epsilon_w \in \{-1,1\}$, we see that 
\begin{equation*}
\left| e^{z^3/3 - w^{3}/3}  \right| = e^{\Real (\alpha + y + \im \epsilon_z y)^3/3 - \Real (\beta -x + \im \epsilon_z x)^3/3 } = e^{(\alpha^3- \beta^3)/3} \cdot e^{\alpha^2 y + \beta^2x} \cdot e^{-(2/3) \cdot (x^3 + y^3)}.
\end{equation*}
The latter implies that there is a constant $C_1 > 0$, depending on $\alpha, \beta$, such that 
\begin{equation}\label{BoundCubicKer}
\left| e^{z^3/3 - w^{3}/3}  \right| \leq \exp(C_1 - |z|^3/6 - |w|^3/6).
\end{equation}
Combining (\ref{BoundPhiRat}) with (\ref{BoundCubicKer}) and the trivial inequality $|e^u| \leq e^{|u|}$ for $u \in \mathbb{C}$, we see that for some $D_0, D_1 > 0$ (depending on $t_1, t_2, \alpha, \beta, c^{\pm}, A^{\pm}, B^{\pm}, C$ and $C_1$)
\begin{equation}\label{IntegrandKerB}
\left| \frac{\Phi_{a,b,c}(z + t_1) }{\Phi_{a,b,c}(w + t_2)} \cdot \frac{e^{z^3/3 -x_1z - w^3/3 + x_2w}}{z + t_1  -w - t_2} \right| \leq \frac{e^{D_0 + D_1(|z| + |w|) + |z||x_1| + |w||x_2| - |z|^3/6 - |w|^3/6}}{|z + t_1  -w - t_2|}.
\end{equation}
The right side of (\ref{IntegrandKerB}) is integrable over $\Gamma_{\alpha}^+ \times \Gamma_{\beta}^-$ due to the already established local integrability, and the fact that near infinity the cubic terms in the exponential provide sufficient decay.\\

Our work so far shows that the kernel $K_{a,b,c}$ is well-defined. We next proceed to show that the value of the kernel does not depend on $\alpha$ and $\beta$ as long as $\alpha + t_1 < \underline{a}$ and $\beta + t_2 > \underline{b}$. Let us fix $\alpha_1, \alpha_2, \beta_1, \beta_2$ such that $\alpha_i + t_1 < \underline{a}$ and $\beta_i + t_2 > \underline{b}$ for $i = 1,2$. We further denote by $K^{1,\alpha,\beta}_{a,b,c}$ and $K^{3,\alpha,\beta}_{a,b,c}$ the kernels in $K^{1}_{a,b,c}$ and $K^{3}_{a,b,c}$ in Definition \ref{3BPKernelDef}, reflecting the dependence of the latter on the choice of $\alpha, \beta$. Note that $K^2_{a,b,c}$ already does not depend on $\alpha, \beta$ and so we wish to prove that 
\begin{equation*}
K^{1,\alpha_1, \beta_1}_{a,b,c}(t_1, x_1; t_2, x_2)  + K^{3,\alpha_1, \beta_1}_{a,b,c}(t_1, x_1; t_2, x_2) = K^{1,\alpha_2, \beta_2}_{a,b,c}(t_1, x_1; t_2, x_2)  + K^{3,\alpha_2, \beta_2}_{a,b,c}(t_1, x_1; t_2, x_2).
\end{equation*} 
Fix $\alpha_3 < \min(-t_1, \alpha_1, \alpha_2, \beta_1 + t_2 - t_1, \beta_2 + t_2 - t_1)$ and $\beta_3 > \max(-t_2, \beta_1, \beta_2, \alpha_1 + t_1 - t_2, \alpha_2 + t_1 - t_2 )$. To establish the above we show that for $i = 1,2$ we have
\begin{equation*}
K^{1,\alpha_i, \beta_i}_{a,b,c}(t_1, x_1; t_2, x_2)  + K^{3,\alpha_i, \beta_i}_{a,b,c}(t_1, x_1; t_2, x_2) = K^{1,\alpha_3, \beta_3}_{a,b,c}(t_1, x_1; t_2, x_2)  + K^{3,\alpha_3, \beta_3}_{a,b,c}(t_1, x_1; t_2, x_2).
\end{equation*}
By symmetry it suffices to show the above for $i = 1$ only, and then the statement will follow from showing that
\begin{equation}\label{Indep1}
\begin{split}
&K^{1,\alpha_1, \beta_1}_{a,b,c}(t_1, x_1; t_2, x_2)  \hspace{-1mm} +\hspace{-1mm} K^{3,\alpha_1, \beta_1}_{a,b,c}(t_1, x_1; t_2, x_2) = K^{1,\alpha_1, \beta_3}_{a,b,c}(t_1, x_1; t_2, x_2)  \hspace{-1mm} + \hspace{-1mm}K^{3,\alpha_1, \beta_3}_{a,b,c}(t_1, x_1; t_2, x_2) \\
&K^{1,\alpha_1, \beta_3}_{a,b,c}(t_1, x_1; t_2, x_2)  \hspace{-1mm}  + \hspace{-1mm} K^{3,\alpha_1, \beta_3}_{a,b,c}(t_1, x_1; t_2, x_2) = K^{1,\alpha_3, \beta_3}_{a,b,c}(t_1, x_1; t_2, x_2)  \hspace{-1mm} + \hspace{-1mm} K^{3,\alpha_3, \beta_3}_{a,b,c}(t_1, x_1; t_2, x_2).
\end{split}
\end{equation}

Let us denote by $u_{\pm}$ with $\Imag(u_+) \geq \Imag(u_-)$ the two intersection points of $\Gamma_{\alpha_1 + t_1}^+$ and $\Gamma_{\beta_3 + t_2}^-$ -- note that these contours indeed intersect since $\alpha_1 + t_1 < \beta_3 + t_2$ by the definition of $\beta_3$. We also let $v_{\pm}$ be the intersection points of $\Gamma_{\alpha_1 + t_1}^+$ and $\Gamma_{\beta_1 + t_2}^-$ if $\alpha_1 + t_1 <\beta_1 + t_2$ and otherwise we put $v_+ = v_- = \alpha_1 + t_1$. As before we assume $\Imag(v_+) \geq \Imag(v_-)$. We split the contour $\Gamma^+_{\alpha}$ into two parts: $\Gamma^{+,1}_{\alpha}$ -- this consists of the two segments connecting $(u_+ - t_1)$ with $(v_+ - t_1)$ and $(u_- - t_1)$ with $(v_- -t_1)$, and $\Gamma^{+,2}_{\alpha}$ -- given by $\Gamma^{+,2}_{\alpha} = \Gamma^+_{\alpha} \setminus \Gamma^{+,1}_{\alpha}$. All the contours are oriented in the direction of increasing imaginary part. See Figure \ref{S41}.

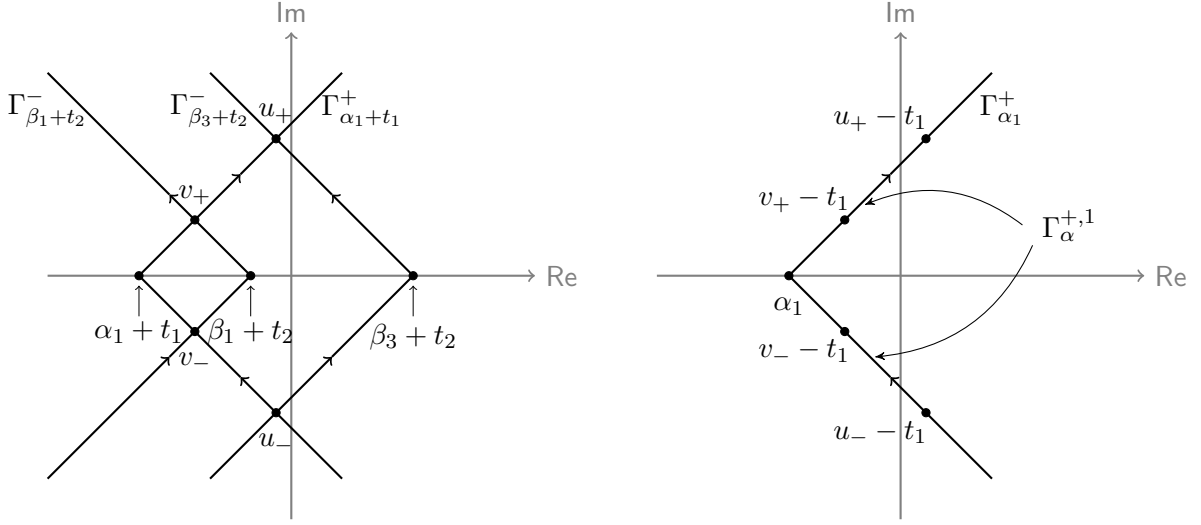
\begin{figure}[h]
    \centering
     \begin{tikzpicture}[scale=2.7]

        \def\tra{3} 
        % Picture on the left
        % Coordinate System
        \draw[->, thick, gray] (-1.2,0)--(1.2,0) node[right]{$\Real$};
        \draw[->, thick, gray] (0,-1.2)--(0,1.2) node[above]{$\Imag$};

        % The contour Gamma_{beta_1 + t_2}^-
        \draw[-,thick][black] (-0.2,0) -- (-0.6,-0.4);
        \draw[->,thick][black]  (-1.2,-1) -- (-0.6,-0.4);
        \draw[black, fill = black] (-0.2,0) circle (0.02);
        \draw (-0.2,-0.275) node{$\beta_1 + t_2$};
        \draw[->,very thin][black] (-0.2,-0.2) -- (-0.2, -0.05);
        \draw[->,thick][black] (-0.2,0) -- (-0.6,0.4);
        \draw[-,thick][black]  (-1.2,1) -- (-0.6,0.4);

         % The contour Gamma_{beta_3 + t_2}^-
        \draw[-,thick][black] (0.6,0) -- (0.2,-0.4);
        \draw[->,thick][black] (-0.4,-1) -- (0.2,-0.4);
        \draw[black, fill = black] (0.6,0) circle (0.02);
        \draw (0.6,-0.3) node{$\beta_3 + t_2$};
        \draw[->,very thin][black] (0.6,-0.2) -- (0.6, -0.05);
        \draw[->,thick][black] (0.6,0) -- (0.2,0.4);
        \draw[-,thick][black]  (-0.4,1) -- (0.2,0.4);       

        % The contour Gamma_{alpha_1}^+
        \draw[-,thick][black] (-0.75, 0) -- (-0.25,-0.5);
        \draw[->,thick][black] (0.25, -1) -- (-0.25, -0.5);
        \draw[black, fill = black] (-0.75,0) circle (0.02);
        \draw (-0.75,-0.275) node{$\alpha_1 + t_1$};
        \draw[->,very thin][black] (-0.75,-0.2) -- (-0.75, -0.05);
        \draw[->,thick][black] (-0.75,0) -- (-0.25,0.5);
        \draw[-,thick][black]  (-0.25,0.5) -- (0.25,1);

        % Intersection point
        \draw[black, fill = black] (-0.475,0.275) circle (0.02);
        \draw[black, fill = black] (-0.475,-0.275) circle (0.02);
        \draw[black, fill = black] (-0.075,0.675) circle (0.02);
        \draw[black, fill = black] (-0.075,-0.675) circle (0.02);
        \draw (-0.475,0.4) node{$v_+$};
        \draw (-0.475,-0.425) node{$v_-$};
        \draw (-0.075,0.8) node{$u_+$};
        \draw (-0.075,-0.825) node{$u_-$};

        % Contour names
        \draw (0.35,0.825) node{$\Gamma^+_{\alpha_1 + t_1}$};
        \draw (-0.4,0.825) node{$\Gamma^-_{\beta_3 + t_2}$};
        \draw (-1.2,0.825) node{$\Gamma^-_{\beta_1 + t_2}$};

        % Picture on the right assuming t_1 = 0.2, t_2 = -0.2
        % Coordinate System
        \draw[->, thick, gray] ({\tra -1.2},0)--({\tra + 1.2},0) node[right]{$\Real$};
        \draw[->, thick, gray] ({\tra + 0},-1.2)--({\tra + 0},1.2) node[above]{$\Imag$};

        % The contour Gamma_{alpha_1}^-
        \draw[-,thick][black] ({\tra -0.55}, 0) -- ({\tra -0.05},-0.5);
        \draw[->,thick][black] ({\tra + 0.45}, -1) -- ({\tra -0.05}, -0.5);
        \draw[black, fill = black] ({\tra -0.55},0) circle (0.02);
        \draw ({\tra -0.55},-0.15) node{$\alpha_1$};
        \draw[->,thick][black] ({\tra -0.55},0) -- ({\tra -0.05},0.5);
        \draw[-,thick][black]  ({\tra -0.05},0.5) -- ({\tra + 0.45},1);

        % Intersection point
        \draw[black, fill = black] ({\tra  -0.275},0.275) circle (0.02);
        \draw[black, fill = black] ({\tra -0.275},-0.275) circle (0.02);
        \draw[black, fill = black] ({\tra + 0.125},0.675) circle (0.02);
        \draw[black, fill = black] ({\tra + 0.125},-0.675) circle (0.02);
        \draw ({\tra -0.475},0.375) node{$v_+ - t_1$};
        \draw ({\tra -0.475},-0.375) node{$v_- - t_1$};
        \draw ({\tra - 0.1},0.775) node{$u_+ - t_1$};
        \draw ({\tra - 0.1},-0.775) node{$u_- - t_1$};

        % Contour names
        \draw ({\tra + 0.5},0.825) node{$\Gamma^+_{\alpha_1}$};
        \draw[->, >=stealth'] ({\tra + 0.6},0.25) node[anchor=west]{{ $\Gamma_{\alpha}^{+,1}$}} to[bend right] ({\tra -0.175},0.35);
        \draw[->, >=stealth'] ({\tra + 0.65},0.15) node[anchor=west]{} to[bend left] ({\tra - 0.125},-0.4);

    \end{tikzpicture} 
    \caption{The left part depicts the contours $\Gamma_{\alpha_1 + t_1}^+, \Gamma_{\beta_1 + t_2}^-, \Gamma_{\beta_3 + t_2}^-$ and their intersection points $v_{\pm}$, $u_{\pm}$. The right part depicts the contours $\Gamma_{\alpha_1}^+, \Gamma_{\alpha_1}^{+,1}$, where the latter consists of the two segments connecting $v_{\pm} - t_1$ with $u_{\pm} - t_1$.}
    \label{S41}
\end{figure}

We have from (\ref{3BPKer}) that
\begin{equation}\label{Indep2}
\begin{split}
K^{3, \alpha_1, \beta_1}_{a,b,c} (t_1, x_1; t_2, x_2) =  &\frac{1}{(2\pi \im)^2} \int_{\Gamma^{+,1}_{\alpha_1 }} d z \int_{\Gamma_{\beta_1}^-} dw \frac{e^{z^3/3 -x_1z - w^3/3 + x_2w}}{z + t_1 - w - t_2} \cdot \frac{\Phi_{a,b,c}(z + t_1) }{\Phi_{a,b,c}(w + t_2)} \\
&+ \frac{1}{(2\pi \im)^2} \int_{\Gamma^{+,2}_{\alpha_1 }} d z \int_{\Gamma_{\beta_1}^-} dw \frac{e^{z^3/3 -x_1z - w^3/3 + x_2w}}{z + t_1 - w - t_2} \cdot \frac{\Phi_{a,b,c}(z + t_1) }{\Phi_{a,b,c}(w + t_2)}.
\end{split}
\end{equation}
We now proceed to deform the contours $\Gamma_{\beta_1}^-$ to $\Gamma_{\beta_3}^-$ in the above two integrals. We note that the only singularity we can encounter is a simple pole at $w = z + t_1 - t_2$, and in the second line of (\ref{Indep2}) this pole lies outside of the region enclosed by $\Gamma_{\beta_1}^-$ and $\Gamma_{\beta_3}^-$. By Cauchy's theorem the second line of (\ref{Indep2}) remains unchanged if we deform $\Gamma_{\beta_1}^-$ to $\Gamma_{\beta_3}^-$. We mention that the decay estimates necessary to deform the contours near infinity come from (\ref{IntegrandKerB}), which one can analogously show for all $w$ in the region enclosed by $\Gamma_{\beta_1}^-$ and $\Gamma_{\beta_3}^-$ (rather than $w$ on the contours themselves). If we now deform $\Gamma_{\beta_1}^-$ to $\Gamma_{\beta_3}^-$ in the first line of (\ref{Indep2}) we obtain by the residue theorem for each $ z \in \Gamma^{+,1}_{\alpha_1 }$
\begin{equation*}
\begin{split}
&  \int_{\Gamma_{\beta_1}^-} \hspace{-2mm} dw \frac{e^{z^3/3 -x_1z - w^3/3 + x_2w}}{z + t_1 - w - t_2} \cdot \frac{\Phi_{a,b,c}(z + t_1) }{\Phi_{a,b,c}(w + t_2)} =   \int_{\Gamma_{\beta_3}^-}\hspace{-2mm} dw \frac{e^{z^3/3 -x_1z - w^3/3 + x_2w}}{z + t_1 - w - t_2} \cdot \frac{\Phi_{a,b,c}(z + t_1) }{\Phi_{a,b,c}(w + t_2)} \\
& - 2 \pi \im \cdot \exp \left( z^3/3 - x_1 z - (z + t_1 - t_2)^3/3 + x_2 (z + t_1 -t_2) \right).
\end{split}
\end{equation*}
Combining the last few observations with (\ref{Indep2}) we conclude that 
\begin{equation}\label{Indep3}
\begin{split}
&K^{3, \alpha_1, \beta_1}_{a,b,c} (t_1, x_1; t_2, x_2) = K^{3, \alpha_1, \beta_3}_{a,b,c} (t_1, x_1; t_2, x_2) \\
&+ \frac{1}{2\pi \im} \int_{\Gamma^{+,1}_{\alpha_1 }} d z  \exp \left( z^3/3 - x_1 z - (z + t_1 - t_2)^3/3 + x_2 (z + t_1 -t_2) \right).
\end{split}
\end{equation}
We now perform a change of variables in the second line of (\ref{Indep3}) to $w = z + t_1$. This gives
\begin{equation}\label{Indep4}
\begin{split}
&\frac{1}{2\pi \im} \int_{\Gamma^{+,1}_{\alpha_1 }} d z  \exp \left( z^3/3 - x_1 z - (z + t_1 - t_2)^3/3 + x_2 (z + t_1 -t_2) \right)  \\
& = \frac{1}{2\pi \im} \int_{v_+}^{u_+} e^{(t_2 - t_1)w^2 + (t_1^2 - t_2^2) w + w (x_2-x_1) + x_1 t_1 - x_2 t_2 - t_1^3/3 + t_2^3/3} dw \\
& + \frac{1}{2\pi \im} \int_{u_-}^{v_-} e^{(t_2 - t_1)w^2 + (t_1^2 - t_2^2) w + w (x_2-x_1) + x_1 t_1 - x_2 t_2 - t_1^3/3 + t_2^3/3} dw  \\
& = K^{1, \alpha_1, \beta_3}_{a,b,c} (t_1, x_1; t_2, x_2) - K^{1, \alpha_1, \beta_1}_{a,b,c} (t_1, x_1; t_2, x_2) ,
\end{split}
\end{equation}
where in the last equality we used Cauchy's theorem and that by definition
$$K^{1, \alpha_1, \beta_1}_{a,b,c} (t_1, x_1; t_2, x_2) = \frac{1}{2\pi \im} \int_{v_-}^{v_+} e^{(t_2 - t_1)w^2 + (t_1^2 - t_2^2) w + w (x_2-x_1) + x_1 t_1 - x_2 t_2 - t_1^3/3 + t_2^3/3} dw$$
$$K^{1, \alpha_1, \beta_3}_{a,b,c} (t_1, x_1; t_2, x_2) = \frac{1}{2\pi \im} \int_{u_-}^{u_+} e^{(t_2 - t_1)w^2 + (t_1^2 - t_2^2) w + w (x_2-x_1) + x_1 t_1 - x_2 t_2 - t_1^3/3 + t_2^3/3} dw.$$
The first line of (\ref{Indep1}) now follows from (\ref{Indep3}) and (\ref{Indep4}). The second line of (\ref{Indep1}) is proved analogously so we omit the proof.\\

We finally show that $K_{a,b,c}(t_1, \cdot; t_2, \cdot)$ is continuous  on $ \mathbb{R}^2$. Suppose that $x_i^n \rightarrow x_i$ as $n \rightarrow \infty$ for $i = 1,2$. The continuity of $K_{a,b,c}(t_1, \cdot; t_2, \cdot)$ at $(x_1, x_2)$ follows from
\begin{equation}\label{ContK}
\begin{split}
\lim_{n \rightarrow \infty} K^i_{a,b,c}(t_1, x_1^n; t_2, x_2^n) \rightarrow K^i_{a,b,c}(t_1, x_1; t_2, x_2) \mbox{ for }i = 1,2,3. 
\end{split}
\end{equation}
Using (\ref{3BPKer}), we have that (\ref{ContK}) holds for $i = 1$ by the bounded convergence theorem, for $i = 2$ directly from the definition of $K^2_{a,b,c}$ and for $i = 3$ by the dominated convergence theorem with dominating function as in (\ref{IntegrandKerB}).
\end{proof}

%%%%%%%%%%%%%%%%%%%%%%%%%%%%%%%%%%%%%%%%%%%%%%%%%%%%%%%%%%%%%%%%%%%%%
%
%    Section 4.2
%
%%%%%%%%%%%%%%%%%%%%%%%%%%%%%%%%%%%%%%%%%%%%%%%%%%%%%%%%%%%%%%%%%%%%%
\subsection{Parameter scaling and prelimit formula} \label{Section4.2}

In the following definition we explain how we scale the parameters in the ascending Schur process of Section \ref{Section3.1}.
\begin{definition}\label{ParScale}
Assume that $\{a_i^+\}_{ i \geq 1}$, $\{a_i^-\}_{ i \geq 1}$, $\{b_i^+\}_{ i \geq 1}$, $\{b_i^-\}_{ i \geq 1}$, $c^+$, $c^-$ are as in Definition \ref{DLP}. Fix $m \in \mathbb{N}$, $q \in (0,1)$ and $t_1, \dots, t_m \in \mathbb{R}$ with $t_1 < \cdots < t_m$. Define also the constants 
\begin{equation}\label{SigmaQ}
\sigma_q = \frac{q^{1/3} (1 + q)^{1/3}}{1- q} \mbox{ and } f_q = \frac{q^{1/3}}{2 (1 + q)^{2/3}}.
\end{equation}

For $N \in \mathbb{N}$ increasing to infinity, we put $M_k(N) = N + \lfloor t_k N^{2/3} \rfloor$ for $k = 1, \dots, m$. We consider five numbers $A_N, B_N, C_N^{+}, C_N^{-}, D_N$ and sequences $\{x^N_i \}_{ i\geq 1}$ and $\{y^N_i\}_{i \geq 1}$ such that
\begin{equation}\label{ABSeq}
\begin{split}
&x_i^N =1 - \frac{1}{N^{1/3} b_i^+ \sigma_q } \mbox{ for $i = 1, \dots, B_N$, and } y_i^N =1 -\frac{1}{N^{1/3} a_i^+ \sigma_q } \mbox{ for $i = 1, \dots, A_N$},
\end{split}
\end{equation}
where $B_N \leq \min\left(\lfloor N^{1/12} \rfloor, J_b^+ \right)$ is the largest integer such that $x^N_{B_N} \geq q$, and $A_N \leq \min\left(\lfloor N^{1/12} \rfloor, J_a^+ \right)$ is the largest integer such that $y^N_{A_N} \geq q$. Here, we use the convention $x_0^N = y_0^N = 1$ so that $A_N =0$ and $B_N = 0$ are possible. We also have
\begin{equation}\label{DSeq}
\begin{split}
&D_N = \min \left( \max(J_a^-, J_b^-), \lfloor N^{1/12} \rfloor \right)  \\
&x^N_{B_N  + i} = 1 - \frac{b_i^- }{N^{1/3}\sigma_q}   , \hspace{2mm}   y^N_{A_N + i} = 1 - \frac{a_i^-}{N^{1/3} \sigma_q}  \mbox{ for } i = 1, \dots, D_N;
\end{split}
\end{equation}
\begin{equation}\label{CSeq}
\begin{split}
&C^-_N =  \begin{cases} 0 &\mbox{ if $c^- = 0$, } \\  \lfloor N^{1/12} \rfloor &\mbox{ if } c^- > 0 \end{cases}, \hspace{2mm} C^+_N =  \begin{cases} 0 &\mbox{ if $c^+ = 0$, } \\  \lfloor N^{1/12} \rfloor  &\mbox{ if } c^+ > 0 \end{cases}, \hspace{2mm}  \\
&x_{B_N + D_N +  i }^N = y_{A_N + D_N + i}^N =1 - \frac{2}{N^{1/4}c^+ \sigma_q} \mbox{ for } i = 1 ,\dots, C_N^+  \mbox{ and }\\
& x_{B_N + D_N + C_N^+ + i }^N = y_{A_N + D_N +  C_N^+ + i}^N = 1  - \frac{c^-}{2N^{5/12} \sigma_q} \mbox{ for $i = 1, \dots, C_N^-$};
\end{split}
\end{equation}
\begin{equation}\label{RemSeq}
\begin{split}
\mbox{ $x_i^N = q$ for $i > B_N + D_N + C_N^+ + C_N^-$ and $y_i^N = q$ for $i > A_N + D_N + C_N^+ + C_N^-$.}
\end{split}
\end{equation}
We let $N_0 \in \mathbb{N}$ be sufficiently large (depending on $q, a_1^-, b_1^-, c^+, c^-$ and $t_1, \dots, t_m$) so that $M_m(N) > \cdots > M_1(N) >  B_N + D_N + C_N^+ + C_N^-$,  $N > A_N + D_N + C_N^+ + C_N^-$ and $x_i^N, y_i^N \geq q$ for all $i \in \mathbb{N}$, provided that $N \geq N_0$. We also record for future use $\tilde{N} = N - A_N - C_N^+ - C_N^- - D_N$. 

Note that if $N \geq N_0$ and $M \geq M_m(N)$ we can define the ascending Schur process of Section \ref{Section3.1} with parameters $N, M, \{x_i^N\}_{i = 1}^N$ and $\{y_i^N\}_{ i = 1}^M$ as above. We observe that the condition $x^N_i y^N_j \in [0,1)$ is satisfied by the construction above. Indeed, $x^N_i \in [0, 1)$ and $y^N_j \in [0,1)$ in (\ref{ABSeq}), (\ref{CSeq}) and (\ref{RemSeq}). In (\ref{DSeq}) we could have that $x^N_i \in [0,1]$ or $y^N_j \in [0,1]$; however, by the definition of $D_N$ we cannot have simultaneously $x_i^N = y_j^N = 1$. For $N \geq N_0$ we denote the distribution (\ref{SP}) with the above set of parameters by $\mathbb{P}_N$.
\end{definition}

Now that we have explained how we scale parameters with $N$, we proceed to derive an alternative formula for the correlation kernel from Proposition \ref{PropCK}, which is suitable for taking the $N \rightarrow \infty$ limit. This alternative formula is recorded in Lemma \ref{PrelimitKernel} and we proceed to introduce the relevant contours that appear in that result and its proof.

\begin{definition}\label{finiteContours} Let $a \in \mathbb{R}$, and $\delta \in (0, 1/2]$. With this data we define two contours $\gamma_{a,\delta}^{+}$ and $\gamma_{a, \delta}^-$. The contour $\gamma_{a, \delta}^{+}$ consists of two segments connecting $a$ with $a + \delta \pm \im \delta$, as well as a circular arc of a $0$-centered circle that passes through the points $a + \delta \pm \im \delta$. Analogously, $\gamma_{a, \delta}^{-}$ consists of two segments connecting $a$ with $a - \delta \pm \im \delta$, as well as a circular arc of a $0$-centered circle that passes through the points $a - \delta \pm \im \delta$. Both $\gamma_{a,\delta}^{+}$ and $\gamma_{a, \delta}^-$ are positively oriented. We will call the portion of $\gamma_{a,\delta}^{\pm}$ that consists of the straight segments by $\gamma_{a,\delta}^{\pm, \vert}$ and the circular portions by $\gamma_{a,\delta}^{\pm, \circ}$. 

If we are further given a $\rho \in [0, \delta)$, we will consider the contours $\gamma_{a,\delta, \rho}^{+}$ and $\gamma_{a, \delta, \rho}^-$. These are obtained from $\gamma_{a,\delta}^{+}$ (resp. $\gamma_{a, \delta}^-$), by replacing the two segments that connect $a$ to $a+ \rho \pm \im \rho$ (resp. $a - \rho \pm  \im \rho$) with a straight vertical segment that connects $a+ \rho \pm \im \rho$ (resp. $a - \rho \pm  \im \rho$). The contours in this definition are depicted in Figure \ref{S42}.
\end{definition}
\begin{figure}[h]
    \centering
     \begin{tikzpicture}[scale=2.7]

        \def\tra{3} 
        % Picture on the left
        % Coordinate System
        \draw[->, thick, gray] (-1.2,0)--(1.2,0) node[right]{$\Real$};
        \draw[->, thick, gray] (0,-1.2)--(0,1.2) node[above]{$\Imag$};
        \def\radA{0.824} % big radius
        \def\radB{0.447} % small radius

        % The contour gamma_{a,delta}^+
        \draw[->,thick][black] (0.6,0) -- (0.7,0.1);
        \draw[-,thick][black] (0.7,0.1) -- (0.8,0.2);
        \draw[-,thick][black] (0.7,-0.1) -- (0.6,0);
        \draw[->,thick][black] (0.8,-0.2) -- (0.7,-0.1);
        \draw[black, fill = black] (0.6,0) circle (0.02);
        \draw (0.6,-0.1) node{$a$};
        \draw[->,thick][black] (0.8,0.2) arc (14.05:180:\radA);
        \draw[-,thick][black] (0.8,0.2) arc (14.05:360 - 14.05:\radA);

        % The contour gamma_{a,delta}^-
        \draw[->,thick][black] (0.6,0) -- (0.5,0.1);
        \draw[-,thick][black] (0.5,0.1) -- (0.4,0.2);
        \draw[-,thick][black] (0.5,-0.1) -- (0.6,0);
        \draw[->,thick][black] (0.4,-0.2) -- (0.5,-0.1);
        \draw[->,thick][black] (0.4,0.2) arc (26.58:180:\radB);
        \draw[-,thick][black] (0.4,0.2) arc (26.58:360 - 26.58:\radB);

        % Draw delta  
        \draw[dashed,very thin][gray] (0.35,0.2) -- (0.95, 0.2);
        \draw[dashed,very thin][gray] (0.35,-0.2) -- (0.95, -0.2);      
        \draw[->,very thin][black] (0.9,0) -- (0.9, 0.2);
        \draw[->,very thin][black] (0.9,0.2) -- (0.9, 0);
        \draw[->,very thin][black] (0.9,0) -- (0.9, -0.2);
        \draw[->,very thin][black] (0.9,-0.2) -- (0.9, 0);
        \draw (0.95,0.1) node{$\delta$};
        \draw (0.95,-0.1) node{$\delta$};

        % Contour names
        \draw (0.5,0.8) node{$\gamma_{a,\delta}^{+,\circ}$};
        \draw (0.35,0.45) node{$\gamma_{a,\delta}^{-, \circ}$};
         \draw[->, >=stealth'] (0.3,0.15) node[anchor=east]{{ $\gamma_{a,\delta}^{-, \vert}$}} to[bend right] (0.45,0.15);
        \draw[->, >=stealth'] (0.8,0.4)  to[bend right] (0.75,0.15);
        \draw (0.94,0.42) node{$\gamma_{a,\delta}^{+,\vert}$};

        % Picture on the right
        % Coordinate System
        \draw[->, thick, gray] ({\tra -1.2},0)--({\tra + 0.9},0) node[right]{$\Real$};
        \draw[->, thick, gray] ({\tra + 0},-1.2)--({\tra + 0},1.2) node[above]{$\Imag$};
        \def\radA{0.824} % big radius
        \def\radB{0.447} % small radius

        % The contour gamma_{a,delta,rho}^+
        \draw[->,thick][black] ({\tra + 0.7},-0.1) -- ({\tra + 0.7},0);
        \draw[-,thick][black] ({\tra + 0.7},0) -- ({\tra + 0.7},0.1);
        \draw[-,thick][black] ({\tra + 0.7},0.1) -- ({\tra + 0.8},0.2);
        \draw[-,thick][black] ({\tra + 0.8},-0.2) -- ({\tra + 0.7},-0.1);
        \draw[black, fill = black] ({\tra + 0.6},0) circle (0.02);
        \draw ({\tra + 0.6},-0.1) node{$a$};
        \draw[->,thick][black] ({\tra + 0.8},0.2) arc (14.05:180:\radA);
        \draw[-,thick][black] ({\tra + 0.8},0.2) arc (14.05:360 - 14.05:\radA);

        % The contour gamma_{a,delta,rho}^-       
        \draw[-,thick][black] ({\tra + 0.5},0.1) -- ({\tra + 0.4},0.2);
        \draw[-,thick][black] ({\tra + 0.5},-0.1) -- ({\tra + 0.4},-0.2);
        \draw[->,thick][black] ({\tra + 0.5},-0.1) -- ({\tra + 0.5},0);
        \draw[-,thick][black] ({\tra + 0.5},0) -- ({\tra + 0.5},0.1);
        \draw[->,thick][black] ({\tra + 0.4},0.2) arc (26.58:180:\radB);
        \draw[-,thick][black] ({\tra + 0.4},0.2) arc (26.58:360 - 26.58:\radB);
        
        % Draw rho and delta  
        \draw ({\tra + 0.3},0.1) node{$\delta$};
        \draw ({\tra + 0.3},-0.1) node{$\delta$};
        \draw[dashed,very thin][gray] ({\tra + 0.35},0.2) -- ({\tra + 0.8}, 0.2);
        \draw[dashed,very thin][gray] ({\tra + 0.35},-0.2) -- ({\tra + 0.8}, -0.2);
        \draw[dashed,very thin][gray] ({\tra + 0.5},0.1) -- ({\tra + 1.2}, 0.1);
        \draw[dashed,very thin][gray] ({\tra + 0.5},-0.1) -- ({\tra + 1.2}, -0.1);
        \draw[->,very thin][black] ({\tra + 0.35},0) -- ({\tra + 0.35}, 0.2);
        \draw[->,very thin][black] ({\tra + 0.35},0.2) -- ({\tra + 0.35}, 0);
        \draw[->,very thin][black] ({\tra + 0.35},0) -- ({\tra + 0.35}, -0.2);
        \draw[->,very thin][black] ({\tra + 0.35},-0.2) -- ({\tra + 0.35}, 0);
        \draw[->,very thin][black] ({\tra + 1.2},0) -- ({\tra + 1.2}, -0.1);
        \draw[->,very thin][black] ({\tra + 1.2},0) -- ({\tra + 1.2}, 0.1);
        \draw ({\tra + 1.3},0) node{$2\rho$};

        % Contour names
        \draw ({\tra + 0.5},0.8) node{$\gamma_{a,\delta,\rho}^+$};
        \draw ({\tra + 0.35},0.45) node{$\gamma_{a,\delta, \rho}^-$};

    \end{tikzpicture} 
    \caption{The left part depicts the contours $\gamma_{a,\delta}^{\pm,\circ}, \gamma_{a,\delta}^{\pm, \vert}$. The right part depicts the contours $\gamma_{a,\delta,\rho}^-, \gamma_{a,\delta, \rho}^+$.}
    \label{S42}
\end{figure}
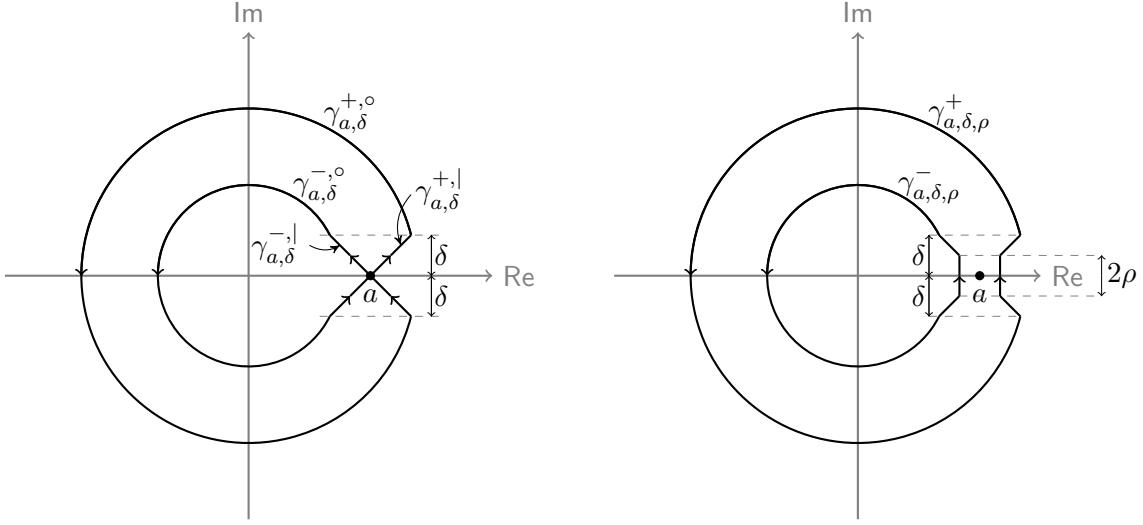

\begin{lemma}\label{PrelimitKernel}
Assume the same notation in Definitions \ref{DLP} and \ref{ParScale}, and fix $\delta \in (0, 1/2]$. For $N$ large enough so that $N \geq N_0$ and $\delta > N^{-1/3}$, we let $\mathfrak{S}(\lambda)$ be the point process as in (\ref{PPDef}) for $\lambda$ distributed according to $\mathbb{P}_N$. Then, $\mathfrak{S}(\lambda)$ is a determinantal point process on $\{1, \dots, m\} \times \mathbb{R}$ with reference measure given by the counting measure on $\{1, \dots, m\} \times \mathbb{Z}$ and correlation kernel 
\begin{equation}\label{CKDefAlt}
K_N(u,x; v, y) = K^1_N(u,x;v,y) + K^2_N(u,x;v,y) + K^3_N(u,x;v,y),
\end{equation}
where 
\begin{equation}\label{PKF1}
\begin{split}
K^1_N(u,x;v,y) =\left( 1- q \right)^{M_u(N) - M_v(N)} \cdot  \frac{1}{2\pi \im} \int_{1 -  \im N^{-1/3}}^{1 + \im N^{-1/3}} (1 - qw)^{{M_v}(N) - {M_u}(N)} \cdot w^{y - x - 1} dw,
\end{split}
\end{equation}
\begin{equation}\label{PKF2}
\begin{split}
K^2_N(u,x;v,y) =-\left( 1- q \right)^{M_u(N) - M_v(N)} \cdot  \frac{{\bf 1} \{ u> v\}}{2\pi \im} \cdot \int_{C_1} (1 - qw)^{{M_v}(N) - {M_u}(N)} \cdot w^{y - x - 1} dw,
\end{split}
\end{equation}
with $C_1$ being the positively oriented, zero-centered circle of radius $1$. The kernel $K^3_N$ is given by
\begin{equation}\label{PKF3}
\begin{split}
K^3_N(u,x;v,y) =\frac{1}{(2\pi \im)^2} \int_{\gamma_{z}}dz  \int_{\gamma_{w}}dw \frac{1}{z-w} \cdot \prod_{i = 1}^7 H_i^N(z,w) \cdot w^y z^{-x-1} ,
\end{split}
\end{equation}
where $\gamma_{w} = \gamma_{1 + N^{-1/3}, \delta}^{-}$ and $\gamma_{z} = \gamma_{1-N^{-1/3}, \delta}^+$ with $\gamma_{a,\delta}^{\pm}$ as in Definition \ref{finiteContours}. The functions $H^N_i(z,w)$ are given by 
\begin{equation}\label{PKF41}
H^N_1(z,w) =  \prod_{i = 1}^{A_N} \frac{1 - \frac{1}{z} \cdot  \left(1 - \frac{1}{N^{1/3} a_i^+ \sigma_q} \right)}{1 -  \frac{1}{w} \cdot  \left(1 - \frac{1}{N^{1/3} a_i^+ \sigma_q} \right)}, \hspace{2mm} H^N_2(z,w) =\prod_{i = 1}^{B_N} \frac{1 - w \cdot \left(1 - \frac{1}{N^{1/3} b_i^+ \sigma_q} \right)}{1 - z \cdot \left(1 - \frac{1}{N^{1/3} b_i^+ \sigma_q} \right)}, 
\end{equation}
\begin{equation}\label{PKF42}
H^N_3(z,w) = \prod_{i = 1}^{D_N} \frac{1 - w \cdot\left(1 - \frac{b_i^-}{N^{1/3}\sigma_q}  \right)}{1 - z \cdot\left(1 - \frac{b_i^-}{N^{1/3}\sigma_q}  \right)} \cdot \frac{1 - \frac{1}{z} \cdot \left(1 - \frac{a_i^-}{N^{1/3} \sigma_q}  \right)}{1 - \frac{1}{w} \cdot \left(1 - \frac{a_i^-}{N^{1/3} \sigma_q}  \right)},
\end{equation}
\begin{equation}\label{PKF43}
H^N_4(z,w) = \left( \frac{1 - w \cdot \left(1 - \frac{c^-}{2N^{5/12} \sigma_q}\right)}{1 - z \cdot \left(1 - \frac{c^-}{2N^{5/12} \sigma_q} \right)}  \cdot \frac{1 - \frac{1}{z} \cdot \left(1 - \frac{c^-}{2N^{5/12} \sigma_q} \right)}{1 - \frac{1}{w} \cdot \left(1  - \frac{c^-}{2N^{5/12} \sigma_q} \right)} \right)^{C_N^-}, 
\end{equation}
\begin{equation}\label{PKF44}
H^N_5(z,w) = \left( \frac{1 - w \cdot\left(1 - \frac{2}{N^{1/4} c^+ \sigma_q} \right)}{1 - z \cdot \left(1 - \frac{2}{N^{1/4} c^+ \sigma_q} \right)} \cdot \frac{1 - \frac{1}{z} \cdot \left(1 - \frac{2}{N^{1/4} c^+ \sigma_q} \right)}{1 - \frac{1}{w} \cdot \left(1 - \frac{2}{N^{1/4} c^+ \sigma_q} \right)} \right)^{C_N^+}, 
\end{equation}
\begin{equation}\label{PKF45}
H^N_6(z,w) = \left( \frac{(1 - q/z)(1 - qw)}{(1 - q/w)(1 - qz)}\right)^{\tilde{N}}, H_7^N(z,w) = \frac{(1 - qw)^{\tilde{M}_v}}{(1 - qz)^{\tilde{M}_u}} \cdot (1-q)^{\tilde{M}_u- \tilde{M}_v},
\end{equation}
with $\tilde{M}_r = M_r(N) +A_N -B_N - N $ for $r = 1, \dots, m$. Part of the statement is that all integrals above are well-defined and finite for each $N$ as in the lemma.
\end{lemma}
\begin{proof} Let $\overline{y} = \max_{1 \leq i \leq N} y_i^N$ and $\underline{x} = \min_{1 \leq j \leq M_m} (x_j^N)^{-1}$. As explained in Definition \ref{ParScale}, we have that $\overline{y} < \underline{x}$ provided that $N \geq N_0$ and so we can apply Proposition \ref{PropCK}. By conjugating the kernel (\ref{CKDef}) in that proposition by $(1-q)^{M_u - M_v}$, and applying part (4) of Proposition \ref{PropLem}, we see that $\mathfrak{S}(\lambda)$ is a determinantal point process on $\{1, \dots, m\} \times \mathbb{R}$ with reference measure given by the counting measure on $\{1, \dots, m\} \times \mathbb{Z}$ and correlation kernel 
\begin{equation}\label{CKDefAlt1.5}
\tilde{K}_N(u,x; v, y) = \frac{1}{(2\pi \im)^2} \oint_{C_{r_1}}dz \oint_{C_{r_2}}dw \frac{1}{z-w} \cdot\prod_{i = 1}^7H^N_i(z,w)  \cdot w^y z^{-x -1},
\end{equation}
where $\overline{y} < r_1, r_2 < \underline{x}$ and we have $r_1 < r_2$ when $u > v$, and $r_2 < r_1$ when $u \leq v$. We mention that in deriving (\ref{CKDefAlt1.5}) we also substituted the sequences $x_i^N, y_i^N$ from Definition \ref{ParScale} and used the definition of $H^N_i(z,w)$ from (\ref{PKF41} - \ref{PKF45}).

If $u > v$, we can deform the $z$ contour outside of the $w$ contour, and in the process we pick up a residue at $z = w$ coming from the $z-w$ in the denominator. Observe that $H^N_i(w,w) = 1$ for all $i = 1, \dots, 6$ and $H^N_7(w,w) = (1-qw)^{M_v(N) - M_u(N)} \cdot (1-q)^{M_u(N) - M_v(N)}$. We conclude that
\begin{equation}\label{CKDefAlt2}
\tilde{K}_N(u,x; v, y) = \frac{1}{(2\pi \im)^2} \oint_{C_{r_1}}dz \oint_{C_{r_2}}dw \frac{1}{z-w} \cdot\prod_{i = 1}^7H^N_i(z,w)  \cdot w^y z^{-x -1} + K_N^2(u,x;v,y),
\end{equation}
where $\overline{y} < r_2 < r_1 < \underline{x}$ for all $u,v \in \{1, \dots, m\}$. What remains is to show that the double integral in (\ref{CKDefAlt2}) is equal to $ K_N^1(u,x;v,y) +  K_N^3(u,x;v,y)$.

We now define $\rho_1, \rho_2 \in (0, \delta)$ as follows. If $\overline{y} = 1$, then we can find $\epsilon > 0$ such that $\underline{x} > 1 + \epsilon$, and we proceed to fix $\rho_1$ so that $ N^{-1/3} < \rho_1 < \min(\delta, N^{-1/3} + \epsilon, 2N^{-1/3})$  and then fix $\rho_2$ so that $2N^{-1/3} - \rho_1 < \rho_2 < N^{-1/3}$. If $\overline{y} < 1$, we can find $\epsilon > 0$ such that $\overline{y} < 1- \epsilon$, and we proceed to fix $\rho_2$ so that $N^{-1/3} < \rho_2 < \min(\delta, N^{-1/3} + \epsilon, 2N^{-1/3})$  and then fix $\rho_1$ so that $2N^{-1/3} - \rho_2 < \rho_1 < N^{-1/3}$. By Cauchy's theorem we can deform $C_{r_1}$ and $C_{r_2}$ in (\ref{CKDefAlt2}) to $\gamma^+_{\rho_1}:=\gamma^+_{1 - N^{-1/3}, \delta, \rho_1}$ and $\gamma^-_{\rho_2}:= \gamma^-_{1 + N^{-1/3}, \delta, \rho_2}$, respectively, without crossing any poles and hence without changing the value of the integral. We recall that the contours $\gamma_{a,\delta, \rho}^{\pm}$ are as in Definition \ref{finiteContours}, see also Figure \ref{S43}. We have thus reduced the proof to showing that 
\begin{equation}\label{CKDefAlt3}
 \frac{1}{(2\pi \im)^2} \oint_{\gamma^+_{\rho_1}}dz \oint_{\gamma^-_{\rho_2}}dw \frac{1}{z-w} \cdot\prod_{i = 1}^7H^N_i(z,w)  \cdot w^y z^{-x -1} = K_N^1(u,x;v,y) +  K_N^3(u,x;v,y).
\end{equation}
Below we show that the right side of (\ref{CKDefAlt3}) is well-defined. Afterwards, we will deform the $\gamma_z$ and $\gamma_w$ contours in the definition of $K_N^3(u,x;v,y)$ to $\gamma^+_{\rho_1}$ and $\gamma^-_{\rho_2}$, respectively. In the process of the deformation, we will pick up a residue, which will precisely cancel with $ K_N^1(u,x;v,y) $, hence establishing (\ref{CKDefAlt3}). We turn to providing the details below.

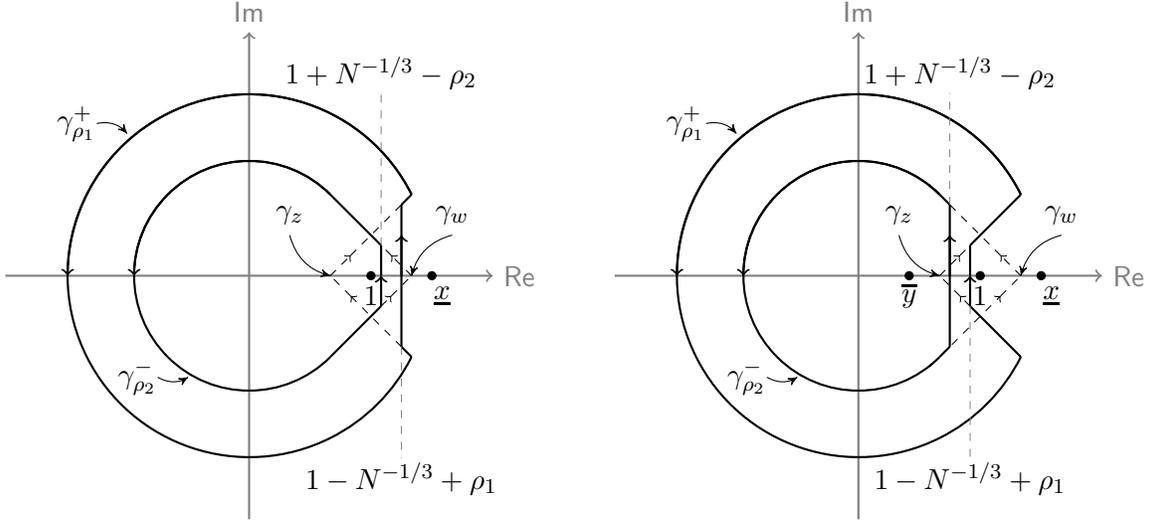
\begin{figure}[h]
    \centering
     \begin{tikzpicture}[scale=2.7]

        \def\tra{3} 
        % Picture on the left
        % Coordinate System
        \draw[->, thick, gray] (-1.2,0)--(1.2,0) node[right]{$\Real$};
        \draw[->, thick, gray] (0,-1.2)--(0,1.2) node[above]{$\Imag$};
        \def\radA{0.894} % big radius
        \def\radB{0.566} % small radius

        % Parts of gamma_w and gamma_z
        \draw[->,dashed][black] (0.4,0) -- (0.5,0.1);
        \draw[-,dashed][black] (0.5,0.1) -- (0.75,0.35);
        \draw[->,dashed][black] (0.75,-0.35) -- (0.5,-0.1);
        \draw[-,dashed][black] (0.5,-0.1) -- (0.4,0);
        \draw[->,dashed][black] (0.8,0) -- (0.7,0.1);
        \draw[-,dashed][black] (0.7,0.1) -- (0.65,0.15);
        \draw[->,dashed][black] (0.65,-0.15) -- (0.7,-0.1);
        \draw[-,dashed][black] (0.7,-0.1) -- (0.8,0);

        % The contour gamma_{rho_1}^+
        \draw[->,thick][black] (0.75,-0.35) -- (0.75,0.2);
        \draw[-,thick][black] (0.75,0) -- (0.75,0.35);
        \draw[-,thick][black] (0.75,-0.35) -- (0.8,-0.4);
        \draw[-,thick][black] (0.75,0.35) -- (0.8,0.4);
        \draw[black, fill = black] (0.6,0) circle (0.02);
        \draw (0.6,-0.1) node{$1$};
        
        \draw[->,thick][black] (0.8,0.4) arc (26.57:180:\radA);
        \draw[-,thick][black] (0.8,0.4) arc (26.57:360 - 26.57:\radA);
        \draw[black, fill = black] (0.9,0) circle (0.02);
        \draw (0.95,-0.1) node{$\underline{x}$};  

        % The contour gamma_{rho_2}^-
        \draw[->,thick][black] (0.65,-0.15) -- (0.65,0);
        \draw[-,thick][black] (0.65,0.15) -- (0.65,0);
        \draw[-,thick][black] (0.65,0.15) -- (0.4,0.4);
        \draw[-,thick][black] (0.65,-0.15) -- (0.4,-0.4);
        \draw[->,thick][black] (0.4,0.4) arc (45:180:\radB);
        \draw[-,thick][black] (0.4,0.4) arc (45:360 - 45:\radB);

        % Contour names
        \draw[dashed,very thin][gray] (0.75,-0.35) -- ( 0.75, -0.9);
        \draw[dashed,very thin][gray] (0.65,0.15) -- ( 0.65, 0.9);
        \draw ( 0.75,-1) node{$1 - N^{-1/3} + \rho_1$};
        \draw ( 0.65,1) node{$1 + N^{-1/3} - \rho_2$};

        \draw (- 0.55,-0.5) node{$\gamma_{\rho_2}^-$};
        \draw (- 0.85,0.75) node{$\gamma_{\rho_1}^+$};
        \draw (0.2,0.3) node{$\gamma_{z}$};
        \draw (1,0.3) node{$\gamma_{w}$};
        
        \draw[->, >=stealth'] ( - 0.45, -0.5)  to[bend right] ( - 0.3,-0.5);
        \draw[->, >=stealth'] ( - 0.75, 0.75)  to[bend left] ( - 0.60,0.7);
        \draw[->, >=stealth'] ( 1, 0.2)  to[bend right] ( 0.8,0);
        \draw[->, >=stealth'] ( 0.2, 0.2)  to[bend right] ( 0.4,0);

        % Picture on the right
        % Coordinate System
        \draw[->, thick, gray] ({\tra -1.2},0)--({\tra + 1.2},0) node[right]{$\Real$};
        \draw[->, thick, gray] ({\tra + 0},-1.2)--({\tra + 0},1.2) node[above]{$\Imag$};

        % Parts of gamma_w and gamma_z
        \draw[->,dashed][black] ({\tra + 0.4},0) -- ({\tra + 0.5},0.1);
        \draw[-,dashed][black] ({\tra + 0.5},0.1) -- ({\tra + 0.55},0.15);
        \draw[->,dashed][black] ({\tra + 0.55},-0.15) -- ({\tra + 0.5},-0.1);
        \draw[-,dashed][black] ({\tra + 0.5},-0.1) -- ({\tra + 0.4},0);
        \draw[->,dashed][black] ({\tra + 0.8},0) -- ({\tra + 0.7},0.1);
        \draw[-,dashed][black] ({\tra + 0.7},0.1) -- ({\tra + 0.35},0.45);
        \draw[->,dashed][black] ({\tra + 0.35},-0.45) -- ({\tra + 0.7},-0.1);
        \draw[-,dashed][black] ({\tra + 0.7},-0.1) -- ({\tra + 0.8},0);

        % The contour gamma_{rho_1}^+
        \draw[->,thick][black] ({\tra + 0.55},-0.15) -- ({\tra + 0.55},0);
        \draw[-,thick][black] ({\tra + 0.55},0) -- ({\tra + 0.55},0.15);
        \draw[-,thick][black] ({\tra + 0.55},0.15) -- ({\tra + 0.8},0.4);
        \draw[-,thick][black] ({\tra + 0.8},-0.4) -- ({\tra + 0.55},-0.15);
        \draw[black, fill = black] ({\tra + 0.6},0) circle (0.02);
        \draw ({\tra + 0.6},-0.1) node{$1$};
        \draw[black, fill = black] ({\tra + 0.9},0) circle (0.02);
        \draw ({\tra + 0.95},-0.1) node{$\underline{x}$};       
        \draw[->,thick][black] ({\tra + 0.8},0.4) arc (26.57:180:\radA);
        \draw[-,thick][black] ({\tra + 0.8},0.4) arc (26.57:360 - 26.57:\radA);

        % The contour gamma_{rho_2}^-     
        \draw[-,thick][black] ({\tra + 0.45},0) -- ({\tra + 0.45},0.35);
        \draw[->,thick][black] ({\tra + 0.45},-0.35) -- ({\tra + 0.45},0.2);
        \draw[-,thick][black] ({\tra + 0.4},-0.4) -- ({\tra + 0.45},-0.35);
        \draw[-,thick][black] ({\tra + 0.4},0.4) -- ({\tra + 0.45},0.35);       
        \draw[->,thick][black] ({\tra + 0.4},0.4) arc (45:180:\radB);
        \draw[-,thick][black] ({\tra + 0.4},0.4) arc (45:360 - 45:\radB);
        \draw[black, fill = black] ({\tra + 0.25},0) circle (0.02);
        \draw ({\tra + 0.25},-0.1) node{$\overline{y}$};

        % Contour names
        \draw[dashed,very thin][gray] ({\tra + 0.45},0.35) -- ({\tra + 0.45}, 0.9);
        \draw[dashed,very thin][gray] ({\tra +0.55},-0.15) -- ({\tra + 0.55}, -0.9);
        \draw ({\tra + 0.5},1) node{$1 + N^{-1/3} - \rho_2$};
        \draw ({\tra + 0.55},-1) node{$1 - N^{-1/3} + \rho_1$};

        \draw ({\tra - 0.55},-0.5) node{$\gamma_{\rho_2}^-$};
        \draw ({\tra - 0.85},0.75) node{$\gamma_{\rho_1}^+$};
        \draw ({\tra + 0.2},0.3) node{$\gamma_{z}$};
        \draw ({\tra + 1},0.3) node{$\gamma_{w}$};

        \draw[->, >=stealth'] ({\tra - 0.45}, -0.5)  to[bend right] ({\tra - 0.3},-0.5);
        \draw[->, >=stealth'] ({\tra - 0.75}, 0.75)  to[bend left] ({\tra - 0.60},0.7);
        \draw[->, >=stealth'] ( {\tra + 1}, 0.2)  to[bend right] ( {\tra +0.8},0);
        \draw[->, >=stealth'] ( {\tra + 0.2}, 0.2)  to[bend right] ( {\tra + 0.4},0);

    \end{tikzpicture} 
    \caption{The left part depicts the contours $\gamma_{\rho_1}^+$ and $\gamma_{\rho_2}^-$ when $\overline{y} = 1$. The right part depicts the contours $\gamma_{\rho_1}^+$ and $\gamma_{\rho_2}^-$ when $\overline{y} < 1$. The figures also depict the contours $\gamma_z,\gamma_w$ from the statement of the lemma, which partially overlap with the contours $\gamma_{\rho_1}^+, \gamma_{\rho_2}^-$, respectively. The parts that do not overlap are drawn in dashed black lines.} 
    \label{S43}
\end{figure}

We observe that the all the functions $H_i^N(z,w)$ are bounded and continuous on $\gamma_{z} \times \gamma_w$, and the only singularity of the integrand in (\ref{PKF3}) comes from $z-w$  in the denominator, which vanishes on $\gamma_{z} \times \gamma_w$ only when $z = w = 1 \pm \im N^{-1/3}$. Arguing as in (\ref{IntSing}), we conclude that $|z-w|^{-1}$ is locally integrable near the points $1 \pm \im N^{-1/3}$ and so the integral in (\ref{PKF3}) is well-defined and finite. The integrals in (\ref{PKF1}) and (\ref{PKF2}) are also well-defined and finite as the integrands are continuous and the contours compact. This shows that the right side of (\ref{CKDefAlt3}) is well-defined. 

Suppose first that $\overline{y} = 1$. In this case, we first deform the $\gamma_w$ contour to $\gamma^-_{\rho_2}$, which does not affect the value of $K_N^3(u,x;v,y)$ by Cauchy's theorem as we do not cross any poles. Subsequently, for each $w \in \gamma^-_{\rho_2}$ we deform the $\gamma_z$ contour to $\gamma^+_{\rho_1}$. Let us denote by $\gamma^-_{\rho_2,1}$ the portion of $\gamma^-_{\rho_2}$ that is to the right of $1$. In words, $\gamma^-_{\rho_2,1}$ consists of three segments connecting $1 - \im N^{-1/3}$ to $1 + N^{-1/3} - \rho_2 - \im \rho_2$, connecting $1 + N^{-1/3} - \rho_2 - \im \rho_2$ to $1 + N^{-1/3} - \rho_2 + \im \rho_2$ and connecting $1 + N^{-1/3} - \rho_2 + \im \rho_2$ to $1 + \im N^{-1/3}$. If $w \not \in \gamma^-_{\rho_2,1}$, then the deformation of $\gamma_z$ to $\gamma^+_{\rho_1}$ does not cross any poles, but if $w \in \gamma^-_{\rho_2,1}$, then the deformation crosses a simple pole at $z = w$. From the residue theorem we conclude that
\begin{equation}\label{Deform1}
\begin{split}
K_N^3(u,x;v,y) = & \frac{1}{(2\pi \im)^2} \oint_{\gamma^+_{\rho_1}}dz \oint_{\gamma^-_{\rho_2}}dw \frac{1}{z-w} \cdot\prod_{i = 1}^7H^N_i(z,w)  \cdot w^y z^{-x -1}  \\
& - \frac{1}{2\pi \im} \int_{\gamma^-_{\rho_2,1}} (1-qw)^{M_v(N) - M_u(N)} \cdot (1-q)^{M_u(N) - M_v(N)} w^{y - x - 1} dw.
\end{split}
\end{equation}
By Cauchy's theorem we see that the second line in (\ref{Deform1}) is precisely $- K_N^1(u,x;v,y)$ and so (\ref{Deform1}) implies (\ref{CKDefAlt3}) when $\overline{y} = 1$.

If we now suppose that $\overline{y} < 1$, then we first deform the $\gamma_z$ contour to $\gamma^+_{\rho_1}$, which does not cross any poles. Afterwards, we deform the $\gamma_w$ contour to $\gamma^-_{\rho_2}$. If we denote by $\gamma^+_{\rho_1,1}$ the portion of $\gamma^+_{\rho_1}$ to the left of $1$, then for $z \not \in \gamma^+_{\rho_1,1}$ the latter deformation crosses no poles. If $z \in   \gamma^+_{\rho_1,1}$, then the latter deformation crosses the simple pole at $w = z$.  From the residue theorem we conclude that
\begin{equation}\label{Deform2}
\begin{split}
K_N^3(u,x;v,y) = & \frac{1}{(2\pi \im)^2} \oint_{\gamma^+_{\rho_1}}dz \oint_{\gamma^-_{\rho_2}}dw \frac{1}{z-w} \cdot\prod_{i = 1}^7H^N_i(z,w)  \cdot w^y z^{-x -1}  \\
&-  \frac{1}{2\pi \im} \int_{\gamma^+_{\rho_2,1}} (1-qz)^{M_v(N) - M_u(N)} \cdot (1-q)^{M_u(N) - M_v(N)} z^{y - x - 1} dz.
\end{split}
\end{equation}
By Cauchy's theorem we see that the second line in (\ref{Deform2}) is precisely $- K_N^1(u,x;v,y)$ and so (\ref{Deform2}) implies (\ref{CKDefAlt3}) when $\overline{y}  < 1$. This suffices for the proof.
\end{proof}

%%%%%%%%%%%%%%%%%%%%%%%%%%%%%%%%%%%%%%%%%%%%%%%%%%%%%%%%%%%%%%%%%%%%%
%
%    Section 4.3
%
%%%%%%%%%%%%%%%%%%%%%%%%%%%%%%%%%%%%%%%%%%%%%%%%%%%%%%%%%%%%%%%%%%%%%
\subsection{Proof of Theorem \ref{T1}} \label{Section4.3} In the following statement we derive the limits of the kernels $K_N^i(u,x;v,y)$ for $i = 1,2,3$ from Lemma \ref{PrelimitKernel} and use the result to prove Theorem \ref{T1}.

\begin{proposition}\label{LimitKernelProp} Assume the same notation as in Lemma \ref{PrelimitKernel} and set $\delta = N^{-1/12}$. Suppose that $\tilde{x}_N, \tilde{y}_N$ are sequences of reals that converge to $\tilde{x}, \tilde{y} \in \mathbb{R}$ as $N \rightarrow \infty$, and fix $u,v \in \{1, \dots, m\}$. With this data we define
\begin{equation}\label{ScaleXY}
\begin{split}
&x_N = \frac{2q}{1-q} \cdot \tilde{N} + \frac{q t_u}{1- q} \cdot N^{2/3} + \sigma_q \tilde{x}_N \cdot N^{1/3}  \\
&y_N = \frac{2q}{1-q} \cdot \tilde{N} + \frac{q t_v}{1- q} \cdot N^{2/3} + \sigma_q \tilde{y}_N \cdot N^{1/3} .
\end{split}
\end{equation}
Then, we have the following limits
\begin{equation}\label{LimitK1}
\lim_{N \rightarrow \infty} N^{1/3} \sigma_q K^1_N(u, x_N; v, y_N) = \frac{1}{2\pi \im} \int_{-\im \sigma_q}^{\im \sigma_q} dw \exp \left( (\tilde{y} - \tilde{x}) w - f_q(t_v - t_u) w^2 \right);
\end{equation}
\begin{equation}\label{LimitK2}
\lim_{N \rightarrow \infty} N^{1/3}\sigma_q K^2_N(u, x_N; v, y_N) = -   \frac{ {\bf 1}\{t_u > t_v\} }{\sqrt{4 \pi f_q (t_u - t_v)} } \cdot  \exp \left( - \frac{(\tilde{y}- \tilde{x})^2}{4 f_q (t_u - t_v)} \right);
\end{equation}
\begin{equation}\label{LimitK3}
\begin{split}
&\lim_{N \rightarrow \infty} N^{1/3} \sigma_q K^3_N(u, x_N;v, y_N) = \frac{1}{(2\pi \im)^2} \int_{\Gamma^+_{-\sigma_q}} dz \int_{\Gamma^-_{\sigma_q}} dw \hspace{2mm}  e^{ \frac{z^3}{3} - \tilde{x} z + f_q t_u z^2 - \frac{w^3}{3} + \tilde{y} w- f_q t_v w^2 }\\
& \times \frac{1}{z - w} \cdot e^{c^+z + c^-/z - c^+w - c^-/w} \cdot \prod_{i = 1}^{\infty} \frac{(1 - b_i^+w) (1 - b_i^-/w) (1 + a_i^+ z) (1 + a_i^-/z) }{(1 - b_i^+z) (1 - b_i^-/z) (1 + a_i^+w) (1 + a_i^-/w) },
\end{split}
\end{equation}
where we recall from (\ref{SigmaQ}) that $f_q =  \frac{q^{1/3}}{2 (1 + q)^{2/3}}$, and the contours $\Gamma^{\pm}_a$ are as in Definition \ref{DefContInf}. Part of the statement is that the above integrals are all convergent.
\end{proposition}
\begin{remark}
We mention that in equations (\ref{LimitK1}), (\ref{LimitK2}) and (\ref{LimitK3}) we have that $x_N, y_N$ are reals (and not necessarily integers). The kernels $K^i_N(u, x; v,y)$ are well-defined for any choice of $x,y \in \mathbb{R}$ if we write 
$$w^{y} =  \exp \left( y \cdot \log w \right) \mbox{ and } z^{-x} = \exp \left( - x \cdot \log z \right), $$
where as usual we take the principal branch of the logarithm.
\end{remark}
Proposition \ref{LimitKernelProp} is proved in Section \ref{Section5}. In the remainder of this section we assume the validity of Proposition \ref{LimitKernelProp} and conclude the proof of Theorem \ref{T1}.

\begin{proof}[Proof of Theorem \ref{T1}] Set $t_i = f_q^{-1} \cdot s_i$ and $\mathcal{T} = \{t_1, \dots, t_m\}$. Assume the same notation as in Lemma \ref{PrelimitKernel} with $\delta = N^{-1/12}$. We define 
\begin{equation}\label{S4PVC1}
X_i^{j,N} = \sigma_q^{-1} N^{-1/3} \cdot \left( \lambda_i^{M_j(N)} - \frac{2q \tilde{N} }{1-q} - \frac{qt_j N^{2/3}}{1-q} - i\right) \mbox{ for $j \in \{1, \dots, m\}$ and $i \in \mathbb{N}$.}
\end{equation}
Let $f$ be any piece-wise linear strictly increasing bijection from $\mathbb{R}$ to $\mathbb{R}$ such that $f(j) = s_j$ for $j = 1, \dots, m$. Setting $a_N = \sigma_q^{-1} N^{-1/3} $ and for $t \in \mathbb{R}$
$$b_N(t) = \sigma_q^{-1} N^{-1/3} \cdot \left( - \frac{2q \tilde{N}}{1-q} - \frac{qt  N^{2/3}}{1-q} \right),$$
we define $\phi_N : \mathbb{R}^2 \rightarrow \mathbb{R}^2$ via $\phi_N(t,x) = (f(t), a_N x + b_N(f_q^{-1}f(t)))$. We readily observe that 
$$(s_j, X_i^{j,N}) =  \phi_N(j, \lambda_i^{M_j(N)} - i) \mbox{ for $j \in \{1, \dots, m\}$ and $i \in \mathbb{N}$,}$$
which implies that if $M_X^N$ is the random measure on $\mathbb{R}^2$ formed by $\{(s_j, X_i^{j,N}): 1\leq j \leq m, i \geq 1\}$, then $M_X^N = \mathfrak{S}(\lambda) \phi_N^{-1}$. From Lemma \ref{PrelimitKernel} we know that $\mathfrak{S}(\lambda)$ is a determinantal point process on $\mathbb{R}^2$ with correlation kernel $K_N(u,x; v,y)$ as in (\ref{CKDefAlt}) for all large $N$. Combining the latter with parts (5) and (6) of Proposition \ref{PropLem}, we conclude that for all large $N$ the random measure $M_X^N$ is a determinantal point process on $\mathbb{R}^2$ with correlation kernel 
\begin{equation}\label{S4ScaledKer1}
K^{X}_N(s_i, \tilde{x}_N; s_j, \tilde{y}_N) = N^{1/3} \sigma_q K_N(i, x_N; j, y_N)
\end{equation}
and reference measure given by $\mu_{\mathcal{A}, \nu^N}$. Here, $\mu_{\mathcal{A}, \nu^N}$ is as in Definition \ref{DefSlices} with $\mathcal{T} = \mathcal{A} = \{s_1, \dots, s_m\}$, and $\nu^N = (\nu_{s_1}^N, \dots, \nu_{s_m}^N)$ with $\nu^N_{s}$ being $\sigma_q^{-1/3} N^{-1/3}$ times the counting measure on $a_N \cdot \mathbb{Z} +  b_N(f_q^{-1}s)$. We mention that the relationship between $\tilde{x}_N, \tilde{y}_N$ and $x_N$, $y_N$ is as in (\ref{ScaleXY}).\\

We denote the right sides of (\ref{LimitK1}), (\ref{LimitK2}) and (\ref{LimitK3}) by $K_{\infty}^i(t_u, \tilde{x}; t_v, \tilde{y})$ for $i = 1,2,3$. By a straightforward computation we check for $i = 1,2,3$ that
\begin{equation}\label{MatchingFormula}
K^i_{\infty}(t_u, x_1; t_v, x_2) = e^{   (x_1 f_q t_u +  2f_q^3 t_u^3/3) - (x_2 f_q t_v +  2f_q^3 t_v^3/3) }  \cdot K^i_{a,b,c}(f_q t_v, x_2 + f_q^2 t_v^2; f_q t_u,x_1 + f_q^2 t_u^2).
\end{equation}
To verify (\ref{MatchingFormula}) we will compute the right side and compare with the expressions in (\ref{LimitK1}), (\ref{LimitK2}) and (\ref{LimitK3}). We take in Definition \ref{3BPKernelDef} $\alpha = - f_q t_v - \sigma_q$ and $\beta = - f_q t_u + \sigma_q$. The latter means that the contours $\Gamma_{\alpha + f_q t_v}^+$ and $\Gamma_{\beta + f_q t_u}^-$ have two intersection points at $\pm \im \sigma_q$. From (\ref{3BPKer}) we get 
\begin{equation*}
\begin{split}
&e^{ (x_1 f_q t_u +  2f_q^3 t_u^3/3) - (x_2 f_q t_v +  2f_q^3 t_v^3/3) }  K^1_{a,b,c}\left(f_q t_v, x_2 + f_q^2 t_v^2; f_q t_u,x_1 + f_q^2 t_u^2 \right)  \\
& = \frac{ e^{ (x_1 f_q t_u +  2f_q^3 t_u^3/3) - (x_2 f_q t_v +  2f_q^3 t_v^3/3) }  }{2\pi \im} \int_{- \im \sigma_q}^{\im \sigma_q} \hspace{-2mm} dw e^{f_q(t_u - t_v)w^2 + f_q^2(t_v^2 - t_u^2) w + w (x_1+ f_q^2 t_u^2 - x_2 - f_q^2 t_v^2) }  \\
&\times e^{ (x_2 + f_q^2 t_v^2) f_q t_v -(x_1 + f_q^2 t_u^2) f_q t_u - f_q^3t_v^3/3 + f_q^3 t_u^3/3 }\\
& = \frac{1}{2\pi \im}  \int_{- \im \sigma_q}^{\im \sigma_q}dw \exp \left( f_q(t_u - t_v) w^2 + w(x_1 - x_2) \right) =  \frac{1}{2\pi \im}  \int_{- \im \sigma_q}^{\im \sigma_q}dw \exp \left( f_q(t_u - t_v) w^2 + w(x_2 - x_1) \right) 
\end{split}
\end{equation*}
which matches the right side of (\ref{LimitK1}) when $\tilde{x} = x_1, \tilde{y} = x_2$. In the last equality above we changed variables $w \rightarrow -w$. 

We similarly have that
\begin{equation*}
\begin{split}
&e^{ (x_1 f_q t_u +  2f_q^3 t_u^3/3) - (x_2 f_q t_v +  2f_q^3 t_v^3/3) }   K^2_{a,b,c}\left(f_q t_v, x_2 + f_q^2 t_v^2; f_q t_u,x_1 + f_q^2 t_u^2 \right)   \\
& = - e^{ (x_1 f_q t_u +  2f_q^3 t_u^3/3) - (x_2 f_q t_v +  2f_q^3 t_v^3/3) }  \cdot  \frac{{\bf 1}\{ t_u > t_v\} }{\sqrt{4\pi f_q (t_u - t_v)}} \\
&\times  e^{ - \frac{(x_1 + f_q^2 t_u^2 - x_2 - f_q^2 t_v^2)^2}{4f_q (t_u - t_v)} - \frac{f_q(t_u - t_v)(x_1 + f_q^2t_u^2 + x_2 +  f_q^2 t_v^2)}{2} + \frac{f_q^3(t_v - t_u)^3}{12} } = -  \frac{{\bf 1}\{ t_u > t_v\} }{\sqrt{4\pi f_q (t_u - t_v)}} \cdot e^{ - \frac{(x_2 - x_1)^2}{4 f_q (t_u - t_v)}},
\end{split}
\end{equation*}
which matches the right side of (\ref{LimitK2}) when $\tilde{x} = x_1, \tilde{y} = x_2$.

Finally, we have from (\ref{3BPKer}) upon changing variables $\tilde{z} = z + f_q t_v$ and $\tilde{w} = w + f_q t_u$ that
\begin{equation*}
\begin{split}
&e^{ (x_1 f_q t_u +  2f_q^3 t_u^3/3) - (x_2 f_q t_v +  2f_q^3 t_v^3/3) }   K^3_{a,b,c}\left(f_q t_v, x_2 + f_q^2 t_v^2; f_q t_u,x_1 + f_q^2 t_u^2 \right)   \\
& = e^{ (x_1 f_q t_u +  2f_q^3 t_u^3/3) - (x_2 f_q t_v +  2f_q^3 t_v^3/3) }  \cdot \frac{1}{(2\pi \im)^2} \int_{\Gamma_{-\sigma_q }^+} d \tilde{z} \int_{\Gamma_{\sigma_q}^-} d\tilde{w}  \frac{\Phi_{a,b,c}(\tilde{z}) }{\Phi_{a,b,c}(\tilde{w})} \cdot \frac{1}{\tilde{z} - \tilde{w}} \\
&\times e^{(\tilde{z} - f_q t_v)^3/3 -(x_2 + f_q^2 t_v^2) (\tilde{z} - f_q t_v) - (\tilde{w} - f_q t_u)^3/3 + (x_1 + f_q^2 t_u^2)(\tilde{w} - f_q t_u)}  \\
&  = \frac{1}{(2\pi \im)^2} \int_{\Gamma_{-\sigma_q }^+} d \tilde{z} \int_{\Gamma_{\sigma_q}^-} d\tilde{w} \frac{\Phi_{a,b,c}(\tilde{z}) }{\Phi_{a,b,c}(\tilde{w})} \cdot \frac{e^{\tilde{z}^3/3 - \tilde{z}^2 f_q t_v - \tilde{z} x_2  -\tilde{w}^3/3 + \tilde{w}^2 f_q t_u + \tilde{w} x_1}}{\tilde{z} - \tilde{w}},
\end{split}
\end{equation*}
which matches the right side of (\ref{LimitK3}) when $\tilde{x} = x_1, \tilde{y} = x_2$ once we change variables $w = - \tilde{z}$ and $z = -\tilde{w}$.\\

We are now ready to finish up. From (\ref{S4ScaledKer1}), (\ref{MatchingFormula}) and Proposition \ref{LimitKernelProp} we have for each $A > 0$
\begin{equation}\label{S4VL1}
\lim_{N \rightarrow \infty} \max_{s,t \in \mathcal{A}} \sup_{- A \leq x,y \leq A} \left| K^X_N(s, x; t ,y) -  e^{   (x s +  2s^3/3) - (yt +  2t^3/3) }  K_{a,b,c}(t, y + t^2; s ,x + s^2) \right| = 0.
\end{equation}
By Lemma \ref{WellDefKer} we know that $e^{   (x s +  2s^3/3) - (yt +  2t^3/3) }  K_{a,b,c}(t, y + t^2; s ,x + s^2)$ is continuous in $(x,y) \in \mathbb{R}^2$ for each fixed $s,t \in \mathcal{A}$. One also observes for each $s \in \mathcal{A}$ that $\nu^N_s$ converges vaguely to the Lebesgue measure on $\mathbb{R}$ and so by Proposition \ref{PropWC1} we conclude that there is a determinantal point process $\tilde{M}$ on $\mathbb{R}^2$ with correlation kernel 
$$\tilde{K}(s,x; t,y) = e^{   (x s +  2s^3/3) - (yt +  2t^3/3) }  K_{a,b,c}(t, y + t^2; s ,x + s^2),$$
and reference measure $\mu_{\mathcal{A}} \times \lambda$. Furthermore, $M_X^N$ converge weakly to $\tilde{M}$. We let $\phi(s,x) = (s, x+s^2)$ and set $M = \tilde{M} \phi^{-1}$. From parts (4) and (5) of Proposition \ref{PropLem} we see that $M$ is a determinantal point process on $\mathbb{R}^2$ with correlation kernel $K_{a,b,c}(s, x; t ,y)$ and reference measure $\mu_{\mathcal{A}} \times \lambda$, which concludes the proof.
\end{proof}

%%%%%%%%%%%%%%%%%%%%%%%%%%%%%%%%%%%%%%%%%%%%%%%%%%%%%%%%%%%%%%%%%%%%%
%
%    Section 5
%
%%%%%%%%%%%%%%%%%%%%%%%%%%%%%%%%%%%%%%%%%%%%%%%%%%%%%%%%%%%%%%%%%%%%%
\section{Kernel convergence}\label{Section5} In this section we prove Proposition \ref{LimitKernelProp}. In Section \ref{Section5.1} we prove (\ref{LimitK1}) and in Section \ref{Section5.2} we prove (\ref{LimitK2}). In Section \ref{Section5.3} we explain how one can truncate the contours in the definition of $K^3_N$ near $1$ without affecting the value of the kernel significantly, and prove (\ref{LimitK3}) in Section \ref{Section5.4}. Throughout this section we assume without mention that $N \geq N_0$ and $N^{-1/12} \in (0,1/2]$, where we recall that $N_0$ is as in Definition \ref{ParScale}.

%%%%%%%%%%%%%%%%%%%%%%%%%%%%%%%%%%%%%%%%%%%%%%%%%%%%%%%%%%%%%%%%%%%%%
%
%    Section 5.1
%
%%%%%%%%%%%%%%%%%%%%%%%%%%%%%%%%%%%%%%%%%%%%%%%%%%%%%%%%%%%%%%%%%%%%%
\subsection{Proof of (\ref{LimitK1})} \label{Section5.1} Throughout this section all constants in the big $O$ notations depend on $q, t_1, \dots, t_m$. In view of (\ref{PKF1}) and (\ref{ScaleXY}) we have that 
\begin{equation}\label{OM1}
\begin{split}
&N^{1/3} \sigma_q K_N^1(u, x_N; v, y_N) = \frac{1}{2\pi \im} \int_{- \im  \sigma_q}^{\im  \sigma_q}d\tilde{w} \exp \left( [M_v(N) - M_u(N)] \cdot F_N(\tilde{w}) \right) \\
&\times \exp \left( \log (1 + N^{-1/3}  \sigma_q^{-1} \tilde{w}) \cdot \left( \frac{q(t_v - t_u) N^{2/3}}{1-q} + \sigma_q N^{1/3} (\tilde{y}_N - \tilde{x}_N) -1 \right) \right),
\end{split}
\end{equation}
where we have performed the change of variables $w = 1 + N^{-1/3} \sigma_q^{-1} \tilde{w}$ and the function $F_N(\tilde{w})$ is 
\begin{equation}\label{OM2}
F_N(\tilde{w}) = \log \left(1 - q \left(1 + N^{-1/3} \sigma_q^{-1} \tilde{w} \right) \right) - \log(1 - q).
\end{equation} 
By straightforward Taylor expansion near $0$ we have that 
$$F_N(\tilde{w}) = - \frac{q}{1-q} \cdot N^{-1/3} \sigma_q^{-1} \tilde{w} - \frac{q^2}{2(1-q)^2} \cdot N^{-2/3} \sigma_q^{-2} \tilde{w}^2 + O(N^{-1}), $$
and also
\begin{equation*}
\begin{split}
&\exp \left( \log (1 + N^{-1/3}  \sigma_q^{-1} \tilde{w}) \cdot \left( \frac{q(t_v - t_u) N^{2/3}}{1-q} + \sigma_q N^{1/3} (\tilde{y}_N - \tilde{x}_N) - 1\right) \right)  \\
& = \exp \left(\frac{q(t_v - t_u)}{1-q} N^{1/3}  \sigma_q^{-1} \tilde{w} + (\tilde{y}_N - \tilde{x}_N) \tilde{w} - \frac{q(t_v - t_u)}{2(1-q)} \sigma_q^{-2} \tilde{w}^2   + O\left( (1 + |\tilde{x}_N - \tilde{y}_N|)N^{-1/3} \right) \right).
\end{split}
\end{equation*}
Combining the latter with $M_v(N) - M_u(N) = N^{2/3} (t_v - t_u) + O(1)$, we see that 
\begin{equation*}
\begin{split}
N^{1/3} \sigma_q K_N^1(u, x_N; v, y_N) = &\frac{1}{2\pi \im} \int_{- \im  \sigma_q}^{\im  \sigma_q}d\tilde{w} \exp \left( (\tilde{y}_N - \tilde{x}_N) \tilde{w} - \frac{q(t_v-t_u)}{2(1-q)^2} \sigma_q^{-2} \tilde{w}^2   \right) \\
& \times \exp \left( O\left( (1 + |\tilde{x}_N - \tilde{y}_N|)N^{-1/3} \right) \right).
\end{split}
\end{equation*}
We may now let $N \rightarrow \infty$ above and conclude (\ref{LimitK1}) from the bounded convergence theorem.

%%%%%%%%%%%%%%%%%%%%%%%%%%%%%%%%%%%%%%%%%%%%%%%%%%%%%%%%%%%%%%%%%%%%%
%
%    Section 5.2
%
%%%%%%%%%%%%%%%%%%%%%%%%%%%%%%%%%%%%%%%%%%%%%%%%%%%%%%%%%%%%%%%%%%%%%
\subsection{Proof of (\ref{LimitK2})} \label{Section5.2} Throughout this section all constants in the big $O$ notations depend on $q, t_1, \dots, t_m$. The equality in (\ref{LimitK2}) is clear when $v \geq u$ and so we assume that $u > v$ in the sequel. In view of (\ref{PKF2}) and (\ref{ScaleXY}) we have that 
\begin{equation}\label{OM3}
\begin{split}
&N^{1/3} \sigma_q K_N^2(u, x_N; v, y_N) = - \frac{\sigma_q}{2\pi \im} \int_{- \im  \pi  N^{1/3} }^{\im  \pi N^{1/3} }\hspace{-2mm} d\tilde{w} \hspace{2mm} e^{ [{M_v}(N) - {M_u}(N)] G(\tilde{w}N^{-1/3})+ \tilde{w} \sigma_q (\tilde{y}_N - \tilde{x}_N) } \cdot e^{c_N \tilde{w}},
\end{split}
\end{equation}
where we have performed the change of variables $w = \exp (N^{-1/3} \tilde{w})$, and set
\begin{equation}\label{OM4}
\begin{split}
&G(z) = \log(1 - qe^{z}) - \log (1 -q) + \frac{q}{1-q} \cdot z \mbox{ and } \\
& c_N = N^{-1/3} \cdot \frac{q}{1-q} \cdot  \left[t_v N^{2/3} - t_u N^{2/3} - M_v(N) + M_u(N) \right] = O(N^{-1/3}).
\end{split}
\end{equation} 
The following lemma summarizes the relevant properties we require for the function $G(z)$.
\begin{lemma}\label{PropertiesOfG} Let $q \in (0,1)$ and $G(z)$ be as in (\ref{OM4}). There exist $q$-dependent constants $\delta \in (0,1)$ and $C > 0$  such that $G(z)$ is well-defined for $\Real(z) \in [-\delta, \delta]$ and satisfies
\begin{equation}\label{OM5}
\left| G(z) + \frac{q}{2(1-q)^2} \cdot z^2 \right| \leq C |z|^3 \mbox{, provided that } |z| \leq \delta.
\end{equation}
There also exists a $q$-dependent constant $\epsilon > 0$ such that for $x \in [-\pi, \pi]$ we have
\begin{equation}\label{OM6}
\Real G( \im x) \geq \epsilon x^2.
\end{equation}
\end{lemma}
\begin{proof} Note that $G$ is well-defined provided that $|qe^{z}| < 1$, which can be ensured if $\delta \in (0,1) $ is chosen small enough so that $e^{\delta} < q^{-1}$ and $\Real(z) \in [-\delta, \delta]$. By a direct computation we have 
$$G'(z) = -\frac{q e^z}{1 - q e^z} + \frac{q}{1-q} = \frac{1}{1-q} - \frac{1}{1 - qe^z} \mbox{ and } G''(z) =  \frac{-q e^z}{(1 - qe^z)^2}.$$ 
In particular, we see that $G(0) = G'(0) = 0$ and $G''(0) = - \frac{q}{(1-q)^2}$, which implies (\ref{OM5}) if we take $C$ sufficiently large.

Observe that $\Real G(\im x) = (1/2) \log ( 1 +q^2 - 2q \cos(x))  - \log (1-q)$ and so
$$\frac{d}{dx} \Real G(\im x) = \frac{q \sin(x) }{1 + q^2 - 2q \cos(x)}.$$
We conclude that $\Real G(\im x)$ increases on $(0, \pi)$ and decreases on $(-\pi, 0)$. I.e. $\Real G(\im x)$ increases as $x$ moves away from $0$ either in the positive or negative direction. 

From (\ref{OM5}) we can find a positive constant $\epsilon_1 \in (0,\delta)$ such that for $|x| \leq \epsilon_1$ we have 
$$\left| \Real G(\im x) - \frac{qx^2 }{2(1-q)^2} \right| \leq C|x|^3 \leq C \epsilon_1 \cdot x^2 <\frac{q}{4(1-q)^2} \cdot x^2 .$$
The latter shows that for $|x| \leq \epsilon_1$ we have
$$\Real G(\im x) \geq \frac{q}{4(1-q)^2} \cdot x^2$$
On the other hand, if $x \in [-\pi, \pi]$, then $y = \frac{\epsilon_1}{\pi} \cdot x$ satisfies $|y| \leq \epsilon_1$ and $|x| \geq |y|$. The monotonicity of $\Real G(\im x)$ now gives
$$\Real G(\im x) \geq \Real G(\im y) \geq \frac{q}{4(1-q)^2} \cdot y^2 =  \frac{q}{4(1-q)^2} \cdot (\epsilon_1/\pi)^2 \cdot x^2,$$
which establishes (\ref{OM6}) with $\epsilon = \frac{q}{4(1-q)^2} \cdot (\epsilon_1/\pi)^2$. 
\end{proof}

We now turn to analyzing the limit of (\ref{OM3}) as $N \rightarrow \infty$. It follows from (\ref{OM5}) and the fact that $M_v(N) - M_u(N) = (t_v- t_u)N^{2/3} + O (1)$ that for each $\tilde{w} \in \mathbb{C}$
$$\lim_{N \rightarrow \infty} e^{ [{M_v}(N) - {M_u}(N)] G(\tilde{w}N^{-1/3})+ \tilde{w} \sigma_q (\tilde{y}_N - \tilde{x}_N) } \cdot e^{c_N \tilde{w}} = \exp \left( - \frac{ q (t_v - t_u) }{2 (1-q)^2}\cdot \tilde{w}^2 + \tilde{w} \sigma_q (\tilde{y} - \tilde{x}) \right).$$
We can find a constant $C_1 > 0$ such that $|\tilde{x}_N| \leq C_1$ and $|\tilde{y}_N| \leq C_1$ for all $N \in \mathbb{N}$. Combining the latter with (\ref{OM6}) we conclude that for some $C_2 > 0$ and all large $N \in \mathbb{N}$ we have
$$\left| e^{ [{M_v}(N) - {M_u}(N)] G(\tilde{w}N^{-1/3})+ \tilde{w} \sigma_q (\tilde{y}_N - \tilde{x}_N) } \cdot e^{c_N \tilde{w}}  \right| \leq C_2 \cdot \exp \left( 2C_1 \sigma_q |\tilde{w}| - \epsilon (t_u - t_v) |\tilde{w}|^2  \right),$$
where we used that $t_u > t_v$ (which was assumed in the beginning of the section). From the dominated convergence theorem with dominating function $C_2 \cdot \exp \left( 2C_1 \sigma_q |\tilde{w}| - \epsilon (t_u - t_v) |\tilde{w}|^2  \right)$ we get 
\begin{equation}\label{OM7}
\begin{split}
\lim_{N \rightarrow \infty} N^{1/3}\sigma_q  K_N^2(u, x_N; v, y_N) =- \frac{\sigma_q}{2\pi \im} \int_{- \im  \infty }^{\im  \infty }\hspace{-2mm} d\tilde{w} \hspace{2mm} \exp \left( \frac{q (t_u - t_v) }{2(1-q)^2}\cdot \tilde{w}^2 + \tilde{w} \sigma_q (\tilde{y} - \tilde{x})\right) .
\end{split}
\end{equation}
Changing variables $ \tilde{w} = \im w$ we recognize above the characteristic function of a Gaussian random variable, which precisely evaluates to the right side of (\ref{LimitK2}).

%%%%%%%%%%%%%%%%%%%%%%%%%%%%%%%%%%%%%%%%%%%%%%%%%%%%%%%%%%%%%%%%%%%%%
%
%    Section 5.3
%
%%%%%%%%%%%%%%%%%%%%%%%%%%%%%%%%%%%%%%%%%%%%%%%%%%%%%%%%%%%%%%%%%%%%%
\subsection{Contour truncation} \label{Section5.3} Let us denote $\gamma_{z,N}$ the contour $\gamma_{a, \delta}^{+, \vert}$ from Definition \ref{finiteContours} with $a = 1 - N^{-1/3}$ and $\delta = N^{-1/12}$. We also let $\gamma_{w,N}$ denote the contour $\gamma_{a,\delta}^{-, \vert}$ from Definition \ref{finiteContours} with $a = 1 + N^{-1/3}$ and $\delta = N^{-1/12}$. In Lemma \ref{LTrunc} below we show that if we replace the contours $\gamma_z$ and $\gamma_w$ in the definition of $K_N^3(u, x_N; v, y_N)$ from (\ref{PKF3}) with the contours $\gamma_{z,N}$ and $\gamma_{w,N}$ as above, then the integral remains asymptotically the same as $N \rightarrow \infty$. What this means is that we can truncate the contours $\gamma_z$, $\gamma_w$ near the point $1$ without affecting the asymptotics of the kernel. We require the following result for the proof of Lemma \ref{LTrunc}.
\begin{lemma}\label{TaylorS} Fix $q \in (0,1)$ and define the function 
\begin{equation}\label{DefS}
S(z) = \log (1 - q/z) - \log(1 - qz) - \frac{2q}{1-q} \cdot \log z
\end{equation}
for $z \in \mathbb{C} \setminus \{0\}$, where we take the principal branch of the logarithm. There exists $\delta_0 \in (0, 1/2)$ and $C_0 > 0$ (depending on $q$) such that 
\begin{equation}\label{EqTayS}
\left| S(z) - \sigma_q^3 (z-1)^3/3 \right| \leq C_0 \cdot |z-1|^4 \mbox{ if } |z-1| \leq \delta_0.
\end{equation}
Furthermore, we have the following inequalities
\begin{equation}\label{RealS}
\begin{split}
&\frac{d}{d\theta} \Real S(r e^{ \pm \im \theta}) > 0 \mbox{ for } \theta \in (0, \pi) \mbox{ and } r \in (0,1) \\
& \frac{d}{d\theta} \Real S(r e^{ \pm \im \theta}) < 0 \mbox{ for } \theta \in (0, \pi) \mbox{ and } r > 1.
\end{split}
\end{equation}
\end{lemma}
\begin{proof} By a direct computation we have
\begin{equation}\label{SDer}
\begin{split}
& S'(z) = \frac{q}{z^2 - qz} + \frac{q}{1-qz} - \frac{2q}{(1-q) z}, \hspace{2mm} S''(z) = \frac{-q(2z - q)}{(z^2 - qz)^2} + \frac{q^2}{(1-qz)^2} + \frac{2q}{(1-q)z^2},  \\
&S'''(z) = \frac{-2q(z^2 - qz)^2 + 2q(2z - q)^2(z^2 -qz) }{(z^2 - qz)^4} + \frac{2 q^3}{(1-qz)^3} - \frac{4q}{(1-q)z^3}.
\end{split}
\end{equation}
This gives $S(1) = S'(1) = S''(1) = 0$ and 
$$S'''(1) = \frac{-2q(1-q) + 2q (2-q)^2 + 2q^3 - 4q (1-q)^2}{(1-q)^3} = \frac{2 q(1+q)}{(1-q)^3} = 2 \sigma_q^3.$$
Taylor expanding $S(z)$ near $z =1$ we conclude that there are $\delta_0, C_0$ satisfying (\ref{EqTayS}).\\

In the remainder we prove (\ref{RealS}). Fix $\epsilon \in \{-1, 1\}$ and note that
$$\Real S(re^{\epsilon \theta}) = \frac{1}{2} \log (1 + q^2 r^{-2} -2qr^{-1} \cos (\theta) ) - \frac{1}{2} \log(1  + q^2 r^2 -2qr \cos(\theta)) - \frac{2q}{1-q} \log r.$$
This gives
\begin{equation*}
\begin{split}
&\frac{d}{d\theta} \Real S(re^{\epsilon \theta}) = \frac{q r^{-1} \sin(\theta) }{1 + q^2r^{-2} - 2q r^{-1} \cos(\theta) } - \frac{qr \sin(\theta)}{1 + q^2 r^2 - 2qr \cos(\theta)} \\
& = \frac{\sin(\theta) q (1-q^2) (r^{-1} - r)}{[1 + q^2r^{-2} - 2q r^{-1} \cos(\theta)] [1 + q^2 r^2 - 2qr \cos(\theta)]},
\end{split}
\end{equation*}
which clearly implies (\ref{RealS}).
\end{proof}

\begin{lemma}\label{LTrunc} Assume the same notation as in Proposition \ref{LimitKernelProp} and suppose that $\gamma_{z,N}$ and $\gamma_{w,N}$ are as in the beginning of this section. Then
\begin{equation}\label{TruncEq}
\lim_{N \rightarrow \infty} N^{1/3} K_N^3(u, x_N; v ,y_N) -  \frac{N^{1/3}}{(2\pi \im)^2} \int_{\gamma_{z,N}}dz  \int_{\gamma_{w,N}}dw \frac{1}{z-w} \cdot \prod_{i = 1}^7 H_i^N(z,w) \cdot w^{y_N} z^{-x_N-1} = 0.
\end{equation}
\end{lemma}
\begin{proof} Let us denote $\gamma_w^0 = \gamma_{w,N}$, $\gamma_w^1 = \gamma_w \setminus \gamma_{w,N}$ and similarly $\gamma_z^0 = \gamma_{z,N}$, $\gamma_z^1 = \gamma_z \setminus \gamma_{z,N}$. Using the formula for $K_N^3(u, x_N; v ,y_N)$ from (\ref{PKF3}) with $x_N, y_N$ scaled as in (\ref{ScaleXY}), and $\delta = N^{-1/12}$ we get
\begin{equation}\label{TE1}
\begin{split}
&K_N^3(u, x_N; v ,y_N) -  \frac{1}{(2\pi \im)^2} \int_{\gamma_{z,N}}dz  \int_{\gamma_{w,N}}dw \frac{1}{z-w} \cdot \prod_{i = 1}^7 H_i^N(z,w) \cdot w^{y_N} z^{-x_N-1}   \\
& = K_N^{3,0,1}(u, x_N; v ,y_N) + K_N^{3,1,0}(u, x_N; v ,y_N) + K_N^{3,1,1}(u, x_N; v ,y_N), 
\end{split}
\end{equation}
where for $j_1, j_2 \in \{0,1\}$
\begin{equation}\label{TE1.5}
\begin{split}
&K_N^{3,j_1,j_2}(u, x_N; v ,y_N) \int_{\gamma^{j_1}_z}dz \int_{\gamma^{j_2}_w } dw  \cdot \frac{1}{z-w}\prod_{i = 1}^5 H_i^N(z,w) \cdot H^N_7(z,w) \cdot H^N_8(z,w) \\
& \times \exp \left(\log w \cdot \sigma_q \tilde{y}_N \cdot N^{1/3} -  \log z \cdot \sigma_q \tilde{x}_N \cdot N^{1/3}  \right) \cdot \exp \left(\tilde{N} \cdot [S(z) - S(w)] \right).
\end{split}
\end{equation}
where $S(z)$ is as in (\ref{DefS}), and $H_8(z,w)$ is given by
\begin{equation}\label{H8Def}
H^N_8(z,w) = \exp \left( \log w \cdot \frac{qt_v}{1-q} \cdot N^{2/3} - \log z \cdot \frac{qt_u}{1- q}\cdot N^{2/3} - \log z\right).
\end{equation}

Suppose that $z,w \in \mathbb{C}$ are such that $|z|, |w| \in [1/2, 2]$, $x \in [q,1]$ and there is a positive constant $\tilde{\delta} > 0$ such that $|z - x^{-1}| \geq \tilde{\delta} N^{-1/3}$ and $|w - x| \geq \tilde{\delta} N^{-1/3}$. We then have that 
\begin{equation}\label{RatBound1}
\left|\frac{1 - w x}{1 - z x} \right| = O(N^{1/3}) \mbox{ and } \left| \frac{1 - x/z}{1 - x/w} \right| = O(N^{1/3}),
\end{equation}
where the constant in the big $O$ notation depends on $q, \tilde{\delta}$. 

We observe from the definition of $\gamma_z, \gamma_w$ that if $z \in \gamma_z \cup \gamma_w$, then $1 - N^{-1/12} \leq |z| \leq 1 +N^{-1/12}$ and so for large $N$ we have $|z| \in [1/2, 2]$. In addition, we note that $\gamma_z$ is at least distance $2^{-1/2} N^{-1/3}$ away from $[1, q^{-1}]$ and $\gamma_w$ is at least distance $2^{-1/2} N^{-1/3}$ away from $[q,1]$. In particular, the estimates in (\ref{RatBound1}) hold for all $z \in \gamma_z$ and $w \in \gamma_w$ and $x \in [q,1]$. Since $H^N_i(z,w)$ are products of $O(N^{1/12})$ terms of the form (\ref{RatBound1}) for $i = 1, \dots, 5$, we conclude that there is a constant $c_0  > 0$ such that for $z \in \gamma_z$ and $w \in \gamma_w$ we have
\begin{equation}\label{Tr1}
\left| \prod_{i = 1}^5 H_i^N(z,w) \right| \leq \exp \left(c_0  \cdot N^{1/12} \cdot \log N  \right).
\end{equation}

We next note from the boundedness of the contours $\gamma_z, \gamma_w$ that there are constants $c_1 > 0$ (depending on $q$, and the sequences $\{\tilde{x}_n\}_{n \geq 1}$, $\{\tilde{y}_n\}_{n \geq 1}$) and $c_2 > 0$ (depending on $q$ and $t_1, \dots, t_m$) such that for $z \in \gamma_z$ and $w \in \gamma_w$ we have
\begin{equation}\label{Tr2}
\begin{split}
&\left| \exp \left(\log w \cdot \sigma_q \tilde{y}_N \cdot N^{1/3} -  \log z \cdot \sigma_q \tilde{x}_N \cdot N^{1/3}\right)  \right| \leq \exp \left( c_1 \cdot N^{1/3} \right) \mbox{ and }\\
&\left| H^N_8(z,w)  \right| \leq \exp \left( c_2 \cdot N^{2/3} \right).
\end{split}
\end{equation}
Since $\gamma_z$ and $\gamma_w$ are bounded contours and are also bounded away from $q^{-1}$ we have
\begin{equation}\label{Tr3}
\begin{split}
&\left| H^N_7(z,w)  \right| \leq \exp \left( c_3 \cdot N^{2/3} \right),
\end{split}
\end{equation}
for some $c_3 > 0$ (depending on $q$ and $t_1, \dots, t_m$) provided that $z \in \gamma_z$ and $w \in \gamma_w$. Notice that for $(z,w) \in \gamma_z \times \gamma_w \setminus\gamma_z^0 \times  \gamma_w^0$ we have $|z-w| \geq \rho N^{-1/12}$ for some universal $\rho > 0$ and so for such $z,w$
\begin{equation}\label{Tr4}
\begin{split}
&\left| \frac{1}{z - w} \right| \leq \exp \left( c_4 \cdot \log N \right),
\end{split}
\end{equation}
where $c_4$ is a universal constant.

We finally focus on the term $\exp \left(\tilde{N} \cdot [S(z) - S(w)] \right)$. We have from (\ref{EqTayS}) that for all large $N$
\begin{equation}\label{RealSB}
\begin{split}
&\Real S(z) \leq C_0 |z-1|^4 + (\sigma_q^3/3) \Real (z-1)^3 \leq c_5 N^{-1} - (\sigma_q^3/24)  |z-1|^3 \mbox{ if } z \in \gamma_z^0\\
&\Real S(w) \geq -C_0 |z-1|^4 + (\sigma_q^3/3) \Real (w-1)^3 \geq - c_5 N^{-1} + (\sigma_q^3/24)  |w-1|^3 \mbox{ if } w \in \gamma_w^0,
\end{split}
\end{equation}
for some $c_5 > 0 $ that depends on $q$. The latter and (\ref{RealS}) also gives
\begin{equation*}
\begin{split}
&\Real S(z) \leq \Real S(1 - N^{-1/3} + N^{-1/12} + \im N^{-1/12}) \leq c_5N^{-1} -  (\sigma_q^3/24) N^{-1/4} \mbox{ if } z \in \gamma_{z}^1 \\
&\Real S(w) \geq \Real S(1 + N^{-1/3} - N^{-1/12} - \im N^{-1/12}) \geq - c_5N^{-1} +  (\sigma_q^3/24) N^{-1/4} \mbox{ if } w \in \gamma_w^1 .
\end{split}
\end{equation*}
Combining the last two sets of inequalities we conclude for $(z,w) \in \gamma_z \times \gamma_w \setminus\gamma_z^0 \times  \gamma_w^0$ 
\begin{equation}\label{Tr5}
\left|\exp \left(\tilde{N} \cdot [S(z) - S(w)] \right)\right| = \exp \left(\tilde{N}[\Real S(z)  - \Real S(w)] \right) \leq \exp \left(2c_5 - (\sigma_q^3/24) \cdot \tilde{N} N^{-1/4} \right).
\end{equation}
Combining (\ref{Tr1}), (\ref{Tr2}), (\ref{Tr3}), (\ref{Tr4}), (\ref{Tr5}), the fact that $\tilde{N} = N + O(N^{1/12})$, and the fact that $\gamma_z, \gamma_w$ have bounded (in $N$) length, we see that the second line of (\ref{TE1}) is bounded by $\exp \left( C N^{2/3} - (\sigma_q^3/24) N^{3/4} \right)$ for some large enough constant $C$. This implies (\ref{TruncEq}).
\end{proof}

%%%%%%%%%%%%%%%%%%%%%%%%%%%%%%%%%%%%%%%%%%%%%%%%%%%%%%%%%%%%%%%%%%%%%
%
%    Section 5.4
%
%%%%%%%%%%%%%%%%%%%%%%%%%%%%%%%%%%%%%%%%%%%%%%%%%%%%%%%%%%%%%%%%%%%%%
\subsection{Proof of (\ref{LimitK3})} \label{Section5.4} In this section we prove (\ref{LimitK3}). We denote
\begin{equation}\label{NS1}
 \tilde{K}_N^3(u, x_N; v ,y_N) =  \frac{1}{(2\pi \im)^2} \int_{\gamma_{z,N}}dz  \int_{\gamma_{w,N}}dw \frac{1}{z-w} \cdot \prod_{i = 1}^7 H_i^N(z,w) \cdot w^{y_N} z^{-x_N-1},
\end{equation}
where $\gamma_{z,N}$ and $\gamma_{w,N}$ are as in Section \ref{Section5.3}. In view of Lemma \ref{LTrunc} it suffices to prove that 
\begin{equation}\label{LimitK3R1}
\begin{split}
&\lim_{N \rightarrow \infty} N^{1/3} \sigma_q \tilde{K}^3_N(u, x_N;v, y_N) = \frac{1}{(2\pi \im)^2} \int_{\Gamma^+_{-\sigma_q}} dz \int_{\Gamma^-_{\sigma_q}} dw \hspace{2mm}  e^{ \frac{z^3}{3} - \tilde{x} z + f_q t_u z^2 - \frac{w^3}{3} + \tilde{y} w- f_q t_v w^2 }\\
& \times \frac{1}{z - w} \cdot e^{c^+z + c^-/z - c^+w - c^-/w} \cdot \prod_{i = 1}^{\infty} \frac{(1 - b_i^+w) (1 - b_i^-/w) (1 + a_i^+ z) (1 + a_i^-/z) }{(1 - b_i^+z) (1 - b_i^-/z) (1 + a_i^+w) (1 + a_i^-/w) },
\end{split}
\end{equation}

We perform the change of variables $z = 1 + N^{-1/3} \tilde{z}$ and $w = 1 + N^{-1/3} \tilde{w}$ and get
\begin{equation}\label{LimitK3R2}
\begin{split}
& N^{1/3} \sigma_q \tilde{K}_N^3(u, x_N; v ,y_N) = \frac{\sigma_q}{(2\pi \im)^2} \int_{\Gamma_{-1}^+}d\tilde{z}  \int_{\Gamma_1^-}d \tilde{w} {\bf 1}\{ |\Imag ( \tilde{w})| \leq N^{1/4}, |\Imag (\tilde{z})| \leq N^{1/4}\} \\
& \times   \frac{1}{\tilde{z} - \tilde{w}}  \times \prod_{i = 1}^{8} \tilde{H}^N_i(\tilde{z}, \tilde{w}) ,
\end{split}
\end{equation}
where 
\begin{equation}\label{NS2}
\tilde{H}^N_i(\tilde{z}, \tilde{w}) = H_i(1 + N^{-1/3} \tilde{z}, 1 + N^{-1/3} \tilde{w}) \mbox{ for } i = 1, \dots, 5,
\end{equation}
with $H^N_i$ as in (\ref{PKF41} - \ref{PKF45}). In addition, we have
\begin{equation}\label{NS3}
\tilde{H}^N_6(\tilde{z}, \tilde{w}) = \exp \left( \tilde{N} \cdot [S(1 + N^{-1/3} \tilde{z}) - S(1 + N^{-1/3} \tilde{w})] \right),
\end{equation}
where $S(z)$ is as in (\ref{DefS}), 
\begin{equation}\label{NS4}
\begin{split}
&\tilde{H}^N_7(\tilde{z}, \tilde{w}) = \exp \left( \tilde{M}_u Q(1 + N^{-1/3} \tilde{z} ) -\tilde{M}_v Q(1 + N^{-1/3} \tilde{w}) \right) \mbox{, with }\\
&Q(z) = - \log (1 - qz) + \log(1-q) - \frac{q}{1-q} \log z,
\end{split}
\end{equation}
and 
\begin{equation}\label{NS5}
\begin{split}
&\tilde{H}^N_8(\tilde{z}, \tilde{w}) = \exp \left( \sigma_q \tilde{y}_N \cdot N^{1/3} \cdot \log (1 + N^{-1/3} \tilde{w}) -  \sigma_q \tilde{x}_N \cdot N^{1/3} \cdot \log (1 + N^{-1/3} \tilde{z}) \right) \\
& \times \exp \left( \log(1 + N^{-1/3} \tilde{z}) \cdot \left( - \frac{qt_u}{1- q} \cdot N^{2/3} - 1 + \tilde{M}_u \cdot \frac{q}{1-q} \right) \right) \\
&\times \exp \left( \log(1 +N^{-1/3} \tilde{w}) \cdot \left( \frac{q t_v}{1-q} \cdot N^{2/3} - \tilde{M}_v \cdot \frac{q}{1-q} \right) \right).
\end{split}
\end{equation}
We proceed to analyze each $\tilde{H}_i^N$ for $i = 1, \dots, 8$ as $N \rightarrow \infty$. We mention that for each $i = 1, \dots, 8$ we will compute the pointwise limit of $\tilde{H}_i^N$ and find a suitable dominating function. In addition, when we find dominating functions for the functions $\tilde{H}_1^N, \tilde{H}_2^N, \tilde{H}_5^N, \tilde{H}_7^N$ we do so under more general assumptions on $\tilde{z}, \tilde{w}$. We do this with an outlook of using these bounds in later parts of the text.

\subsubsection{The term $\tilde{H}_1^N(\tilde{z}, \tilde{w})$} From (\ref{PKF41}) and (\ref{NS2}) we have that 
 \begin{equation*}
\begin{split}
\tilde{H}_1^N(\tilde{z}, \tilde{w}) = \prod_{i = 1}^{A_N} \frac{1 - \frac{1}{1 + N^{-1/3} \tilde{z} } \cdot  \left(1 - \frac{1}{N^{1/3} a_i^+ \sigma_q} \right)}{1 -  \frac{1}{1 + N^{-1/3} \tilde{w}} \cdot  \left(1 - \frac{1}{N^{1/3}  a_i^+ \sigma_q } \right)} = \prod_{i = 1}^{A_N}  \frac{\frac{\tilde{z}}{N^{1/3}} + \frac{1}{N^{1/3} a_i^+ \sigma_q} - \frac{\tilde{z}}{N^{2/3} a_i^+ \sigma_q} + O(N^{-2/3} |\tilde{z}|^2) }{ \frac{\tilde{w}}{N^{1/3}} + \frac{1}{N^{1/3} a_i^+ \sigma_q} - \frac{\tilde{w}}{N^{2/3} a_i^+ \sigma_q} + O(N^{-2/3} |\tilde{z}|^2) },
\end{split}
\end{equation*}
where the constant in the big $O$ notation depends on $q$ alone and in deriving the above we used that $\left(1 - \frac{1}{N^{1/3}  a_i^+ \sigma_q } \right) \in [q,1]$ from the definition of $A_N$ in Definition \ref{ParScale}. In particular,
\begin{equation}\label{H1E}
\begin{split}
\tilde{H}_1^N(\tilde{z}, \tilde{w}) = \prod_{i = 1}^{A_N}  \frac{1 +  a^+_i \sigma_q \tilde{z} - N^{-1/3} \tilde{z} + O(N^{-1/3} a_i^+ |\tilde{z}|^2) }{1 + a^+_i \sigma_q \tilde{w} - N^{-1/3} \tilde{w} + O(N^{-1/3}   a_i^+ |\tilde{w}|^2) }.
\end{split}
\end{equation}

Using (\ref{H1E}) and the bounds from (\ref{RatBound}) we can find for each $\tilde{\delta} > 0$ constants $C_1 > 0$ and $N_1 \in \mathbb{N}$, depending on $\tilde{\delta}$ and $q$, such that the following holds. For each $N \geq N_1$, $\tilde{z}, \tilde{w} \in \mathbb{C}$ with $|\tilde{z}| , |\tilde{w}| \leq 2N^{1/4}$ and $|1 + a_i^+ \sigma_q \tilde{w}| \geq \tilde{\delta}$ for all $i \geq 1$ we have
\begin{equation}\label{H1B}
\left|\tilde{H}_1^N(\tilde{z}, \tilde{w}) \right|  \leq \exp \left( C_1 + C_1 (|\tilde{z}| + |\tilde{w}|) \sum_{i = 1}^{\infty} a_i^+ \right).
\end{equation}
In particular, we observe that for $\tilde{z} \in \Gamma_{-1}^+, |\Imag(\tilde{z})| \leq N^{1/4}$ and $\tilde{w} \in \Gamma_{1}^-, |\Imag(\tilde{w})| \leq N^{1/4}$, we have that $|\tilde{z}| , |\tilde{w}| \leq 2N^{1/4}$, and also $|1 + a^+_i\sigma_q \tilde{w}| \geq 1/2$. This means that the conditions before (\ref{H1B}) are satisfied for $\tilde{\delta} = 1/2$, and so (\ref{H1B}) holds for some $C_1$ and all large $N$ (depending on $q$ only). The inequality in (\ref{H1B}) for a different choice of $\tilde{\delta}$ will be used later in the text. 

We next claim that for each fixed $\tilde{z} \in \Gamma_{-1}^+$ and $\tilde{w} \in \Gamma_1^-$ 
\begin{equation}\label{H1L}
\lim_{N \rightarrow \infty} \tilde{H}_1^N(\tilde{z}, \tilde{w}) = \prod_{i = 1}^{J_a^+} \frac{1 + a_i^+ \sigma_q \tilde{z}   }{1 + a_i^+ \sigma_q \tilde{w}  } = \prod_{i = 1}^{\infty} \frac{1 + a_i^+ \sigma_q \tilde{z}   }{1 + a_i^+ \sigma_q \tilde{w}  },
\end{equation}
where we recall that $J_a^+$ is as in Definition \ref{DLP}. If $J_a^+ < \infty$, we have that (\ref{H1L}) follows from (\ref{H1E}) as in this case $A_N = J_a^+$ for all large $N$. If $J_{a}^+ = \infty$, we can set $K \in \mathbb{N}$ to be sufficiently large (depending on $q$ and the sequence $\{a_i^+\}_{i \geq 1})$ so that $a_i^+ \sigma_q \in [0, 1/2]$ for all $i \geq K$. If $i \geq K$ and $\tilde{z} \in \Gamma_{-1}^+$ we have that $|1 + \sigma_q a_i^+ \tilde{z}| \geq 1/2$ and so
\begin{equation}\label{H1FI}
1 + \sigma_q a^+_i\tilde{z} - N^{-1/3} \tilde{z} + O(N^{-1/3}a_i^+ |\tilde{z}|^2) =  \left( 1 + \sigma_q a^+_i\tilde{z} \right) \cdot \exp \left( O(N^{-1/3}) \right),
\end{equation}
where the constant in the big $O$ notation depends on $\tilde{z}$ and $q$ but not $i$, provided that $i \geq K$. Using (\ref{H1FI}) and the fact that $A_N = O(N^{1/12})$ we conclude that 
\begin{equation*}
\begin{split}
&\lim_{N \rightarrow \infty}  \prod_{i = 1}^{A_N} \left( 1 +   a^+_i \sigma_q \tilde{z} - N^{-1/3} \tilde{z} + O(N^{-1/3} a_i^+ |\tilde{z}|^2)  \right) =    \prod_{i = 1}^{K} \left(    1 + a_i^+ \sigma_q  \tilde{z} \right) \\
& \times \lim_{N \rightarrow \infty} \prod_{i = K+1}^{A_N} \left( 1 +  a^+_i \sigma_q\tilde{z} \right) \cdot \exp \left( O(N^{-1/4}) \right) = \prod_{i = 1}^{\infty}\left( 1 + a_i^+ \sigma_q \tilde{z}  \right).  
\end{split}
\end{equation*}
On the other hand, if $\tilde{w} \in \Gamma_{1}^-$, we have that $|1 + a^+_i\sigma_q \tilde{w}| \geq 1/2$ for all $i \geq 1$ and so (\ref{H1FI}) holds with $\tilde{z}$ replaced with $\tilde{w}$ and constant now depending on $\tilde{w}$ and $q$. Using the latter and that $A_N = O(N^{1/12})$ we conclude
\begin{equation*}
\begin{split}
&\lim_{N \rightarrow \infty}  \prod_{i = 1}^{A_N} \left( 1 +   a^+_i \sigma_q \tilde{w} - N^{-1/3} \tilde{w} + O(N^{-1/3} a_i^+ |\tilde{w}|^2)  \right)  \\
& =  \lim_{N \rightarrow \infty} \prod_{i = 1}^{A_N} \left( 1 +  a^+_i \sigma_q\tilde{w} \right) \cdot \exp \left( O(N^{-1/4}) \right) = \prod_{i = 1}^{\infty}\left(1 + a_i^+ \sigma_q \tilde{w}  \right).  
\end{split}
\end{equation*}
The last two limits and (\ref{H1E}) imply (\ref{H1L}).

\subsubsection{The term $\tilde{H}_2^N(\tilde{z}, \tilde{w})$} From (\ref{PKF42}) and (\ref{NS2}) we have that 
 \begin{equation*}
\begin{split}
\tilde{H}_2^N(\tilde{z}, \tilde{w}) = \prod_{i = 1}^{B_N} \frac{1 - (1 + N^{-1/3} \tilde{w}) \cdot \left(1 - \frac{1}{N^{1/3} b_i^+ \sigma_q} \right)}{1 - (1 + N^{-1/3} \tilde{z}) \cdot \left(1 - \frac{1}{N^{1/3} b_i^+ \sigma_q} \right)} =   \prod_{i = 1}^{B_N} \frac{1 -  b_i^+ \sigma_q \tilde{w}  + \tilde{w}N^{-1/3} }{1 -  b_i^+ \sigma_q \tilde{z}  + \tilde{z}N^{-1/3}}.
\end{split}
\end{equation*}
Using the latter and (\ref{RatBound}) we can find for each $\tilde{\delta} > 0$ constants $C_2, N_2 > 0$, depending on $\tilde{\delta}, q$ such that the following holds. For each $N \geq N_2$, $\tilde{z}, \tilde{w} \in \mathbb{C}$ with $|\tilde{z}| , |\tilde{w}| \leq 2N^{1/4}$ and $|1 - b_i^+ \sigma_q \tilde{z}| \geq \tilde{\delta}$ for all $i \geq 1$ we have
\begin{equation}\label{H2B}
\left|\tilde{H}_2^N(\tilde{z}, \tilde{w}) \right|  \leq \exp \left( C_2 + C_2 (|\tilde{z}| + |\tilde{w}|) \sum_{i = 1}^{\infty} b_i^+ \right).
\end{equation}
Note that the conditions before (\ref{H2B}) are satisfied when $\tilde{z} \in \Gamma_{-1}^+, |\Imag(\tilde{z})| \leq N^{1/4}$ and $\tilde{w} \in \Gamma_{1}^-, |\Imag(\tilde{w})| \leq N^{1/4}$ with $\tilde{\delta} = 1/2$, and so (\ref{H2B}) holds for some $C_2$ and all large $N$ (depending on $q$ only). The inequality in (\ref{H2B}) for a different choice of $\tilde{\delta}$ will be used later in the text. 

Arguing as in the case of $\tilde{H}_1^N$ we also have for each $\tilde{z} \in \Gamma_{-1}^+$ and $\tilde{w} \in \Gamma_{1}^-$ 
\begin{equation}\label{H2L}
\lim_{N \rightarrow \infty} \tilde{H}_2^N(\tilde{z}, \tilde{w})  = \lim_{N \rightarrow \infty} \prod_{i = 1}^{B_N} \frac{1 -  b_i^+ \sigma_q \tilde{w}  }{1 -  b_i^+ \sigma_q \tilde{z}  } = \prod_{i = 1}^{\infty} \frac{1 -  b_i^+  \sigma_q \tilde{w}  }{1 - b_i^+ \sigma_q \tilde{z} }.
\end{equation}

\subsubsection{The term $\tilde{H}_3^N(\tilde{z}, \tilde{w})$} From (\ref{PKF43}) and (\ref{NS2}) we have that 
\begin{equation*}
\begin{split}
& \tilde{H}_3^N(\tilde{z}, \tilde{w}) = \prod_{i = 1}^{D_N} \frac{1 - (1 + N^{-1/3}\tilde{w}) \cdot\left(1 - \frac{b_i^-}{N^{1/3}\sigma_q}  \right)}{1 - (1 + N^{-1/3}\tilde{z})  \cdot\left(1 - \frac{b_i^-}{N^{1/3}\sigma_q}  \right)} \cdot \frac{1 - \frac{1}{1 + N^{-1/3}\tilde{z} } \cdot \left(1 - \frac{a_i^-}{N^{1/3} \sigma_q}  \right)}{1 - \frac{1}{1 + N^{-1/3}\tilde{w} } \cdot \left(1 - \frac{a_i^-}{N^{1/3} \sigma_q}  \right)} \\
&= \prod_{i = 1}^{D_N} \frac{1 - \frac{b_i^-}{\sigma_q \tilde{w}}- \sigma_q^{-1} b_i^- N^{-1/3} }{1 - \frac{b_i^-}{\sigma_q \tilde{z}}- \sigma_q^{-1} b_i^- N^{-1/3} } \cdot  \frac{1 + \frac{a_i^-}{\sigma_q \tilde{z}}  - \sigma_q^{-1} a_i^- N^{-1/3} + O(N^{-1/3} |\tilde{z}|) }{1 + \frac{a_i^-}{\sigma_q \tilde{w}}  - \sigma_q^{-1} a_i^- N^{-1/3} + O(N^{-1/3} |\tilde{w}|)}.
\end{split}
\end{equation*}
We next note that if $\tilde{z} \in \Gamma^+_{-1}$, and $\tilde{w} \in \Gamma^-_{1}$, then for all $i \geq 1$ 
\begin{equation}\label{H3LB}
\left|1 - \frac{b_i^-}{\sigma_q \tilde{z}}  \right| \geq \frac{1}{2} \mbox{ and } \left| 1 + \frac{a_i^-}{\sigma_q \tilde{w}}  \right| \geq \frac{1}{2}.
\end{equation}
As the two statements are proved analogously, we focus on the first inequality. If $\Real \tilde{z} \leq 0$, then $\Real (1/\tilde{z}) \leq 0$ and so 
$$\left|1 - \frac{b_i^-}{\sigma_q \tilde{z}}  \right| \geq \Real \left(1 - \frac{b_i^-}{\sigma_q \tilde{z}} \right) \geq 1, $$
which proves the first inequality in (\ref{H3LB}). We may thus assume that $\tilde{z} = x \pm \im (x+1)$ for some $x \geq 0$. In the latter case we have that 
$$\Real(1/\tilde{z}) = \frac{x}{x^2 + (x+1)^2} \mbox{ and } \Imag(1/\tilde{z}) = \pm \frac{x + 1}{x^2 + (x+1)^2}.$$
The latter suggests that $\Arg(1/\tilde{z}) \in [\pi/4, 7\pi/4]$, and so $\frac{b_i^-}{\sigma_q \tilde{z}}$ (which has the same argument) is at least distance $1/2$ from $1$, which proves (\ref{H3LB}) for $\tilde{z} \in \Gamma^+_{-1}$ with $\Real \tilde{z} \geq 0$ as well.

We may now combine the lower bounds in (\ref{H3LB}), the fact that $|1/\tilde{w}|$, and $|1/\tilde{z}|$ are uniformly bounded for $\tilde{z} \in \Gamma^+_{-1}$, and $\tilde{w} \in \Gamma^-_{1}$, with (\ref{RatBound}) to conclude that for some $C_3 > 0$, depending on $q$, and all large $N$
\begin{equation}\label{H3B}
\left|\tilde{H}_3^N(\tilde{z}, \tilde{w}) \right|  \leq \exp \left(C_3 + C_3  \sum_{i = 1}^{\infty} (a_i^- + b_i^-) \right),
\end{equation}
provided that $\tilde{z} \in \Gamma_{-1}^+, |\Imag(\tilde{z})| \leq N^{1/4}$ and $\tilde{w} \in \Gamma_{1}^-, |\Imag(\tilde{w})| \leq N^{1/4}$. We mention that in deriving (\ref{H3B}) we used that $D_N = O(N^{1/12})$ and our choice of $\tilde{z}, \tilde{w}$ ensures $|\tilde{z}| = O(N^{1/4})$ and $|\tilde{w}| = O(N^{1/4})$.

Arguing as in the case of $\tilde{H}_1^N$ we also have for each $\tilde{z} \in \Gamma_{-1}^+$ and $\tilde{w} \in \Gamma_{1}^-$ 
\begin{equation}\label{H3L}
\lim_{N \rightarrow \infty} \tilde{H}_3^N(\tilde{z}, \tilde{w})  = \prod_{i = 1}^{\infty} \frac{1 - \frac{b_i^-}{\sigma_q \tilde{w}} }{1 - \frac{b_i^-}{\sigma_q \tilde{z}} } \cdot  \frac{1 + \frac{a_i^-}{\sigma_q \tilde{z}}  }{1 + \frac{a_i^-}{\sigma_q \tilde{w}}  }.
\end{equation}

\subsubsection{The term $\tilde{H}_4^N(\tilde{z}, \tilde{w})$}  From (\ref{PKF44}) and (\ref{NS2}) we have $ \tilde{H}_4^N(\tilde{z}, \tilde{w}) = 1$ if $c^- = 0$ and if $c^- > 0$
\begin{equation*}
\begin{split}
& \tilde{H}_4^N(\tilde{z}, \tilde{w}) =   \left( \frac{1 - (1 + N^{-1/3} \tilde{w}) \cdot \left(1 - \frac{c^-}{2N^{5/12} \sigma_q}\right)}{1 - (1 + N^{-1/3} \tilde{z})  \cdot \left(1 - \frac{c^-}{2N^{5/12} \sigma_q} \right)}  \cdot \frac{1 - \frac{1}{1 + N^{-1/3} \tilde{z} } \cdot \left(1 - \frac{c^-}{2N^{5/12} \sigma_q} \right)}{1 - \frac{1}{1 + N^{-1/3} \tilde{w} } \cdot \left(1  - \frac{c^-}{2N^{5/12} \sigma_q} \right)} \right)^{\lfloor N^{1/12} \rfloor }  \\
& = \left( \frac{ - \frac{\tilde{w}}{N^{1/3}} + \frac{c^-}{2 N^{5/12} \sigma_q} + \frac{c^- \tilde{w} }{2N^{3/4} \sigma_q} }{ - \frac{\tilde{z}}{N^{1/3}} + \frac{c^-}{2 N^{5/12} \sigma_q} + \frac{c^- \tilde{z} }{2N^{3/4} \sigma_q} }   \cdot \frac{\frac{\tilde{z}}{N^{1/3}} + \frac{c^-}{2N^{5/12} \sigma_q} - \frac{c^- \tilde{z}}{2N^{3/4} \sigma_q } + O(|\tilde{z}|^2 N^{-2/3})}{\frac{\tilde{w}}{N^{1/3}} + \frac{c^-}{2N^{5/12} \sigma_q} - \frac{c^- \tilde{w}}{2N^{3/4} \sigma_q} + O(|\tilde{w}|^2 N^{-2/3})}\right)^{\lfloor N^{1/12} \rfloor },
\end{split}
\end{equation*}
where the constant in the big $O$ notation depends on $q$ alone and in deriving the above we used that $ \left(1 - \frac{c^-}{2N^{5/12} \sigma_q}\right) \in [q, 1]$ provided that $N \geq N_0$ as in Definition \ref{ParScale}. Using the above we conclude
\begin{equation*}
\begin{split}
& \tilde{H}_4^N(\tilde{z}, \tilde{w}) =   \left(  \frac{1  - \frac{c^-}{2 N^{1/12} \sigma_q \tilde{w}} - \frac{c^- }{2N^{5/12} \sigma_q} }{  1 - \frac{c^-}{2 N^{1/12} \sigma_q \tilde{z}} - \frac{c^- }{2N^{5/12} \sigma_q} }  \cdot \frac{1 + \frac{c^-}{2N^{1/12} \sigma_q \tilde{z}} - \frac{c^- }{2N^{5/12} \sigma_q } + O(|\tilde{z}| N^{-1/3})}{1 + \frac{c^-}{2N^{1/12} \sigma_q \tilde{w}} - \frac{c^- }{2N^{5/12} \sigma_q} + O(|\tilde{w}| N^{-1/3})}\right)^{\lfloor N^{1/12} \rfloor }.
\end{split}
\end{equation*}
For fixed $\tilde{z}, \tilde{w}$ and $c^- > 0$ we conclude that 
\begin{equation}\label{H4L}
\lim_{N \rightarrow \infty} \tilde{H}_4^N(\tilde{z}, \tilde{w})  = \exp\left( \frac{c^-}{\sigma_q \tilde{z}} - \frac{c^-}{\sigma_q \tilde{w}} \right),
\end{equation}
and of course (\ref{H4L}) still holds if $c^- = 0$ as then $ \tilde{H}_4^N \equiv 1$ for all $N \geq 1$. Since $u^{-1}$ is bounded for $u \in \Gamma_{-1}^+ \cup \Gamma_1^-$ we also conclude that for some $C_4 > 0$, depending on $q$, and all large $N$
\begin{equation}\label{H4B}
\left|\tilde{H}_4^N(\tilde{z}, \tilde{w}) \right|  \leq \exp \left(C_4 (1 + c^-) \right).
\end{equation}

\subsubsection{The term $\tilde{H}_5^N(\tilde{z}, \tilde{w})$}  From (\ref{PKF45}) and (\ref{NS2}) we have $ \tilde{H}_5^N(\tilde{z}, \tilde{w}) = 1$ if $c^+ = 0$ and if $c^+ > 0$
\begin{equation*}
\begin{split}
& \tilde{H}_5^N(\tilde{z}, \tilde{w}) =   \left( \frac{1 - (1 + N^{-1/3} \tilde{w}) \cdot\left(1 - \frac{2}{N^{1/4} c^+ \sigma_q} \right)}{1 -  (1 + N^{-1/3} \tilde{z})  \cdot \left(1 - \frac{2}{N^{1/4} c^+ \sigma_q} \right)} \cdot \frac{1 - \frac{1}{ (1 + N^{-1/3} \tilde{z}) } \cdot \left(1 - \frac{2}{N^{1/4} c^+ \sigma_q} \right)}{1 - \frac{1}{ (1 + N^{-1/3} \tilde{w}) } \cdot \left(1 - \frac{2}{N^{1/4} c^+ \sigma_q} \right)} \right)^{\lfloor N^{1/12} \rfloor }\\
&  \left( \frac{ - \frac{\tilde{w}}{N^{1/3}} + \frac{2}{ N^{1/4} c^+ \sigma_q} + \frac{2\tilde{w} }{N^{7/12} c^+ \sigma_q} }{  - \frac{\tilde{z}}{N^{1/3}} + \frac{2}{ N^{1/4} c^+ \sigma_q} + \frac{2\tilde{z} }{N^{7/12} c^+ \sigma_q}  }   \cdot \frac{\frac{\tilde{z}}{N^{1/3}} + \frac{2}{N^{1/4} c^+ \sigma_q} - \frac{2 \tilde{z}}{N^{7/12} c^+ \sigma_q } + O(|\tilde{z}|^2 N^{-2/3})}{\frac{\tilde{w}}{N^{1/3}} + \frac{2}{N^{1/4} c^+ \sigma_q} - \frac{2 \tilde{w}}{N^{7/12} c^+ \sigma_q } + O(|\tilde{w}|^2 N^{-2/3})}\right)^{\lfloor N^{1/12} \rfloor },
\end{split}
\end{equation*}
where the constant in the big $O$ notation depends on $q$ alone and in deriving the above we used that $ \left(1 - \frac{2}{N^{1/4} c^+ \sigma_q}\right) \in [q, 1]$, provided that $N \geq N_0$ as in Definition \ref{ParScale}. Using the above we conclude
\begin{equation*}
\begin{split}
& \tilde{H}_5^N(\tilde{z}, \tilde{w}) =   \left( \frac{1 - \frac{c^+\sigma_q \tilde{w}}{2N^{1/12}} + \frac{\tilde{w} }{N^{1/3}} }{  1 - \frac{c^+\sigma_q\tilde{z}}{2N^{1/12}} + \frac{\tilde{z} }{N^{1/3}}  }   \cdot \frac{1 + \frac{c^+\sigma_q\tilde{z} }{2N^{1/12}}  - \frac{\tilde{z}}{N^{1/3} } + O(c^+|\tilde{z}|^2 N^{-5/12})}{1 + \frac{c^+ \sigma_q \tilde{w} }{2N^{1/12}} - \frac{ \tilde{w} }{N^{1/3} } + O(c^+|\tilde{w}|^2 N^{-5/12})}\right)^{\lfloor N^{1/12} \rfloor }.
\end{split}
\end{equation*}
From the latter and (\ref{RatBound}) we can find for each $\tilde{\delta} > 0$ constants $C_5, N_5 > 0$, depending on $\tilde{\delta}, q$ such that the following holds. For each $N \geq N_5$, $\tilde{z}, \tilde{w} \in \mathbb{C}$ with $|\tilde{z}|, |\tilde{w}| \leq 2N^{1/4}$ such that 
$$\left|1 - \frac{c^+\sigma_q\tilde{z}}{2N^{1/12}} + \frac{\tilde{z} }{N^{1/3}} \right| \geq \tilde{\delta} \mbox{ and } \left|1 + \frac{c^+  \sigma_q\tilde{w} }{2N^{1/12}} - \frac{\tilde{w}}{N^{1/3} } + O(c^+|\tilde{w}|^2 N^{-5/12}) \right| \geq \tilde{\delta}, $$
we have that
\begin{equation}\label{H5B}
\left|\tilde{H}_5^N(\tilde{z}, \tilde{w}) \right|  \leq \exp \left(C_5 + C_5 c^+ (|z| + |w|)  \right).
\end{equation}
Note that the conditions before (\ref{H5B}) are satisfied when $\tilde{z} \in \Gamma_{-1}^+, |\Imag(\tilde{z})| \leq N^{1/4}$ and $\tilde{w} \in \Gamma_{1}^-, |\Imag(\tilde{w})| \leq N^{1/4}$ with $\tilde{\delta} = 1/2$, and so (\ref{H5B}) holds for some $C_5$ and all large $N$ (depending on $q$ only). The inequality in (\ref{H5B}) for a different choice of $\tilde{\delta}$ will be used later in the text.

For fixed $\tilde{z}, \tilde{w}$ and $c^+ > 0$ we also have by directly taking the limit above 
\begin{equation}\label{H5L}
\lim_{N \rightarrow \infty} \tilde{H}_5^N(\tilde{z}, \tilde{w})  = \exp\left( c^+ \sigma_q \tilde{z} - c^+ \sigma_q \tilde{w} \right),
\end{equation}
and of course (\ref{H5L}) still holds if $c^+ = 0$ as then $ \tilde{H}_5^N \equiv 1$ for all $N \geq 1$.

\subsubsection{The term $\tilde{H}_6^N(\tilde{z}, \tilde{w})$} From (\ref{NS3}) and (\ref{EqTayS}) we have for fixed $\tilde{z} \in \Gamma_{-1}^+$ and $\tilde{w} \in \Gamma_1^-$
\begin{equation}\label{H6L}
\lim_{N \rightarrow \infty} \tilde{H}_6^N(\tilde{z}, \tilde{w})  =\lim_{N \rightarrow \infty}  \exp \left( \sigma_q^3 \tilde{z}^3/3 - \sigma_q^3 \tilde{w}^3/3 + O(N^{-1/3}) \right) = \exp \left( \sigma_q^3 \tilde{z}^3/3 - \sigma_q^3 \tilde{w}^3/3\right),
\end{equation}
where the constant in the big $O$ notation depends on $\tilde{z}, \tilde{w}$. We mention that in deriving (\ref{H6L}) we used that $\tilde{N} = N + O(N^{1/12})$ with a universal constant in the last big $O$ notation. From (\ref{RealSB}) we also have that for some $C_6 > 0$  and all large $N$
\begin{equation}\label{H6B}
\left|\tilde{H}_6^N(\tilde{z}, \tilde{w}) \right|  \leq \exp \left( C_6 - (\sigma_q^3/24) |\tilde{z}|^3 - (\sigma_q^3/24) |\tilde{w}|^3   \right),
\end{equation}
provided that $\tilde{z} \in \Gamma_{-1}^+, |\Imag(\tilde{z})| \leq N^{1/4}$ and $\tilde{w} \in \Gamma_{1}^-, |\Imag(\tilde{w})| \leq N^{1/4}$.

\subsubsection{The term $\tilde{H}_7^N(\tilde{z}, \tilde{w})$} Recall from (\ref{NS4}) that $Q(z) = - \log (1 - qz) + \log(1-q) - \frac{q}{1-q} \log z$. By a direct computation
$$Q'(z) = \frac{q}{1- qz} -\frac{q}{(1-q)z} \mbox{ and } Q''(z) = \frac{q}{(1-qz)^2} + \frac{q}{(1-q) z^2},$$
which gives $Q(1) = Q'(1) = 0$ and so for $|x| \leq 1/2$ we have
\begin{equation}\label{H7Q}
Q(1 + x) = x^2 \cdot \frac{q}{2(1-q)^2} + O(x^3).
\end{equation}
From (\ref{NS4}), (\ref{H7Q}) and the fact that $\tilde{M}_u = t_u \cdot N^{2/3} + O(N^{1/12})$, $\tilde{M}_v = t_v N^{2/3} + O(N^{1/12})$ we conclude that for fixed $\tilde{z} \in \Gamma_{-1}^+$ and $\tilde{w} \in \Gamma_{1}^-$
\begin{equation}\label{H7L}
\begin{split}
\lim_{N \rightarrow \infty} \tilde{H}_7^N(\tilde{z}, \tilde{w})  &=\lim_{N \rightarrow \infty}  \exp \left( \tilde{M}_u Q(1 + N^{-1/3} \tilde{z} ) -\tilde{M}_v Q(1 + N^{-1/3} \tilde{w}) \right) \\
&= \exp \left( t_u \cdot \frac{q}{2(1-q)^2} \cdot \tilde{z}^2 -  t_v \cdot \frac{q}{2(1-q)^2} \cdot \tilde{w}^2 \right).
\end{split}
\end{equation}
From (\ref{H7Q}) we also conclude that for some $C_7 > 0$, depending on $t_u, t_v, q$,  and all large $N$ we have
\begin{equation}\label{H7B}
\left|\tilde{H}_7^N(\tilde{z}, \tilde{w}) \right|  \leq \exp \left( C_7  + C_7(|\tilde{z}|^2 + |\tilde{w}|^2)  \right),
\end{equation}
provided that $|\tilde{z}| \leq 2N^{1/4}$ and $|\tilde{w}| \leq 2N^{1/4}$ and so in particular if $\tilde{z} \in \Gamma_{-1}^+, |\Imag(\tilde{z})| \leq N^{1/4}$ and $\tilde{w} \in \Gamma_{1}^-, |\Imag(\tilde{w})| \leq N^{1/4}$.

\subsubsection{The term $\tilde{H}_8^N(\tilde{z}, \tilde{w})$} From (\ref{NS5}) and the fact that $\tilde{M}_u = t_u N^{2/3} + O(N^{1/12})$, $\tilde{M}_v = t_v N^{2/3} + O(N^{1/12})$ we conclude that for fixed $\tilde{z} \in \Gamma_{-1}^+$ and $\tilde{w} \in \Gamma_{1}^-$
\begin{equation}\label{H8L}
\begin{split}
\lim_{N \rightarrow \infty} \tilde{H}_8^N(\tilde{z}, \tilde{w})  = \exp \left(\tilde{y} \sigma_q  \tilde{w} - \tilde{x} \sigma_q  \tilde{z} \right).
\end{split}
\end{equation}
Furthermore, if $|\tilde{z}| \leq 2N^{1/4} $ and $|\tilde{w}| \leq 2N^{1/4}$ we have 
\begin{equation}\label{TBO1}
\begin{split}
&\log(1 + N^{-1/3} \tilde{z}) \cdot \left( - \frac{qt_u}{1- q} \cdot N^{2/3} - 1 +  \frac{q\tilde{M}_u}{1-q} \right) = \log(1 + N^{-1/3} \tilde{z})  \cdot O(N^{1/12}) = O(1),\\
& \log(1 +N^{-1/3} \tilde{w}) \cdot \left( \frac{q t_v}{1-q} \cdot N^{2/3} - \frac{q \tilde{M}_v}{1-q} \right) =  \log(1 + N^{-1/3} \tilde{w})  \cdot O(N^{1/12}) = O(1),
\end{split}
\end{equation}
where the constants in the big $O$ notations depend on $t_u, t_v, q$. Combining the latter with (\ref{NS5}), we conclude that for some $C_8 > 0$, depending on $t_u, t_v, q$ and the sequences $\tilde{x}_n, \tilde{y}_n$, and all large $N$ 
\begin{equation}\label{H8B}
\left|\tilde{H}_8^N(\tilde{z}, \tilde{w}) \right|  \leq \exp \left( C_8  + C_8(|\tilde{z}|  + |\tilde{w}|)  \right),
\end{equation}
provided that $\tilde{z} \in \Gamma_{-1}^+, |\Imag(\tilde{z})| \leq N^{1/4}$ and $\tilde{w} \in \Gamma_{1}^-, |\Imag(\tilde{w})| \leq N^{1/4}$.

\subsubsection{Proving  (\ref{LimitK3R1})} It follows from (\ref{H1L}), (\ref{H2L}), (\ref{H3L}), (\ref{H4L}), (\ref{H5L}), (\ref{H6L}), (\ref{H7L}) and (\ref{H8L}) that for each $\tilde{z} \in \Gamma_{-1}^+$ and $\tilde{w} \in \Gamma_1^-$ such that $(\tilde{z}, \tilde{w}) \neq (\pm \im, \pm \im)$ (these are the points where $\tilde{z} - \tilde{w}$ vanishes) we have
\begin{equation}\label{H9L}
\begin{split}
&\lim_{N \rightarrow \infty}  \frac{{\bf 1}\{ |\Imag ( \tilde{w})| \leq N^{1/4}, |\Imag (\tilde{z})| \leq N^{1/4}\}}{\tilde{z} - \tilde{w}}  \times \prod_{i = 1}^{8} \tilde{H}^N_i(\tilde{z}, \tilde{w}) \\
&=  \frac{e^{\frac{\sigma_q^3 \tilde{z}^3}{3} + \frac{t_uq}{2(1-q)^2} \tilde{z}^2 - \sigma_q \tilde{x} \tilde{z} - \frac{\sigma_q^3 \tilde{w}^3}{3} - \frac{t_uq}{2(1-q)^2} \tilde{w}^2 + \sigma_q \tilde{y} \tilde{w}} }{\tilde{z} - \tilde{w}} \cdot  e^{\frac{c^-}{\sigma_q \tilde{z}} - \frac{c^-}{\sigma_q \tilde{w}} + c^+ \sigma_q \tilde{z} - c^+ \sigma_q \tilde{w}} \\
&\times \prod_{i = 1}^{\infty}  \frac{(1 - b_i^+\tilde{w}\sigma_q) (1 - b_i^-/(\tilde{w}\sigma_q)) (1 + a_i^+ \tilde{z}\sigma_q) (1 + a_i^-/(\tilde{z}\sigma_q)) }{(1 - b_i^+\tilde{z}\sigma_q) (1 - b_i^-/(\tilde{z}\sigma_q)) (1 + a_i^+\tilde{w}\sigma_q) (1 + a_i^-/(\tilde{w}\sigma_q)) }.
\end{split}
\end{equation}
On the other hand, we have from (\ref{H1B}), (\ref{H2B}), (\ref{H3B}), (\ref{H4B}), (\ref{H5B}), (\ref{H6B}), (\ref{H7B}) and (\ref{H8B}) that there are constants $D_0, D_1, D_2 > 0$, depending on $t_u, t_v, q, c^+, c^-$ and the sequences $\tilde{x}_n, \tilde{y}_n, a^{\pm}_n, b^{\pm}_n$, such that for all large $N$
\begin{equation}\label{H9B}
\begin{split}
&\left| \frac{{\bf 1}\{ |\Imag ( \tilde{w})| \leq N^{1/4}, |\Imag (\tilde{z})| \leq N^{1/4}\}}{\tilde{z} - \tilde{w}}  \times \prod_{i = 1}^{8} \tilde{H}^N_i(\tilde{z}, \tilde{w}) \right| \\
&\leq \exp \left(D_0 + D_1 (|\tilde{z}| + |\tilde{w}|) + D_2 (|\tilde{z}|^2 + |\tilde{w}|^2) - (\sigma_q^3/24) |\tilde{z}|^3 - (\sigma_q^3/24)|\tilde{w}|^3  \right) \cdot \frac{1}{|\tilde{z} - \tilde{w}|}.
\end{split}
\end{equation}
Note that the function on the second line of (\ref{H9B}) is integrable over $(\tilde{z}, \tilde{w}) \in \Gamma_{-1}^+ \times \Gamma_1^-$. Indeed, the integrability near infinity is ensured by the cubic terms in the exponential, while the integrability of $|\tilde{z}- \tilde{w}|^{-1}$ near the singularities at $(\pm \im,  \pm \im)$ follows from the argument in (\ref{IntSing}). Using (\ref{H9L}) and the dominated convergence theorem, with dominating function given by the second line of (\ref{H9B}), we conclude from (\ref{LimitK3R2}) that
\begin{equation*}
\begin{split}
& \lim_{N \rightarrow \infty} N^{1/3} \sigma_q \tilde{K}_N^3(u, x_N; v ,y_N) = \frac{\sigma_q}{(2\pi \im)^2} \int_{\Gamma_{-1}^+}d\tilde{z}  \int_{\Gamma_1^-}d \tilde{w} \frac{e^{\frac{\sigma_q^3 \tilde{z}^3}{3} + \frac{t_uq}{2(1-q)^2} \tilde{z}^2 - \sigma_q \tilde{x} \tilde{z} - \frac{\sigma_q^3 \tilde{w}^3}{3} - \frac{t_uq}{2(1-q)^2} \tilde{w}^2 + \sigma_q \tilde{y} \tilde{w}} }{\tilde{z} - \tilde{w}}    \\
&\times e^{\frac{c^-}{\sigma_q \tilde{z}} - \frac{c^-}{\sigma_q \tilde{w}} + c^+ \sigma_q \tilde{z} - c^+ \sigma_q \tilde{w}} \prod_{i = 1}^{\infty}  \frac{(1 - b_i^+\tilde{w}\sigma_q) (1 - b_i^-/(\tilde{w}\sigma_q)) (1 + a_i^+ \tilde{z}\sigma_q) (1 + a_i^-/(\tilde{z}\sigma_q)) }{(1 - b_i^+\tilde{z}\sigma_q) (1 - b_i^-/(\tilde{z}\sigma_q)) (1 + a_i^+\tilde{w}\sigma_q) (1 + a_i^-/(\tilde{w}\sigma_q)) }.
\end{split}
\end{equation*}
Changing variables in the last equation $z = \sigma_q \tilde{z}$ and $w = \sigma_q \tilde{w}$, and using that $\frac{q}{2(1-q)^2} \cdot \sigma_q^{-2} = f_q$ we obtain (\ref{LimitK3R1}).

%%%%%%%%%%%%%%%%%%%%%%%%%%%%%%%%%%%%%%%%%%%%%%%%%%%%%%%%%%%%%%%%%%%%%
%
%    Section 6
%
%%%%%%%%%%%%%%%%%%%%%%%%%%%%%%%%%%%%%%%%%%%%%%%%%%%%%%%%%%%%%%%%%%%%%
\section{Upper tail estimates}\label{Section6} In this section we obtain upper tail estimates for the kernel $K_N(u,x;v,y)$ from (\ref{CKDefAlt}) when $x,y$ are large -- this is Proposition \ref{KUTE} below. We further use this result to estimate the upper tails of the distribution of  $\lambda^{M_k(N)}_1$ under $\mathbb{P}_N$ as in Definition \ref{ParScale} -- this is Proposition \ref{PVC} below.

%%%%%%%%%%%%%%%%%%%%%%%%%%%%%%%%%%%%%%%%%%%%%%%%%%%%%%%%%%%%%%%%%%%%%
%
%    Section 6.1
%
%%%%%%%%%%%%%%%%%%%%%%%%%%%%%%%%%%%%%%%%%%%%%%%%%%%%%%%%%%%%%%%%%%%%%
\subsection{Kernel estimate}\label{Section6.1} The goal of this section is to obtain the following upper tail estimate.
\begin{proposition}\label{KUTE} Assume the same notation in Definitions \ref{DLP} and \ref{ParScale}, and suppose that $a_1^- = b_1^- = c^- = 0$. For each $L \geq 0$ there exist $N_1 \geq N_0$, $R_1 > 0$ and $r_1 > 0$, depending on $L$, $q$, $c^+$, $\{t_1, \dots, t_m\}$ and the sequences $\{a_i^+\}_{i \geq 1}$, $\{b_i^+\}_{i \geq 1}$, such that the following holds. For any $N \geq N_1$, $u, v \in \{1, \dots, m\}$, $\tilde{x}_N, \tilde{y}_N \in [-L, \infty)$,
\begin{equation}\label{EqUTE}
N^{1/3} \sigma_q \left|  K_N(u,x_N;v,y_N) - K^2_N(u,x_N;v,y_N)   \right| \leq R_1 \cdot \exp \left( - r_1 \cdot |\tilde{x}_N| - r_1 \cdot |\tilde{y}_N| \right),
\end{equation}
where $x_N, y_N$ are as in (\ref{ScaleXY}), $K_N(u,x;v,y)$ is as in (\ref{CKDefAlt}) with $\delta = N^{-1/12} \in (0, 1/2]$ and $K^2_N(u,x;v,y)$ is as in (\ref{PKF2}).
\end{proposition}
\begin{proof} Combining (\ref{CKDefAlt}), (\ref{CKDefAlt2}) and (\ref{CKDefAlt3}), we see that for $N \geq N_0$ 
\begin{equation}\label{S61UT1}
K_N(u,x; v, y) - K^2_N(u,x;v,y) =  \frac{1}{(2\pi \im)^2} \oint_{C_{r_1}}dz \oint_{C_{r_2}}dw \frac{1}{z-w} \cdot\prod_{i = 1}^7H^N_i(z,w)  \cdot w^y z^{-x -1}.
\end{equation}
We recall that the functions $H_i^N$ are as in (\ref{PKF41}-\ref{PKF45}), and $C_{r_i}$ are positively oriented, zero-centered circles with radii that satisfy $\max_{1 \leq i \leq N} y_i^N = \overline{y} < r_2 < r_1 < \underline{x} = \min_{1 \leq j \leq M_m} (x_j^N)^{-1}$. 

For clarity we split the remainder of the proof into three steps, which we describe presently. In view of (\ref{S61UT1}) we need to find upper bounds for the double integral on the right side. In Step 1, we deform the contours $C_{r_1}$ and $C_{r_2}$ to contours $\gamma_w$ and $\gamma_z$ that are similar to the ones in Section \ref{Section4}, although they do not cross each other. Choosing separated contours is possible in the present setting due to our assumption that $a_1^- = b_1^- = c^- = 0$. We further split the contours $\gamma_{z}, \gamma_w$ into two pieces -- ones close to the critical point $1$, and ones well-separated from it. This allows us to rewrite the expression in (\ref{S61UT1}) as the sum of four terms -- see (\ref{sliceK3}). In Step 2 we investigate the terms in (\ref{sliceK3}) that involve integration over a contour that is well-separated from $1$. For these terms we can effectively use the bounds in Section \ref{Section5.3} to obtain the required estimate. In Step 3 we deal with the term in (\ref{sliceK3}) that involves integration close to $1$, and obtain estimates by utilizing the bounds in Section \ref{Section5.4}. We now turn to supplying the details.\\

{\bf \raggedleft Step 1.} Suppose that $\delta_1 > 0$ is sufficiently small so that 
\begin{equation}\label{S6Delta1}
1 + 2\delta_1 < q^{-1}, \hspace{3mm} 1 - 2\delta_1 \geq q, \hspace{3mm} \delta_1 < 1/2, \hspace{3mm} \delta_1 b_1^+ \sigma_q < 1/2, \hspace{3mm} \delta_1 a_1^+ \sigma_q < 1/2 ,\mbox{ and } \delta_1 c^+ \sigma_q < 1.
\end{equation} 
Throughout the proof we assume that $N$ is large so that $N \geq N_0$ and $ N^{-1/12} \in (\delta_1 \cdot N^{-1/3}, 1/2]$.

We deform $C_{r_1}$ and $C_{r_2}$ in (\ref{S61UT1}) to the contours $\gamma_z = \gamma^+_{1 + \delta_1 N^{-1/3}, N^{-1/12}}$ and $\gamma_w = \gamma^-_{1 - \delta_1 N^{-1/3}, N^{-1/12}}$, respectively, where we recall that $\gamma_{a,\delta}^{\pm}$ was defined in Definition \ref{finiteContours}. From (\ref{S6Delta1}) and the scaling in Definition \ref{ParScale} we see that
\begin{equation}\label{S6Parbound}
1+ \delta_1N^{-1/3} \leq \underline{x} - \delta_1N^{-1/3} \mbox{ and } 1 - \delta_1 N^{-1/3} \geq \delta_1N^{-1/3} + \overline{y}.
\end{equation}
In particular, we conclude that we do not cross any poles and by Cauchy's theorem we do not change the value of the integral in the process of the deformation, which gives
\begin{equation}\label{S61SpecEq}
K_N(u,x; v, y) - K^2_N(u,x;v,y) =  \frac{1}{(2\pi \im)^2} \oint_{\gamma_z}dz \oint_{\gamma_w}dw \frac{1}{z-w} \cdot\prod_{i = 1}^7H^N_i(z,w)  \cdot w^y z^{-x -1}.
\end{equation}
We next let 
$$\gamma_w^0 = \gamma^{-, \vert}_{1 - \delta_1 N^{-1/3}, N^{-1/12}}, \hspace{3mm}  \gamma_w^1 = \gamma^{-, \circ}_{1 - \delta_1 N^{-1/3}, N^{-1/12}},$$ 
$$\gamma_z^0 = \gamma^{+, \vert}_{1 + \delta_1 N^{-1/3}, N^{-1/12}}, \hspace{3mm} \gamma_z^1 = \gamma^{+, \circ}_{1 + \delta_1 N^{-1/3}, N^{-1/12}},$$
where we recall that $\gamma_{a,\delta}^{\pm, \vert}$ and $\gamma_{a, \delta}^{\pm, \circ}$ were defined in Definition \ref{finiteContours}, see also Figure \ref{S61}. 

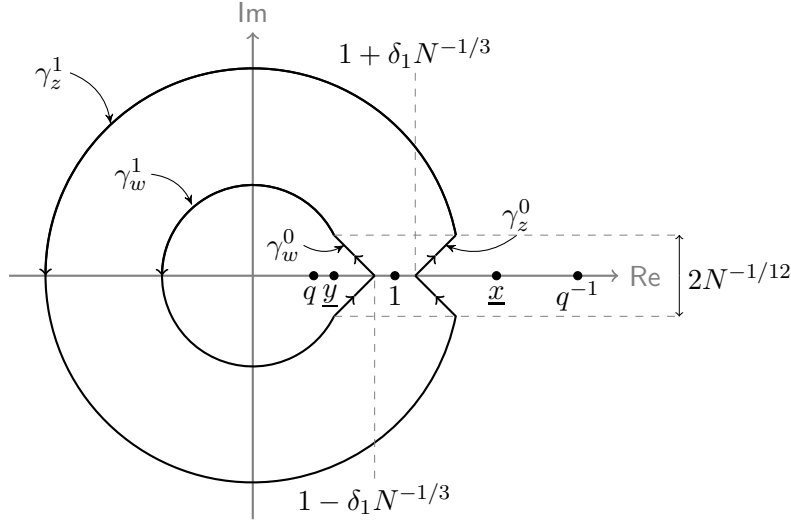
\begin{figure}[h]
    \centering
     \begin{tikzpicture}[scale=2.7]

        \def\tra{3} 
        % Picture on the left
        % Coordinate System
        \draw[->, thick, gray] (-1.2,0)--(1.8,0) node[right]{$\Real$};
        \draw[->, thick, gray] (0,-1.2)--(0,1.2) node[above]{$\Imag$};
        \def\radA{1.020} % big radius
        \def\radB{0.447} % small radius

        % The contour gamma_{z}
        \draw[->,thick][black] (0.8,0) -- (0.9,0.1);
        \draw[-,thick][black] (0.9,0.1) -- (1,0.2);
        \draw[-,thick][black] (0.8,0) -- (0.9,-0.1);
        \draw[->,thick][black] (1,-0.2) -- (0.9,-0.1);
        \draw[->,thick][black] (1,0.2) arc (11.21:180:\radA);
        \draw[-,thick][black] (1,0.2) arc (11.21:360 - 11.21:\radA);

        % The contour gamma_{w}
        \draw[->,thick][black] (0.6,0) -- (0.5,0.1);
        \draw[-,thick][black] (0.5,0.1) -- (0.4,0.2);
        \draw[-,thick][black] (0.6,0) -- (0.5,-0.1);
        \draw[->,thick][black] (0.4,-0.2) -- (0.5,-0.1);        
        \draw[->,thick][black] (0.4,0.2) arc (26.57:180:\radB);
        \draw[-,thick][black] (0.4,0.2) arc (26.57:360 - 26.57:\radB);

        % Various points
        \draw[black, fill = black] (0.7,0) circle (0.02);
        \draw (0.7,-0.1) node{$1$}; 
        \draw[black, fill = black] (0.4,0) circle (0.02);
        \draw (0.38,-0.1) node{$\underline{y}$};
        \draw[black, fill = black] (0.3,0) circle (0.02);
        \draw (0.28,-0.1) node{$q$};
        \draw[black, fill = black] (1.2,0) circle (0.02);
        \draw (1.2,-0.1) node{$\underline{x}$}; 
        \draw[black, fill = black] (1.6,0) circle (0.02);
        \draw (1.6,-0.1) node{$q^{-1}$};
       
        % Contour names
        \draw (-1,1) node{$\gamma_{z}^{1}$};
        \draw[->, >=stealth'] ( - 0.9, 1)  to[bend left] ( - 0.70,0.75);
        \draw (1.3,0.3) node{$\gamma_{z}^{0}$};
        \draw[->, >=stealth'] ( 1.3, 0.2)  to[bend left] ( 0.95,0.15);
        \draw (-0.6,0.5) node{$\gamma_{w}^{1}$};
        \draw[->, >=stealth'] ( - 0.5, 0.5)  to[bend left] ( - 0.3,0.35);
        \draw (0.15,0.15) node{$\gamma_{w}^{0}$};
        \draw[->, >=stealth'] ( 0.25, 0.15)  to[bend right] ( 0.45,0.15);
     
        \draw (0.6,-1.1) node{$1 - \delta_1 N^{-1/3}$};
        \draw[-,very thin, dashed][gray] (0.6,-1) -- (0.6, 0);
        \draw (0.8,1.1) node{$1 + \delta_1 N^{-1/3}$};
        \draw[-,very thin, dashed][gray] (0.8,1) -- (0.8,0);

        \draw[-,very thin, dashed][gray] (0.4,0.2) -- (2.1,0.2);
        \draw[-,very thin, dashed][gray] (0.4,-0.2) -- (2.1,-0.2);
        \draw[->,very thin][black] (2.1,0) -- (2.1, -0.2);
        \draw[->,very thin][black] (2.1,0) -- (2.1, 0.2);
        \draw (2.4,0) node{$2N^{-1/12}$};

    \end{tikzpicture} 
    \caption{The figure depicts the contours $\gamma_z^i, \gamma_w^{i}$ for $i = 0,1$. We have $\gamma_z = \gamma_z^0 \cup \gamma_z^1$ and $\gamma_w = \gamma_w^0 \cup \gamma_w^1$.} 
    \label{S61}
\end{figure}

For $i,j \in \{0,1\}$ we define
\begin{equation}\label{sliceK2}
\begin{split}
&K^{i,j}_N(u,x_N; v, y_N) = \frac{1}{(2\pi \im)^2} \int_{\gamma^i_{z}}dz  \int_{\gamma^j_{w}}dw\frac{1}{z-w}\prod_{i = 1}^5 H_i^N(z,w) \cdot H^N_7(z,w) \cdot H^N_8(z,w) \\
& \exp \left(\log w \cdot \sigma_q \tilde{y}_N \cdot N^{1/3} -  \log z \cdot \sigma_q \tilde{x}_N \cdot N^{1/3}  \right) \cdot \exp \left(\tilde{N} \cdot [S(z) - S(w)] \right),
\end{split}
\end{equation}
where we recall that $H_8^N$ is as in (\ref{H8Def}), and $S(z)$ is as in (\ref{DefS}). From (\ref{S61SpecEq}) 
\begin{equation}\label{sliceK3}
K_N(u,x_N; v, y_N) - K^2_N(u,x_N; v, y_N)  = \sum_{i, j \in \{0,1\}} K^{i,j}_N(u,x_N; v, y_N).
\end{equation}
and so it suffices to analyze and bound the four kernels separately. \\

Before we proceed with the next step we isolate a quick estimate, which will be used in Section \ref{Section6.2}. Let us fix $N$ large enough so that $N \geq N_0$ and $N^{-1/12} \in (\delta_1 \cdot N^{-1/3}, 1/2]$. We claim that for any $t \in \mathbb{R}$ we can find $R(N,t), r(N,t) > 0$ that depend on $N$, $t$ and the parameters $q, t_1, \dots, t_m, c^+$ and sequences $\{a_i^+\}_{i \geq 1}$, $\{b_i^+\}_{i \geq 1}$ such that for $\tilde{x}_N, \tilde{y}_N \geq t$
\begin{equation}\label{S61SpecEq2}
\left|K_N(u,x_N; v, y_N) - K^2_N(u,x_N;v,y_N) \right| \leq R(N,t) \cdot \exp\left(-r(N,t) |\tilde{x}_N| - r(N,t) |\tilde{y}_N| \right).
\end{equation}
To see (\ref{S61SpecEq2}) we note that $H_i^N(z,w)$ for $i = 1, \dots, 7$ are bounded (they do not depend on $\tilde{x}_N, \tilde{y}_N$) as are the lengths of the contours $\gamma_z$ and $\gamma_w$. The latter, (\ref{ScaleXY}) and (\ref{S61SpecEq}) produce the bound
\begin{equation*}
\left|K_N(u,x_N; v, y_N) - K^2_N(u,x_N;v,y_N) \right| \leq \tilde{R}(N) \cdot \max_{z \in \gamma_z} |z|^{- \sigma_q N^{1/3} \tilde{x}_N} \cdot  \max_{w \in \gamma_w} |w|^{\sigma_q N^{1/3} y_N},
\end{equation*}
for some constant $\tilde{R}(N) > 0$ that depends on the parameters we listed before (\ref{S61SpecEq2}) except for $t$. The last equation gives (\ref{S61SpecEq2}) once we observe that $|w| \leq 1 - \delta_1 N^{-1/3}$ for $w \in \gamma_w$ and $|z| \geq 1 + \delta_1 N^{-1/3}$ for $z \in \gamma_z$.\\

{\bf \raggedleft Step 2.} In this step we estimate $K^{i,j}_N(u,x_N; v, y_N)$ when $i + j > 0$. Notice that if $z \in \gamma_z$ and $w\in \gamma_w$, we have $|z|, |w| \in [1/2,2]$. The latter follows from our assumption that $\delta_1 \in (0,1/2)$ in (\ref{S6Delta1}), and $N^{-1/12} \in (0,1/2]$. In addition, we have $(x_i^N)^{-1} \in [\underline{x}, q^{-1}]$ and $y_i^N \in [q , \underline{y}]$. In view of (\ref{S6Parbound}), we know that the segments $[\underline{x}, q^{-1}]$ and $[q , \underline{y}]$ are at least distance $(\delta_1/2)N^{-1/3}$ away from $\gamma_z \cup \gamma_w$. All of this implies that the conditions above (\ref{RatBound1}) are satisfied with $\tilde{\delta} = \delta_1/2$, which implies that $H^N_i(z,w)$ are products of $O(N^{1/12})$ terms, each being $O(N^{1/3})$ for $i = 1,2,5$. Since $a_1^- = b_1^- = c^- =0 $ by assumption we also have $H_3^N(z,w) = H_4^N(z,w) = 1$. The last few observations imply that we have the bound in (\ref{Tr1}) for some $c_0$ depending on $\delta_1$ and $q$ alone.

Directly from the definition of $H_8^N(z,w)$ in (\ref{H8Def}) we conclude the second line in (\ref{Tr2}) with a constant $c_2 > 0$, depending on $q$ and $t_1, \dots, t_m$. Since $\gamma_z$, $\gamma_w$ are bounded contours that are bounded away from $q^{-1}$, we also have the bound in (\ref{Tr3})
for some $c_3 > 0$ (depending on $q$ and $t_1, \dots, t_m$). In addition, we have that $\gamma_z$, $\gamma_w$ are at least distance $2\delta_1 N^{-1/3}$ away from each other, and so (\ref{Tr4}) holds with a constant $c_4$ that depends on $\delta_1$ alone. Arguing as in the derivation of (\ref{Tr5}), we also have for $(z,w) \in \gamma_z \times \gamma_w \setminus \gamma_z^0 \times \gamma_w^0$ that
\begin{equation}\label{S6Tr5}
\left|\exp \left(\tilde{N} \cdot [S(z) - S(w)] \right)\right| = \exp \left(\tilde{N}[\Real S(z)  - \Real S(w)] \right) \leq \exp \left(c_5 - (\sigma_q^3/24) \cdot  N^{3/4} \right),
\end{equation}
for some $c_5 > 0$ that depends on $q$. We finally investigate the terms
\begin{equation*}
\begin{split}
&\left| \exp \left(\log w \cdot \sigma_q \tilde{y}_N \cdot N^{1/3}  \right) \right| \mbox{ and } \left| \exp\left(-  \log z \cdot \sigma_q \tilde{x}_N \cdot N^{1/3}  \right)  \right|.
\end{split}
\end{equation*}

We observe that for $w \in \gamma_w$ we have $|w| \leq 1 - \delta_1 N^{-1/3}$ and for $z \in \gamma_z$ we have $|z| \geq 1 + \delta_1 N^{-1/3}$. The latter implies that we can find $\rho_1 > 0$, depending on $q$ and $\delta_1$, such that for all large $N$, $w \in \gamma_w$ and $z \in \gamma_z$ we have
$$N^{1/3} \cdot  \sigma_q \cdot \Real[\log w ] \leq - \rho_1 \mbox{ and } - N^{1/3} \cdot \sigma_q \cdot \Real[\log z ] \leq -\rho_1.$$
The latter gives the estimates
\begin{equation*}
    \begin{split}
    &\left| \exp \left(\log w \cdot \sigma_q \tilde{y}_N \cdot N^{1/3}  \right) \right| \leq \exp \left( - \rho_1 \tilde{y}_N \right) \mbox{ if } \tilde{y}_N \geq 0,\\
    & \left| \exp\left(-  \log z \cdot \sigma_q \tilde{x}_N \cdot N^{1/3}  \right)  \right|\leq \exp \left( - \rho_1 \tilde{x}_N \right) \mbox{ if } \tilde{x}_N \geq 0.
    \end{split}
\end{equation*}
On the other hand, we also have the trivial bounds
\begin{equation*}
    \begin{split}
    &\left| \exp \left(\log w \cdot \sigma_q \tilde{y}_N \cdot N^{1/3}  \right) \right| \leq \exp \left( c_1 N^{1/3} \right) \mbox{ if } \tilde{y}_N \in[-L,0],\\
    & \left| \exp\left(-  \log z \cdot \sigma_q \tilde{x}_N \cdot N^{1/3}  \right)  \right|\leq \exp \left( c_1 N^{1/3}\right) \mbox{ if } \tilde{x}_N \in[-L,0].
    \end{split},
\end{equation*}
where $c_1 > 0$ depends on $q$ and $L$ alone. Combining the last two inequalities we conclude
\begin{equation}\label{S6Tr6}
    \begin{split}
    &\left| \exp \left(\log w \cdot \sigma_q \tilde{y}_N \cdot N^{1/3}  \right) \right| \leq \exp \left( - \rho_1 |\tilde{y}_N| + \rho_1 L + c_1 N^{1/3}  \right) \mbox{ if } \tilde{y}_N \geq -L,\\
    & \left| \exp\left(-  \log z \cdot \sigma_q \tilde{x}_N \cdot N^{1/3}  \right)  \right|\leq \exp \left( - \rho_1 |\tilde{x}_N| + \rho_1 L + c_1 N^{1/3} \right) \mbox{ if } \tilde{x}_N \geq -L.
    \end{split}
\end{equation}
Combining (\ref{Tr1}), the second line in (\ref{Tr2}), (\ref{Tr3}), (\ref{Tr4}), (\ref{S6Tr5}), (\ref{S6Tr6}) and the boundedness of $\gamma_z$, $\gamma_w$ we conclude for $i + j > 0$
\begin{equation}\label{S6Tr7}
\left| K^{i,j}_N(u,x_N; v, y_N) \right| \leq \exp\left(CN^{2/3} -  (\sigma_q^3/24) \cdot  N^{3/4}  - \rho_1 |\tilde{x}_N| - \rho_1 |\tilde{y}_N| \right),
\end{equation}
for some sufficiently large constant $C$ (depending on $c_0-c_5, \rho_1, L, q ,\delta_1)$ and all large $N$.\\

{\bf \raggedleft Step 3.} In this step we estimate $K_N^{0,0}(u, x_N; v, y_N)$. We perform the change of variables $z = 1 + N^{-1/3} \tilde{z}$ and $w = 1 + N^{-1/3} \tilde{w}$ and get
\begin{equation}\label{S6SK1}
\begin{split}
& N^{1/3} \sigma_q K_N^{0,0}(u, x_N; v ,y_N) = \frac{\sigma_q}{(2\pi \im)^2} \int_{\Gamma_{\delta_1}^+}d\tilde{z}  \int_{\Gamma_{-\delta_1}^-}d \tilde{w} {\bf 1}\{ |\Imag ( \tilde{w})| \leq N^{1/4}, |\Imag (\tilde{z})| \leq N^{1/4}\} \\
& \times   \frac{1}{\tilde{z} - \tilde{w}}  \times \prod_{i = 1}^{8} \tilde{H}^N_i(\tilde{z}, \tilde{w}) ,
\end{split}
\end{equation}
where $\Gamma_{a}^{\pm}$ are as in Definition \ref{DefContInf}, and the functions $\tilde{H}^N_i(\tilde{z}, \tilde{w})$ are as in (\ref{NS2}-\ref{NS5}). We next proceed to discuss which bounds from Section \ref{Section5.4} go through for our new contours. 

By our choice of $\delta_1$ satisfying (\ref{S6Delta1}), we have for each $\tilde{w} \in \Gamma_{-\delta_1}^-$ with $|\Imag(\tilde{w})| \leq N^{1/4}$ that 
$$|1 + a_i^+ \sigma_q \tilde{w}| \geq 1/4 \mbox{ and } |\tilde{w}| \leq \sqrt{2} N^{1/4} + \delta_1 \leq 2N^{1/4}.$$ 
The latter satisfy the conditions prior to (\ref{H1B}) with $\tilde{\delta} = 1/4$ and so the bound for $\tilde{H}_1^N$ in (\ref{H1B}) holds. Similarly, for each $\tilde{z} \in \Gamma_{\delta_1}^+$ with $|\Imag(\tilde{z})| \leq N^{1/4}$ we have 
$$|1 - b_i^+ \sigma_q \tilde{z}| \geq 1/4 \mbox{ and } |\tilde{z}| \leq \sqrt{2} N^{1/4} + \delta_1 \leq 2N^{1/4}.$$
The latter satisfy the conditions prior to (\ref{H2B}) with $\tilde{\delta} = 1/4$ and so the bound for $\tilde{H}_2^N$ in (\ref{H2B}) holds. Since $a_1^- = b_1^- = c^- = 0$, we also have that 
\begin{equation}\label{S6H3B}
\tilde{H}_3^N(\tilde{z}, \tilde{w}) = \tilde{H}_4^N(\tilde{z}, \tilde{w}) = 1.
\end{equation}
Note that the conditions before (\ref{H5B}) are satisfied when $\tilde{z} \in \Gamma_{\delta_1}^+, |\Imag(\tilde{z})| \leq N^{1/4}$ and $\tilde{w} \in \Gamma_{-\delta_1}^-, |\Imag(\tilde{w})| \leq N^{1/4}$ with $\tilde{\delta} = 1/4$, and so the bound for $\tilde{H}_5^N$ in (\ref{H5B}) holds.

We next note that from (\ref{EqTayS}) the inequalities in (\ref{RealSB}) continue to hold for the contours $\gamma_z^0$, $\gamma_w^0$ in the present proposition, with a possibly different constant $c_5$. Using the latter, we see that the bound for $\tilde{H}_6^N$ in (\ref{H6B}) holds for $\tilde{z} \in \Gamma_{\delta_1}^+, |\Imag(\tilde{z})| \leq N^{1/4}$ and $\tilde{w} \in \Gamma_{-\delta_1}^-, |\Imag(\tilde{w})| \leq N^{1/4}$. Using that $|\tilde{z}| \leq 2N^{1/4}$ and $|\tilde{w}| \leq 2N^{1/4}$ we also conclude that the bound for $\tilde{H}_7^N$ in (\ref{H7B}) continues to hold. Since $\Gamma_{\delta_1}^+$ and $\Gamma_{-\delta_1}^-$ are at least $2\delta_1$ away from each other, we get the bound
\begin{equation}\label{S6RatB}
\left| \frac{1}{\tilde{z} - \tilde{w}} \right| \leq \frac{1}{2\delta_1}.
\end{equation}

We finally focus our attention on $\tilde{H}^N_8$, which is essentially the only term of the integrand in (\ref{S6SK1}) we need to treat differently than in Section \ref{Section5.4}. Using the bound in (\ref{TBO1}), which holds when $|\tilde{z}| \leq 2N^{1/4}$ and $|\tilde{w}| \leq 2N^{1/4}$, and (\ref{NS5}) we get for some large $A_8$ (depending on $t_1, \dots, t_m$ and $q$)
\begin{equation}\label{KNUT0}
    \begin{split}
\left| \tilde{H}^N_8(\tilde{z}, \tilde{w}) \right| \leq A_8 \cdot \left| \exp \left( \sigma_q \tilde{y}_N \cdot N^{1/3} \cdot \log (1 + N^{-1/3} \tilde{w}) -  \sigma_q \tilde{x}_N \cdot N^{1/3} \cdot \log (1 + N^{-1/3} \tilde{z}) \right) \right|.
    \end{split}
\end{equation}
Using that $|\log(1+z)| \leq 2|z|$ for all small $z$, we conclude
\begin{equation}\label{KNUT1}
    \begin{split}
    & \left| \exp \left( \sigma_q \tilde{y}_N \cdot N^{1/3} \cdot \log (1 + N^{-1/3} \tilde{w}) \right) \right| \leq \exp \left( a_8 |\tilde{w}| \right) \mbox{ if } \tilde{y}_N \in[-L,0],\\
    &\left| \exp \left(- \sigma_q \tilde{x}_N \cdot N^{1/3} \cdot \log (1 + N^{-1/3} \tilde{z}) \right) \right| \leq \exp \left( a_8 |\tilde{z}| \right) \mbox{ if } \tilde{x}_N \in[-L,0],
    \end{split}
\end{equation}
where the constant $a_8$ depends on $q$ and $L$ alone. 

Writing $\tilde{w} = -\delta_1 - x + \sigma x$ for $x \in [0, N^{1/4}]$ and $\sigma \in \{-1, 1\}$ we get 
\begin{equation}\label{KNUT2}
    \begin{split}
&2 \Real \log (1 + N^{-1/3} \tilde{w})  = \log \left( 1 - 2(\delta_1 + x)N^{-1/3} + (\delta_1^2 + 2x^2 + 2x \delta_1)N^{-2/3} \right)   \\
& \leq \log \left(1 - \delta_1 N^{-1/3} - x N^{-1/3} \right) \leq \log \left(1 - \delta_1 N^{-1/3} \right) \leq -\rho_2 N^{-1/3},
    \end{split}
\end{equation}
where the inequalities on the second line hold for all large $N$, and $\rho_2 > 0$ is a small positive constant depending on $\delta_1$ alone. One analogously shows that if $\tilde{z} = \delta_1 + x + \sigma x$ for $x \in [0, N^{1/4}]$ and $\sigma \in \{-1, 1\}$ we get for some $\delta_1$-dependent $\rho_3 > 0$ and all large $N$
$$2 \Real \log (1 + N^{-1/3} \tilde{z}) \geq \log \left(1 + \delta_1 N^{-1/3} \right) \geq \rho_3 N^{-1/3}.$$
The last two inequalities imply that for some $\rho'_1 > 0$, depending on $\delta_1$ and $q$, and all large $N$
\begin{equation}\label{KNUT3}
    \begin{split}
    & \left| \exp \left( \sigma_q \tilde{y}_N \cdot N^{1/3} \cdot \log (1 + N^{-1/3} \tilde{w}) \right) \right| \leq \exp \left( -\rho_1'  \tilde{y}_N \right) \mbox{ if } \tilde{y}_N \geq 0,\\
    &\left| \exp \left(- \sigma_q \tilde{x}_N \cdot N^{1/3} \cdot \log (1 + N^{-1/3} \tilde{z}) \right) \right| \leq \exp \left( -\rho_1'  \tilde{x}_N \right) \mbox{ if } \tilde{x}_N \geq 0,
    \end{split}
\end{equation}
Combining (\ref{KNUT1}) and (\ref{KNUT3}) we get for $\tilde{z} \in \Gamma_{\delta_1}^+, |\Imag(\tilde{z})| \leq N^{1/4}$ and $\tilde{w} \in \Gamma_{-\delta_1}^-, |\Imag(\tilde{w})| \leq N^{1/4}$
\begin{equation}\label{KNUT4}
    \begin{split}
    & \left| \exp \left( \sigma_q \tilde{y}_N \cdot N^{1/3} \cdot \log (1 + N^{-1/3} \tilde{w}) \right) \right| \leq \exp \left( a_8 |\tilde{w}| -\rho_1'  |\tilde{y}_N| + \rho_1' L \right) \mbox{ if } \tilde{y}_N \geq -L,\\
    &\left| \exp \left(- \sigma_q \tilde{x}_N \cdot N^{1/3} \cdot \log (1 + N^{-1/3} \tilde{z}) \right) \right| \leq \exp \left( a_8|\tilde{w}| -\rho_1'  |\tilde{x}_N|  + \rho_1' L\right) \mbox{ if } \tilde{x}_N \geq -L.
    \end{split}
\end{equation}
Combining (\ref{H1B}), (\ref{H2B}), (\ref{H5B}), (\ref{H6B}), (\ref{H7B}), (\ref{S6H3B}), (\ref{S6RatB}), (\ref{KNUT0}) and (\ref{KNUT4}) we conclude that for all large $N$, $\tilde{x}_N \geq - L$, $\tilde{y}_N \geq -L$
\begin{equation}\label{S6H9B}
\begin{split}
&\left| \frac{{\bf 1}\{ |\Imag ( \tilde{w})| \leq N^{1/4}, |\Imag (\tilde{z})| \leq N^{1/4}\}}{\tilde{z} - \tilde{w}}  \times \prod_{i = 1}^{8} \tilde{H}^N_i(\tilde{z}, \tilde{w}) \right| \leq \exp \left( -\rho_1'  |\tilde{x}_N| - \rho_1'  |\tilde{y}_N| \right) \\
&\times \exp \left(  D_0 + D_1 (|\tilde{z}| + |\tilde{w}|) + D_2 (|\tilde{z}|^2 + |\tilde{w}|^2) - (\sigma_q^3/24) |\tilde{z}|^3 - (\sigma_q^3/24)|\tilde{w}|^3  \right),
\end{split}
\end{equation}
where $D_0$, $D_1$, $D_2$ are constants that depend on the same parameters as in the statement of the proposition. Note that the function on the second line of (\ref{S6H9B}) is integrable over $(\tilde{z}, \tilde{w}) \in \Gamma_{\delta_1}^+ \times \Gamma_{-\delta_1}^-$ due to the cubic terms in the exponents. In particular, from (\ref{S6SK1}) and (\ref{S6H9B})  we conclude that there is a constant $R_1 > 0$ such that 
\begin{equation}\label{S6H10B}
\begin{split}
&\left| N^{1/3} \sigma_q K_N^{0,0}(u, x_N; v ,y_N) \right| \leq (R_1/2) \cdot \exp \left( -\rho_1'  |\tilde{x}_N| - \rho_1'  |\tilde{y}_N| \right).
\end{split}
\end{equation}
Equations (\ref{S6Tr7}) and (\ref{S6H10B}) together imply (\ref{EqUTE}) with $r_1 = \min(\rho_1, \rho_1')$ and the same choice of $R_1$ if $N$ is large enough.
\end{proof}

%%%%%%%%%%%%%%%%%%%%%%%%%%%%%%%%%%%%%%%%%%%%%%%%%%%%%%%%%%%%%%%%%%%%%
%
%    Section 6.2
%
%%%%%%%%%%%%%%%%%%%%%%%%%%%%%%%%%%%%%%%%%%%%%%%%%%%%%%%%%%%%%%%%%%%%%
\subsection{Probability estimate}\label{Section6.2} We continue with the same notation and assumptions as in Proposition \ref{KUTE}. The goal of this section is to establish the following result.

\begin{proposition}\label{PVC} There exist constants $N_2 \geq N_0$, $R_2 > 0$, $r_2 > 0$, depending on $q$, $c^+$, $\{t_1, \dots, t_m\}$ and the sequences $\{a_i^+\}_{i \geq 1}$, $\{b_i^+\}_{i \geq 1}$, such that for each $u \in \{1, \dots, m\}$, $N \geq N_2$ and $a \geq 0$
\begin{equation}\label{ePVC}
\mathbb{P} \left( \sigma_q^{-1} N^{-1/3} \cdot \left( \lambda_1^{M_u(N)} - \frac{2q \tilde{N}}{1-q} - \frac{qt_u N^{2/3}}{1-q} - 1 \right) 
 > a \right) \leq R_2 \cdot \exp \left( - r_2 a \right).
\end{equation}
\end{proposition}
\begin{proof} Let $M^u_N$ be the point process formed by $\left\{\lambda^{M_u(N)}_i - i\right\}_{i \geq 1}$. From Lemmas \ref{LemmaSlice} and \ref{PrelimitKernel} we conclude that $M_N^u$ is a determinantal point process with correlation kernel $K_N(u, \cdot; u, \cdot)$ and reference measure $\rho$ -- the counting measure on $\mathbb{Z}$. If we set 
\begin{equation}\label{S6PVC1}
X_i^N = \sigma_q^{-1} N^{-1/3} \cdot \left( \lambda_i^{M_u(N)} - \frac{2q \tilde{N} }{1-q} - \frac{qt_u N^{2/3}}{1-q} - i\right),
\end{equation}
then $\{X_i^N\}_{i \geq 1}$ forms a point process $M^{u,X}_N$ on $\mathbb{R}$ such that
$$M_N^{u,X} = M_N^u \phi_N^{-1}, \mbox{ where } \phi_N(y) = a_Ny + b_N$$
with
$$ a_N = \sigma_q^{-1} N^{-1/3} \mbox{ and } b_N = \sigma_q^{-1} N^{-1/3} \cdot \left( - \frac{2q \tilde{N}}{1-q} - \frac{qt_u N^{2/3}}{1-q} \right).$$
From Proposition \ref{PropLem}, parts (5) and (6), we conclude that $M_N^{u,X}$ is determinantal with correlation kernel $K^u_N$ 
\begin{equation}\label{S6PVC2}
K_N^u(\tilde{x}_N, \tilde{y}_N) = \sigma_q N^{1/3} \cdot K_N(u, x_N; u, y_N),
\end{equation}
and reference measure $\rho^{u}_N$, given by $\sigma_q^{-1} N^{-1/3}$ times the counting measure on $a_N \cdot \mathbb{Z} + b_N$.\\

Let $\delta_1 > 0$ be small enough, depending on $q, a_1^+, b_1^+, c^+$ so that (\ref{S6Delta1}) holds. Suppose further that $N$ is large so that $N \geq N_0$ and $ N^{-1/12} \in (\delta_1 \cdot N^{-1/3}, 1/2]$. From (\ref{S61SpecEq2}) and (\ref{S6PVC2}) we have for each $t \in \mathbb{R}$ that there are $R(N,t), r(N,t) > 0$ such that for $x, y \geq t$ 
$$\left| K_N^u(x,y) \right| \leq R(N,t) \cdot \exp \left( -r(N,t) |x| - r(N,t)|y| \right).$$
From Hadamard's inequality (\ref{Hadamard}), we conclude for $x_1, \dots, x_n \geq t$
\begin{equation*}
\left| \det\left[ K_N^u(x_i,x_j) \right]_{i,j = 1}^n \right| \leq R(N,t)^n \cdot n^{n/2} \cdot \prod_{i = 1}^n \exp \left( -r(N,t) |x_i| \right),
\end{equation*}
which implies 
\begin{equation*}
1 + \sum_{n \geq 1} \frac{1}{n!} \int_{(t, \infty)^n} \left| \det\left[ K_N^u(x_i,x_j) \right]_{i,j = 1}^n \right| (\rho_N^{u})^n(dx) < \infty.
\end{equation*}
We conclude that the conditions of Proposition \ref{PropLast} are satisfied and so for all $a \in \mathbb{R}$
\begin{equation}\label{S6PVC4}
\mathbb{P}(X_1^N > a) = 1 - \mathbb{P}(X_1^N \leq a) = \sum_{n = 1}^{\infty} \frac{(-1)^{n-1}}{n!}  \int_{(a, \infty)^n}\det\left[ K_N^u(x_i,x_j) \right]_{i,j = 1}^n(\rho_N^{u})^n(dx). 
\end{equation}

Let $N_1, R_1, r_1$ be as in Proposition \ref{KUTE} for $L = 0$. Using (\ref{EqUTE}), note that $K_N^2(u,x;u,y) = 0$, we have for $x, y \geq 0$ and all large $N$
\begin{equation}\label{S6PVC3}
\left| K_N^u(x,y) \right| \leq R_1 \cdot \exp \left( -r_1 |x| - r_1|y| \right).
\end{equation}
Using (\ref{S6PVC3}) and Hadamard's inequality (\ref{Hadamard}), we conclude for all $n \geq 1$, $x_i \geq 0$, and $N \geq N_1$
\begin{equation*}
\left| \det\left[ K_N^u(x_i,x_j) \right]_{i,j = 1}^n \right| \leq R_1^n \cdot n^{n/2} \cdot \prod_{i = 1}^n \exp \left( -r_1 x_i \right),
\end{equation*}
which implies for $a \geq 0$
\begin{equation*}
\begin{split}
&\int_{(a,\infty)^n} \left| \det\left[ K_N^u(x_i,x_j) \right]_{i,j = 1}^n \right|(\rho_N^{u})^n(dx) \leq R_1^n \cdot n^{n/2} \cdot  \int_{(a,\infty)^n} \prod_{i = 1}^n \exp \left( -r_1 x_i \right) (\rho_N^{u})^n(dx) \\
& \leq R_1^n \cdot n^{n/2} \cdot e^{n r_1 a_N} \cdot \left(\int_{(a,\infty)} e^{-r_1 x} dx \right)^n \leq C^n n^{n/2} e^{-n a r_1 }, 
\end{split}
\end{equation*}
where $dx$ is the Lebesgue measure, and $C$ is a constant that depends on $q, R_1,r_1$. Combining the last estimate with (\ref{S6PVC4}), we conclude for all $a \geq 0$ and large $N$ (depending on the parameters in the proposition)
$$\mathbb{P}(X_1^N > a) \leq \sum_{n =1}^{\infty} \frac{n^{n/2}}{n!} \cdot C^n e^{-na r_1/2} \leq e^{-a r_1} \sum_{n =1}^{\infty} \frac{n^{n/2}}{n!} \cdot C^n,  $$
which implies (\ref{ePVC}) with $r_2 = r_1$ and $R_2 = \sum_{n =1}^{\infty} \frac{n^{n/2}}{n!} \cdot C^n$.
\end{proof}

%%%%%%%%%%%%%%%%%%%%%%%%%%%%%%%%%%%%%%%%%%%%%%%%%%%%%%%%%%%%%%%%%%%%%
%
%    Section 7
%
%%%%%%%%%%%%%%%%%%%%%%%%%%%%%%%%%%%%%%%%%%%%%%%%%%%%%%%%%%%%%%%%%%%%%
\section{Finite-dimensional convergence}\label{Section7} The goal of this section is to establish Theorem \ref{T2}. In Section \ref{Section7.1} we re-introduce the random elements $(X_i^{j,N}: i \geq 1, j = 1, \dots, m)$ from (\ref{S4PVC1}) and show that as long as $a_1^- = b_1^- = c^- = 0$, the resulting sequence (in $N$) is tight -- the precise statement is Proposition \ref{S7MP}. In the same section we prove Theorem \ref{T2} by combining the tightness from Proposition \ref{S7MP}, the point process convergence established in Section \ref{Section4.3}, and Proposition \ref{PropWC2}. The proof of Proposition \ref{S7MP} is given in Section \ref{Section7.2} and relies on the upper tail estimate from Proposition \ref{PVC}, the tightness criterion of Proposition \ref{TightnessCrit} and the monotone coupling of Proposition \ref{MonCoup}.

%%%%%%%%%%%%%%%%%%%%%%%%%%%%%%%%%%%%%%%%%%%%%%%%%%%%%%%%%%%%%%%%%%%%%
%
%    Section 7.1
%
%%%%%%%%%%%%%%%%%%%%%%%%%%%%%%%%%%%%%%%%%%%%%%%%%%%%%%%%%%%%%%%%%%%%%
\subsection{Proof of Theorem \ref{T2}}\label{Section7.1} We start with a useful proposition, which will be used to establish Theorem \ref{T2}, and whose proof is given in Section \ref{Section7.2}.

\begin{proposition}\label{S7MP} Assume the same notation in Definitions \ref{DLP} and \ref{ParScale}, and suppose that $a_1^- = b_1^- = c^- = 0$. If $\lambda$ has law $\mathbb{P}_N$ as in Definitions \ref{ParScale}, we define as in (\ref{S4PVC1}) the variables
\begin{equation}\label{S7PVC1}
X_i^{j,N} = \sigma_q^{-1} N^{-1/3} \cdot \left( \lambda_i^{M_j(N)} - \frac{2q \tilde{N} }{1-q} - \frac{q t_j N^{2/3}}{1-q} - i\right) \mbox{ for $j \in \{1, \dots, m\}$ and $i \in \mathbb{N}$.}
\end{equation}
Then, for each $i \geq 1$ and $j \in \{1, \dots, m\}$ the sequence $\{X_i^{j,N}\}_{N \geq N_0}$ is tight.
\end{proposition}
\begin{proof}[Proof of Theorem \ref{T2}] For clarity we split the proof into two steps. In the first step we construct the sequence of processes $\{Y_i\}_{i \geq 1}$ under the assumption that $a_1^- = b_1^- = c^- = 0$, i.e. we assume that $J_a^- = J_b^- = 0$. The idea is to use Proposition \ref{S7MP} to verify the conditions of Proposition \ref{PropWC2} and conclude that $X_i^{j,N}$ converge in the finite-dimensional sense. By Kolmogorov's extension theorem these finite dimensional limits can be lifted to processes, which after some modification and variable change will yield $\{Y_i\}_{i \geq 1}$ satisfying the conditions of the proposition. In the second step we deal with the general case when $J_a^- + J_b^- < \infty$. It turns out that we can pick appropriate parameters $\{ \tilde{a}_i^+ \}_{i \geq 1}, \{ \tilde{b}_i^+\}_{i \geq 1}, \tilde{c}^+$ and $\tilde{a}_1^- = \tilde{b}_1^- = \tilde{c}^- = 0$, and an appropriate translation $\Delta \in \mathbb{R}$, such that if $\{\tilde{Y}_i\}_{i \geq 1}$ satisfy the conditions of the theorem for the ``tilde'' parameters, then $\{\tilde{Y}_i (\cdot + \Delta) \}_{i \geq 1}$ satisfy them for the original ones. In words, we will construct processes for the case $J_a^- + J_b^- < \infty$ by translating the ones for $J_a^- = J_b^- = 0$, which we constructed in Step 1.\\

{\bf \raggedleft Step 1.} In this step we prove the theorem under the additional assumption that $a_1^- = b_1^- = c^- = 0$. Let us fix $m \in \mathbb{N}$, $s_1, \dots, s_m \in \mathbb{R}$ with $s_1 < \dots < s_m$ and set $\mathcal{A} = \{s_1, \dots, s_m\}$. We assume the same notation as in Definition \ref{ParScale} for $t_i = f_q^{-1} s_i$, and let $X_i^{j,N}$ be as in (\ref{S7PVC1}) for this choice of parameters. As was shown in the proof of Theorem \ref{T1} in Section \ref{Section4.3}, we have that if $M_X^N$ is the random measure on $\mathbb{R}^2$ formed by $\{(s_j, X_i^{j,N}): 1\leq j \leq m, i \geq 1\}$, then $M_X^N$ converges weakly to a determinantal point process $\tilde{M}$ on $\mathbb{R}^2$ with correlation kernel 
$$\tilde{K}(s,x; t,y) = e^{   (x s +  2s^3/3) - (yt +  2t^3/3) }  K_{a,b,c}(t, y + t^2; s ,x + s^2)$$
and reference measure $\mu_{\mathcal{A}} \times \lambda$. From Proposition \ref{S7MP} we further know that $\{X_i^{j,N}\}_{N \geq N_0}$ is tight. 

The last few observations show that the sequence $(X_i^{j,N}: i \geq 1, j \in \{1, \dots, m\})$ satisfies the conditions of Proposition \ref{PropWC2}. Let $X$ be as in that proposition, and define the random element $X^{\mathcal{A}} = (X_i^{\mathcal{A}}(s_j): i \geq 1, j \in \{1, \dots, m\})$ of $(\mathbb{R}^{\infty},\mathcal{R}^{\infty})^{\mathcal{A}}$ (this is the space of real-valued functions from $\mathcal{A} = \{s_1, \dots, s_m\}$ into $\mathbb{R}^{\infty}$, which is homeomorphic to $\mathbb{R}^{\infty}$) via 
$$X_i^{\mathcal{A}}(s_j, \omega) := X_i^j(\omega) \mbox{ for } i \geq 1, j \in \{1, \dots, m\}.$$
From Proposition \ref{PropWC2} we have that $(\{X_i^{1,N}\}_{i \geq 1}, \dots, \{X_i^{m,N}\}_{i \geq 1})$ converge in the finite-dimensional sense to $(\{X_i^{\mathcal{A}}(s_1)\}_{i \geq 1}, \dots, \{X_i^{\mathcal{A}}(s_m)\}_{i \geq 1})$. In addition, we have that 
$$X_i^{\mathcal{A}}(s_j, \omega) \geq X_{i + 1}^{\mathcal{A}}(s_j, \omega) \mbox{ for } i \geq 1, j \in \{1, \dots, m\},$$
and also the random measure 
$$M^{\mathcal{A}}(\omega, A) = \sum_{i \geq 1} \sum_{j = 1}^m {\bf 1} \{ (s_j, X_i^{\mathcal{A}}(\omega, s_j)) \in A \}$$
has the same distribution as $\tilde{M}$ as random elements in $(S, \mathcal{S})$ (as in Section \ref{Section2.1} for $k = 2$). 

As mentioned earlier, $\tilde{M}$ is a determinantal point process, and then so is $M^{\mathcal{A}}$, which from part (1) of Proposition \ref{PropLem} implies that $M^{\mathcal{A}}$ is a simple point process almost surely. We conclude that there is an event $E$ such that $\mathbb{P}(E) = 1$ and for $\omega \in E$ we have 
\begin{equation}\label{S7PVC2}
X_i^{\mathcal{A}}(s_j, \omega) > X_{i + 1}^{\mathcal{A}}(s_j, \omega) \mbox{ for } i \geq 1, j \in \{1, \dots, m\}.
\end{equation}
We now define the random elements $\tilde{X}^{\mathcal{A}} = (\tilde{X}_i^{\mathcal{A}}(s_j): i \geq 1, j \in \{1, \dots, m\})$ of $(\mathbb{R}^{\infty},\mathcal{R}^{\infty})^{\mathcal{A}}$ via
\begin{equation}\label{S7PVC3}
\tilde{X}_i^{\mathcal{A}}(s_j, \omega) = \begin{cases} \tilde{X}_i^{\mathcal{A}}(s_j, \omega) &\mbox{ if } \omega \in E \\ -i &\mbox{ if } \omega \in E^c . \end{cases}
\end{equation}
By construction we have that $\tilde{X}^{\mathcal{A}}$ is a random element in $(G,\mathcal{G})^{\mathcal{A}}$, where $G = \{(x_1, x_2, \dots) \in \mathbb{R}^{\infty}: x_1 > x_2 > \cdots \}$ and $\mathcal{G}$ is its Borel $\sigma$-algebra inherited from $\mathbb{R}^{\infty}$. Since $\mathbb{P}(E) = 1$, we also have that $\tilde{X}^{\mathcal{A}}$ is a modification of $X^{\mathcal{A}}$. We observe that the laws of $X^{\mathcal{A}}$ over varying sets $\mathcal{A}$ are consistent, as these elements are weak limits of consistent random elements. We conclude that the laws of $\tilde{X}^{\mathcal{A}}$ over varying sets $\mathcal{A}$ are consistent, which by Kolmogorov's extension theorem, see \cite[Theorem IV.4.18]{Cinlar}, implies that there is a probability space $(\Omega, \mathcal{F}, \mathbb{P})$ that carries a process $(\tilde{X}(t): t\in \mathbb{R})$ with $\tilde{X}(t,\omega) = \{\tilde{X}_i(t,\omega)\}_{i \geq 1} \in (G,\mathcal{G})$ such that for any real $s_1 < \cdots < s_m$ and $\mathcal{A} = \{s_1, \dots, s_m\}$ the random element $(\tilde{X}(s_1), \dots, \tilde{X}(s_m))$ has the same law as $(\tilde{X}^{\mathcal{A}}(s_1), \dots, \tilde{X}^{\mathcal{A}}(s_m))$. We mention that when applying \cite[Theorem IV.4.18]{Cinlar} we used that $(G, \mathcal{G})$ is a standard Borel space, which follows from the fact that $G$ is a $G_{\delta}$ set of the Polish space $(\mathbb{R}^{\infty}, \mathcal{R}^{\infty})$, see \cite[Theorem 2.1.22]{Durrett}.\\

We finally define the sequence of processes $\{Y_i\}_{i \geq 1}$ on $(\Omega, \mathcal{F}, \mathbb{P})$ via
\begin{equation}\label{S7PVC4}
Y_i(t, \omega) = \tilde{X}_i(t, \omega) + t^2.
\end{equation}
Since $\tilde{X}(t,\omega) \in G$ for each $t \in \mathbb{R}$ and $\omega \in \Omega$, we conclude that (\ref{T2E2}) is satisfied. In addition, we have that if $M$ is as in (\ref{T2E1}), then $M = \tilde{M}^{\mathcal{A}} \phi^{-1}$, where $\phi(s,x) = (s, x + s^2)$ and 
$$\tilde{M}^{\mathcal{A}}(\omega, A) = \sum_{i \geq 1} \sum_{j = 1}^m {\bf 1} \{ (s_j, \tilde{X}_i(\omega, s_j)) \in A \}.$$
Since $(\tilde{X}(s_1), \dots, \tilde{X}(s_m))$ has the same law as $(X^{\mathcal{A}}(s_1), \dots, X^{\mathcal{A}}(s_m))$, we conclude from Corollary \ref{CorWC2} that $\tilde{M}^{\mathcal{A}}$ has the same law as $\tilde{M}$. In particular, $\tilde{M}^{\mathcal{A}}$ is also a determinantal point process with kernel $\tilde{K}(s,x;t,y)$. From parts (4) and (5) of Proposition \ref{PropLem} we see that $M$ is a determinantal point process on $\mathbb{R}^2$ with correlation kernel $K_{a,b,c}(s, x; t ,y)$ and reference measure $\mu_{\mathcal{A}} \times \lambda$, which concludes the proof of the theorem in this case.\\

{\bf \raggedleft Step 2.} In this step we handle the general case of parameters as in Definition \ref{DLP} with $c^- = 0$ and $J_a^- + J_b^- < \infty$. Our first task is to construct a new set of parameters $\{\tilde{a}^{\pm}_i\}_{i \geq 1}$, $\{\tilde{b}^{\pm}_i\}_{i \geq 1}$, $\tilde{c}^{\pm}$ that satisfy the conditions Definition \ref{DLP} and such that $\tilde{a}_1^- = \tilde{b}_1^- = \tilde{c}^- = 0$ (i.e. all ``minus'' parameters are equal to zero).

Let $\{\hat{a}^+_i\}_{i \geq 1}$ be the decreasing sequence formed by the terms in $\{a_i^+\}_{i \geq 1}$ (counted with multiplicities) as well as the terms $1/a_1^-, \dots, 1/ a_{J_a^-}^-$. Similarly, we let $\{\hat{b}^+_i\}_{i \geq 1}$ be the decreasing sequence formed by the terms in $\{b_i^+\}_{i \geq 1}$ (counted with multiplicities) as well as the terms $1/b_1^-, \dots, 1/ b_{J_b^-}^-$. Finally, we set $\hat{c}^+ = c$, $\hat{c}^- = 0$ and $\hat{a}_i^- = \hat{b}_i^- = 0$ for all $i \geq 1$. Observe that 
\begin{equation*}
\sum_{i \geq 1} \hat{a}_i^+ = \sum_{i \geq 1} a_i^+ + \sum_{j = 1}^{J_a^-} 1/ a_j^- < \infty, \mbox{ and } \sum_{i \geq 1} \hat{b}_i^+ = \sum_{i \geq 1} b_i^+ + \sum_{j = 1}^{J_b^-} 1/ b_j^- < \infty,
\end{equation*}
since $J_a^-$ and $J_b^-$ are both finite by assumption. In particular, the parameters $\{\hat{a}^{\pm}_i\}_{i \geq 1}$, $\{\hat{b}^{\pm}_i\}_{i \geq 1}$, $\hat{c}^{\pm}$ satisfy the conditions of Definition \ref{DLP}. Using the identities
$$1 + b/z = b/z \cdot (1 + z/b) \mbox{ and } 1- a/z = a/z \cdot (1 - z/a), $$
we also see that 
\begin{equation}\label{S7Q2}
\frac{\Phi_{a,b,c}(z)}{\Phi_{a,b,c}(w )} = \left(\frac{z}{w} \right)^{J_a^- - J_b^-} \cdot \frac{\Phi_{\hat{a},\hat{b},\hat{c}}(z)}{\Phi_{\hat{a},\hat{b},\hat{c}}(w)}.
\end{equation}

We next fix $\Delta \in \mathbb{R}$ as follows. If $J_a^- = J_b^-$, we set $\Delta = 0$. If $J_a^- > J_b^-$, we let $\Delta > 0$ be close enough to zero so that $1 - \Delta a_1^+ > 0$. If $J_a^- < J_b^-$, we let $\Delta < 0$ be close enough to zero so that $1 + \Delta \hat{b}_1^+ >0$. We now define the parameters $\{\tilde{a}^{+}_i\}_{i \geq 1}$, $\{\tilde{b}^{+}_i\}_{i \geq 1}$, $\tilde{c}^{+}$ using $\{\hat{a}_i^+\}_{i \geq 1}$, $\{\hat{b}_i^+\}_{i \geq 1}$, $\hat{c}^+$ and $\Delta$. We set $\tilde{c}^+ = \hat{c}^+$. If $J_a^- = J_b^-$ we let $\tilde{a}_i^+ = \hat{a}_i^+$ and $\tilde{b}_i^+ = \hat{b}_i^+$ for $i \geq 1$. If $\tilde{J}_a^- > \tilde{J}_b^-$, we set $h = \tilde{J}_a^- - \tilde{J}_b^-$ and let $\{\tilde{b}_i^+\}_{i \geq 1}$ be the decreasing sequence formed by the terms in $\{ \hat{b}_i^+/ (1 + \Delta \hat{b}_i^+) \}_{i \geq 1}$ (counted with multiplicities) and $h$ copies of $\Delta^{-1}$, while $\tilde{a}^+_i = \hat{a}_i^+/ (1 - \Delta \hat{a}_i^+)$ for $i \geq 1$. If $\tilde{J}_b^- > \tilde{J}_a^-$, we set $h = \tilde{J}_b^- - \tilde{J}_a^-$ and let $\{\tilde{a}_i^+\}_{i \geq 1}$ be the decreasing sequence formed by the terms in $\{ \hat{a}_i^+/ (1 - \Delta \hat{a}_i^+) \}_{i \geq 1}$ (counted with multiplicities) and $h$ copies of $-\Delta^{-1}$, while $\tilde{b}^+_i = \hat{b}_i^+/ (1 + \Delta \hat{b}_i^+)$ for $i \geq 1$. We readily observe that $\{\tilde{a}^{\pm}_i\}_{i \geq 1}$, $\{\tilde{b}^{\pm}_i\}_{i \geq 1}$, $\tilde{c}^{\pm}$ satisfy the conditions Definition \ref{DLP}, since
\begin{equation*}
\sum_{i \geq 1} \tilde{a}_i^+ + \sum_{i \geq 1} \tilde{b}_i^+ = \sum_{i \geq 1} \frac{\hat{a}_i^+}{1 - \Delta \hat{a}_i^+} + \sum_{i \geq 1} \frac{\hat{b}_i^+}{1 + \Delta \hat{b}_i^+} + {\bf 1} \{J_a^- \neq J_b^-\} \cdot |J_a^- - J_b^-| \cdot |\Delta^{-1}| < \infty.
\end{equation*}
We recall that ``minus'' parameters are equal to zero.\\

The above work finishes our construction of the parameters $\{\tilde{a}^{\pm}_i\}_{i \geq 1}$, $\{\tilde{b}^{\pm}_i\}_{i \geq 1}$, $\tilde{c}^{\pm}$. We next observe using the identities
$$\frac{1 + b(z+s + \Delta )}{1 + b(w+t + \Delta)} = \frac{1 + \frac{b}{1 +b \Delta} (z+s)}{1 + \frac{b}{1 +b \Delta} (w+t)} \mbox{, }\frac{1 - a(w+t + \Delta )}{1 - a(z+s + \Delta)} = \frac{1 - \frac{a}{1 -a \Delta} (w+t)}{1 - \frac{a}{1 -a \Delta} (z+s)}, \mbox{ and } $$
$$\left(\frac{z + s + \Delta }{w + t + \Delta} \right)^{J_a^- - J_b^-} = \begin{cases} 1 &\mbox{ if } J_a^- = J_b^-, \\ \left( \frac{1 + \Delta^{-1} (z + s) }{1 + \Delta^{-1} (w + t)} \right)^{J_a^- - J_b^-}  &\mbox{ if } J_a^- > J_b^-, \\ \left( \frac{1 - [-\Delta^{-1}] (w + t) }{1 -[-\Delta^{-1}] (z + s)} \right)^{J_b^- - J_a^-} &\mbox{ if } J_{a}^- < J_b^{-}, \end{cases}$$
that we have the following equality
\begin{equation}\label{S7Q3}
 \left(\frac{z + s + \Delta}{w + t + \Delta} \right)^{J_a^- - J_b^-} \cdot \frac{\Phi_{\hat{a},\hat{b},\hat{c}}(z + s + \Delta)}{\Phi_{\hat{a},\hat{b},\hat{c}}(w + t + \Delta)} = \frac{\Phi_{\tilde{a},\tilde{b},\tilde{c}}(z + s)}{\Phi_{\tilde{a},\tilde{b},\tilde{c}}(w + t )}.
\end{equation}

Let $\{\tilde{Y}_i \}_{i \geq 1}$ be as in the statement of the theorem for the parameters $\{\tilde{a}^{\pm}_i\}_{i \geq 1}$, $\{\tilde{b}^{\pm}_i\}_{i \geq 1}$, $\tilde{c}^{\pm}$. Since by construction $\tilde{a}_1^- = \tilde{b}_1^- = \tilde{c}^- = 0$, we see that the existence of such processes is ensured by our work in Step 1. We define the processes 
\begin{equation}\label{S7Q4}
Y_i(t,\omega) = \tilde{Y}_i(t - \Delta, \omega),
\end{equation}
and claim that they satisfy the conditions of the theorem. Since $\{ \tilde{Y}_i \}_{i \geq 1}$ satisfy (\ref{T2E2}), we conclude the same for $\{Y_i\}_{i \geq 1}$. Let us now fix $m \in \mathbb{N}$, and $s_1, \dots, s_m \in \mathbb{R}$ with $s_1 < s_2 < \cdots < s_m$. If $M$ is as in (\ref{T2E1}) and 
$$\tilde{M}(\omega, A) = \sum_{i \geq 1} \sum_{j = 1}^m {\bf 1}\{ (s_j - \Delta, \tilde{Y}_i(s_j - \Delta, \omega)) \in A\},$$
we see from (\ref{S7Q4}) that $M = \tilde{M} \phi^{-1}$, where $\phi(s,x) = (s + \Delta, x)$. From our work in Step 1, we know that $\tilde{M}$ is a determinantal point process with correlation kernel $K_{\tilde{a}, \tilde{b}, \tilde{c}}(s,x; t, y)$ and reference measure $\mu_{\mathcal{B}} \times \lambda$, where $\mathcal{B} = \{s_1 - \Delta, \dots, s_m - \Delta\}$. From part (5) of Proposition \ref{PropLem} we conclude that $M$ is a determinantal point process with correlation kernel 
$$
\tilde{K}(s,x; t,y) = K_{\tilde{a}, \tilde{b}, \tilde{c}}(s - \Delta,x; t - \Delta, y)$$
and reference measure $\mu_{\mathcal{A}} \times \lambda$. In order to conclude the proof it remains to check that 
\begin{equation}\label{S7Q5}
K_{\tilde{a}, \tilde{b}, \tilde{c}}(s - \Delta,x; t - \Delta, y) = K_{a,b,c}(s, x; t ,y).
\end{equation}

We fix $\alpha, \beta \in \mathbb{R}$ such that $\alpha + s - \Delta < 0$, $\alpha + s < 0$, $\beta + t - \Delta > 0$ and $\beta + t > 0$. Using (\ref{3BPKer}), we see that to prove (\ref{S7Q5}) it suffices to show that
\begin{equation}\label{S7Q6}
K^i_{\tilde{a}, \tilde{b}, \tilde{c}}(s - \Delta,x; t - \Delta, y) = K^i_{a,b,c}(s, x; t ,y) \mbox{ for } i = 1,2,3,
\end{equation}
where we have taken the above $\alpha, \beta$ in (\ref{3BPKer}) for both sides of (\ref{S7Q6}). The equality in (\ref{S7Q5}) for $i = 2$ is immediate from (\ref{3BPKer}). We also have from (\ref{3BPKer})  
\begin{equation*}
    \begin{split}
        &K^3_{\tilde{a}, \tilde{b}, \tilde{c}}(s - \Delta,x; t - \Delta, y) = \frac{1}{(2\pi \im)^2} \int_{\Gamma_{\alpha }^+} d z \int_{\Gamma_{\beta}^-} dw \frac{e^{z^3/3 -x_1z - w^3/3 + x_2w}}{z + s - w - t} \cdot \frac{\Phi_{\tilde{a},\tilde{b},\tilde{c}}(z + s - \Delta) }{\Phi_{\tilde{a},\tilde{b},\tilde{c}} (w + t - \Delta)} \\
        & = \frac{1}{(2\pi \im)^2} \int_{\Gamma_{\alpha }^+} d z \int_{\Gamma_{\beta}^-} dw \frac{e^{z^3/3 -x_1z - w^3/3 + x_2w}}{z + s - w - t} \cdot \frac{\Phi_{a,b,c}(z + s) }{\Phi_{a,b,c} (w + t )} = K^3_{\tilde{a}, \tilde{b}, \tilde{c}}(s,x; t, y),
    \end{split}
\end{equation*}
where the middle equality used (\ref{S7Q2}) and (\ref{S7Q3}). This proves (\ref{S7Q6}) when $i = 3$. 

Finally, if we set $u_{\pm}$ to be the two intersection points of $\Gamma^+_{\alpha + s - \Delta}$ and $\Gamma^-_{\beta + t - \Delta}$, such that $\Imag(u_+) > 0$, and $\Imag(u_-) < 0$, we have that $\Gamma^+_{\alpha + s}$ and $\Gamma^-_{\beta + t }$ intersect at $u_+ + \Delta$ and $u_- + \Delta$. From (\ref{3BPKer}) we get  
\begin{equation*}
    \begin{split}
        &K^2_{\tilde{a}, \tilde{b}, \tilde{c}}(s - \Delta,x; t - \Delta, y) \\
        &=  \frac{1}{2\pi \im} \int_{u_-}^{u_+}d\tilde{w} \cdot  e^{(t - s)\tilde{w}^2 + ((s- \Delta)^2 - (t- \Delta)^2) \tilde{w} + \tilde{w} (y-x) + x (s- \Delta) - y (t- \Delta) - (s- \Delta)^3/3 + (t- \Delta)^3/3} \\
        & =  \frac{1}{2\pi \im} \int_{u_- + \Delta}^{u_+ + \Delta}dw \cdot  e^{(t - s)w^2 + (s^2 - t^2) w + w (y-x) + x s - y t - s^3/3 + t^3/3} = K^2_{a,b,c}(s, x; t ,y),
    \end{split}
\end{equation*}
where the middle equality is verified by changing variables $\tilde{w} = w - \Delta$. We conclude that (\ref{S7Q6}) holds for $i = 2$ as well, which concludes the proof of that equation and hence the theorem.
\end{proof}

%%%%%%%%%%%%%%%%%%%%%%%%%%%%%%%%%%%%%%%%%%%%%%%%%%%%%%%%%%%%%%%%%%%%%
%
%    Section 7.2
%
%%%%%%%%%%%%%%%%%%%%%%%%%%%%%%%%%%%%%%%%%%%%%%%%%%%%%%%%%%%%%%%%%%%%%
\subsection{Establishing tightness}\label{Section7.2} In this section we present the proof of Proposition \ref{S7MP}. For clarity we split the proof into two steps. In the first step we prove the proposition under the assumption that $a_1^+ = b_1^+ = a_1^- = b_1^- = c^+ = c^- = 0$ in Definition \ref{DLP} by verifying the conditions of Proposition \ref{TightnessCrit}. In the second step we prove the full proposition by using the result from Step 1, and the monotone coupling in Proposition \ref{MonCoup}.\\

{\bf \raggedleft Step 1.} In this step we assume that $a_1^+ = b_1^+ = a_1^- = b_1^- = c^+ = c^- = 0$. Fix $j \in \{1, \dots, m\}$. As shown in Step 1 of the proof of Proposition \ref{PVC} we have that for all large $N$ the sequence $\{X_i^{j,N}\}_{i \geq 1}$ forms a determinantal point process $M^{j,X}_N$ on $\mathbb{R}$ with correlation kernel 
\begin{equation*}
K_N^j(\tilde{x}_N, \tilde{y}_N) = \sigma_q N^{1/3} \cdot K_N(j, x_N; j, y_N),
\end{equation*}
where $x_N, y_N$ are as in (\ref{ScaleXY}), and reference measure $\rho^j_N$ given by $\sigma_q^{-1} N^{-1/3}$ times the counting measure on $a_N \cdot \mathbb{Z} + b_N$ with
$$ a_N = \sigma_q^{-1} N^{-1/3} \mbox{ and } b_N = \sigma_q^{-1} N^{-1/3} \cdot \left( - \frac{2q \tilde{N}}{1-q} - \frac{qt_j N^{2/3}}{1-q} \right).$$
We also recall that $K_N(i,x;j,y)$ is as in (\ref{CKDefAlt}) with $\delta = N^{-1/12} \in (0, 1/2]$.

From (\ref{S4ScaledKer1}) and (\ref{S4VL1}) we conclude that $K_N^j(x, y)$ converges uniformly over compact sets in $\mathbb{R}^2$ to $e^{f_q t_j (x-y)} K_{a,b,c}(f_q t_j, y + f_q^2 t_j^2; f_qt_j, x + f_q^2 t_j^2)$. Using the continuity of $K_{a,b,c}(t, x; t, y)$ in $(x,y) \in \mathbb{R}^2$ from Lemma \ref{WellDefKer}, and the fact that $\rho^j_N$ converges vaguely to the Lebesgue measure on $\mathbb{R}$, we see that the conditions of Proposition \ref{PropWC0} are satisfied. We conclude that there exists a determinantal point process $M^j$ on $\mathbb{R}$ with correlation kernel
$$K^j(x,y) = e^{f_q t_j (x-y)} K_{a,b,c}(f_q t_j, y + f_q^2 t_j^2; f_qt_j, x + f_q^2 t_j^2),$$
and reference measure given by the Lebesgue measure on $\mathbb{R}$, and moreover $M^{j,X}_N$ converge weakly to $M^j$ as random elements in $(S,\mathcal{S})$. 

The above work shows that condition (1) in Proposition \ref{TightnessCrit} is satisfied by $(X_i^{j,N}: i \geq 1)$. In addition condition (3) is satisfied in view of Proposition \ref{PVC}. In order to conclude the tightness of $X^{j,N}_i$ we thus only need to verify the second condition in Proposition \ref{TightnessCrit}, i.e.
\begin{equation}\label{S72E1}
\mathbb{P}(M^j(\mathbb{R}) = \infty) = 1.
\end{equation}
From parts (5) and (6) in Proposition \ref{PropLem}, we see that $\tilde{M}^j = M^j \phi^{-1}$ with $\phi(x) = x - f_q^2 t_j^2$ is a determinantal point process on $\mathbb{R}$ with kernel 
$$\tilde{K}^j(x,y) = K_{a,b,c}(f_q t_j, y ; f_qt_j, x ),$$
which we recognize from (\ref{3BPKer}) as nothing but the Airy kernel. I.e. $\tilde{M}^j$ is the Airy point process $\mathcal{A}$ on $\mathbb{R}$. We would obtain (\ref{S72E1}) if we can show that 
\begin{equation}\label{S72E2}
\mathbb{P}(\mathcal{A}(\mathbb{R}) = \infty) = 1.
\end{equation}
It appears that (\ref{S72E2}) is well-known in the literature, but as we could not find a proof of it we include a short argument below based on Tracy and Widom's original paper \cite{TW94}.\\

For $s \in \mathbb{R}$ and $n \in \mathbb{Z}_{\geq 0}$ we define
$$E(n;s) = \mathbb{P}(\mathcal{A}(s,\infty) = n) \mbox{ and } r(n; s) = \frac{E(n;s)}{E(0;s)}.$$
Fix $k \in \mathbb{N}$ and note that for each $t > 0$
$$\mathbb{P}(\mathcal{A}(\mathbb{R}) \leq k) \leq \mathbb{P}(\mathcal{A}(-2t, \infty) \leq k) = \sum_{n = 0}^k E(n;-2t) = E(0;-2t) \cdot \sum_{n = 0}^k r(n;-2t). $$
Using \cite[(1.19) and (1.24)]{TW94} we conclude that we can find constants $C$ and $t_0 > 0$, depending on $k$, such that for all $t \geq t_0$ 
$$\mathbb{P}(\mathcal{A}(\mathbb{R}) \leq k) \leq C \cdot \exp \left( \frac{8 k t^{3/2}}{3} - \frac{2t^3}{3}\right).$$
Letting $t \rightarrow \infty$ establishes (\ref{S72E2}). \\

{\bf \raggedleft Step 2.} We suppose that $\lambda$ has law $\mathbb{P}_N$ for $N \geq N_0$ as in Definition \ref{ParScale} for the parameters as in the statement of the proposition and $X_i^{j,N}$ are as in (\ref{S7PVC1}). Since $X_1^{j,N} \geq X_u^{j,N}$ almost surely, we see 
$$\mathbb{P}( X_u^{j,N} \geq a)  \leq \mathbb{P}( X_1^{j,N} \geq a),$$
and so from Proposition \ref{PVC} we conclude for each $u \geq 1$ and $j \in \{1, \dots, m\}$
\begin{equation}\label{S72E3}
\begin{split}
&\lim_{a \rightarrow \infty} \limsup_{N \rightarrow \infty} \mathbb{P}( X_u^{j,N} \geq a) = 0.
\end{split}
\end{equation}

Suppose now that $N_0^0$ is as in Definition \ref{ParScale} for the same $t_1, \dots, t_m$ but with $a_1^+ = b_1^+ = a_1^- = b_1^- = c^+ = c^- = 0$. Note that in this case we have $A_N = B_N = C^+_N = C^-_N = D_N = 0$, $x_i^N = y_j^N = q$, and so in particular $\tilde{N} = N$. If $\tilde{\mathbb{P}}_N$ is the law in Definition \ref{ParScale} for this choice of parameters, we let $\tilde{\lambda}$ have law $\tilde{\mathbb{P}}_N$ for $N \geq N_0^0$. We also define for $j \in \{1, \dots, m\}$ and $i \in \mathbb{N}$
\begin{equation}\label{S72E4}
\begin{split}
\tilde{X}_i^{j,N} &= \sigma_q^{-1} N^{-1/3} \cdot \left( \tilde{\lambda}_i^{M_j(N)} - \frac{2q N }{1-q} - \frac{q t_j N^{2/3}}{1-q} - i\right) \\
& = \sigma_q^{-1} N^{-1/3} \cdot \left( \tilde{\lambda}_i^{M_j(N)} - \frac{2q \tilde{N} }{1-q} - \frac{q t_j N^{2/3}}{1-q} - i\right) - \frac{2q (N - \tilde{N})}{(1-q) \sigma_q N^{1/3}}.
\end{split}
\end{equation}

Notice that by construction the $x_i^N$, $y_j^N$ parameters in Definition \ref{ParScale} are at least $q$ for all $N \geq N_0$ (this is part of the definition of $N_0$). We conclude from Proposition \ref{MonCoup} that for all $N \geq \max(N_0, N_0^0)$, we can couple $\lambda$ and $\tilde{\lambda}$ on the same probability space so that almost surely
\begin{equation}\label{S72E5}
\sum_{i = 1}^k \lambda_i^{M_j(N)} \geq \sum_{i = 1}^k \tilde{\lambda}_i^{M_j(N)} \mbox{ for each } k \in \mathbb{N}, j \in \{1, \dots, m\}. 
\end{equation}
Since $N \geq \tilde{N}$, we conclude from (\ref{S7PVC1}), (\ref{S72E4}) and (\ref{S72E5}) that for $N \geq \max(N_0, N_0^0)$ we can couple $X_i^{j,N}$ and $\tilde{X}_i^{j,N}$ on the same probability space so that almost surely
\begin{equation}\label{S72E55}
\sum_{i = 1}^k X_i^{j,N} \geq \sum_{i = 1}^k \tilde{X}_i^{j,N} \mbox{ for each } k \in \mathbb{N}, j \in \{1, \dots, m\}. 
\end{equation}

Using that almost surely $X_1^{j,N} \geq X_2^{j,N} \geq \cdots$ and (\ref{S72E3}) we have for each $k \in \mathbb{N}$
$$\limsup_{a \rightarrow \infty} \limsup_{N \rightarrow \infty} \mathbb{P}\left( \sum_{i = 1}^k X_i^{j,N} \geq a \right) \leq \limsup_{a \rightarrow \infty} \limsup_{N \rightarrow \infty} \mathbb{P}(  X_1^{j,N} \geq a/k ) = 0.$$
On the other hand, from Step 1, we know that $\tilde{X}_i^{j,N}$ is a tight sequence (in $N$) for each fixed $i \geq 1$ and $j \in \{1, \dots, m\}$. The latter implies that $\sum_{i = 1}^k \tilde{X}_i^{j,N}$ is tight (in $N$) for each fixed $k \geq 1$ and $j \in \{1, \dots, m\}$. The latter and (\ref{S72E5}) shows
$$\limsup_{a \rightarrow \infty} \limsup_{N \rightarrow \infty} \mathbb{P}\left( \sum_{i = 1}^k X_i^{j,N} \leq -a \right) \leq \limsup_{a \rightarrow \infty} \limsup_{N \rightarrow \infty} \mathbb{P}\left(  \sum_{i = 1}^k \tilde{X}_i^{j,N} \leq -a \right) = 0.$$
The last two displayed equations show that $\sum_{i = 1}^k X_i^{j,N} $ is tight (in $N$) for each fixed $k \geq 1$ and $j \in \{1, \dots, m\}$, which implies that $X_u^{j,N}$ is tight (in $N$) for each fixed $u \geq 1$ and $j \in \{1, \dots, m\}$. This suffices for the proof.

%%%%%%%%%%%%%%%%%%%%%%%%%%%%%%%%%%%%%%%%%%%%%%%%%%%%%%%%%%%%%%%%%%%%%
%
%    Section 8
%
%%%%%%%%%%%%%%%%%%%%%%%%%%%%%%%%%%%%%%%%%%%%%%%%%%%%%%%%%%%%%%%%%%%%%
\section{Construction of line ensembles}\label{Section8} Throughout this section we use freely the notions of a {\em line ensemble} as defined in \cite[Definition 2.1]{CorHamA} and the {\em Brownian Gibbs property} as defined in \cite[Definition 2.2]{CorHamA}, see also \cite[Section 2.1]{DEA21}. We will always assume that $\Sigma = \mathbb{N}$ and $\Lambda = \mathbb{R}$ in \cite[Definition 2.2]{CorHamA}.

%%%%%%%%%%%%%%%%%%%%%%%%%%%%%%%%%%%%%%%%%%%%%%%%%%%%%%%%%%%%%%%%%%%%%
%
%    Section 8.1
%
%%%%%%%%%%%%%%%%%%%%%%%%%%%%%%%%%%%%%%%%%%%%%%%%%%%%%%%%%%%%%%%%%%%%%
\subsection{Proof of Theorem \ref{T3}}\label{Section8.1} In the course of the proof of Theorem \ref{T3} we will require the following lemma, whose proof is given in the next section. In the setup of the lemma we suppose that $\mathcal{L}^N = \{\mathcal{L}_i^N\}_{i \geq 1}$ is a sequence of $\mathbb{N}$-indexed line ensembles on $\mathbb{R}$, which satisfy the Brownian Gibbs property. We also suppose that each $\mathcal{L}^N$ is determinantal with a correlation kernel $K_N(x,t;y,s)$ in the following sense: for any $t_1, \dots, t_r \in \mathbb{R}$ with $t_1 < \cdots < t_r$ the random measure on $\mathbb{R}^2$
\begin{equation}\label{S81RM}
M^N(A) = \sum_{i \geq 1} \sum_{j = 1}^r {\bf 1} \{ (t_j, \sqrt{2} \cdot \mathcal{L}_i^N(t_j) + t_j^2) \in A\}.
\end{equation}
is a determinantal point process on $\mathbb{R}^2$ with correlation kernel $K_N(s,x; t,y)$ and reference measure $\mu_{\mathcal{T}} \times \lambda$, where $\mu_{\mathcal{T}}$ is the counting measure on $\mathcal{T} = \{t_1, \dots, t_r\}$ and $\lambda$ is the Lebesgue measure on $\mathbb{R}$.

\begin{lemma}\label{S8Lemma} Assume the same notation as in the preceding paragraph. Suppose that for fixed $s, t \in \mathbb{R}$ and sequences $x_N, y_N \in \mathbb{R}$ such that $\lim_N x_N = x$ and $\lim_N y_N = y$ we have
\begin{equation}\label{S81E1}
\lim_{N \rightarrow \infty} K_N(s, x_N; t, y_N) = K_{a,b,c}(s,x;t,y),
\end{equation}
where $K_{a,b,c}$ is as in Definition \ref{3BPKernelDef} for parameters as in Definition \ref{DLP} such that $c^- = 0$ and $J_a^- + J_b^- < \infty$. Suppose further that for each $t \in \mathbb{R}$ we have 
\begin{equation}\label{S81E2}
\lim_{ a \rightarrow \infty} \limsup_{N \rightarrow \infty} \mathbb{P}(\mathcal{L}^N_1(t) \geq a) = 0.
\end{equation}
Then, there exists a line ensemble $\mathcal{L}^{a,b,c}$ satisfying the conditions of Theorem \ref{T3}. Moreover $\mathcal{L}^N$ converge weakly to $\mathcal{L}^{a,b,c}$ as random elements in $C(\mathbb{N} \times \mathbb{R})$.
\end{lemma}
\begin{remark} As shown in \cite[Lemma 2.2]{DEA21} the space $C(\mathbb{N} \times \mathbb{R})$ is Polish and then the weak convergence in Lemma \ref{S8Lemma} is that of random elements in $C(\mathbb{N} \times \mathbb{R})$, cf. \cite[Section 3]{Billing}.
\end{remark}

With the above result in place we may proceed with the proof of Theorem \ref{T3}.

\begin{proof}[Proof of Theorem \ref{T3}] The starting point of the proof is to note that if $J_{a}^+ = J_b^+ = m \geq 0$ and $c^- = c^+ = J_a^- = J_b^- = 0$, then we may pick $\alpha, \beta$ in Definition \ref{3BPKer} such that $\alpha + t_1 < \beta + t_2$ and then $K_{a,b,c}$ agrees with $K^{\mathrm{Airy}}_{X,Y}$ from (\ref{S1EAKS}) for $x_i = (a_i^+)^{-1}$ and $y_i = - (b_i^+)^{-1}$ for $i = 1, \dots, m$. The existence of a line ensemble $\mathcal{L}^{a,b,c}$ as in the statement of the theorem is then ensured by \cite[Theorem 3.8 and Proposition 3.12]{CorHamA}. We mention that in \cite[Proposition 3.12]{CorHamA} there is a small typo and it should read ``$r_i \leq \sqrt{2} N$'' as opposed to ``$r_i \leq 2N$''. In addition, the parameters $p_i,q_i$ in that proposition are related to ours via $p_i = \sqrt{2} \cdot y_i = -\sqrt{2} \cdot (b_i^+)^{-1}$ and $q_i = -\sqrt{2} \cdot x_i = -\sqrt{2} \cdot (a_i^+)^{-1}$ for $i = 1, \dots, m$.

The goal in the remainder of the proof is to obtain $\mathcal{L}^{a,b,c}$ for general choices of parameters as in the statement of the theorem by taking appropriate limits of the ones we obtained above and applying Lemma \ref{S8Lemma}. For clarity, we split the remainder of the proof into several steps.\\

{\bf \raggedleft Step 1.} In this step we construct our ensembles when $c^- = c^+ = J_a^- = J_b^- = 0$ and $J_a^+ + J_b^+ < \infty$. Pick $m \geq \max(J_a^+, J_b^+)$ and consider an $N$-indexed sequence of ensembles $\mathcal{L}^{N}$ as in the beginning of the proof with $\tilde{c}^- = \tilde{c}^+ = J_{\tilde{a}}^- = J_{\tilde{b}}^- = 0$, and 
$$\tilde{a}_i^+ = a_i^+ \mbox{ for } i = 1, \dots, \min(m, J_a^+), \hspace{1mm} \tilde{a}_i^+ = 1/N \mbox{ for } i = \min(m, J_a^+) +1 , \dots, m, \hspace{1mm} \tilde{a}_i^+ = 0 \mbox{ for } i \geq m+1,$$
$$\tilde{b}_i^+ = b_i^+ \mbox{ for } i = 1, \dots, \min(m, J_b^+), \hspace{1mm} \tilde{b}_i^+ = 1/N \mbox{ for } i = \min(m, J_b^+) +1 , \dots, m , \hspace{1mm} \tilde{b}_i^+ = 0 \mbox{ for } i \geq m+1.$$
We mention that the tilde parameters depend on $N$, but we will suppress it from the notation. In the remainder of this step we prove that $\mathcal{L}^{N}$ satisfy the conditions of Lemma \ref{S8Lemma}, which if true would imply the statement of the theorem in the present case.\\

Observe first that by assumption we know that $\mathcal{L}^{N}$ satisfy the Brownian Gibbs property and the measure in (\ref{S81RM}) is determinantal with correlation kernel $K_N(s,x;t,y) = K_{\tilde{a}, \tilde{b}, \tilde{c}}(s,x;t,y)$ and reference measure $\mu_{\mathcal{T}} \times \lambda$. In Definition \ref{3BPKernelDef} we pick $\alpha, \beta$ so that $s+ \alpha > 0$, $t + \beta < 0$, $a_1^+(s + \alpha) < 1/2$, and $-b_1^+(t + \alpha) < 1/2$. We also assume that $N$ is large enough so that $(1/N)(s + \alpha) < 1/2$ and $-(1/N)(t + \alpha) < 1/2$, and then for $s,t \in \mathcal{T}$ and $x,y \in \mathbb{R}$
\begin{equation}\label{S81Q1}
\begin{split}
&K_N(s,x;t,y) = -  \frac{{\bf 1}\{ t > s\} }{\sqrt{4\pi (t - s)}} \cdot e^{ - \frac{(y - x)^2}{4(t - s)} - \frac{(t - s)(y + x)}{2} + \frac{(t - s)^3}{12} }\\
&+\frac{1}{(2\pi \im)^2} \int_{\Gamma_{\alpha }^+} d z \int_{\Gamma_{\beta}^-} dw \frac{e^{z^3/3 -xz - w^3/3 + yw}}{z + s - w - t} \cdot \prod_{i =1}^m \frac{(1 + \tilde{b}^+_i(z+s))(1 - \tilde{a}_i^+(w+t))}{(1 + \tilde{b}^+_i(w+t))(1 - \tilde{a}_i^+(z+s))}.
\end{split}
\end{equation}
We also have
\begin{equation}\label{S81Q2}
\begin{split}
&K_{a,b,c}(s,x;t,y) = -  \frac{{\bf 1}\{ t > s\} }{\sqrt{4\pi (t - s)}} \cdot e^{ - \frac{(y - x)^2}{4(t - s)} - \frac{(t - s)(y + x)}{2} + \frac{(t - s)^3}{12} }\\
&+\frac{1}{(2\pi \im)^2} \int_{\Gamma_{\alpha }^+} d z \int_{\Gamma_{\beta}^-} dw \frac{e^{z^3/3 -xz - w^3/3 + yw}}{z + s - w - t} \cdot \prod_{i =1}^{J_a^+} \frac{1 - a_i^+(w+t)}{1 - a_i^+(z+s)} \prod_{j =1}^{J_b^+}  \frac{1 + b^+_j(z+s)}{1 + b^+_j(w+t)}.
\end{split}
\end{equation}

We next note the following consequence of (\ref{RatBound}). Suppose that $\{\hat{a}_i\}_{i \geq 1}$ and $\{\hat{b}_i\}_{i \geq 1}$ are non-negative and $\sum_{i \geq 1} (\hat{a}_i^+ + \hat{b}_i^+) \in [0,A]$ for some $A > 0$. Suppose further that $s+ \alpha > 0$, $t + \beta < 0$, $\hat{a}_i^+(s + \alpha) < 1/2$, and $-\hat{b}_i^+(t + \alpha) < 1/2$. Then, for $z \in \Gamma_{\alpha}^+$ we have $|1 - a^+_i (z + s)| \geq 1/4$, and for $w \in \Gamma_{\beta}^+$ that $|1 - b_i^+(w+t)| \geq 1/4$. The latter and (\ref{RatBound}) show that there is a constant $c_A$ such that
\begin{equation}\label{S81Q3}
\begin{split}
\left|\prod_{i =1}^{\infty} \frac{1 - \hat{a}_i^+(w+t)}{1 - \hat{a}_i^+(z+s)} \prod_{j =1}^{\infty}  \frac{1 + \hat{b}^+_j(z+s)}{1 + \hat{b}^+_j(w+t)} \right| \leq \exp( c_A |w+t| + c_A |z+s|),
\end{split}
\end{equation}
provided that $z \in \Gamma_{\alpha}^+$ and $w \in \Gamma_{\beta}^-$. In addition, by analyzing the real part of $z^3/3$ and $w^3/3$, we see that we can find a constants $D_1 > 0$ (depending on $\alpha, \beta$) such that 
\begin{equation}\label{S81Q4}
\begin{split}
\left|e^{z^3/3 - w^3/3} \right| \leq D_1 \cdot \exp \left(-|z|^3/24 - |w|^3/24 \right).
\end{split}
\end{equation}
Since by construction $\Gamma_{\alpha + s}^+$ is well-separated from $\Gamma_{\beta + t}^-$, we also have for some $D_2 > 0$ (depending on $\alpha, \beta,s ,t$)
\begin{equation}\label{S81Q5}
\left|\frac{1}{z + s - w - t} \right| \leq D_2.
\end{equation}
Finally, for a sequence $x_N$, and $y_N$ converging to $x$ and $y$ we have that for all large $N$
\begin{equation}\label{S81Q6}
\left|e^{-x_N z  + y_Nw}\right| \leq \exp \left( (|x|+1)|z| + (|y|+1) |w| \right).
\end{equation}

We now have that the first line of (\ref{S81Q1}) for $K_N(s, x_N; t, y_N)$ converges to that of (\ref{S81Q2}) by continuity as $N \rightarrow \infty$. The integrands on the second line of (\ref{S81Q1}) converge pointwise to that on the second line of (\ref{S81Q2}). We thus conclude that 
$$\lim_{N \rightarrow \infty} K_N(s, x_N; t, y_N) = K_{a,b,c}(s,x;t,y)$$
from the dominated convergence theorem with a dominating function given by the product of the right sides of (\ref{S81Q3}), (\ref{S81Q4}), (\ref{S81Q5}) and (\ref{S81Q6}). This establishes (\ref{S81E1}). \\

In the remainder of this step we show that (\ref{S81E2}) holds. This is almost a consequence of Proposition \ref{PVC} but we provide the necessary details. In the sequel we fix $t \in \mathbb{R}$ and consider the point processes $M^N$ on $\mathbb{R}$ formed by $\{\sqrt{2} \cdot \mathcal{L}^N_i(t) + t^2\}_{i \geq 1}$. From Lemma \ref{LemmaSlice} and part (4) of Proposition \ref{PropLem} we conclude $M^N$ is a determinantal point process on $\mathbb{R}$ with reference measure $\lambda$ and correlation kernel
\begin{equation}\label{S81Q7}
\begin{split}
&K_N^t(x,y) = \frac{e^{t(y-x)}}{(2\pi \im)^2} \int_{\Gamma_{\alpha }^+} d z \int_{\Gamma_{\beta}^-} dw \frac{e^{z^3/3 -xz - w^3/3 + yw}}{z - w} \cdot \prod_{i =1}^m \frac{(1 + \tilde{b}^+_i(z+t))(1 - \tilde{a}_i^+(w+t))}{(1 + \tilde{b}^+_i(w+t))(1 - \tilde{a}_i^+(z+t))}.
\end{split}
\end{equation}
Using that $\Real(z+t) > 0$ for all $z \in \Gamma_{\alpha}^+$ and $\Imag(w+t) < 0$ for all $w \in \Gamma_{\beta}^-$, we see that there is a constant $d_1 > 0$ (depending on $\alpha, \beta, t$) such that for $z\in \Gamma_{\alpha}^+$ and $w \in \Gamma_{\beta}^-$ 
\begin{equation}\label{S81Q8}
\begin{split}
\left|e^{-x (z + t)} \right| \leq \begin{cases} e^{-d_1 x} &\mbox{ if } x \geq 0, \\ e^{|L||z + t|} &\mbox{ if } x \in [-L, 0], \end{cases} \hspace{3mm} \left|e^{(w + t)y} \right| \leq \begin{cases} e^{-d_1 y}, &\mbox{ if } y \geq 0 \\ e^{|L||w + t|} &\mbox{ if } y \in [-L, 0]. \end{cases}
\end{split}
\end{equation}
Combining (\ref{S81Q3}), (\ref{S81Q4}), (\ref{S81Q5}), (\ref{S81Q7}) and (\ref{S81Q8}) we see that we can find a constant $B_L$ (depending on $\alpha, \beta, t$,  the sum $\sum_{i \geq 1} (a_i^+ + b_i^+)$ and $L$) such that for all large $N$ and $x, y \geq -L$
\begin{equation}\label{S81Q9}
\begin{split}
&\left|K_N^t(x,y) \right| \leq B_L \cdot \exp\left( -d_1 |x| - d_1 |y| \right).
\end{split}
\end{equation}
From Hadamard's inequality (\ref{Hadamard}), we conclude for $x_1, \dots, x_n \geq -L$
\begin{equation*}
\left| \det\left[ K_N^t(x_i,x_j) \right]_{i,j = 1}^n \right| \leq B_L^n \cdot n^{n/2} \cdot \prod_{i = 1}^n \exp \left( - d_1 |x_i| \right),
\end{equation*}
which implies for any $a \in \mathbb{R}$
\begin{equation*}
1 + \sum_{n \geq 1} \frac{1}{n!} \int_{(a, \infty)^n} \left| \det\left[ K_N^t(x_i,x_j) \right]_{i,j = 1}^n \right| \lambda^n(dx) < \infty.
\end{equation*}
We conclude that the conditions of Proposition \ref{PropLast} are satisfied and so for all $a \in \mathbb{R}$
\begin{equation*}
\mathbb{P}\left( \sqrt{2} \cdot \mathcal{L}^N_1(t) + t^2 > a \right)  = \sum_{n = 1}^{\infty} \frac{(-1)^{n-1}}{n!}  \int_{(a, \infty)^n}\det\left[ K_N^t(x_i,x_j) \right]_{i,j = 1}^n\lambda^n(dx). 
\end{equation*}
Combining the estimates after and including (\ref{S81Q9}), we get for all $a \geq 0$
\begin{equation*}
\mathbb{P}\left( \sqrt{2} \cdot \mathcal{L}^N_1(t) + t^2 > a \right) \leq \sum_{n = 1}^{\infty} \frac{n^{n/2} B_0^n }{n!} \cdot \left(\int_{a}^{\infty} e^{-d_1 x} dx\right)^n \leq e^{-d_1a} \cdot \sum_{n = 1}^{\infty} \frac{n^{n/2} B_0^n }{n! \cdot d_1^n},
\end{equation*}
which proves (\ref{S81E2}).\\

{\bf \raggedleft Step 2.} In this step we construct $\mathcal{L}^{a,b,c}$ for general parameters. We start with the case when $c^-  = c^+ = J_{a}^- = J_b^-  = 0$. Consider an $N$-indexed sequence of ensembles $\mathcal{L}^{N}$ as constructed in Step 1 for parameters $\tilde{c}^- = \tilde{c}^+ = J_{\tilde{a}}^- = J_{\tilde{b}}^- = 0$, and 
$$\tilde{a}_i^+ = a_i^+ \mbox{ for } i = 1, \dots, \min(N, J_a^+), \hspace{3mm} \tilde{a}_i^+ = 0 \mbox{ for } i \geq \min(N, J_a^+)+1,$$
$$\tilde{b}_i^+ = b_i^+ \mbox{ for } i = 1, \dots, \min(N, J_b^+), \hspace{3mm} \tilde{b}_i^+ = 0 \mbox{ for } i \geq \min(N, J_a^+)+1.$$
At this time we can repeat verbatim the argument in Step 1, to conclude that $\mathcal{L}^N$ satisfy the conditions of Lemma \ref{S8Lemma}, which implies the statement of the theorem in the present case.\\

We next consider the case of $c^-  = J_{a}^- = J_b^-  = 0$ but arbitrary $c^+ \geq 0$. Consider an $N$-indexed sequence of line ensembles $\mathcal{L}^{N}$ of the form $\mathcal{L}^{\tilde{a},\tilde{b},\tilde{c}}$ such that $\tilde{c}^-  = \tilde{c}^+ = J_{a}^- = J_b^-  = 0$, $\{\tilde{a}_i^+\}_{i \geq 1}$ is the decreasing sequence formed by $\{{a}_i^+\}_{i \geq 1}$ (counted with multiplicities) and $N$ copies of $c^+/N$, while $\{\tilde{b}_i^+\}_{i \geq 1}$ is equal to $\{{b}_i^+\}_{i \geq 1}$. Our work so far shows that the measure in (\ref{S81RM}) is determinantal with correlation kernel $K_N(s,x;t,y) = K_{\tilde{a}, \tilde{b}, \tilde{c}}(s,x;t,y)$, given by
\begin{equation}\label{S81P1}
\begin{split}
&K_N(s,x;t,y) = -  \frac{{\bf 1}\{ t > s\} }{\sqrt{4\pi (t - s)}} \cdot e^{ - \frac{(y - x)^2}{4(t - s)} - \frac{(t - s)(y + x)}{2} + \frac{(t - s)^3}{12} }\\
&+\frac{1}{(2\pi \im)^2} \int_{\Gamma_{\alpha }^+} d z \int_{\Gamma_{\beta}^-} dw \frac{e^{z^3/3 -xz - w^3/3 + yw}}{z + s - w - t} \cdot \prod_{i =1}^\infty \frac{(1 + b^+_i(z+s))(1 - a_i^+(w+t))}{(1 + b^+_i(w+t))(1 - a_i^+(z+s))} \\
&\times  \left( \frac{(1 - c^+(w+t)/N)}{(1 - c^+(z+s)/N)} \right)^N,
\end{split}
\end{equation}
where we pick $\alpha, \beta$ so that $s+ \alpha > 0$, $t + \beta < 0$, $a_1^+(s + \alpha) < 1/2$, and $-b_1^+(t + \alpha) < 1/2$. We also suppose that $N$ is large enough so that $(c^+/N)(s + \alpha) < 1/2$ and $-(c^+/N)(t + \alpha) < 1/2$. As before, we can repeat verbatim the argument in Step 1, to conclude that $\mathcal{L}^N$ satisfy the conditions of Lemma \ref{S8Lemma}. The only minor differences are that in showing that the integrands of $K_N(s,x_N;t,y_N)$ converge pointwise to the integrand in $K_{a,b,c}(s,x;t,y)$ we need to use the simple limit
$$\lim_{N \rightarrow \infty} \left( \frac{(1 - c^+(w+t)/N)}{(1 - c^+(z+s)/N)} \right)^N = \exp \left( c^+(z+ s) - c^+(w+t) \right),$$
and the $B_L$ in (\ref{S81Q9}) depends on $c^+ + \sum_{i \geq 1} (a_i^+ + b_i^+)$ as opposed to just $\sum_{i \geq 1} (a_i^+ + b_i^+)$.\\

We finally suppose that $c^- = 0$ and $J_a^- + J_b^- < \infty$. Starting from our parameters, we let $\{\tilde{a}^{\pm}_i\}_{i \geq 1}$, $\{\tilde{b}^{\pm}_i\}_{i \geq 1}$, $\tilde{c}^{\pm}$ be the parameters we constructed in Step 2 of the proof of Theorem \ref{T2}, and also let $\Delta$ be as in that step. We recall that in this construction we have $\tilde{a}_1^- = \tilde{b}_1^- = \tilde{c}^- = 0$ (i.e. all ``minus'' parameters are equal to zero). In particular, our work so far shows that there exists a line ensemble $\mathcal{L}^{\tilde{a}, \tilde{b}, \tilde{c}}$ as in the statement of the theorem for this choice of parameters. From (\ref{S7Q4}) and (\ref{S7Q5}), we know that if $m \in \mathbb{N}$, $s_1 < s_2 < \cdots < s_m$ and $\mathcal{A} = \{s_1, \dots, s_m\}$, then
\begin{equation}\label{S81P2}
\tilde{M}(A) = \sum_{i \geq 1} \sum_{j = 1}^m {\bf 1} \{ (s_j - \Delta, \sqrt{2} \cdot \mathcal{L}_i^{\tilde{a}, \tilde{b}, \tilde{c}}(s_j- \Delta) + (s_j-\Delta)^2) \in A\}
\end{equation}
is a determinantal point process with correlation kernel $K_{a,b,c}$ and reference measure $\mu_{\mathcal{A}} \times \lambda$. From part (3) of Proposition \ref{PropLem} and Corollary \ref{CorWC2} we conclude that we have the following equality in the sense of finite-dimensional distributions
$$\left(\sqrt{2} \cdot \mathcal{L}_i^{\tilde{a}, \tilde{b}, \tilde{c}}(t - \Delta) + (t-\Delta)^2: i \geq 1, t \in \mathbb{R} \right)  = \left(Y_i(t): i \geq 1 , t \in \mathbb{R} \right),$$
where we recall that $\{Y_i\}_{i \geq 1}$ are as in Theorem \ref{T2}. In particular, if we define 
$$\mathcal{L}^{a,b,c}_i(t) = \mathcal{L}_i^{\tilde{a}, \tilde{b}, \tilde{c}}(t - \Delta)  + 2^{-1/2} \cdot \left( (t- \Delta)^2 - t^2 \right),$$
we see that $\mathcal{L}^{a,b,c}$ is a line ensemble, which satisfies (\ref{S1FDE}). Since $\mathcal{L}^{\tilde{a}, \tilde{b}, \tilde{c}}$ satisfies the Brownian Gibbs property and the latter is preserved under translations and deterministic shifts by functions of the form $at + b$, we conclude that $\mathcal{L}^{a,b,c}$ also satisfies the Brownian Gibbs property. 
\end{proof}

%%%%%%%%%%%%%%%%%%%%%%%%%%%%%%%%%%%%%%%%%%%%%%%%%%%%%%%%%%%%%%%%%%%%%
%
%    Section 8.2
%
%%%%%%%%%%%%%%%%%%%%%%%%%%%%%%%%%%%%%%%%%%%%%%%%%%%%%%%%%%%%%%%%%%%%%
\subsection{Proof of Lemma \ref{S8Lemma}}\label{Section8.2} We start by verifying that the sequence $\mathcal{L}^N$ satisfies the hypotheses of \cite[Definition 3.3]{CorHamA} for each $k \geq 1$ and $T > 0$. The fact that hypothesis $(H1)_{k,T}$ is satisfied follows from our assumption that $\mathcal{L}^N$ satisfy the Brownian Gibbs property. We next check that $(H2)_{k,T}$ is satisfied, for which it suffices to show that 
\begin{equation}\label{S82E1}
\left( \mathcal{L}_i^N(t): i \geq 1, t \in \mathbb{R} \right) \Rightarrow \left( 2^{-1/2} Y_i(t) - 2^{-1/2} t^2 : i \geq 1, t \in \mathbb{R} \right),
\end{equation}
where in (\ref{S82E1}) we have finite-dimensional convergence and $\{Y_i\}_{i \geq 1}$ are as in Theorem \ref{T2} for parameters $a,b,c$ as in the statement of the present lemma. 

We first have from Lemma \ref{LemmaSlice} that for a fixed $t \in \mathbb{R}$ the variables $\{ \sqrt{2} \cdot \mathcal{L}_i^N(t) + t^2\}_{i \geq 1}$ form a determinantal point process $M^N$ on $\mathbb{R}$ with correlation kernel $K^t_N(x,y) = K_N(t,x;t,y)$ and reference measure $\lambda$. From (\ref{S81E1}) and Proposition \ref{PropWC0} we conclude that $M^N$ converge to a determinantal point process $M$ on $\mathbb{R}$ with correlation kernel $K(x,y) = K_{a,b,c}(t,x; t,y)$. From Theorem \ref{T2} we also know that $\mathbb{P}(M(\mathbb{R}) = \infty) = 1$. The latter few observations show that conditions (1) and (2) in Proposition \ref{TightnessCrit} are satisfied for the sequence $X^N = (\sqrt{2} \cdot \mathcal{L}_i^N(t) + t^2: i \geq 1)$, while condition (3) is satisfied in view of (\ref{S81E2}). We conclude from Proposition \ref{TightnessCrit} that $\{ \sqrt{2} \cdot \mathcal{L}_i^N(t) + t^2\}_{i \geq 1}$ is a tight sequence of random elements in $(\mathbb{R}^{\infty}, \mathcal{R}^{\infty})$ for each $t \in \mathbb{R}$.

From (\ref{S81E1}), we see that $K_N(s,x;t,y)$ satisfy the conditions of Proposition \ref{PropWC1}, which together with the tightness statement we established in the previous paragraph shows that the conditions of Proposition \ref{PropWC2} are satisfied. In particular, for each $m \in \mathbb{N}$, and real $t_1 < t_2 < \cdots < t_m$, we conclude that the random elements in $(\mathbb{R}^{\infty}, \mathcal{R}^{\infty})$, given by
$$\left( \sqrt{2} \cdot \mathcal{L}_i^N(t_j) + t_j^2: i \geq 1 \mbox{ and } j = 1,\dots, m \right)$$
converge in the finite-dimensional sense to some random element $\left(  X_i(t_j) : i \geq 1 \mbox{ and } j = 1,\dots, m \right)$. Moreover, (\ref{S81E1}) and Propositions \ref{PropWC1} and \ref{PropWC2} imply that 
$$\tilde{M}(\omega, A) = \sum_{i \geq 1} \sum_{j = 1}^r {\bf 1}\{(t_j, X_i(t_j, \omega)) \in A \} $$
is a determinantal point process on $\mathbb{R}^2$ with correlation kernel $K_{a,b,c}(s,x;t,y)$ and reference measure $\mu_{\mathcal{T}} \times \lambda$. Corollary \ref{CorWC2} shows that $\left(  X_i(t_j) : i \geq 1 \mbox{ and } j = 1,\dots, m \right)$ has the same finite-dimensional distribution as $\left(  Y_i(t_j) : i \geq 1 \mbox{ and } j = 1,\dots, m \right)$, which with our earlier statements shows (\ref{S82E1}). 

Finally, we have that $(H3)_{k,T}$ is satisfied from the finite-dimensional convergence in (\ref{S82E1}) and the fact that
$$Y_i(t, \omega) > Y_{i+1}(t, \omega) \mbox{ for each $t \in \mathbb{R}$ and $\omega \in \Omega$, }$$
as shown in (\ref{T2E2}). Our verification of \cite[Definition 3.3]{CorHamA} is now complete.\\

At this point most of the work is done and we just quickly verify the statements in the lemma. From \cite[Theorem 3.8]{CorHamA} and (\ref{S82E1}) we know that there exists a line ensemble $\mathcal{L}^{a,b,c}$, which satisfies the Brownian Gibbs property and satisfies (\ref{S1FDE}). The fact that $\mathcal{L}^N$ converges weakly to $\mathcal{L}^{a,b,c}$ as random elements of $C(\mathbb{N} \times \mathbb{R})$ follows from \cite[Proposition 3.6]{CorHamA}.

%%%%%%%%%%%%%%%%%%%%%%%%%%%%%%%%%%%%%%%%%%%%%%%%%%%%%%%%%%%%%%%%%%%%%
%
%    Appendix
%
%%%%%%%%%%%%%%%%%%%%%%%%%%%%%%%%%%%%%%%%%%%%%%%%%%%%%%%%%%%%%%%%%%%%%
\begin{appendix}

%%%%%%%%%%%%%%%%%%%%%%%%%%%%%%%%%%%%%%%%%%%%%%%%%%%%%%%%%%%%%%%%%%%%%
%
%    Appendix A
%
%%%%%%%%%%%%%%%%%%%%%%%%%%%%%%%%%%%%%%%%%%%%%%%%%%%%%%%%%%%%%%%%%%%%%
\section{Proofs of results from Section \ref{Section2}} \label{AppendixA} In this section we present the proofs of various results from Section \ref{Section2} after establishing a few auxiliary results in Section \ref{AppendixA1}. We mention that the results in Section \ref{AppendixA1} are somewhat standard, but as we could not find them in the literature, we included them here for the sake of completeness.

%%%%%%%%%%%%%%%%%%%%%%%%%%%%%%%%%%%%%%%%%%%%%%%%%%%%%%%%%%%%%%%%%%%%%
%
%    Appendix A.1
%
%%%%%%%%%%%%%%%%%%%%%%%%%%%%%%%%%%%%%%%%%%%%%%%%%%%%%%%%%%%%%%%%%%%%%
\subsection{Auxiliary lemmas} \label{AppendixA1} In this section we derive some basic results about the factorial moment problem and dissecting semi-rings. We continue with the same notation as in Section \ref{Section2}.

\begin{lemma}\label{LemmaMGF} Fix $m \in \mathbb{N}$ and let $(X_1, \dots, X_m)$ and $(Y_1, \dots, Y_m)$ be random vectors taking values in $\mathbb{R}_{\geq 0}^m$. Suppose that there exists $a > 0$ such that for all $a_1, \dots, a_m \in [0, a]$ we have 
$$\mathbb{E}\left[ e^{a_1X_1 + \cdots + a_m X_m} \right] = \mathbb{E}\left[ e^{a_1Y_1 + \cdots + a_m Y_m} \right] < \infty.$$
Then, $(X_1,\dots, X_m)$ has the same distribution as $(Y_1, \dots, Y_m)$.
\end{lemma}
\begin{proof} We prove by induction on $n$ that 
\begin{equation}\label{P3Red2}
\mathbb{E}\left[ \prod_{i = 1}^n e^{z_i X_i} \cdot \prod_{i = n+1}^m e^{a_i X_i} \right] = \mathbb{E}\left[ \prod_{i = 1}^n e^{z_i Y_i} \cdot \prod_{i = n+1}^m e^{a_i Y_i}  \right]
\end{equation}
whenever $z_1, \dots, z_n \in D_a := \{z \in \mathbb{C}: |\Real(z)|< a\}$ and $a_{n+1}, \dots, a_m \in [0, a]$. Part of the statement is that both sides of (\ref{P3Red2}) are well-defined and finite.

The base case $n = 0$ holds by assumption and we suppose that we have proved the result for $n = v \leq m-1$. Assume now that $n = v + 1$. We know that 
$$\left| \prod_{i = 1}^n e^{z_i X_i} \cdot \prod_{i = n+1}^m e^{a_i X_i}  \right| = \prod_{i = 1}^n e^{\Real(z_i) X_i} \cdot \prod_{i = n+1}^m e^{a_i X_i} \leq \prod_{i = 1}^n e^{a X_i} \prod_{i = n+1}^m e^{a_i X_i} ,$$
and so both sides of (\ref{P3Red2}) are well-defined and finite. We define for $z \in D_a$ 
$$f_X(z) = \mathbb{E} \left[ \prod_{i = 1}^{n-1} e^{z_i X_i} \cdot e^{z X_n} \cdot \prod_{i = n+1}^m e^{a_i X_i}  \right] \mbox{ and }f_Y(z) = \mathbb{E} \left[ \prod_{i = 1}^{n-1} e^{z_i Y_i} \cdot e^{z Y_n} \cdot \prod_{i = n+1}^m e^{a_i Y_i}  \right].$$
Fix $z_0 \in D_a$, and let $\epsilon \in (0, a)$ be such that $B(z_0,\epsilon) \subset D_a$ and $h_w \rightarrow 0$. Then, we have 
\begin{equation*}
\begin{split}
&\lim_{w \rightarrow \infty} \frac{\prod_{i = 1}^{n-1} e^{z_i X_i} \cdot e^{(z_0 + h_w) X_n} \cdot \prod_{i = n+1}^m e^{a_i X_i} - \prod_{i = 1}^{n-1} e^{z_i X_i} \cdot e^{z_0 X_n} \cdot \prod_{i = n+1}^m e^{a_i X_i}}{h_w} \\
& =  \prod_{i = 1}^{n-1} e^{z_i X_i} \cdot X_n e^{z_0 X_n} \cdot \prod_{i = n+1}^m e^{a_i X_i}.
\end{split}
\end{equation*}

We also have for $|h_w| \leq \epsilon/2$
$$\left| e^{z_0X_n} \cdot \frac{e^{h_w X_n} - 1}{h_w} \right| \leq e^{(a-\epsilon) X_n} \cdot \left|\frac{e^{h_w X_n} - 1}{h_w}  \right| \leq e^{(a-\epsilon) X_n} \cdot \frac{e^{|h_w| X_n} - 1}{|h_w|} \leq  e^{(a-\epsilon) X_n} \cdot X_n  e^{(\epsilon/2) X_n} \leq \frac{2 e^{a X_n} }{\epsilon} ,$$
where in the third inequality we used 
$$\frac{e^{|h_w| X_n} - 1}{|h_w|} = \sum_{r = 1}^{\infty} \frac{|h_w|^{r-1} X_n^r}{r!}  \leq X_n \cdot \sum_{r = 1}^{\infty} \frac{|h_w|^{r-1} X_n^{r-1}}{r!}  \leq  X_n \cdot \sum_{r = 1}^{\infty} \frac{ |h_w|^{r-1} X_n^{r-1}}{(r-1)!} \leq X_n e^{(\epsilon/2)X_n}.$$
The latter implies for $|h_w| \leq \epsilon/2$
\begin{equation*}
\begin{split}
&\left| \frac{\prod_{i = 1}^{n-1} e^{z_i X_i} \cdot e^{(z_0 + h_w) X_n} \cdot \prod_{i = n+1}^m e^{a_i X_i} - \prod_{i = 1}^{n-1} e^{z_i X_i} \cdot e^{z_0 X_n} \cdot \prod_{i = n+1}^m e^{a_i X_i}}{h_w} \right| \leq \frac{2}{\epsilon} \prod_{i = 1}^{m} e^{a X_i}.
\end{split}
\end{equation*}

By the dominated convergence theorem with dominating function $\frac{2}{\epsilon} \prod_{i = 1}^{m} e^{a X_i}$ we conclude that $f_X(z)$ is holomorphic at $z_0$ and has derivative 
$$\mathbb{E}\left[ \prod_{i = 1}^{n-1} e^{z_i X_i} \cdot X_n e^{z_0 X_n} \cdot \prod_{i = n+1}^m e^{a_i X_i}  \right].$$
Analogously, $f_Y(z)$ is holomorphic in $D_a$. Since $f_X(z)$ and $f_Y(z)$ are two holomorphic functions on $D_a$ (which is connected) that agree when $z \in (0,a)$ (by the induction hypothesis) we conclude that they agree on all of $D_a$, see \cite[Corollary 4.9 in Chapter 2]{Stein}. This proves (\ref{P3Red2}) when $n = v+1$ and the general result now follows by induction on $n$.\\

Since the characteristic function uniquely determines the distribution of a random vector, see \cite[Theorem 3.10.5]{Durrett}, it suffices to show that for each $t_1, \dots, t_m \in \mathbb{R}$ we have 
\begin{equation*}
\mathbb{E}\left[ e^{\sum_{i = 1}^m \im t_i X_i} \right] = \mathbb{E}\left[ e^{\sum_{i = 1}^m \im t_i Y_i} \right].
\end{equation*}
The latter is a special case of (\ref{P3Red2}) for $n = m$ and $z_i =  \im t_i$ for $i = 1, \dots, m$.
\end{proof}

\begin{lemma}\label{MomentProblem} Fix $m \in \mathbb{N}$ and let $(X_1, \dots, X_m)$ and $(Y_1, \dots, Y_m)$ be random vectors taking values in $\mathbb{Z}_{\geq 0}^m$. Suppose that there is a sequence $c_n \in [0,\infty)$ such that for all $n_1, \dots, n_m \in \mathbb{Z}_{\geq 0}$ with $n_1 + \cdots + n_m = n$ we have
$$\mathbb{E}\left[ \prod_{i = 1}^m \frac{X_i!}{(X_i-n_i)!} \right] = \mathbb{E}\left[ \prod_{i = 1}^m \frac{Y_i!}{(Y_i-n_i)!} \right] \leq c_n, \mbox{ where } \frac{1}{m!} = 0 \mbox{ for $m<0$ and }\sum_{n = 1}^{\infty} \frac{c_n}{n!} \cdot a^n < \infty $$
 for some $a > 0$. Then, $(X_1, \dots, X_m)$ has the same distribution as $(Y_1, \dots, Y_m)$.
\end{lemma}
\begin{proof} Let $a_1, \dots, a_m \geq 0$. Using the Binomial theorem and Tonelli's theorem we get
\begin{equation} \label{P3Red3}
\mathbb{E} \left[ \prod_{i = 1}^m(1 + a_i)^{X_i} \right] = 1 + \sum_{n = 1}^{\infty} \frac{1}{n!} \sum_{n_1 + \cdots +n_m = n} \binom{n}{n_1, \dots, n_m} \prod_{j = 1}^m a_j^{n_j}\mathbb{E}\left[ \prod_{i = 1}^m \frac{X_i!}{(X_i-n_i)!} \right],
\end{equation}
where the latter series is allowed to be $+\infty$. By our assumption on the factorial moments we have 
\begin{equation*}
\begin{split}
& 1 + \sum_{n = 1}^{\infty} \frac{1}{n!} \sum_{n_1 + \cdots +n_m = n} \binom{n}{n_1, \dots, n_m} \prod_{j = 1}^m a_j^{n_j}\mathbb{E}\left[ \prod_{i = 1}^m \frac{X_i!}{(X_i-n_i)!} \right] \leq \\
&1 + \sum_{n = 1}^{\infty} \frac{c_n}{n!} \sum_{n_1 + \cdots +n_m = n} \binom{n}{n_1, \dots, n_m} \prod_{j = 1}^m a_j^{n_j} = 1 + \sum_{n = 1}^{\infty} \frac{c_n (a_1 + \cdots + a_m)^n}{n!}, 
\end{split}
\end{equation*}
and so we conclude that the series in (\ref{P3Red3}) is finite if $a_i \in [0, a/m]$. Arguing analogously for $(Y_1, \dots, Y_k)$, we get
$$ \mathbb{E} \left[ \prod_{i = 1}^m(1 + a_i)^{X_i} \right]  = \mathbb{E} \left[ \prod_{i = 1}^m(1 + a_i)^{Y_i} \right] < \infty \iff \mathbb{E} \left[ \prod_{i = 1}^me^{c_i X_i} \right]  = \mathbb{E} \left[ \prod_{i = 1}^me^{c_i Y_i} \right] < \infty,$$
provided that $a_i \in [0, a/m]$ or $c_i = \log (1 + a_i) \in [0, \log (1+a/m)]$ for $i = 1, \dots, m$. From Lemma \ref{LemmaMGF} we conclude that $(X_1, \dots, X_m)$ and $(Y_1, \dots, Y_m)$ have the same distribution.
\end{proof}

\begin{lemma}\label{P1L1} Let $\mathcal{I}$ be a semi-ring and $A_1, \dots, A_n \in \mathcal{I}$. We can find finitely many $B_1, \dots, B_m \in \mathcal{I}$ such that $B_i \cap B_j = \emptyset$ when $i \neq j$, $\cup_{i = 1}^n A_i = \cup_{j = 1}^m  B_j$ and $A_i = \cup_{j \in D_i} B_j $ where $D_i = \{j : B_j \subseteq A_i\}$.
\end{lemma}
\begin{proof} We proceed by induction on $n$. When $n = 1$, we can just let $B_1 = A_1$. Suppose we have proved the lemma when $n = k$ and that $n = k +1$. By induction hypothesis we can find $B_1, \dots, B_a$ such that $B_i \cap B_j = \emptyset$ when $i \neq j$ and each $A_i$ is the union of finitely many $B_j$'s for $i = 1,\dots, k$. Note that since $\mathcal{I}$ is a semi-ring, $A_{k+1} \setminus B_1 = A_{k+1} \setminus (B_1 \cap A_{k+1})$ is a finite disjoint union of sets in $\mathcal{I}$. Iterating this, we see that $A_{k+1} \setminus (B_1 \cup \cdots \cup B_a) = \sqcup_{c_{a+1} = 1}^{C_{a+1}} I_{a+1, c_{a+1}}$ where $I_{a+1, c_{a+1}} \in \mathcal{I}$. We also have $B_i \setminus A_{k+1} = \sqcup_{c_i = 1}^{C_i} I_{i,c_i}$ for $i = 1, \dots, a$ where $I_{i,c_i} \in \mathcal{I}$. We now consider the sets $\{I_{i,c_i}\}$ for $i = 1,\dots a+1$ and $c_i = 0,\dots, C_i$, where $I_{i,0} = B_i \cap A_{k+1}$ for $i = 1, \dots, a$ and $I_{a+1, 0} = \emptyset$. One readily verifies that these sets are pairwise disjoint sets in $\mathcal{I}$ and that $B_i = \sqcup_{c_i = 0}^{C_i} I_{i,c_i}$ for $i = 1, \dots, a$ and $A_{k+1} = \sqcup_{c_{a+1} = 1}^{C_{a+1}} I_{a+1, c_{a+1}} \sqcup \sqcup_{i = 1}^a I_{i, 0}$. Consequently, the family $\{I_{i,c_i}\}$ satisfies the conditions of the lemma for $A_1, \dots, A_{k+1}$. The general result now follows by induction.
\end{proof}

\begin{lemma}\label{DSRLemma}  Suppose that $\mathcal{I}$ is a dissecting semi-ring in $(E, \mathcal{E})$. Then, for each $n \in \mathbb{N}$ the family $\mathcal{I}_n = \{A_1 \times \cdots \times A_n :A_i \in \mathcal{I}\}$ is a dissecting semi-ring in $(E^n, \mathcal{E}^{\otimes n})$. In addition, every $A \in \mathcal{I}_n$ can be written as a finite disjoint union of sets of the form $B_1 \times \cdots \times B_n \in \mathcal{I}_n$ such that $B_i = B_j$ or $B_i \cap B_j = \emptyset$ for each $1 \leq i\leq j \leq n$.
\end{lemma}
\begin{proof} The fact that $\mathcal{I}_n$ is dissecting follows from \cite[Lemma 1.9]{Kall}.  Let $A = A_1 \times \cdots \times A_n \in \mathcal{I}_n$. From Lemma \ref{P1L1} we can find finitely many $I_1, \dots, I_m \in \mathcal{I}$ that are pairwise disjoint and $A_i = \sqcup_{j \in D_i} I_j$ for $i = 1, \dots, n$, where $D_i = \{j: I_j \subseteq A_i\}$. We now consider the sets $\{I_{a_1} \times \cdots \times I_{a_n}: a_i \in D_i \mbox{ for } i = 1, \dots, n\}$. The latter sets are clearly pairwise disjoint elements of $\mathcal{I}_n$ and also their union is precisely $A$. This suffices for the proof.
\end{proof}

\begin{lemma}\label{VaguePairLim} Suppose that $\lambda_n$ is a sequence in $S$ that converges vaguely to $\lambda$. Suppose that $B \subset E$ is a bounded Borel set with $\lambda(\partial B) = 0$ and $f$ is a bounded continuous function on $B$. Then,
\begin{equation}\label{UI1}
\lim_{n \rightarrow \infty} \int_{B} f(x) \lambda_n(dx) = \int_{B} f(x) \lambda(dx).
\end{equation}   
\end{lemma}
\begin{proof} If $\lambda(B) = 0$, then we have from \cite[Lemma 4.1]{Kall}
$$\limsup_{n \rightarrow \infty} \left|\int_{B} f(x) \lambda_n(dx) \right| \leq \sup_{x \in B} |f(x)| \cdot \limsup_{n \rightarrow \infty} \lambda_n(B) = \sup_{x \in B} |f(x)| \cdot \lambda(B) = 0 = \int_{B} f(x) \lambda(dx).$$
We may thus assume that $A = \lambda(B) \in (0,\infty)$. 

Setting $A_n = \lambda_n(B)$ we have that $A_n \rightarrow A$ and by possibly passing to a subsequence we may assume that $A_n > 0$ for all $n \in \mathbb{N}$. Setting 
$$\mu_n (F)  = A_n^{-1} \cdot \lambda_n(F) \mbox{ and } \mu(F) = A^{-1} \cdot \lambda(F)$$
for Borel $F \subseteq B$, we see that $\mu_n$ is a sequence of probability measures that weakly converge to $\mu$. The quickest way to see the latter is to compare \cite[Lemma 4.1 (iv)]{Kall} with the Portmanteau theorem \cite[Theorem 2.1]{Billing}. From the weak convergence we conclude
$$\lim_{n \rightarrow \infty} \int_{B} f(x) \mu_n(dx) = \int_{B} f(x) \mu(dx),$$
which implies (\ref{UI1}).
\end{proof}

%%%%%%%%%%%%%%%%%%%%%%%%%%%%%%%%%%%%%%%%%%%%%%%%%%%%%%%%%%%%%%%%%%%%%
%
%    Appendix A.2
%
%%%%%%%%%%%%%%%%%%%%%%%%%%%%%%%%%%%%%%%%%%%%%%%%%%%%%%%%%%%%%%%%%%%%%
\subsection{Proof of Lemma \ref{FormL}}\label{AppendixA2} Let $A_1 = [-1,1]^k$, $I_1 = [0,1]$ and define inductively for $n \geq 2$ the sets $A_n = [-n,n]^k \setminus A_{n-1}$ and $I_n = [0,n] \setminus I_{n-1}$. In particular, we see that $E = \mathbb{R}^k = \sqcup_{n \geq 1} A_n$ and $[0,\infty) = \sqcup_{n \geq 1} I_n$. From the Borel isomorphism theorem, see \cite[Theorem 3.3.13]{Srivastava}, we can find a bijection $\phi_n: A_n \rightarrow I_n$ such that $\phi_n$ and $\phi_n^{-1}$ are both measurable (with respect to the $\sigma$-algebras of Borel subsets of $A_n$ and $I_n$). The latter implies that the map
$\phi: \overline{E} \rightarrow [0, \infty]$, defined by 
$$\phi(x) = \begin{cases} \phi_n(x) &\mbox{ if } x \in A_n \\ \infty &\mbox{ if } x = \partial, \end{cases}$$
defines a bijection between $\overline{E}$ and $[0, \infty]$, such that $\phi$ and $\phi^{-1}$ are both measurable. Note that $\phi(E) = [0, \infty)$ and $\phi([-n,n]^k) = [0,n]$ for all $n \geq 1$.

Let $M'$ be the push-forward measure of $M$ under $\phi$. I.e. for each Borel $A \subset [0, \infty)$ and $\omega \in \Omega$
$$M'(\omega, A) = M (\omega, \phi^{-1}(A)).$$
Since $M$ is a point process, we conclude the same for $M'$. From \cite[Chapter 6, Proposition 1.11]{Cinlar} and its proof we see that for $A \subset [0, \infty)$ we have
$$M'(\omega,A) = \sum_{n \geq 1} {\bf 1}\{T_n(\omega) \in A \},$$
where the (extended) random variables $T_n$ are given by 
$$T_n(\omega) = \inf\{ t \geq 0: M'(\omega, [0,t]) \geq n \},$$
with the usual convention that $\inf \emptyset = \infty$. The random elements $X_n(\omega):= \phi^{-1}(T_n(\omega))$ now clearly satisfy the conditions of the lemma.  

%%%%%%%%%%%%%%%%%%%%%%%%%%%%%%%%%%%%%%%%%%%%%%%%%%%%%%%%%%%%%%%%%%%%%
%
%    Appendix A.3
%
%%%%%%%%%%%%%%%%%%%%%%%%%%%%%%%%%%%%%%%%%%%%%%%%%%%%%%%%%%%%%%%%%%%%%
\subsection{Proof of Proposition \ref{PPConv}}\label{AppendixA3} For clarity we split the proof into four steps.\\

{\bf \raggedleft Step 1.} In this step we construct candidates for $M^{\infty}$ and $T^{\infty}$ as in the statement of the proposition. We will establish that they have the required properties in the next steps. 

Let $F$ be a bounded set in $E$. Then, we can find $B \in I_{T}$ such that $F \subseteq B$. From Chebyshev's inequality we conclude 
$$\lim_{r \rightarrow \infty} \sup_{N \geq 1} \mathbb{P}\left(M^N(F) \geq r \right) 
 \leq \lim_{r \rightarrow \infty} \sup_{N \geq 1} \mathbb{P}\left(M^N(B) \geq r \right) \leq \lim_{r \rightarrow \infty} \limsup_{N \rightarrow \infty} r^{-1} \cdot \mathbb{E}\left[M^N(B)  \right]= 0,$$
where in the last equality we used that $\limsup_{N \rightarrow \infty} \mathbb{E}[M^N(B)] < \infty$ as follows from (\ref{RM11}). From \cite[Theorem 4.10]{Kall} we conclude that $\{ M^N \}_{N \geq 1}$ is a tight sequence in $S$.

Let $M'$ be any subsequential limit of $M^N$, and let $\{a_w\}_{w \geq 1}$ be a sequence such that $M^{a_w} \Rightarrow M'$. By Skorohod's Representation Theorem, see \cite[Theorem 6.7]{Billing} we may assume that $M^{a_w}$ and $M'$ are all defined on the same probability space $(\Omega, \mathcal{F}, \mathbb{P})$ and for each $\omega \in \Omega$ we have $M^{a_w}(\omega) \xrightarrow{v} M'(\omega)$. We recall that $\xrightarrow{v}$ denotes convergence in the vague topology, and that the application of \cite[Theorem 6.7]{Billing} is allowed since $S$ is a Polish space. The random measure $M'$ is our candidate for $M^{\infty}$.\\

If $F \subset E$ is a bounded Borel set, $B \in \mathcal{I}_T$ is such that $F \subseteq B^{\circ}$ and $\mu'$ denotes the mean measure of $M'$ we have from \cite[Lemma 4.1]{Kall} and Fatou's lemma
$$\mu'(F) = \mathbb{E}[M'(F)] \leq  \mathbb{E}[M'(B^{\circ})] \leq \mathbb{E}\left[\liminf_{w \rightarrow \infty} M^{a_w}(B)\right] \leq \liminf_{w \rightarrow \infty}\mathbb{E}[M^{a_w}(B)] = C(1;B), $$
where in the last equality we used (\ref{RM11}). We conclude that $\mu'$ is locally finite.

For $i = 1, \dots, k$ we let $\pi_i: E \rightarrow \mathbb{R}$ be such that
$$\pi_i(x) = x_i \mbox{ for } x = (x_1, \dots, x_k) \in E,$$
and also
$$H_i^n = \left\{y \in \mathbb{R}: \mu'\left(\pi_i^{-1}(y) \cap [-n, n ]^k\right) > 1/n \right\}.$$
Since $\mu'([-n,n]^k)$ is finite for each $n$, we conclude that $H_i^n$ is finite for each $i = 1, \dots, k$ and $n \in \mathbb{N}$. We define 
$$T' = T \cup \cup_{i = 1}^k \cup_{n \geq 1} H_i^n,$$
which is a countable subset of $\mathbb{R}$. This is our candidate for $T^{\infty}$. We mention that if $t \not \in T'$, we have for each $n \in \mathbb{N}$ and $i = 1, \dots, k$ 
$$\mu'\left(\pi_i^{-1}(t) \cap [-n, n ]^k\right) \leq 1/n,$$
which implies that
\begin{equation}\label{AP0}
\mu'\left( \pi_i^{-1}(t) \right) = 0.
\end{equation}

{\bf \raggedleft Step 2.} Let $A_1, \dots, A_m \in \mathcal{I}_{T'}$ be pairwise disjoint, where $T'$ is as in Step 1. In this step we show that for each $n_1, \dots, n_m \in \mathbb{N}$
\begin{equation}\label{AP1}
\mathbb{E} \left[ \prod_{i = 1}^m \prod_{j = 1}^{n_i} (M'(A_i) - j + 1) \right] = C(n_1, \dots, n_m; A_1, \dots, A_m).
\end{equation}

From (\ref{RM11}) we have
\begin{equation}\label{AP2}
\lim_{w \rightarrow \infty} \mathbb{E} \left[ \prod_{i = 1}^m \prod_{j = 1}^{n_i} (M^{a_w}(A_i) - j + 1) \right] = C(n_1, \dots, n_m; A_1, \dots, A_m).
\end{equation}
Let $A = [c_1, d_1) \times \cdots [c_k, d_k) \in \mathcal{I}_{T'}$ so that $c_1, \dots, c_k, d_1, \dots, d_k \not \in T'$. From (\ref{AP0}) we conclude that 
\begin{equation}\label{AP3}
\mu'\left(\partial A \right) \leq \mu' \left( \cup_{i = 1}^k  \pi_i^{-1}(c_i) \cup  \cup_{i = 1}^k \pi_i^{-1}(d_i)  \right) = 0.
\end{equation}
The latter shows that 
\begin{equation}\label{AP31}
M'(\partial A_i) = 0 \mbox{ a.s. for $i = 1,  \dots, k $},
\end{equation}
which by \cite[Lemma 4.1]{Kall} and the fact that $M^{a_w} \xrightarrow{v} M'$ for each $\omega \in \Omega$ implies
\begin{equation}\label{AP32}
\lim_{w \rightarrow \infty} M^{a_w}(A_i) = M'(A_i) \mbox{ a.s. for $i = 1,  \dots, k $}.
\end{equation}
We conclude that we have the almost sure convergence
\begin{equation}\label{AP3}
\lim_{w \rightarrow \infty}  \prod_{i = 1}^m \prod_{j = 1}^{n_i} (M^{a_w}(A_i) - j + 1)  = \prod_{i = 1}^m \prod_{j = 1}^{n_i} (M'(A_i) - j + 1) .
\end{equation}
Equations (\ref{AP2}) and (\ref{AP3}) would imply (\ref{AP1}) if we show that the sequence $\prod_{i = 1}^m \prod_{j = 1}^{n_i} (M^{a_w}(A_i) - j + 1) $ is uniformly integrable, cf. \cite[Theorem 3.5]{Billing}.

Using the trivial inequalities 
$$0 \leq \prod_{i = 1}^m \prod_{j = 1}^{n_i} (M^{a_w}(A_i) - j + 1)  \leq \prod_{i = 1}^m (M^{a_w}(A_i))^{n_i} \leq \sum_{i = 1}^m (M^{a_w}(A_i))^{h},$$
where $h = n_1 + \cdots +n_m$, we see that it suffices to prove for each $i = 1, \dots, k$ that 
\begin{equation}\label{AP4}
\limsup_{w \rightarrow \infty} \mathbb{E} \left[ (M^{a_w}(A_i))^{2h} \right] < \infty,
\end{equation}
see \cite[(3.18)]{Billing}. Consider now the polynomials $P_0(x) = 1$ and for $r \geq 1$
$$P_r(x) = x (x-1) \cdots (x-r+1).$$
By opening the brackets we have for each $r$ that 
$$P_r(x) = x^r + \sum_{i = 0}^{r-1} \alpha_{i,r} x^i.$$
The latter shows that if $\vec{u}_r = [P_0(x), \dots,  P_r(x)]^T$ and $\vec{v}_r = [1, x,  \dots,  x^r]^T$, then 
$$\vec{u}_r = L_r \vec{v}_r,$$
where $L_r$ is an upper-triangular matrix with all entries on the diagonal being equal to $1$. This means that $L_r^{-1}$ is also upper-triangular and has all $1$'s on the diagonal. In particular, $x^r$ is a finite linear combination of $P_0, \dots, P_r$ for all $r \in \mathbb{Z}_{\geq 0}$. From (\ref{AP2}) we know that $\lim_{w \rightarrow \infty} \mathbb{E}[P_r(M^{a_w}(A_i))]$ exists and is finite for all $r$ and $i = 1, \dots, k$, and so we conclude by linearity that the same is true for $\lim_{w \rightarrow \infty} \mathbb{E}[(M^{a_w}(A_i))^r]$, which clearly implies (\ref{AP4}). \\

{\bf \raggedleft Step 3.} From \cite[Lemma 1.9]{Kall} we can find a countable subclass $\tilde{\mathcal{I}} \subset \mathcal{I}_{T'}$ such that $\tilde{\mathcal{I}}$ is also a dissecting semi-ring. From (\ref{AP31}) and (\ref{AP32}) we can find an event $U \in \mathcal{F}$ such that $\mathbb{P}(U) = 1$ and for all $\omega \in U$ and $A \in \tilde{\mathcal{I}}$ we have 
\begin{equation}\label{AP5}
\lim_{w \rightarrow \infty} M^{a_w}(\omega, A) = M'(\omega, A)
\end{equation}
We now let $T^{\infty} = T'$ as in Step 1, and $M^{\infty}$ be the random measure that equals $M'$ on $U$ and $M^{\infty}$ is the zero measure on $U^c$. The goal of this step is to show that $M^{\infty}$ and $T^{\infty}$ satisfy the conditions of the proposition. Essentially, what remains is to show that $M^{\infty}$ is a point process. Indeed, if the latter is true, then since $M^{\infty}$ is a.s. equal to $M'$ we see that (\ref{AP1}) implies (\ref{RM12}). In addition, by construction $T^{\infty}$ is countable.

Fix an $\omega \in \Omega$, and let as usual $M^{\infty}_\omega = M^{\infty}(\omega, \cdot)$ be the measure on $E$. We seek to show that for each bounded Borel set $B \subset E$ we have 
\begin{equation}\label{AP6}
M^{\infty}_\omega(B) \in \mathbb{Z}_{\geq 0}.
\end{equation}
From (\ref{AP5}) and the fact that $M^{a_w}(\omega, A) \in \mathbb{Z}_{\geq 0}$ for all $A \in \tilde{\mathcal{I}}$ and $\omega \in \Omega$, we conclude $M'(\omega, A) \in \mathbb{Z}_{\geq 0}$ for all $\omega \in U$ and $A \in \tilde{\mathcal{I}}$. Since by definition we have
$$M^{\infty}(\omega, A) = \begin{cases} M'(\omega, A) & \mbox{ if } \omega \in U, \\ 0  & \mbox{ if } \omega \in U^c, \end{cases}$$
we conclude that (\ref{AP6}) holds if $B \in \tilde{\mathcal{I}}$.

From Lemma \ref{P1L1} for any $I_1, \dots, I_R \in \tilde{\mathcal{I}}$ we can write $\cup_{r = 1}^RI_r $ as a {\em finite disjoint} union of sets in $\tilde{\mathcal{I}}$, which by finite additivity of $M^{\infty}_\omega$ allows us to conclude (\ref{AP6}) when $B$ is a finite union of sets in $\tilde{\mathcal{I}}$. Suppose next that $G \subset E$ is a bounded open set. By the dissecting property we can find a sequence $I_1, I_2, \dots$ in $\tilde{\mathcal{I}}$ such that $G = \cup_{r \geq 1}I_r$. By the monotone convergence theorem 
$$M_{\omega}^{\infty}(G) = \lim_{R \rightarrow \infty}M_{\omega}^{\infty}\left( \cup_{r = 1}^RI_r \right),$$
and since the terms on the right are in $\mathbb{Z}_{\geq 0}$ and the limit is finite, we conclude (\ref{AP6}) holds when $B$ is a bounded open set. Finally, if $B$ is any bounded Borel set, we have that 
\begin{equation}\label{AP7}
M_{\omega}^{\infty}(B) = \inf_{B \subset G} M_{\omega}^{\infty}(G),
\end{equation}
where the infimum is over bounded open sets $G$ that contain $B$. As the elements on the right of (\ref{AP7}) are in $\mathbb{Z}_{\geq 0}$, we conclude that (\ref{AP6}) holds for all bounded Borel sets. We mention that the equality in (\ref{AP7}) can be deduced by a straightforward adaptation of \cite[Theorem 1.1]{Billing}.\\

{\bf \raggedleft Step 4.} In this final step we show that $M^N$ converge weakly under the additional assumption that (\ref{RM13}) holds. From Step 1 we know that $\{M^N\}_{N \geq 1}$ is a tight sequence in $S$. Consequently, we only need to show that if $M''$ is any subsequential limit of $\{M^N\}_{N \geq 1}$, then $M''$ has the same distribution as $M^{\infty}$ from Step 3 as random elements in $S$.

Let $M''$ be any subsequential limit of $\{M^N\}_{N \geq 1}$. Our work from Steps 1,2 and 3 shows that $M''$ has a modification (which we continue to call $M''$) that is a point process, and there exists a countable set $T'' \subset \mathbb{R}$, containing $T$, such that (\ref{AP1}) holds for $M'$ replaced with $M''$ and pairwise disjoint $A_1, \dots, A_m \in \mathcal{I}_{T''}$.

In particular, we conclude that for any pairwise disjoint $A_1, \dots, A_m \in \mathcal{I}_{T^{\infty} \cup T''}$ we have 
\begin{equation}\label{AP8}
\mathbb{E} \left[ \prod_{i = 1}^m \frac{M^{\infty}(A_i)!}{(M^{\infty}(A_i) - n_i)!} \right] = \mathbb{E} \left[ \prod_{i = 1}^m \frac{M''(A_i)!}{(M''(A_i) - n_i)!} \right] = C(n_1, \dots, n_m; A_1, \dots, A_m).
\end{equation}
Fix $B \in \mathcal{I}_{T^{\infty} \cup T''}$ sufficiently large so that $\cup_{i = 1}^m A_i \subseteq B$. From (\ref{RM17}) and (\ref{AP8}) we get
\begin{equation}\label{AP9}
C(n_1, \dots, n_m; A_1, \dots, A_m) \leq \mathbb{E} \left[ \prod_{i = 1}^m \frac{M^{\infty}(B)!}{(M^{\infty}(B) - n)!} \right] = C(n; B).
\end{equation}
Equations (\ref{AP8}), (\ref{AP9}) and the summability condition in (\ref{RM13}) together imply that the conditions of Lemma \ref{MomentProblem} are satisfied for the random vectors $(X_1, \dots, X_m) = (M^{\infty}(A_1), \dots, M^{\infty}(A_m))$, $(Y_1, \dots, Y_m) = (M''(A_1), \dots, M''(A_m))$, $c_n = C(n; B)$ and $a = \epsilon$ as in (\ref{RM13}). We conclude that $(M^{\infty}(A_1), \dots, M^{\infty}(A_m))$ has the same distribution as $(M''(A_1), \dots, M''(A_m))$.

Suppose now that $A_1, \dots, A_n \in \mathcal{I}_{T^{\infty} \cup T''}$ are not necessarily pairwise disjoint, and let $B_1, \dots, B_m \in \mathcal{I}_{T^{\infty} \cup T''}$ be as in Lemma \ref{P1L1}. Our work in the previous paragraph shows that the random vector $(M^{\infty}(B_1), \dots, M^{\infty}(B_m))$ has the same distribution as $(M''(B_1), \dots, M''(B_m))$. Since $A_i$'s are disjoint unions of $B_j$'s we know by finite additivity that the random vectors $(M^{\infty}(A_1), \dots, M^{\infty}(A_n))$ and $(M''(A_1), \dots, M''(A_n))$ can be obtained by applying the same linear transformation to the random vectors $(M^{\infty}(B_1), \dots, M^{\infty}(B_m))$ and $(M''(B_1), \dots, M''(B_m))$, respectively. As linear mappings are continuous we conclude that $(M^{\infty}(A_1), \dots, M^{\infty}(A_n))$ and $(M''(A_1), \dots, M''(A_n))$ have the same distribution. The last statement, combined with \cite[Lemma 4.7]{Kall} and the $\pi-\lambda$ theorem imply that $M^{\infty}$ and $M''$ have the same distribution as random elements in $S$. 

%%%%%%%%%%%%%%%%%%%%%%%%%%%%%%%%%%%%%%%%%%%%%%%%%%%%%%%%%%%%%%%%%%%%%
%
%    Appendix A.4
%
%%%%%%%%%%%%%%%%%%%%%%%%%%%%%%%%%%%%%%%%%%%%%%%%%%%%%%%%%%%%%%%%%%%%%
\subsection{Proof of Lemma \ref{Cons}}\label{AppendixA4} Let $I_1, \dots, I_m \in \mathcal{I}$ be pairwise disjoint. From (\ref{RM3}), (\ref{RN3}), the symmetry of $\mu_n$ and $u_n$ we have
$$\mu_n( B_1 \times \cdots \times B_n) = \mathbb{E} \left[ \prod_{j = 1}^m \frac{M(I_j)!}{(M(I_j) - n_j)!} \right] = \int_{B_1 \times \cdots \times B_n}  u_n(x_1 ,\dots, x_n) \lambda^n(dx) < \infty,$$
where $B_1, \dots, B_n$ is any permutation of $n_1$ copies of $I_1$, $n_2$ copies of $I_2$ etc. From Lemma \ref{DSRLemma} any $A \in \mathcal{I}_n$ can be written as a finite disjoint union of sets of the form $B_1 \times \cdots \times B_n$ as above. By linearity we conclude 
\begin{equation}\label{asd2}
\mu_n(A) = \int_{A}   u_n(x_1, \dots, x_n) \lambda^n(dx) 
\end{equation}
for all $A \in \mathcal{I}_n$. Let $A_1, \dots, A_r \in \mathcal{I}_n$. From Lemma \ref{P1L1} we can find finitely many $B_1, \dots, B_m \in \mathcal{I}_n$ that are pairwise disjoint and 
$$\cup_{i = 1}^r A_i = \cup_{j = 1}^m B_j.$$
By additivity and (\ref{asd2}) applied to $B_j$'s we conclude for $B = \cup_{i = 1}^r A_i = \cup_{j = 1}^m B_j$ 
$$\mu_n(B) = \sum_{j = 1}^m \mu_n(B_j) = \sum_{j = 1}^m  \int_{B_j}   u_n(x_1, \dots, x_n) \lambda^n(dx)  =  \int_{B}   u_n(x_1, \dots, x_n) \lambda^n(dx).$$
This shows that (\ref{asd2}) holds for {\em finite} unions of elements in $\mathcal{I}_n$. Fix now any open set $G \subset E^n$. Since $\mathcal{I}_n$ is dissecting we can find $I_m \in \mathcal{I}_n$ such that $G = \cup_{m \geq 1}I_m$. This shows that 
\begin{equation}\label{fg}
\begin{split}
&\mu_n(G) = \mu_n\left( \cup_{m \geq 1} I_m \right) = \lim_{N \rightarrow \infty} \mu_n \left( \cup_{m = 1}^N I_m \right) = \lim_{N \rightarrow \infty}\int_{ \cup_{m = 1}^N I_m }   u_n(x_1, \dots, x_n) \lambda^n(dx)  \\
& = \int_{ \cup_{m = 1}^\infty I_m }   u_n(x_1, \dots, x_n) \lambda^n(dx) = \int_{G}   u_n(x_1, \dots, x_n) \lambda^n(dx),
\end{split}
\end{equation}
where in the second equality on the first line we used the monotone convergence theorem, in the third equality we used that (\ref{asd2}) holds for finite unions of elements in $\mathcal{I}_n$, and in the first equality on the second line we used that monotone convergence theorem again. 

We now fix a bounded open set $G$ in $E^n$ and consider the collection $\mathcal{A}_G$ of subsets of $G$ that satisfy (\ref{asd2}). One readily observes that $\mathcal{A}_G$ is a $\lambda$-system, which contains the $\pi$-system of sets in $\mathcal{I}_n$ that are subsets of $G$. By the $\pi -\lambda$ theorem we conclude that (\ref{asd2} holds for any Borel subset of $G$ and since $G$ was arbitrary we see that (\ref{asd2}) holds for all bounded Borel sets. Finally, if $A$ is any Borel set, then $A = \cup_{ m \geq 1} A_m$ where $A_m$ are bounded and Borel. One can repeat (\ref{fg}) with $G$ replaced with $A$ and $I_m$ replaced with $A_m$ to conclude that (\ref{asd2}) holds for $A$ as well. This proves that (\ref{asd2}) holds for any Borel set, which completes the proof.

%%%%%%%%%%%%%%%%%%%%%%%%%%%%%%%%%%%%%%%%%%%%%%%%%%%%%%%%%%%%%%%%%%%%%
%
%    Appendix A.45
%
%%%%%%%%%%%%%%%%%%%%%%%%%%%%%%%%%%%%%%%%%%%%%%%%%%%%%%%%%%%%%%%%%%%%%
\subsection{Proof of Proposition \ref{PropWC0}}\label{AppendixA45} For clarity we split the proof into two steps.\\

{\bf \raggedleft Step 1.} Since $\lambda$ is the vague limit of $\lambda_n \in S$, we conclude that $\lambda \in S$ too, i.e. $\lambda$ is locally bounded. For $i = 1, \dots, k$ we let $\pi_i: E \rightarrow \mathbb{R}$ be such that
$$\pi_i(x) = x_i \mbox{ for } x = (x_1, \dots, x_k) \in E,$$
and also
$$H_i^n = \left\{y \in \mathbb{R}: \lambda\left(\pi_i^{-1}(y) \cap [-n, n ]^k\right) > 1/n \right\}.$$
Since $\lambda([-n,n]^k) < \infty$, we conclude that $H_i^n$ is finite for each $i = 1, \dots, k$ and $n \in \mathbb{N}$. We define $T = \cup_{i = 1}^k \cup_{n \geq 1} H_i^n,$
which is a countable subset of $\mathbb{R}$. We mention that if $t \not \in T$, we have for each $n \in \mathbb{N}$ and $i = 1, \dots, k$ 
$$\lambda\left(\pi_i^{-1}(t) \cap [-n, n ]^k\right) \leq 1/n,$$
which implies that
\begin{equation}\label{A45AP0}
\lambda\left( \pi_i^{-1}(t) \right) = 0.
\end{equation}
Fix pairwise disjoint $A_1, \dots, A_m \in \mathcal{I}_T$, as in (\ref{rect}), and $n_1, \dots, n_m \in \mathbb{N}$ with $n = n_1 + \cdots + n_m$. We claim that 
\begin{equation}\label{A45BP1}
\lim_{N \rightarrow \infty}\mathbb{E} \left[ \prod_{i = 1}^m \frac{M^N(A_i)!}{(M^N(A_i) - n_i)!}\right] = \int_{A_1^{n_1} \times \cdots \times A_m^{n_m}} \det \left[ K(x_i,x_j) \right]_{i,j = 1}^n \lambda^n(dx).
\end{equation}
We prove (\ref{A45BP1}) in the second step. Here, we assume its validity and conclude the proof of the proposition.\\

From Hadamard's inequality (\ref{Hadamard}) and the local boundedness of $K$ and $\lambda$, for each $B \in \mathcal{I}_{T}$ we can find a constant $C \in (0, \infty)$ such that 
$$\left| \int_{B^n} \det \left[ K(x_i, x_j) \right]_{i,j = 1}^n \lambda^n(dx) \right| \leq C^n \cdot n^{n/2}.$$
The last inequality shows that (\ref{RM13}) holds for any $\epsilon > 0$ and so by Proposition \ref{PPConv} we conclude that there exists a countable $T^{\infty} \subset \mathbb{R}$ such that $T \subseteq T^{\infty}$, and a point process $M^{\infty}$ such that $M^N$ converge weakly to $M^{\infty}$ and for all pairwise disjoint $A_1, \dots, A_m \in \mathcal{I}_{T^{\infty}}$
\begin{equation}\label{A45BP2}
\mathbb{E} \left[ \prod_{i = 1}^m \frac{M^{\infty}(A_i)!}{(M^{\infty}(A_i) - n_i)!}\right]= \int_{A_1^{n_1} \times \cdots \times A_m^{n_m}} \det \left[ K(x_i, x_j) \right]_{i,j = 1}^n \lambda^n(dx).
\end{equation}
From Lemma \ref{Cons} and Definition \ref{DPP} we conclude that $M^{\infty}$ satisfies the conditions of the proposition.\\

{\bf \raggedleft Step 2.} We know that $M^N$ is determinantal with correlation kernel $K_N$ and reference measure $\lambda_N$. Using the latter and (\ref{RN2}) we see that to show (\ref{A45BP1}) it suffices to prove 
\begin{equation}\label{A45BP3}
\lim_{N \rightarrow \infty}\int_{A_1^{n_1} \times \cdots \times A_m^{n_m}} \det \left[ K_N( x_i, x_j) \right]_{i,j = 1}^n \lambda_N^n(dx) = \int_{A_1^{n_1} \times \cdots \times A_m^{n_m}} \det \left[ K(x_i,x_j) \right]_{i,j = 1}^n \lambda^n(dx).
\end{equation}
Since $K_N(x,y)$ converge uniformly on $A_1^{n_1} \times \cdots \times A_m^{n_m}$ to $K(x,y)$, while by \cite[Lemma 4.1]{Kall}
$$\limsup_{N \rightarrow \infty}\lambda_N^n(A_1^{n_1} \times \cdots \times A_m^{n_m}) \leq \prod_{i = 1}^m \lambda(\overline{A_i})^{n_i} < \infty ,$$
we conclude that to prove (\ref{A45BP3}) it suffices to show that 
\begin{equation}\label{A45BP4}
\lim_{N \rightarrow \infty}\int_{A_1^{n_1} \times \cdots \times A_m^{n_m}} \det \left[ K(x_i,x_j) \right]_{i,j = 1}^n \lambda_N^n(dx) = \int_{A_1^{n_1} \times \cdots \times A_m^{n_m}} \det \left[ K(x_i,x_j) \right]_{i,j = 1}^n \lambda^n(dx).
\end{equation}
Equation (\ref{A45BP4}) would follow from the continuity of $\det \left[ K(x_i, x_j) \right]_{i,j = 1}^n$ (which we have by the assumptions in the proposition), and Lemma \ref{VaguePairLim}, provided we can show 
\begin{equation}\label{A45BP5}
\lambda_N^n \xrightarrow{v} \lambda^n \mbox{ and } \lambda^n\left(\partial (A_1^{n_1} \times \cdots \times A_m^{n_m}) \right) = 0.
\end{equation}
In the remainder of the step we establish (\ref{A45BP5}).

Fix $B \in \mathcal{I}_T$. Writing $B = [a_1, b_1) \times \cdots [a_k, b_k)$, we see that 
$$\partial B \subset \cup_{i = 1}^k \left(\pi_i^{-1}(a_i) \cup \pi_i^{-1}(b_i) \right).$$
By assumption we know $a_i \not \in T$, $b_i \not \in T$, which from (\ref{A45AP0}) shows $\lambda(\partial B) = 0$. From \cite[Lemma 4.1]{Kall} we conclude for each $B \in \mathcal{I}_T$
$$\lim_{N \rightarrow \infty} \lambda_N(B) = \lambda(B),$$
which implies for each $B_1, \dots, B_n \in \mathcal{I}_T$ that 
$$\lim_{N \rightarrow \infty} \lambda^n_N(B_1 \times \cdots \times B_n) = \lambda(B_1 \times \cdots \times B_n).$$
From Lemma \ref{DSRLemma} we know that the family $\{B_1 \times \cdots \times B_n: B_i \in \mathcal{I}_{T}\}$ is a dissecting semi-ring and so applying \cite[Lemma 4.1]{Kall} a second time gives the left side of (\ref{A45BP5}).

If we now set $I = B_1 \times \cdots \times B_n$ with $B_i \in \mathcal{I}_{T}$, we observe 
$$\partial I \subset \cup_{i = 1}^n E^{i-1} \times \partial B_i \times E^{n-i},$$
which shows $\lambda^n(\partial I) = 0$. The latter implies the right side of (\ref{A45BP5}).

%%%%%%%%%%%%%%%%%%%%%%%%%%%%%%%%%%%%%%%%%%%%%%%%%%%%%%%%%%%%%%%%%%%%%
%
%    Appendix A.5
%
%%%%%%%%%%%%%%%%%%%%%%%%%%%%%%%%%%%%%%%%%%%%%%%%%%%%%%%%%%%%%%%%%%%%%
\subsection{Proof of Proposition \ref{PropWC1}}\label{AppendixA5} For clarity we split the proof into two steps.\\

{\bf \raggedleft Step 1.} Let $T = \{y \in \mathbb{R}: \nu_t(\{y\}) > 0 \mbox{ for some } t \in \mathcal{T} \}$. Since $\nu_t$ are locally finite, we know that $T$ is countable. Fix pairwise disjoint $A_1, \dots, A_m \in \mathcal{I}_T$, as in (\ref{rect}), and $n_1, \dots, n_m \in \mathbb{N}$ with $n = n_1 + \cdots + n_m$. We claim that 
\begin{equation}\label{BP1}
\lim_{N \rightarrow \infty}\mathbb{E} \left[ \prod_{i = 1}^m \frac{M^N(A_i)!}{(M^N(A_i) - n_i)!}\right] = \int_{A_1^{n_1} \times \cdots \times A_m^{n_m}} \det \left[ K(s_i, x_i; s_j, x_j) \right]_{i,j = 1}^n (\mu_{\mathcal{T}, \nu})^n(dsdx).
\end{equation}
We prove (\ref{BP1}) in the second step. Here, we assume its validity and conclude the proof of the proposition.\\

From Hadamard's inequality (\ref{Hadamard}) and the local boundedness of $K$ and $\nu_t$, for each $B \in \mathcal{I}_{T}$  we can find a constant $C \in (0, \infty)$ such that 
$$\left| \int_{B^n} \det \left[ K(s_i, x_i; s_j, x_j) \right]_{i,j = 1}^n (\mu_{\mathcal{T} ,\nu})^n(dsdx) \right| \leq C^n \cdot n^{n/2}.$$
The last inequality shows that (\ref{RM13}) holds for any $\epsilon > 0$ and so by Proposition \ref{PPConv} we conclude that there exists a countable $T^{\infty} \subset \mathbb{R}$ such that $T \subseteq T^{\infty}$, and a point process $M^{\infty}$ such that $M^N$ converge weakly to $M^{\infty}$ and for all pairwise disjoint $A_1, \dots, A_m \in \mathcal{I}_{T^{\infty}}$
\begin{equation}\label{BP2}
\mathbb{E} \left[ \prod_{i = 1}^m \frac{M^{\infty}(A_i)!}{(M^{\infty}(A_i) - n_i)!}\right]= \int_{A_1^{n_1} \times \cdots \times A_m^{n_m}} \det \left[ K(s_i, x_i; s_j, x_j) \right]_{i,j = 1}^n (\mu_{\mathcal{T} ,\nu} )^n(dsdx).
\end{equation}
From Lemma \ref{Cons} and Definition \ref{DPP} we conclude that $M^{\infty}$ satisfies the conditions of the proposition.\\

{\bf \raggedleft Step 2.} We know that $M^N$ is determinantal with correlation kernel $K_N$ and reference measure $\mu_{\mathcal{T},\nu^N}$. Using the latter and (\ref{RN2}) we see that to show (\ref{BP1}) it suffices to prove that for each $(s_1, \dots, s_n) \in \mathcal{T}^n$ we have 
\begin{equation}\label{BP3}
\lim_{N \rightarrow \infty}\int_{I} \det \left[ K_N(s_i, x_i; s_j, x_j) \right]_{i,j = 1}^n \prod_{i = 1}^n\nu^N_{s_i}(dx_i) = \int_{I} \det \left[ K(s_i, x_i; s_j, x_j) \right]_{i,j = 1}^n \prod_{i = 1}^n\nu_{s_i}(dx_i),
\end{equation}
where $I = (a_1, b_1] \times \cdots \times (a_n, b_n]$ with $a_i, b_i \not \in T$ for each $i = 1, \dots, n$. Since $K_N(s,x; t,y)$ converge uniformly on $I$ to $K(s,x; t,y)$, while by \cite[Lemma 4.1]{Kall}
$$\limsup_{N \rightarrow \infty}\prod_{i = 1}^n \nu^N_{s_i}((a_i, b_i])  \leq \prod_{i = 1}^n \nu_{s_i}([a_i, b_i]) < \infty ,$$
we conclude that to prove (\ref{BP3}) it suffices to show that 
\begin{equation}\label{BP4}
\lim_{N \rightarrow \infty}\int_{I} \det \left[ K(s_i, x_i; s_j, x_j) \right]_{i,j = 1}^n  \prod_{i = 1}^n\nu^N_{s_i}(dx_i)  = \int_{I} \det \left[ K(s_i, x_i; s_j, x_j) \right]_{i,j = 1}^n  \prod_{i = 1}^n\nu_{s_i}(dx_i) .
\end{equation}
Equation (\ref{BP4}) would follow from the continuity of $\det \left[ K(s_i, x_i; s_j, x_j) \right]_{i,j = 1}^n$ (which we have by the assumptions in the proposition), and Lemma \ref{VaguePairLim}, provided we can show 
\begin{equation}\label{BP5}
 \nu^N_{s_1} \times \cdots \times  \nu^N_{s_n}  \xrightarrow{v}  \nu_{s_1} \times \cdots \times  \nu_{s_n}  \mbox{ and } ( \nu_{s_1} \times \cdots \times  \nu_{s_n} ) (\partial I) = 0.
\end{equation}
In the remainder of the step we establish (\ref{BP5}).

Note that
$$\partial I \subset \cup_{i = 1}^n \mathbb{R}^{i-1} \times \{a_i, b_i\} \times \mathbb{R}^{n-i},$$
and since $a_i, b_i \not \in T$ we have $\nu_{s_i}(\{a_i\}) = \nu_{s_i}(\{b_i\}) = 0$. The latter implies the right side of (\ref{BP5}). Since $\nu^N_{s_i} \xrightarrow{v} \nu_{s_i}$, and $a_i, b_i \not \in T$, we have from \cite[Lemma 4.1]{Kall} that
$$\lim_{N \rightarrow \infty} \nu^N_{s_i}((a_i, b_i]) = \nu_{s_i}((a_i, b_i]) \mbox{ for } i = 1, \dots, n,$$
which implies
$$\lim_{N \rightarrow \infty} (\nu^N_{s_1} \times \cdots \times  \nu^N_{s_n} )((a_1, b_1] \times \cdots \times (a_n, b_n]) = (\nu_{s_1} \times \cdots \times  \nu_{s_n} )((a_1, b_1] \times \cdots \times (a_n, b_n]).$$
Applying \cite[Lemma 4.1]{Kall} a second time gives the left side of (\ref{BP5}).
\end{appendix}

\bibliographystyle{alpha}
\bibliography{PD}

\end{document}